\documentclass[reqno]{amsart}
\usepackage[T1]{fontenc}
\usepackage{lmodern}
\usepackage{amssymb,mathtools,booktabs,tabularx,graphicx,xcolor,multirow,array,textcomp}
\numberwithin{equation}{section}
\usepackage[expansion=false]{microtype}
\usepackage{enumitem,mathrsfs,xurl,placeins}
\usepackage[hidelinks,hypertexnames=false]{hyperref}
\usepackage{cleveref}

\theoremstyle{plain}
\newtheorem{theorem}{Theorem}[section]
\newtheorem{lemma}[theorem]{Lemma}
\newtheorem{proposition}[theorem]{Proposition}
\newtheorem{corollary}[theorem]{Corollary}
\newtheorem{assumption}[theorem]{Assumption}
\theoremstyle{definition}

\theoremstyle{remark}
\newtheorem{remark}[theorem]{Remark}

\crefname{assumption}{assumption}{assumptions}
\Crefname{assumption}{Assumption}{Assumptions}

\newcommand{\E}{\mathcal{E}}
\newcommand{\V}{\mathcal{V}}

\newcommand{\F}{\mathfrak{F}}
\newcommand{\R}{\mathbb{R}}
\newcommand{\norm}[1]{\left\lVert#1\right\rVert}
\newcommand{\ip}[2]{\left\langle#1,#2\right\rangle}
\DeclareMathOperator*{\esssup}{ess\,sup}

\title[Neural--Spectral and Newton--HUM Methods]{Neural--Spectral and Newton--HUM Methods for Reconstructing Controls of Wave Equations}
\thanks{T.-P. Nguyen and M.-B. Tran are funded in part by the National Science Foundation (NSF) CAREER Award DMS-2303146 and NSF Grants DMS-2204795, DMS-2305523, and DMS-2306379.}
\author{Tan-Phuc Nguyen}
\address{Department of Mathematics, Texas A\&M University, 155 Ireland Street, College Station, TX 77843, USA}
\email{tanphuc@tamu.edu}
\author{Minh-Binh Tran}
\thanks{Corresponding author: Minh-Binh Tran.}
\address{Department of Mathematics, Texas A\&M University, 155 Ireland Street, College Station, TX 77843, USA}
\email{minhbinh@tamu.edu}
\author{Son Tu}
\address{Department of Mathematics, Baylor University, 1311 South 5th Street, Waco, TX 76706, USA}
\email{Son\_Tu@baylor.edu}
\subjclass[2020]{35L05, 65M70, 93B05, 68T07}
\keywords{semilinear wave control, Neural--Spectral method, Hilbert Uniqueness Method, Newton correction, validated numerics, spectral Galerkin method}

\begin{document}

\begin{abstract}
We develop a framework for constructing and rigorously validating controls for semilinear wave equations with localized interior control. We analyze the continuous energy-space problem and certify exact control for finite-dimensional Galerkin systems. Under explicit smallness and neighborhood conditions, our a posteriori theorem proves the existence of an exact Galerkin control near a computed candidate. It combines a validated terminal-state error, a bounded right inverse of the derivative of the control-to-terminal-state map, and a Lipschitz bound for that derivative. Under additional uniform right-inverse and neighborhood conditions, we also show that Newton correction converges with quadratic decay of terminal-state errors. We use residual estimates and interval arithmetic to enclose the nonlinear trajectory and its linearized controllability Gramian. For the specified eight-mode tanh problem, we rigorously bound a source-response operator to prove contraction of the exact source-frozen iteration. We compare deterministic controls based on the Hilbert Uniqueness Method with neural approximations of temporal modal coefficients. Two Newton corrections reduce the computed terminal-state errors of all six tested neural candidates below \(10^{-8}\). For the three certified candidates, we obtain smaller bounds on the distance to an exact Galerkin steering control by refining the validation mesh while keeping the computed controls fixed.
\end{abstract}

\maketitle

\section{Introduction}

The control theory of partial differential equations has been the subject of sustained and intensive investigation over the last several decades. For wave, heat, and dispersive equations, a large body of work has established exact, null, and approximate controllability via internal and boundary controls under a variety of geometric, analytic, and structural hypotheses; see, for instance, \cite{Coron2007, Lions1988, Zuazua2005, zuazua2006controllability} and the references therein. In the case of wave equations, the Geometric Control Condition (GCC) plays a central role: in multidimensional domains, exact controllability is closely tied to the requirement that every generalized ray of geometric optics enters the control region within the prescribed time interval~\cite{bardos1992sharp}. For semilinear wave equations, exact controllability has been established under various growth, sign, and geometric conditions on the nonlinearity; see, for instance, \cite{Zuazua1993, ZhouXuLei2010, DehmanLebeau2009, JolyLaurent2014}.

By contrast, the numerical reconstruction of controls remains much less developed. Even when controllability is established at the analytical level, it is often difficult to compute a reliable control that steers the system between prescribed states. Nonlinearities may also introduce loss of compactness, nonconvexity, stiffness, and instability in the numerical reconstruction procedure. Thus, while the existence of controls is well understood in many settings, the design of stable, accurate, and computationally tractable methods for approximating such controls remains a challenging and comparatively underdeveloped area.

Several recent works have begun to address this issue from a constructive and numerical perspective. Least-squares, space-time, and fixed-point approaches have been developed for partial differential equations, including distributed and boundary control problems; see, for example, \cite{BhandariLemoineMunch2023, BottoisLemoineMunch2023, LemoineMarinGayteMunch2021, LemoineMunch2023, MunchTrelat2022}. These works show that numerical control reconstruction is a subtle problem, even in low dimensions. For linear waves, the numerical approximation of HUM controls goes back to Glowinski, Li, and Lions~\cite{GlowinskiLiLions1990}; see also~\cite{GlowinskiLionsHe2008}. Space semi-discretizations of the wave equation may lose uniform observability because of spurious high-frequency numerical modes~\cite{InfanteZuazua1999, Zuazua2005}. This motivates filtering, Tychonoff regularization, and related methods~\cite{ErvedozaZuazua2013}, as well as space--time mixed formulations~\cite{CindeaMunch2015}. In higher dimensions, the difficulty is compounded by the cost of mesh-based discretizations and the need to resolve interactions among the control region, high-frequency modes, and the conservative structure of the equations.

The use of neural networks for control reconstruction is part of a broader connection between machine learning and control theory. Bensoussan and Tran~\cite{BensoussanTran2027} study this connection through approximation, optimization, and dynamic programming. For PDE control, neural networks can represent the unknown state and control, with their parameters chosen to reduce errors in the controlled equation and the prescribed initial, boundary, and terminal conditions.

In the work~\cite{KurbanovPhungTran2026}, the authors introduced WeightedPINN, an operator-decomposed physics-informed neural network (PINN) framework designed specifically for the numerical reconstruction of controls in nonlinear PDEs. The main idea of WeightedPINN is to assign separate adaptive weights to the different components of the controlled PDE residual, such as the time derivative, spatial differential operators, nonlinear response, and control term, as well as to the boundary, initial, and terminal constraints. This produces a control-aware residual metric that is better adapted to gradient imbalance at the operator level and to the mechanism through which the unknown control enters the equation. In this sense, WeightedPINN represents a viable attempt to adapt the PINN methodology to the control-reconstruction setting.

Our present paper develops two different approaches to reconstructing controls for semilinear wave equations: Neural--Spectral optimization and source-frozen iteration. We formulate both approaches on the same spectral Galerkin space, but they compute controls differently. The Neural--Spectral method jointly approximates the state and control by training neural networks for their temporal coefficients. The source-frozen method, instead, constructs a sequence of states and controls by repeatedly solving an affine steering problem. This common spatial formulation allows us to compare the controls produced by the two methods and to apply the same endpoint correction and verification procedures to them.

For the Neural--Spectral method, we expand the state and control in spatial eigenfunctions and represent the time-dependent coefficients using neural networks. The state basis satisfies the homogeneous Dirichlet boundary condition exactly. We train the networks to reduce the residual in the projected equation and the errors in the initial and terminal conditions. For the source-frozen method, we evaluate the nonlinear term along the current trajectory and treat it as a known source. We then compute a control that steers the resulting affine system between the prescribed states and use its trajectory in the next iteration. We provide two realizations of this affine solve: one through the Hilbert Uniqueness Method (HUM) and the other through a reusable controllability Gramian. With exact solves and the same quadratic control cost, these realizations produce the same trajectory update. Thus, the Neural--Spectral method obtains a control through optimization, whereas the source-frozen method obtains one through successive affine steering problems.

Neural spectral methods have also been studied for other objectives. Du, Chalapathi, and Krishnapriyan combine neural learning with classical spectral representations for PDE approximation \cite{DuChalapathiKrishnapriyan2024}. Feng, Chen, Huang, Liu, and Yan develop a neural adaptive spectral method for one-shot approximation of families of optimal-control operators \cite{FengChenHuangLiuYan2025}. In this paper, our approach is different. We seek a steering control for prescribed initial and terminal states, and several controls may achieve this transfer. Our analysis connects temporal neural approximations with HUM and Gramian reconstruction, endpoint correction, and separate forward integration of the recovered control.

For both reconstruction methods, we want to determine what the computed control achieves when applied to the system. We therefore study the endpoint map, which assigns to a control the terminal state of its Galerkin trajectory. Our main a posteriori theorem gives sufficient conditions for an exact steering control to exist near a computed control and provides an explicit bound on their distance. This result is of Newton--Kantorovich type (see, e.g.,~\cite{Deuflhard2011}), with the inverse of the derivative replaced by a bounded right inverse, and its proof uses the quadratic fixed-point radius estimate that also underlies the radii-polynomial approach of validated numerics~\cite{HungriaLessardMireles2016}. These conditions involve the terminal-state error, a bounded HUM right inverse of the endpoint derivative, and a Lipschitz bound for that derivative. The right inverse supplies a correction for the linearized problem, while the Lipschitz bound controls the remaining nonlinear error. Under the stated smallness and neighborhood conditions, these estimates prove the existence of an exact control. Under additional uniform right-inverse and neighborhood conditions, we also prove convergence of Newton--HUM correction with quadratic decay of the terminal-state errors.

To apply the existence theorem rigorously, we derive bounds for the errors in the computed nonlinear trajectory and the time-dependent controllability Gramian. We combine residual estimates with interval arithmetic and bounds for spatial quadrature and rounding errors. Together with an analytical estimate for the variation of the endpoint derivative, these bounds provide verified inputs to the theorem. We use them to certify exact Galerkin steering controls near selected neural and deterministic approximations. This verification draws on established methods of validated numerics \cite{HungriaLessardMireles2016, NakaoPlumWatanabe2019, Rump2010, Tucker2011}.

We also establish results specific to each reconstruction method. For source freezing, we derive sufficient conditions for convergence by estimating how a change in the frozen source affects the steering trajectory. We estimate the direct source response together with the control correction needed to preserve the endpoint conditions. This retains cancellation between the two effects that separate norm bounds would discard. We obtain contraction conditions for the continuous and Galerkin problems and rigorously verify the condition for the specified tanh model. For the Neural--Spectral method, we prove approximation results for an admissible state--control pair and derive residual estimates for the trajectory generated by the recovered control. These estimates do not require convergence of the training algorithm. At a fixed Galerkin dimension, we further consider Neural--Spectral controls converging in \(L^2\) to an exact steering control. If the endpoint derivative is locally Lipschitz and has a bounded right inverse at that exact control, then every term with a sufficiently large index satisfies the local existence and Newton--HUM convergence conditions. 

Our continuous theorems concern the PDE under the stated analytical assumptions, while the computer-assisted certificates establish exact steering for specified Galerkin systems. Extending these certificates to the full PDE would require additional bounds for the unresolved modes and verification of the corresponding infinite-dimensional hypotheses. The numerical experiments compare the two reconstruction approaches, examine endpoint correction, and assess the computed controls in Galerkin systems with more modes.

We organize the paper as follows. In \Cref{sec:model}, we introduce the continuous and Galerkin problems and establish the continuous source-frozen results. In \Cref{sec:aposteriori}, we prove the a posteriori existence and Newton--HUM convergence theorems. We derive the validated endpoint and source-response bounds in \Cref{sec:validated-endpoint,sec:response}, and develop the Neural--Spectral analysis in \Cref{sec:neural}. We present the algorithms and numerical comparisons in \Cref{sec:algorithms,sec:results}, and discuss the conclusions in \Cref{sec:conclusion}. The Supplementary Material contains additional proofs and details of the rigorous bounds.

\section{Controlled wave equation and spectral Galerkin setting}\label{sec:model}

Fix \(N\in\mathbb N\), let \(\Omega=(0,1)^N\), and let \(\omega\subset\Omega\) be a nonempty open subset. For a prescribed final time \(T>0\), we consider the following controlled wave equation with homogeneous Dirichlet boundary conditions:
\begin{equation}\label{eq:pde}
\begin{cases}
\partial_{tt}{u} - \Delta u + g(u) + \alpha u + f\,\mathbf{1}_{\omega} = 0
& \text{in } (0,T)\times \Omega, \\[1mm]
u(t,x) = 0
& \text{on } (0,T)\times \partial\Omega, \\[1mm]
u(0,\cdot) = u^0,\quad \partial_t{u}(0,\cdot) = u^1
& \text{in } \Omega.
\end{cases}
\end{equation}
Here, \(u=u(t,x)\) denotes the state variable, and \(f=f(t,x)\) is an interior control acting on the subdomain \(\omega\), with \(\mathbf{1}_{\omega}\) denoting the characteristic function of \(\omega\). The constant \(\alpha \ge 0\) is real, and \(g:\R\to\R\) represents a nonlinear term. We identify \(f\) with its zero extension outside \(\omega\). Thus, \(\mathbf1_\omega f\) denotes the applied control on \(\Omega\).

\paragraph{\bf Control problem.} For prescribed initial and target states
\((u^0,u^1),(u_T^0,u_T^1) \in H_0^1(\Omega)\times L^2(\Omega),\)
we seek a control \(f\in L^2((0,T)\times\omega)\) whose corresponding solution satisfies \(u(T,\cdot)=u_T^0\) and \(\partial_tu(T,\cdot)=u_T^1\) in \(\Omega\).

We use \(L^2(\Omega)\) for the space of square-integrable functions and \(H_0^1(\Omega)\) for the Sobolev space of functions with square-integrable weak first derivatives and zero boundary trace. On \(H_0^1(\Omega)\), we use the gradient norm \(\norm{u}_{H_0^1(\Omega)}=\norm{\nabla u}_{L^2(\Omega)}\). We work in the energy space \(\E = H_0^1(\Omega) \times L^2(\Omega)\), equipped with the norm
\(\norm{\left(u,z\right)}_{\E}^2 = \norm{\nabla u}_{L^2(\Omega)}^2 + \alpha\norm{u}_{L^2(\Omega)}^2 + \norm{z}_{L^2(\Omega)}^2.\)
Let \(-\Delta_D\) denote the Dirichlet Laplacian on \(L^2(\Omega)\). We define the operator \(A:D(A)\subset \E\to \E\) in \eqref{eq:pde_AC} by

\[
D(A):=D(-\Delta_D)\times H_0^1(\Omega),
\qquad
A(u,z):=
\begin{pmatrix}
-z\\
-\Delta u+\alpha u
\end{pmatrix}.
\]
With this formula, \(A\) is skew-adjoint on \(\E\); see \cite[Proposition~2.6.9]{cazenave1998semi}. We define the associated control space by \(\V=L^2\left((0,T)\times\omega\right)\). Fubini's theorem identifies \(\V\) naturally with the Bochner space \(L^2(0,T;L^2(\omega))\). For normed spaces, \(\mathcal L(\V,\E)\) denotes the space of bounded linear operators from \(\V\) to \(\E\), with the induced operator norm. We write \(\mathcal L(X):=\mathcal L(X,X)\). The superscript \(^*\) denotes the Hilbert adjoint, and \(\ip{\cdot}{\cdot}_{\E}\) is the inner product induced by the energy norm. We let \(S(t)=e^{-tA}\) denote the unitary wave group and define the control operator \(\mathcal{B}:L^2(\omega)\to\E\) by \(\mathcal{B}w=\left(\mathbf{0},\widetilde{w}\right)\in\E\) for \(w\in L^2(\omega)\), where \(\widetilde{w}\) is the zero extension of \(w\) to \(\Omega\).

We set \(\Phi=(u,\partial_tu)\), \(\Phi_0=(u^0,u^1)\), and \(\Phi_T=(u_T^0,u_T^1)\), and we write \(G(u,z)=(\mathbf0,g(u))\) whenever \(g(u)\in L^2(\Omega)\). The controlled evolution takes the form
\begin{equation}\label{eq:pde_AC}
 \Phi'(t)+A\Phi(t)+\mathcal Bf(t)+G(\Phi(t))
 =\mathbf0_{\E},
 \qquad \Phi(0)=\Phi_0.
\end{equation}
The control objective is \(\Phi(T)=\Phi_T\).

We interpret \eqref{eq:pde_AC} through the standard variation-of-constants formula; see \cite[Chapter~4]{cazenave1998semi}. With \(S(t)=e^{-tA}\), a mild solution is a trajectory \(\Phi\in C([0,T];\E)\) with \(G(\Phi)\in L^1(0,T;\E)\) satisfying
\(\Phi(t) =S(t)\Phi_0 -\int_0^t S(t-s) \left(\mathcal Bf(s)+G(\Phi(s))\right)\,\mathrm ds,\ t\in[0,T],\)
where we take the integral in \(\E\).

Whenever we use the full-space nonlinear endpoint map, we assume that the induced map \(u\mapsto g(u)\) is well-defined and locally Lipschitz from \(H_0^1(\Omega)\) to \(L^2(\Omega)\). For the fixed initial state \(\Phi_0\), let \(\mathcal D_T\subset\V\) consist of the controls for which the unique mild solution \(\Phi_f\) exists on \([0,T]\). We define \(\F:\mathcal D_T\to\E\) by \(\F(f)=\Phi_f(T)-\Phi_T\). Thus, a control \(f\in\mathcal D_T\) achieves the prescribed state transfer precisely when \(\F(f)=\mathbf0_{\E}\). If the induced nonlinearity is globally Lipschitz, then \(\mathcal D_T=\V\).

We base the source-frozen algorithms below on the energy-space convergence result in \Cref{thm:full-energy}. We define the linear controllability operator \(\mathcal{C}_T:\V \to \E\) and the continuous controllability Gramian \(\mathscr{W}_T: \E \to \E\) by
\(\mathcal{C}_T f = -\int_0^T S\left(T-s\right)\mathcal{B}f(s)\,\mathrm{d}s\), \(\mathscr{W}_T = \mathcal{C}_T \mathcal{C}_T^*.\)
For the continuous HUM construction, we assume that there exists \(C_{\rm obs}>0\) such that
\begin{equation}\label{eq:full-observability}
 \norm{\xi}_{\E}^2
 \le C_{\rm obs}\norm{\mathcal C_T^*\xi}_{\V}^2,
 \qquad \xi\in\E.
\end{equation}
This is the observability inequality for the adjoint wave evolution. Since
\(\ip{\mathscr W_T\xi}{\xi}_{\E} =\norm{\mathcal C_T^*\xi}_{\V}^2 \ge C_{\rm obs}^{-1}\norm{\xi}_{\E}^2,\)
the bounded self-adjoint operator \(\mathscr W_T\) is coercive and therefore boundedly invertible on \(\E\). The Geometric Control Condition supplies \eqref{eq:full-observability} in settings covered by the corresponding observability theorem. We define the quantities
\[
\begin{aligned}
\Gamma_T &= \norm{\mathscr{W}_T^{-1}}_{\mathcal{L}(\E)}^{1/2},\\
C_T^{\rm aff} &= \sqrt{T}\left(1 + \sqrt{T}\norm{\mathcal{B}}_{\mathcal L(L^2(\omega),\E)}\Gamma_T\right).
\end{aligned}
\]
We fix the initial and target states \(\Phi_0,\Phi_T\in\E\). For \(h\in L^2(0,T;\E)\), the mild solution of the affine initial-value problem is
\(\Psi(t) =S(t)\Phi_0 -\int_0^t S(t-s)\left(\mathcal Bf(s)+h(s)\right)\,\mathrm ds.\)
Consequently, the terminal constraint \(\Psi(T)=\Phi_T\) is equivalent to \(\mathcal C_Tf=d_h\), where
\begin{equation}\label{eq:full-hum-source}
 d_h
 =\Phi_T-S(T)\Phi_0+\int_0^T S(T-s)h(s)\,\mathrm ds,
 \qquad
 f_h=\mathcal C_T^*\mathscr W_T^{-1}d_h.
\end{equation}
We use the identity \(\mathcal C_Tf_h=\mathscr W_T\mathscr W_T^{-1}d_h=d_h\) to verify feasibility. We denote the kernel of \(\mathcal C_T\) by
\[
 \ker\mathcal C_T
 :=\{k\in\V:\mathcal C_Tk=\mathbf0_{\E}\}.
\]
Its orthogonal complement in \(\V\) is
\[
 (\ker\mathcal C_T)^\perp
 :=\{v\in\V:
       \langle v,k\rangle_{\V}=0
       \text{ for every }k\in\ker\mathcal C_T\}.
\]
For every \(k\in\ker\mathcal C_T\), the adjoint identity gives
\[
 \langle f_h,k\rangle_{\V}
 =\langle\mathscr W_T^{-1}d_h,\mathcal C_Tk\rangle_{\E}
 =0.
\]
Thus, it follows that \(f_h\in(\ker\mathcal C_T)^\perp\). Since \(\mathcal C_Tf_h=d_h\), every feasible control has the form \(f_h+k\) with \(k\in\ker\mathcal C_T\). Hence, \(\norm{f_h+k}_{\V}^2 =\norm{f_h}_{\V}^2+\norm{k}_{\V}^2\), so \(f_h\) is the unique feasible control of minimum \(\V\)-norm. We define \(\mathcal H_T(h)\) to be its associated mild trajectory in \(C([0,T];\E)\). The map \(\mathcal H_T\) is affine, with \(\Phi_0\) and \(\Phi_T\) held fixed.

\begin{lemma}
We assume that \eqref{eq:full-observability} holds, and we fix the same initial and target states in the definition of \(\mathcal H_T\). For all \(h_1, h_2 \in L^2\left(0,T;\E\right)\), we have
\begin{align}
 \norm{f_{h_1} - f_{h_2}}_{\V}
 &\le \Gamma_T \sqrt{T} \norm{h_1 - h_2}_{L^2\left(0,T;\E\right)},\label{eq:affine-control-bound}\\
 \norm{\mathcal{H}_T(h_1) - \mathcal{H}_T(h_2)}_{C\left([0,T];\E\right)}
 &\le C_T^{\rm aff} \norm{h_1 - h_2}_{L^2\left(0,T;\E\right)}.\label{eq:affine-trajectory-bound}
\end{align}
These constants depend only on the fixed linear control system and the time \(T\). In particular, \(\Gamma_T\le\sqrt{C_{\rm obs}}\). For the zero-extension control operator used here, \(\norm{\mathcal B}_{\mathcal L(L^2(\omega),\E)}=1\), so \(C_T^{\rm aff}=\sqrt T+T\Gamma_T\).
\end{lemma}

\begin{proof}
We set \(\delta h=h_1-h_2\), \(\delta d=d_{h_1}-d_{h_2}\), and \(\delta f=f_{h_1}-f_{h_2}\). Using the unitarity of \(S(t)\) and H\"{o}lder's inequality, we obtain \(\norm{\delta d}_{\E} \le \sqrt{T} \norm{\delta h}_{L^2\left(0,T;\E\right)}\).
Using \(\delta f=\mathcal C_T^*\mathscr W_T^{-1}\delta d\) and \(\mathscr W_T=\mathcal C_T\mathcal C_T^*\), we obtain
\[
\begin{aligned}
 \norm{\delta f}_{\V}^2 &= \ip{\mathcal C_T\mathcal C_T^* \mathscr W_T^{-1}\delta d} {\mathscr W_T^{-1}\delta d}_{\E} = \ip{\delta d}{\mathscr W_T^{-1}\delta d}_{\E}\\
 &\le \Gamma_T^2\norm{\delta d}_{\E}^2 \le \Gamma_T^2T \norm{\delta h}_{L^2(0,T;\E)}^2.
\end{aligned}
\]
Taking square roots, we obtain \eqref{eq:affine-control-bound}.

Let \(\delta\Psi=\mathcal H_T(h_1)-\mathcal H_T(h_2)\) denote the trajectory difference. By using Duhamel's formula, we obtain, uniformly in \(t \in [0, T]\),
\begin{equation*}
 \norm{\delta \Psi(t)}_{\E}
 \le \int_0^t \norm{\delta h(s)}_{\E} \,\mathrm{d}s
 + \norm{\mathcal{B}}_{\mathcal L(L^2(\omega),\E)} \int_0^t \norm{\delta f(s)}_{L^2(\omega)} \,\mathrm{d}s.
\end{equation*}
Taking the supremum over \(t\in[0,T]\) and applying the Cauchy--Schwarz inequality, we obtain
\[
 \begin{aligned}
 \norm{\delta\Psi}_{C([0,T];\E)}
 &\le
 \sqrt T\,\norm{\delta h}_{L^2(0,T;\E)}
 +\sqrt T\,
 \norm{\mathcal B}_{\mathcal L(L^2(\omega),\E)}
 \norm{\delta f}_{\V}\\
 &\le
 \left(
 \sqrt T+
 T\norm{\mathcal B}_{\mathcal L(L^2(\omega),\E)}\Gamma_T
 \right)
 \norm{\delta h}_{L^2(0,T;\E)}\\
 &=C_T^{\rm aff}\norm{\delta h}_{L^2(0,T;\E)}.
 \end{aligned}
\]
We have thus proved \eqref{eq:affine-trajectory-bound}.
\end{proof}

Let \(\mathcal X=C([0,T];\E)\) denote the space of continuous energy-space trajectories with the supremum norm. For trajectories whose nonlinear source belongs to \(L^2(0,T;\E)\), we define
\(\mathcal T_\infty(\Phi) =\mathcal H_T\!\left( \left(\mathbf0,g(\pi_1\Phi)\right)\right)\), \(\pi_1(u,z)=u.\)
The hypotheses of the following theorem ensure that this defines a map from \(\mathcal X\) to itself. We construct the iterates \(\left\{\Phi^{(n)}\right\}_{n=0}^\infty\) recursively by
\(\Phi^{(n+1)} = \mathcal{T}_\infty\!\left(\Phi^{(n)}\right)\), \(n \ge 0.\)

\begin{theorem}\label{thm:full-energy}
We assume that \eqref{eq:full-observability} holds. We suppose that the induced map \(g\colon H_0^1(\Omega)\to L^2(\Omega)\) satisfies
\begin{equation}\label{eq:full-g-lipschitz}
 \norm{g(u)-g(v)}_{L^2(\Omega)} \le L_g\norm{u-v}_{H_0^1(\Omega)}, \qquad u,v\in H_0^1(\Omega),
\end{equation}
where \(\norm{w}_{H_0^1(\Omega)} =\norm{\nabla w}_{L^2(\Omega)}\). If
\(q_\infty = C_T^{\rm aff} \sqrt{T} L_g < 1,\)
then the mapping \(\mathcal{T}_\infty\) is a contraction on \(\mathcal X\). Consequently, for every initial guess \(\Phi^{(0)} \in \mathcal X\), the sequence \(\left\{\Phi^{(n)}\right\}_{n=0}^\infty\) converges geometrically to the unique fixed point \(\Phi^*\) of \(\mathcal T_\infty\). The fixed point and its associated source-frozen HUM control \(f^*\) form a mild steering pair. Moreover, the iterates satisfy
\begin{equation}\label{eq:full-energy-residual}
 \norm{\Phi^{(n)} - \Phi^*}_{\mathcal X}
 \le \dfrac{\norm{\Phi^{(n)} - \mathcal{T}_\infty\left(\Phi^{(n)}\right)}_{\mathcal X}}{1 - q_\infty}.
\end{equation}
The corresponding sequence of HUM controls also converges in \(\V\).
\end{theorem}

\begin{proof}
For \(\Phi\in\mathcal X\), we set \(h_\Phi=(\mathbf0,g(\pi_1\Phi))\). Using the Lipschitz bound \eqref{eq:full-g-lipschitz}, we obtain
\begin{equation*}
\norm{h_\Phi}_{L^2(0,T;\E)} \le \sqrt T\left( L_g\norm{\Phi}_{\mathcal X} +\norm{g(0)}_{L^2(\Omega)}\right).
\end{equation*}
Thus, \(h_\Phi\in L^2(0,T;\E)\), and we use the affine HUM construction to obtain \(\mathcal T_\infty(\Phi)\in\mathcal X\).

For all trajectories \(\Phi_1, \Phi_2 \in C\left([0,T];\E\right)\), the Lipschitz bound \eqref{eq:full-g-lipschitz} and H\"{o}lder's inequality give
\begin{equation}\label{eq:full-g-trajectory-lip}
 \norm{g(\pi_1\Phi_1) - g(\pi_1\Phi_2)}_{L^2\left(0,T; L^2(\Omega)\right)}
 \le \sqrt{T} L_g \norm{\Phi_1 - \Phi_2}_{C\left([0,T];\E\right)}.
\end{equation}
Combining \eqref{eq:full-g-trajectory-lip} with~\eqref{eq:affine-trajectory-bound}, we deduce that \(\mathcal{T}_\infty\) is globally Lipschitz with constant at most \(q_\infty\). Since \(q_\infty < 1\), we apply the Banach fixed-point theorem~\cite{brezis2011functional} to obtain a unique fixed trajectory \(\Phi^*\). We also derive the error estimate~\eqref{eq:full-energy-residual} from
\[
 \left\|\Phi^{(n)}-\Phi^*\right\|_{\mathcal X}
 \le\left\|\Phi^{(n)}-\mathcal T_\infty\left(\Phi^{(n)}\right)\right\|_{\mathcal X}
 +q_\infty\left\|\Phi^{(n)}-\Phi^*\right\|_{\mathcal X}.
\]

We set \(h^*=(\mathbf0,g(\pi_1\Phi^*))\) and \(f^*=f_{h^*}\).
Using the fixed-point identity \(\Phi^*=\mathcal H_T(h^*)\), we obtain
\(\Phi^*(t) =S(t)\Phi_0 -\int_0^t S(t-s) \left(\mathcal Bf^*(s)+G(\Phi^*(s))\right)\,\mathrm ds.\)
Moreover, using the affine HUM construction, we obtain \(\Phi^*(0)=\Phi_0\) and \(\Phi^*(T)=\Phi_T\). Thus, \((\Phi^*,f^*)\) is a mild semilinear steering pair.

For each \(n\ge0\), we define \(h^{(n)}=(\mathbf0,g(\pi_1\Phi^{(n)}))\) and \(f^{(n+1)}=f_{h^{(n)}}\). Using \eqref{eq:affine-control-bound}, we obtain
\(\norm{f^{(n+1)}-f^*}_{\V} \le \Gamma_T\sqrt T\, \norm{h^{(n)}-h^*}_{L^2(0,T;\E)} \le \Gamma_T T L_g \norm{\Phi^{(n)}-\Phi^*}_{\mathcal X}\). Then, using the geometric convergence of \(\Phi^{(n)}\) established above, we obtain the geometric convergence of \(f^{(n+1)}\) to \(f^*\) in \(\V\).
\end{proof}

The coarse constant \(C_T^{\rm aff}\) in \eqref{eq:affine-trajectory-bound} bounds the direct source response and the HUM endpoint correction separately. To obtain a sharper estimate, for each \(h \in \mathcal{Z}_\infty = L^2\left(\left(0, T\right);L^2(\Omega)\right)\), we define
\begin{equation}\label{eq:full-response-operator}
\begin{aligned}
\mathscr{H}_{\infty,T}h &= \pi_1\!\left(\mathcal{H}_T\!\left(\left(\mathbf{0},h\right)\right) - \mathcal{H}_T((\mathbf{0},\mathbf{0}))\right),\\
\chi_{\infty,T} &= \norm{\mathscr{H}_{\infty,T}}_{\mathcal{L}\left(\mathcal{Z}_\infty\right)}.
\end{aligned}
\end{equation}
Using formula \eqref{eq:full-hum-source}, we see that the selected control depends affinely on the source. The mild solution formula then lets us conclude that \(\mathcal H_T\) is affine. Subtracting \(\mathcal H_T((\mathbf0,\mathbf0))\) removes its source-independent part, so we obtain the linearity of \(\mathscr H_{\infty,T}\).

Let \(C_\Omega\) denote the Poincar\'e constant in
\(\norm{v}_{L^2(\Omega)} \le C_\Omega\norm{v}_{H_0^1(\Omega)}.\)
For \(\Omega=(0,1)^N\), we may take \(C_\Omega=(\pi\sqrt N)^{-1}\). Using \eqref{eq:affine-trajectory-bound}, we obtain
\[
 \begin{aligned}
 \norm{\mathscr H_{\infty,T}h}_{\mathcal Z_\infty}
 &\le
 C_\Omega\sqrt T\,
 \norm{\mathcal H_T((\mathbf0,h))
       -\mathcal H_T((\mathbf0,\mathbf0))}_
      {C([0,T];\E)}\\
 &\le
 C_\Omega\sqrt T\,C_T^{\rm aff}
 \norm{h}_{\mathcal Z_\infty}.
 \end{aligned}
\]
We conclude that \(\mathscr H_{\infty,T}\) is bounded on \(\mathcal Z_\infty\). The operator depends only on the linear wave dynamics, the actuator region \(\omega\), the terminal time \(T\), and the norm on \(\V\) used in the HUM minimization. It is independent of the prescribed endpoint states \(\Phi_0\) and \(\Phi_T\).

For a globally Lipschitz scalar nonlinearity \(g\) with constant \(\ell_g\), we have
\begin{equation}\label{eq:exact-response-g-lipschitz}
 |g(r)-g(s)|\le\ell_g|r-s|, \qquad r,s\in\R.
\end{equation}
We define
\[
 \widehat{\mathcal T}_\infty(u) =\pi_1\mathcal H_T((\mathbf0,g(u))),
 \qquad u\in\mathcal Z_\infty.
\]
By \eqref{eq:exact-response-g-lipschitz},
\begin{equation*}
 \norm{g(u)}_{\mathcal Z_\infty} \le \ell_g\norm{u}_{\mathcal Z_\infty} +(T|\Omega|)^{1/2}|g(0)|.
\end{equation*}
Thus, the source belongs to \(\mathcal Z_\infty\), and we use the affine HUM construction to obtain
\(\widehat{\mathcal T}_\infty(u) \in C([0,T];H_0^1(\Omega)) \subset\mathcal Z_\infty.\)

\begin{theorem}\label{thm:full-exact-response}
We assume that \eqref{eq:full-observability} holds, and we suppose that \eqref{eq:exact-response-g-lipschitz} holds. If
\(\widehat{q}_\infty = \ell_g\,\chi_{\infty,T} < 1,\)
then \(\widehat{\mathcal{T}}_\infty\) is a contraction on \(\mathcal{Z}_\infty\). Consequently, there is a unique fixed displacement \(u^* \in \mathcal{Z}_\infty\) associated with a mild steering pair. For every \(u^{(0)}\in\mathcal Z_\infty\), the iteration \(u^{(n+1)}=\widehat{\mathcal T}_\infty(u^{(n)})\) converges geometrically to \(u^*\) in \(\mathcal Z_\infty\). Furthermore, every trial trajectory \(u \in \mathcal{Z}_\infty\) satisfies the a posteriori error estimate
\begin{equation}\label{eq:exact-response-error}
 \norm{u-u^*}_{\mathcal{Z}_\infty}
 \le
 \dfrac{\norm{u-\widehat{\mathcal{T}}_\infty(u)}_{\mathcal{Z}_\infty}}
 {1-\widehat{q}_\infty}.
\end{equation}
\end{theorem}

\begin{proof}
Using the definitions of \(\widehat{\mathcal T}_{\infty}\), \(\mathscr H_{\infty,T}\), and the linearity of \(\mathscr H_{\infty,T}\), we have
\begin{equation}\label{eq:linearity-difference}
 \widehat{\mathcal{T}}_\infty(u) - \widehat{\mathcal{T}}_\infty(w)
 = \mathscr{H}_{\infty,T}\left(g(u) - g(w)\right).
\end{equation}
Squaring and integrating the scalar Lipschitz bound \eqref{eq:exact-response-g-lipschitz} over \((0,T)\times\Omega\), we obtain
\begin{equation}\label{eq:exact-response-nemytskii-lipschitz}
 \norm{g(u)-g(w)}_{\mathcal Z_\infty}\le\ell_g\norm{u-w}_{\mathcal Z_\infty}.
\end{equation}
Combining \eqref{eq:linearity-difference}, \eqref{eq:exact-response-nemytskii-lipschitz}, and the operator-norm definition \eqref{eq:full-response-operator}, we show that \(\widehat{\mathcal T}_\infty\) is Lipschitz on \(\mathcal Z_\infty\) with constant at most \(\widehat q_\infty\). Because \(\widehat{q}_\infty < 1\), we apply the Banach fixed-point theorem to obtain a unique fixed displacement \(u^* \in \mathcal{Z}_\infty\) satisfying the error estimate~\eqref{eq:exact-response-error}.

We set \(h^*=(\mathbf0,g(u^*))\), \(\Psi^*=\mathcal H_T(h^*)=(\widetilde u^*,z^*)\), and \(f^*=f_{h^*}\).
From the fixed-point identity, we obtain \(\widetilde u^*=u^*\) in \(\mathcal Z_\infty\), and hence \(g(\widetilde u^*)=g(u^*)\) almost everywhere. Using this equality in the mild affine equation, we rewrite it as
\(\Psi^*(t) =S(t)\Phi_0 -\int_0^t S(t-s) \left(\mathcal Bf^*(s)+G(\Psi^*(s))\right)\,\mathrm ds.\)
Since the first components of the source and control vanish, we obtain \(\partial_t\widetilde u^*=z^*\) in the distributional sense from the first evolution equation. Both components are continuous in their energy spaces, so
\(\widetilde u^* \in C([0,T];H_0^1(\Omega)) \cap C^1([0,T];L^2(\Omega)).\)
The trajectory \(\Psi^*\) satisfies the prescribed initial and terminal states by construction. Thus, \((\widetilde u^*,f^*)\) is a mild semilinear steering pair.
\end{proof}

The uniqueness conclusions concern the fixed points of the specified source-frozen HUM maps. They do not assert uniqueness among all semilinear steering pairs. Moreover, \(f^*\) minimizes the control norm for the affine problem with source frozen at \(g(u^*)\). The argument does not establish that \(f^*\) minimizes the norm among all controls for the nonlinear steering problem.

We can also apply the abstract endpoint theorem in \Cref{sec:aposteriori} to the full-PDE endpoint map when we verify all its hypotheses, including well-posedness on the relevant control neighborhood, a Lipschitz bound for the endpoint derivative, a bounded right inverse, and the certificate's smallness and radius conditions.

Our computer-assisted results concern the specified Galerkin systems. The coarse constant in \Cref{thm:full-energy} exceeds one for the reported one-dimensional tanh cases, as discussed in \Cref{sec:response}, and we do not numerically enclose \(\chi_{\infty,T}\). Thus, we do not verify computationally here the contraction condition in \Cref{thm:full-energy} or the contraction condition in \Cref{thm:full-exact-response}.

To extend the local endpoint certificate to the PDE, we need rigorous estimates for the full endpoint map, including the effect of the omitted modes and the bounds required by \Cref{thm:certificate}. To extend the response certificate, we compare the finite-mode response operator \(\mathscr H_{P,T}\) of \Cref{sec:response} with \(\mathscr{H}_{\infty,T}\) within a shared function space and verify the condition \(\ell_g\chi_{\infty,T} < 1\). We must also account for the control norm since the continuous HUM construction uses the physical \(L^2((0,T)\times\omega)\) norm, whereas the finite-mode response certificate uses the Euclidean modal coefficient norm.

For vectors $x,y \in \mathbb{R}^n$, we define \(\langle x,y\rangle_2=x^\top y\) and for real matrices $X, Y \in \mathbb{R}^{m\times n}$, we set \(\langle X,Y\rangle_F=\operatorname{tr}(X^\top Y)\), where \(\operatorname{tr}\) is the sum of the diagonal entries. For a symmetric matrix $Z$, we denote by \(\lambda_{\min}(Z)\) and \(\lambda_{\max}(Z)\) the minimum and maximum eigenvalues of $Z$, respectively. The notation \(Z\succ0\) means that \(Z\) is symmetric positive definite. For a matrix $A \in \mathbb{R}^{m\times n}$, we use \(\|A\|_\infty=\displaystyle\max_{1 \le i \le m}\sum_{j=1}^n|a_{ij}|\). For a matrix \(J\), let \(\sigma(J)\) denote its spectrum and \(r(J)=\displaystyle\max_{\lambda\in\sigma(J)}|\lambda|\) denote its spectral radius. Throughout this paper, we use \(\norm{\cdot}_2\) to denote both the Euclidean norm on \(\R^n\) and the induced spectral norm on \(\R^{m\times n}\). We use \(\norm{\cdot}_F\) for the Frobenius norm. Specifically, for every \(x\in\R^n\) and \(A\in\R^{m\times n}\),
\begin{align*}
  \norm{x}_2 := \left( \sum_{i=1}^{n} x_i^2 \right)^{1/2}, \qquad
  \norm{A}_2 := \sup_{x\neq\mathbf{0}} \dfrac{\norm{Ax}_2}{\norm{x}_2} = \lambda_{\max}\!\left(A^\top A\right)^{1/2}, \\
  \norm{A}_F := \left( \sum_{i,j} a_{ij}^2 \right)^{1/2} = \left(\operatorname{tr}\!\left(A^\top A\right)\right)^{1/2}.
\end{align*}
For \(A\in\R^{2P\times2P}\), we define \( \mu_2(A):=\displaystyle\sup_{z\neq\mathbf{0}}\dfrac{\ip{z}{Az}_2}{\norm{z}_2^{2}} =\lambda_{\max}\!\left(\dfrac{1}{2}\left(A+A^\top\right)\right),\) and for a linear map \(\mathcal L\) on \(\R^{2P\times2P}\), we set \(\mu_F(\mathcal L):=\displaystyle\sup_{Z\neq\mathbf{0}}\dfrac{\ip{Z}{\mathcal LZ}_F}{\norm{Z}_F^{2}}\). These are the logarithmic norms induced by \(\norm{\cdot}_2\) and \(\norm{\cdot}_F\).

On \(\Omega=(0,1)^N\), the normalized scalar Dirichlet modes are
\[
 e_k(x)=2^{N/2}\prod_{\ell=1}^N\sin(k_\ell\pi x_\ell),
 \qquad \mu_k=\pi^2\sum_{\ell=1}^N k_\ell^2,
 \qquad k\in\mathbb N^N,
\]
where \(\mathbb N=\{1,2,\ldots\}\). We relabel a finite set of modes as \(e_1,\ldots,e_P\). The explicit one-dimensional formulas below use \(N=1\), so \(e_j(x)=\sqrt2\sin(j\pi x)\) and \(\mu_j=(j\pi)^2\). We take \(\alpha=0\) in the calculations with the explicit one-dimensional constants and specify the two-dimensional linear benchmark separately. For the selected modes, we approximate the displacement and parametrize the applied control by
\[
 u_P(t,x)=\sum_{j=1}^P a_j(t)e_j(x),
 \qquad
 \mathbf1_\omega(x)f_P(t,x)
 =\mathbf1_\omega(x)\sum_{j=1}^P c_j(t)e_j(x).
\]
Here, we define \(f_P(t,x)=\sum_{j=1}^P c_j(t)e_j(x)\) on \(\Omega\).
Its restriction to \(\omega\) is the control, and
\(\mathbf1_\omega f_P\) is the applied field on \(\Omega\). More generally, for a coefficient vector \(z\in\R^P\) we write \(u_z:=\displaystyle\sum_{j=1}^{P}z_je_j\) for the function it represents, so that \(u_P(t,\cdot)=u_{a(t)}\) and \(\norm{u_z}_{L^2(\Omega)}=\norm{z}_2\) by orthonormality. Next, we introduce the actuator-overlap matrix \(M_P\), whose entries are \(\left(M_P\right)_{jk}=\displaystyle\int_\omega e_j(x)e_k(x)\,\mathrm{d}x\) for \(j,k\in\{1,2,\ldots,P\}\). For every \(z\in\R^P\),
\begin{equation}\label{eq:actuator-contraction}
 z^\top M_Pz =\int_\omega |u_z(x)|^2\,\mathrm dx \le \norm{u_z}_{L^2(\Omega)}^2 =\norm{z}_2^2.
\end{equation}
If \(z^\top M_Pz=0\), then the finite sine sum \(u_z\) vanishes on the nonempty open set \(\omega\). Since \(u_z\) is real analytic on the connected domain \(\Omega\), we conclude that it vanishes throughout \(\Omega\). Using orthonormality, we obtain \(z=\mathbf0_{\R^P}\), and hence \(M_P\succ0\) and \(\norm{M_P}_2\le1\). The applied-control norm satisfies
\(\norm{\mathbf1_\omega f_P}_{L^2((0,T)\times\Omega)}^2 =\int_0^T c(t)^\top M_Pc(t)\,\mathrm dt.\)
In the HUM constructions below, we instead minimize the Euclidean modal coefficient norm \(\norm{c}_{L^2(0,T;\R^P)}\), unless we specify a different cost. We define the projected nonlinearity \(G_P\) component-wise by
\((G_P(a))_j = \int_\Omega g(u_a(x))e_j(x)\,\mathrm dx, a\in\R^P,\quad j=1,\ldots,P\).
Substituting this ansatz into \eqref{eq:pde} and requiring the residual to be orthogonal to every selected mode, we obtain the Galerkin system
\begin{equation}\label{eq:retained}
 a''(t)+\Lambda_P\, a(t)+G_P(a(t))+M_P\,c(t)=\mathbf{0}_{\R^P},
\end{equation}
where \(\Lambda_P=\operatorname{diag}(\mu_1+\alpha,\ldots,\mu_P+\alpha)\) and \(c(t)=(c_1(t),\ldots,c_P(t))\in\R^P\). For the explicit one-dimensional numerical certificates, we take \(\alpha=0\).

We define the initial and target coefficient vectors by
\[
 \begin{aligned}
 a_0&=\left(\ip{u^0}{e_j}_{L^2(\Omega)}\right)_{j=1}^P,
 &a_1&=\left(\ip{u^1}{e_j}_{L^2(\Omega)}\right)_{j=1}^P,\\
 a_T&=\left(\ip{u_T^0}{e_j}_{L^2(\Omega)}\right)_{j=1}^P,
 &a'_T&=\left(\ip{u_T^1}{e_j}_{L^2(\Omega)}\right)_{j=1}^P.
 \end{aligned}
\]
Thus, \(a(0)=a_0\), \(a'(0)=a_1\), and the steering objective is \(a(T)=a_T\), \(a'(T)=a'_T\). For \(g\in C^1(\R)\), the Jacobian of the projected nonlinearity is
\([DG_P(a)]_{jk}=\int_\Omega g'(u_a)e_je_k\,\mathrm dx.\)
Here, \(D\) denotes the Fr\'echet derivative. If \(g\in C^2(\R)\), then
\([D^2G_P(a)[p,q]]_j =\int_\Omega g''(u_a)u_pu_qe_j\,\mathrm dx.\)
To reformulate \eqref{eq:retained} as a first-order system, we introduce the energy-scaled state variable \(Y\), defined by
\begin{equation*}
 Y=\mathsf E_P\binom{a}{a'},
 \qquad
 \mathsf E_P=\operatorname{diag}(\Lambda_P^{1/2},I_P).
\end{equation*}
The blocks of \(Y=(Y_1,Y_2)^\top\) are \(Y_1=\Lambda_P^{1/2}a\) and \(Y_2=a'\). Using the eigenfunction identities, we obtain
\(\norm{(u_a,u_b)}_{\E}^2 =a^\top\Lambda_Pa+\norm{b}_2^2 =\left\|\mathsf E_P\binom{a}{b}\right\|_2^2.\)
Thus, Euclidean distance in the scaled coordinates equals the energy distance between the corresponding Galerkin states. This scaling allows us to rewrite the dynamics as the ODE
\begin{equation}\label{eq:energy-ode}
 Y'=A_0Y+N(Y)+B_E c,
\end{equation}
where we define the operators by
\[
 A_0=
 \begin{pmatrix}0&\Lambda_P^{1/2}\\-\Lambda_P^{1/2}&0\end{pmatrix},
 \quad
 B_E=\begin{pmatrix}\mathbf{0}\\-M_P\end{pmatrix},
 \quad
 N(Y)=\begin{pmatrix}\mathbf{0}\\-G_P\left(\Lambda_P^{-1/2}Y_1\right)\end{pmatrix}.
\]
The matrix \(A_0\) is skew-symmetric. For the endpoint and variational analysis, we assume that \(g\in C^1(\R)\) and fix the initial state. Let \(\mathcal D_{P,T}\subset L^2(0,T;\R^P)\) consist of the controls for which the unique solution of \eqref{eq:energy-ode} exists on \([0,T]\). This set is open, and the control-to-state map \( c\mapsto Y_c \) is continuously Fr\'echet differentiable from \(\mathcal D_{P,T}\) into \(C([0,T];\R^{2P})\). We obtain these properties from continuous dependence and differentiation of the integral equation on a bounded neighborhood of each admissible trajectory. For \(c\in\mathcal D_{P,T}\), let \(Y_c\) solve \eqref{eq:energy-ode} with \(Y_c(0)=\mathsf E_P(a_0,a_1)^\top\), and write \(a_c=\Lambda_P^{-1/2}(Y_c)_1\). We set \(Y_T=\mathsf E_P(a_T,a'_T)^\top\), where \(a_T,a'_T\) are the modal projections of \(u_T^0,u_T^1\). We define the endpoint map \(\F_P:\mathcal D_{P,T}\to\R^{2P}\) by
\begin{equation}\label{eq:retained-endpoint}
 \F_P(c)=Y_c(T)-Y_T.
\end{equation}
Consequently, \(\norm{\F_P(c)}_2\) is the terminal energy error relative to the projected target in the finite-dimensional Galerkin system. For \(g(u)=\beta\tanh u\) with \(\beta\ge0\), the projected vector field is globally Lipschitz, so \(\mathcal D_{P,T}=L^2(0,T;\R^P)\).

We fix \(c\in\mathcal D_{P,T}\). For \(\zeta\in L^2(0,T;\R^P)\), let
\[
 z_\zeta =D\!\left(c\mapsto Y_c\right)(c)\zeta
 \in C([0,T];\R^{2P}).
\]
Differentiating the integral equation for \(Y_c\), we find that \(z_\zeta\) solves
\begin{equation}\label{eq:linearized-energy}
 \begin{aligned}
 z_\zeta'(t)&=A_c(t)z_\zeta(t)+B_E\zeta(t),
 &z_\zeta(0)&=\mathbf0_{\R^{2P}},\\
 A_c(t)&=A_0+
 \begin{pmatrix}
 \mathbf0&\mathbf0\\
 -DG_P(a_c(t))\Lambda_P^{-1/2}&\mathbf0
 \end{pmatrix}.
 \end{aligned}
\end{equation}
The differential equation holds almost everywhere on \((0,T)\). For each fixed \(s \in \left[0, T\right]\), we let \(\Xi_c(t,s)\) denote the state-transition matrix of the homogeneous component of \eqref{eq:linearized-energy}, which is the unique \(\R^{2P \times 2P}\)-valued solution to
\begin{equation}\label{eq:state-transition}
 \partial_t \Xi_c(t,s) = A_c(t) \Xi_c(t,s) \quad \text{with} \quad \Xi_c(s,s) = I_{2P}.
\end{equation}
Since \(z_\zeta(0) = \mathbf{0}_{\R^{2P}}\), we apply the variation-of-constants formula to obtain
\begin{equation}\label{eq:trajectory-solution}
 z_\zeta(t) = \int_0^t \Xi_c(t,s) B_E \zeta(s) \mathrm{d}s.
\end{equation}
Because \(Y_T\) is independent of \(c\), the derivative of \(\F_P\) is \(D\F_P(c) \zeta = z_\zeta(T)\). We define the energy-scaled controllability Gramian by
\begin{equation}\label{eq:gramian}
 W_c(T) := \int_0^T \Xi_c(T,s) B_E B_E^\top \Xi_c(T,s)^\top \mathrm{d}s.
\end{equation}
If \(W_c(T)\) is positive definite, we use \Cref{prop:linearized-right-inverse} to obtain the norm of the minimum-norm right inverse of \(D\F_P(c)\) in \(L^2\!\left(0, T; \R^P\right)\):
\begin{equation}\label{eq:gamma}
 \gamma_P(c) := \lambda_{\min}\left(W_c(T)\right)^{-1/2}.
\end{equation}

\begin{proposition}\label{prop:linearized-right-inverse}
Let \(g\in C^1(\R)\) and \(c\in\mathcal D_{P,T}\). We suppose that \(W_c(T)\succ0\). Then,
\((R_c\xi)(t) =B_E^\top\Xi_c(T,t)^\top W_c(T)^{-1}\xi\)
defines a bounded linear right inverse
\(R_c:\R^{2P}\to L^2(0,T;\R^P)\)
of \(D\F_P(c)\). For every \(\xi\in\R^{2P}\), the control \(R_c\xi\) is the unique minimizer of
\(\min\left\{ \norm{\zeta}_{L^2(0,T;\R^P)}: D\F_P(c)\zeta=\xi \right\}.\)
Moreover,
\(\norm{R_c}_{\mathcal L(\R^{2P},L^2(0,T;\R^P))} =\lambda_{\min}(W_c(T))^{-1/2}.\)
\end{proposition}
\begin{proof}
By \eqref{eq:retained-endpoint} and \eqref{eq:trajectory-solution}, we obtain \(D\F_P(c)=\mathcal C_c\), where
\(\mathcal C_c:L^2\left(0,T;\R^P\right)\to\R^{2P}\), \(\mathcal C_c\zeta=z_\zeta(T)=\int_0^T\Xi_c(T,s)B_E\zeta(s)\,\mathrm ds\).
Its Hilbert adjoint is \(\left(\mathcal{C}_c^* z\right)(t) = B_E^\top \Xi_c\left(T,t\right)^\top z\). Thus, the Gramian in \eqref{eq:gramian} satisfies \(W_c(T) = \mathcal{C}_c \mathcal{C}_c^*\), and, for every \(\xi\in\R^{2P}\),
\(D\F_P(c)R_c \xi=\mathcal{C}_c R_c \xi = \mathcal{C}_c \mathcal{C}_c^* W_c(T)^{-1} \xi = \xi\).
We therefore verify that \(D\F_P(c)R_c=I_{\R^{2P}}\), so \(R_c\) is a right inverse of \(D\F_P(c)\).

For each \(\xi\in\R^{2P}\), the function \(R_c\xi\) has the smallest norm among all \(q\in L^2(0,T;\R^P)\) satisfying \(\mathcal{C}_cq=\xi\). Indeed, for every \(w \in \ker\mathcal{C}_c\), we have
\begin{equation*}
 \ip{R_c\xi}{w}_{L^2} = \ip{W_c\left(T\right)^{-1}\xi}{\mathcal{C}_c w}_{\R^{2P}} = 0.
\end{equation*}
Therefore, we infer that \(R_c\xi \in \left(\ker\mathcal{C}_c\right)^{\perp}\). The solution set of \(\mathcal{C}_c q = \xi\) is the affine space \(R_c\xi + \ker\mathcal{C}_c\), and orthogonality lets us split the squared norm. Thus, \(R_c\xi\) is the unique minimum-norm element. Finally, we have
\begin{equation*}
 \norm{R_c\xi}_{L^2\left(0, T; \R^P\right)}^2 = \ip{\mathcal{C}_c\mathcal{C}_c^*W_c\left(T\right)^{-1}\xi}{W_c\left(T\right)^{-1}\xi}_{\R^{2P}} = \xi^\top W_c\left(T\right)^{-1}\xi.
\end{equation*}
Taking the supremum over all \(\xi\in\R^{2P}\) with \(\norm{\xi}_2 = 1\), we obtain the operator norm
\begin{equation*}
 \norm{R_c}_{\mathcal L\left(\R^{2P},L^2\left(0,T;\R^P\right)\right)} = \lambda_{\max}\!\left(W_c\left(T\right)^{-1}\right)^{1/2} = \lambda_{\min}\!\left(W_c\left(T\right)\right)^{-1/2}.
\end{equation*}
\end{proof}

For the actuator parametrization considered here, \(W_c(T)\) is positive definite for every \(T>0\) and \(c\in\mathcal D_{P,T}\). Indeed, we suppose that \(\xi^\top W_c(T)\xi=0\), and we set
\(p(t)=\Xi_c(T,t)^\top\xi =\binom{p_1(t)}{p_2(t)}.\)
Then
\(0=\int_0^T\norm{B_E^\top p(t)}_2^2\,\mathrm dt =\int_0^T\norm{M_Pp_2(t)}_2^2\,\mathrm dt.\)
Since \(M_P\) is invertible and \(p\) is continuous, we have \(p_2(t)=\mathbf0\) for every \(t\in[0,T]\). From the adjoint equation \(p'=-A_c^\top p\), we obtain \(p_2'=-\Lambda_P^{1/2}p_1\). Hence, \(p_1=\mathbf0\), and therefore \(\xi=p(T)=\mathbf0\). Thus, we have proved \(W_c(T)\succ0\).

\section{The a posteriori Newton--HUM theorem}\label{sec:aposteriori}

We state the certificate in a form independent of the wave equation. In this section, let \(\V\) and \(\E\) be Hilbert spaces, let \(\mathcal U\subset\V\) be open, and let \(\F:\mathcal U\to\E\) be a nonlinear map. For every \(r\ge0\), we write
\[
 \overline B_r(\mathbf0_{\E})
 :=\{y\in\E:\norm{y}_{\E}\le r\}
\]
for the closed ball of radius \(r\) centered at \(\mathbf0_{\E}\).

\begin{theorem}\label{thm:certificate}
Let \(\bar f\in\mathcal U\), and let \(\rho\in(0,\infty]\) satisfy
\[
 B_\rho(\bar f)
 :=\left\{f\in\V:\norm{f-\bar f}_{\V}<\rho\right\}
 \subset\mathcal U.
\]
We suppose that \(\F\) is continuously Fr\'echet differentiable on \(B_\rho(\bar f)\) and that, for some \(L\ge0\),
\begin{equation}\label{eq:DF-Lip}
 \norm{D\F(f)-D\F(w)}_{\mathcal L(\V,\E)}
 \le L\norm{f-w}_{\V},
 \qquad f,w\in B_\rho(\bar f).
\end{equation}
We assume that \(A_{\bar f}=D\F(\bar f)\) has a bounded linear right inverse \(R_{\bar f}\in\mathcal L(\E,\V)\) satisfying
\(A_{\bar f}R_{\bar f}=I_{\E},\ \norm{R_{\bar f}}_{\mathcal L(\E,\V)}\le\gamma,\)
for some \(\gamma>0\). We choose \(\varepsilon\ge\norm{\F(\bar f)}_{\E}\), and we set \(h=2L\gamma^2\varepsilon\). If \(h<1\) and
\begin{equation}\label{eq:rstar}
 r_*=\dfrac{2\varepsilon}{1+\sqrt{1-h}},
 \qquad
 \gamma r_*<\rho,
\end{equation}
then there exists a unique \(y^*\in\overline B_{r_*}(\mathbf0_{\E})\) such that
\(\F(\bar f+R_{\bar f}y^*)=\mathbf0_{\E}.\)
Consequently,
\(f^*=\bar f+R_{\bar f}y^*,\ \norm{f^*-\bar f}_{\V}\le\gamma r_*,\)
and \(f^*\) is the unique zero of \(\F\) in \(\bar f+R_{\bar f}\overline B_{r_*}(\mathbf0_{\E})\).
\end{theorem}
\begin{proof}
We write \(R=R_{\bar f}\) and define
\(H(y)=\F(\bar f+Ry)\), \(Q(y)=H(y)-H(\mathbf0_{\E})-y.\)
These maps are continuously differentiable for \(\norm{y}_{\E}<\rho/\gamma\). Since \(\gamma r_*<\rho\), they are defined on an open neighborhood of \(\overline B_{r_*}(\mathbf0_{\E})\). Using the right-inverse identity, we obtain
\(DH(\mathbf0_{\E})=D\F(\bar f)R=I_{\E}\), \(Q(\mathbf0_{\E})=\mathbf0_{\E}.\)
Moreover, \(DQ(y)=\left(D\F(\bar f+Ry)-D\F(\bar f)\right)R\), so using the derivative estimate \eqref{eq:DF-Lip}, we obtain \(\norm{DQ(y)}_{\mathcal L(\E)} \le L\gamma^2\norm{y}_{\E}\). Integrating along the segment from \(\mathbf0_{\E}\) to \(y\), we obtain
\(Q(y)=\int_0^1 DQ(ty)y\,\mathrm dt\),  \(\norm{Q(y)}_{\E} \le\dfrac{L\gamma^2}{2}\norm{y}_{\E}^2.\)
Similarly, since \(\overline B_{r_*}(\mathbf0_{\E})\) is convex, we obtain
\(\norm{Q(y)-Q(z)}_{\E} \le L\gamma^2r_*\norm{y-z}_{\E},\)
for all \(y,z\) in this ball.

We define \(\mathcal K(y)=-H(\mathbf0_{\E})-Q(y)\). The radius in \eqref{eq:rstar} satisfies
\(\varepsilon+\dfrac{L\gamma^2}{2}r_*^2=r_*.\)
Therefore, for \(\norm{y}_{\E}\le r_*\),
\(\norm{\mathcal K(y)}_{\E} \le\varepsilon+\dfrac{L\gamma^2}{2}r_*^2=r_*.\)
Thus, \(\mathcal K\) maps the closed ball into itself. Its Lipschitz constant is at most \(L\gamma^2r_*=1-\sqrt{1-h}<1\). Since \(\E\) is complete, we apply the Banach fixed-point theorem to obtain a unique fixed point \(y^*\in\overline B_{r_*}(\mathbf0_{\E})\). We can rewrite the identity \(y^*=\mathcal K(y^*)\) as \(H(y^*)=\mathbf0_{\E}\).

We set \(f^*=\bar f+Ry^*\). Then \(\F(f^*)=\mathbf0_{\E}\) and
\(\norm{f^*-\bar f}_{\V} \le\gamma\norm{y^*}_{\E} \le\gamma r_*.\)
Finally, from \(D\F(\bar f)R=I_{\E}\), we infer that \(R\) is injective. Hence, uniqueness of \(y^*\) also lets us conclude that the zero in \(\bar f+R\overline B_{r_*}(\mathbf0_{\E})\) is unique.
\end{proof}

We establish uniqueness only in the restricted affine slice \(\bar f+R_{\bar f}\overline B_{r_*}(\mathbf0_{\E})\) and existence within the control ball of radius \(\gamma r_*\) about \(\bar f\). We do not assert uniqueness throughout that ball or among all exact controls. When \(R_{\bar f}\) is the minimum-norm HUM right inverse, the correction \(f^*-\bar f\) lies in \((\ker D\F(\bar f))^\perp\).

\begin{remark}
\Cref{thm:certificate} is of Newton--Kantorovich type; see, e.g.,~\cite{Deuflhard2011}. If \(D\F(\bar f)\) is invertible and \(R_{\bar f}=D\F(\bar f)^{-1}\), then the condition \(h=2L\gamma^2\varepsilon<1\) is the strict form of the classical Kantorovich condition \(L\gamma^2\varepsilon<1/2\), and, for \(L>0\), the control-distance bound \(\gamma r_*=(1-\sqrt{1-h})/(L\gamma)\) coincides with the classical Kantorovich radius. \Cref{thm:certificate} requires only a bounded right inverse, which suits the underdetermined endpoint equation. The radius identity \(\varepsilon+\frac{L\gamma^2}{2}r_*^2=r_*\) is the equality case of the quadratic radii-polynomial condition used in validated numerics~\cite{HungriaLessardMireles2016}.
\end{remark}

\begin{theorem}\label{thm:newton}
Let \(f_0\in\mathcal U\). We suppose that \(B_\rho(f_0)\subset\mathcal U\) for some \(\rho\in(0,\infty]\). We assume that \(\F\) is continuously Fr\'echet differentiable on this ball and satisfies \eqref{eq:DF-Lip} there with constant \(L\ge0\). For every \(f\in B_\rho(f_0)\), we choose a bounded linear right inverse \(R_f\in\mathcal L(\E,\V)\). We fix a constant \(\gamma>0\), independent of \(f\), such that
\begin{equation}\label{eq:newton-right-inverse-bound}
 D\F(f)R_f=I_{\E}, \qquad \norm{R_f}_{\mathcal L(\E,\V)}\le\gamma.
\end{equation}
We define \(f_{n+1}=f_n-R_{f_n}\F(f_n)\) and \(\varepsilon_n=\norm{\F(f_n)}_{\E}\). If
\begin{equation*}
 q_{\rm N}=\dfrac{L\gamma^2}{2}\varepsilon_0<1, \qquad s_*:=\dfrac{\gamma\varepsilon_0}{1-q_{\rm N}}<\rho,
\end{equation*}
then every iterate belongs to \(B_\rho(f_0)\), and the sequence converges in \(\V\) to a control \(f^*\in B_\rho(f_0)\) satisfying \(\F(f^*)=\mathbf0_{\E}\). Moreover,
\begin{equation}\label{eq:quadratic}
 \varepsilon_{n+1}
 \le\dfrac{L\gamma^2}{2}\varepsilon_n^2,
 \qquad
 \norm{f_{n+1}-f_n}_{\V}\le\gamma\varepsilon_n,
\end{equation}
and
\begin{equation}\label{eq:newton-tail-control}
\norm{f^*-f_n}_{\V} \le\dfrac{\gamma\varepsilon_n}{1-q_{\rm N}}.
\end{equation}
In particular, the endpoint residuals converge at least quadratically, unless an exact control is reached in finitely many steps.
\end{theorem}

\begin{proof}
If \(\varepsilon_0=0\), the iteration is constant. If \(L=0\), then \(D\F\) is constant on the ball. The first step lies in the ball because its norm is at most \(\gamma\varepsilon_0=s_*<\rho\), and using the right-inverse identity we obtain \(\F(f_1)=\mathbf0_{\E}\). We may therefore assume \(0<q_{\rm N}<1\).

We set \(d_n=-R_{f_n}\F(f_n)\). We prove by induction that the iterates are well defined and that
\begin{equation}\label{eq:newton-geometric-estimates}
 \varepsilon_{n+1}\le q_{\rm N}\varepsilon_n, \qquad
 \varepsilon_n\le q_{\rm N}^n\varepsilon_0, \qquad
 \norm{d_n}_{\V}\le\gamma q_{\rm N}^n\varepsilon_0.
\end{equation}
For \(n=0\), we obtain the last two bounds in \eqref{eq:newton-geometric-estimates} from the definition of \(\varepsilon_0\) and the right-inverse bound \eqref{eq:newton-right-inverse-bound}.

We suppose that the iterates through \(f_n\) are well defined, that the last two estimates in \eqref{eq:newton-geometric-estimates} hold through index \(n\), and that the first holds for \(0\le k<n\). For every \(s\in[0,1]\),
\(\norm{f_n+sd_n-f_0}_{\V} \le\sum_{k=0}^{n-1}\norm{d_k}_{\V} +s\norm{d_n}_{\V} \le\gamma\varepsilon_0\sum_{k=0}^{n}q_{\rm N}^k \le s_*<\rho\).
Thus, the entire update segment lies in \(B_\rho(f_0)\), and \(f_{n+1}=f_n+d_n\) is admissible. Since \(D\F(f_n)d_n=-\F(f_n)\), we use the integral form of the mean value theorem to write
\(\F(f_{n+1}) =\int_0^1 \left[D\F(f_n+sd_n)-D\F(f_n)\right]d_n\,\mathrm ds.\)
Applying \eqref{eq:DF-Lip}, we obtain
\(\varepsilon_{n+1} \le\dfrac{L}{2}\norm{d_n}_{\V}^2 \le\dfrac{L\gamma^2}{2}\varepsilon_n^2 \le q_{\rm N}\varepsilon_n \le q_{\rm N}^{n+1}\varepsilon_0\).
Using the right-inverse bound \eqref{eq:newton-right-inverse-bound} at \(f_{n+1}\), we obtain \(\norm{d_{n+1}}_{\V}\le \gamma q_{\rm N}^{n+1}\varepsilon_0\) and complete the induction.

From the geometric step estimate \eqref{eq:newton-geometric-estimates}, we obtain
\begin{equation*}
 \sum_{n=0}^{\infty}\norm{d_n}_{\V} \le\dfrac{\gamma\varepsilon_0}{1-q_{\rm N}}=s_*.
\end{equation*}
Hence, \((f_n)\) is Cauchy in the complete space \(\V\) and converges to a limit \(f^*\) with \(\norm{f^*-f_0}_{\V}\le s_*<\rho\). By continuity of \(\F\) and \(\varepsilon_n\to0\), we conclude that \(\F(f^*)=\mathbf0_{\E}\).

Finally, using the residual contraction in \eqref{eq:newton-geometric-estimates}, we obtain \(\varepsilon_{n+j}\le q_{\rm N}^j\varepsilon_n\). Therefore,
\(\norm{f^*-f_n}_{\V} \le\sum_{k=n}^{\infty}\norm{d_k}_{\V} \le\gamma\sum_{j=0}^{\infty}q_{\rm N}^j\varepsilon_n =\dfrac{\gamma\varepsilon_n}{1-q_{\rm N}}\).
\end{proof}

In \Cref{thm:newton}, quadratic convergence concerns the endpoint residuals, as expressed in \eqref{eq:quadratic}. The control-error estimate \eqref{eq:newton-tail-control} also bounds the distance from each iterate to the limiting exact control. \Cref{thm:certificate,thm:newton} rely on distinct hypotheses, which consequently require different sufficient constants. When the same constants \(L\) and \(\gamma\) are admissible for both theorems and \(\varepsilon\ge\varepsilon_0=\norm{\F(f_0)}_{\E}\), their indicators satisfy
\begin{equation}\label{eq:indicator-comparison}
 q_{\rm N}=\dfrac{L\gamma^2}{2}\varepsilon_0 \le\dfrac{h}{4}.
\end{equation}
Equality in \eqref{eq:indicator-comparison} holds when \(\varepsilon=\varepsilon_0\). The Newton theorem requires a right-inverse bound throughout the control neighborhood, whereas the existence certificate requires such a bound only at the candidate. We show that a right-inverse bound at the center can provide a neighborhood bound. We suppose that \(D\F(f_0)R_0=I_{\E}\) and \(\norm{R_0}\le\gamma_0\). If \(B_r(f_0)\subset\mathcal U\), \eqref{eq:DF-Lip} holds there, and \(L\gamma_0r<1\), then
\begin{align}
 R_f &=R_0\left(D\F(f)R_0\right)^{-1},  \quad f\in B_r(f_0),\\
\norm{R_f} &\le\dfrac{\gamma_0}{1-L\gamma_0r}.\label{eq:newton-neighborhood-right-inverse-bound}
\end{align}
Indeed, \(\norm{D\F(f)R_0-I_{\E}}\le L\gamma_0r<1\), so the inverse exists by the Neumann series. The minimum-norm HUM right inverse satisfies the upper bound \eqref{eq:newton-neighborhood-right-inverse-bound}.

In the deterministic endpoint solver, we use a damped Newton--HUM iteration. After computing a correction direction, we test step lengths \(1,1/2,1/4,\ldots,2^{-12}\) and accept the first trial for which forward integration reaches \(T\) with a finite terminal energy error at most \(1-10^{-4}\varsigma\) times its current value, where \(\varsigma\) is the step length. For the one-dimensional Neural--Spectral corrections, we test step lengths \(1,1/2,1/4,\ldots,2^{-6}\) and accept the first trial whose independently computed terminal energy error is strictly smaller than its current value. We use this backtracking strategy to seek acceptable steps outside the neighborhood covered by \Cref{thm:newton}. Our theorem treats the exact iteration with full steps. Under its hypotheses, every exact full step with a nonzero residual decreases that residual, since \(\varepsilon_{n+1}\le q_{\rm N}\varepsilon_n<\varepsilon_n\). Thus, a line search requiring only strict decrease of the exact residual would accept the full step throughout this neighborhood. Our convergence theorem treats exact endpoint values and exact right-inverse corrections, whereas the implementation approximates these quantities. We therefore apply our rigorous endpoint certificate separately to the resulting stored control.

To apply \Cref{thm:certificate}, we need three inputs: the endpoint defect \(\varepsilon\), the right-inverse bound \(\gamma\), and a Lipschitz constant \(L\) satisfying \eqref{eq:DF-Lip}. We enclose the first two numerically in \Cref{sec:validated-endpoint} and derive a closed-form bound for the third when \(g(u)=\beta\tanh u\). We write \(\ell_1=\norm{g'}_{L^\infty(\R)}\). Using \(\norm{g'}_{L^\infty(\R)}=\beta\), \(\norm{g''}_{L^\infty(\R)}=\dfrac{4\beta}{3\sqrt3}\), and setting \(\kappa:=\dfrac{\ell_1}{2\pi}\), \(C_{\infty,P}:=\left(\dfrac{2}{\pi^2}\displaystyle\sum_{j=1}^{P}\dfrac{1}{j^2}\right)^{1/2}\), we define
\begin{equation}\label{eq:Lendpoint}
L_P^{\mathrm{end}} =\dfrac{4\beta\,C_{\infty,P}}{3\sqrt3\,\pi}\, \dfrac{e^{\kappa T}\left(\cosh\left(\kappa T\right)-1\right)}{\kappa^{2}}.
\end{equation}
Using the actuator contraction \eqref{eq:actuator-contraction} and the first- and second-variation estimates derived in the Supplement, we obtain the endpoint-derivative Lipschitz bound. Thus, \eqref{eq:DF-Lip} holds for \(\F = \F_P\) with \(L = L_P^{\mathrm{end}}\) and \(\rho = \infty\). Since the derivative bound holds on the entire control space, the radius condition in \eqref{eq:rstar} imposes no additional restriction. For \(P=8\), \(T=2.5\), and \(\beta=1\), we evaluate the right-hand side of \eqref{eq:Lendpoint} in outward-rounded interval arithmetic to obtain
\(L_8^{\mathrm{end}} \in \left[ 0.6426066903715131, \, 0.6426066903715133 \right].\)
For a computed control \(\bar c\), let the validated bounds satisfy
\[
 \norm{\F_P(\bar c)}_2\le\bar\varepsilon,
 \qquad
 \norm{R_{\bar c}}_{\mathcal L(\R^{2P},L^2(0,T;\R^P))}\le\bar\gamma,
 \qquad
 L_P^{\mathrm{end}}\le\bar L.
\]
We take \(\bar L\) as the upper endpoint of the interval enclosing \eqref{eq:Lendpoint}. For the parameters above, we may take \(\bar L=0.6426066903715133\). We define
\begin{equation}\label{eq:validated-control-radius}
 \bar h=2\bar L\bar\gamma^2\bar\varepsilon,
 \qquad
 \rho_{\mathrm{cert}}
 =\dfrac{2\bar\gamma\bar\varepsilon}
 {1+\sqrt{1-\bar h}}.
\end{equation}
If \(\bar h<1\), we apply \Cref{thm:certificate} using the global derivative bound for the tanh problem and obtain a control \(c^*\in L^2(0,T;\R^P)\) such that
\(\F_P(c^*)=\mathbf{0}_{\R^{2P}}\), \(\norm{c^*-\bar c}_{L^2(0,T;\R^P)} \le\rho_{\mathrm{cert}}.\)
Thus, we use \(\rho_{\mathrm{cert}}\) to bound the distance from the computed control to a control that satisfies the prescribed terminal conditions exactly in the Galerkin system. We evaluate the indicator and the reported radius with outward rounding. For \(\beta=0\), interpret the quotient in \eqref{eq:Lendpoint} by continuity: \(e^{\kappa T}(\cosh(\kappa T)-1)/\kappa^2\to T^2/2\). The endpoint derivative is constant in this linear case, so we may take its Lipschitz constant to be zero. We provide the derivation of \eqref{eq:Lendpoint} in the Supplementary Material.

\section{Validated endpoint and variational integration}\label{sec:validated-endpoint}

The quantities \(\varepsilon\) and \(\gamma\) in \Cref{thm:certificate} depend on the nonlinear flow and the corresponding variational equation. We cannot rigorously validate either quantity merely by reporting adaptive-solver values with a percentage margin. Instead, we employ a posteriori residuals of the ordinary differential equations (ODEs).

We consider the one-dimensional reduced system with \(\alpha=0\), modes \(e_j(x)=\sqrt2\sin(j\pi x)\), and \(g(u)=\beta\tanh u\), where \(\beta\ge0\). We fix a control \(c\) and let \(Y_c\) denote its exact modal trajectory. The linearization matrix is \(A(Y)=A_0+DN(Y)\), with
\(DN(Y) = \begin{pmatrix} 0 & 0 \\ -DG_P\!\left(\Lambda_P^{-1/2}Y_1\right)\Lambda_P^{-1/2} & 0 \end{pmatrix}\),
so that \(A\!\left(Y_c(t)\right) = A_c(t)\) represents the linearization matrix in \eqref{eq:linearized-energy}. We define the controllability Gramian of \eqref{eq:linearized-energy} on \([0,t]\) by
\[
 W(t) = \int_0^t \Xi_c(t,s) B_E B_E^\top
 \Xi_c(t,s)^\top \,\mathrm{d}s.
\]
This Gramian satisfies \(W(0) = \mathbf{0}_{\R^{2P\times2P}}\) and coincides with the Gramian \eqref{eq:gramian} at \(t=T\). Because \(t\) appears both in the upper limit of integration and inside the integrand, we differentiate \(W(t)\) via Leibniz's rule and obtain the boundary term \(\Xi_c(t,t) B_E B_E^\top \Xi_c(t,t)^\top = B_E B_E^\top\), using \(\Xi_c(t,t) = I_{2P}\), and the integral term \(\displaystyle\int_0^t \partial_t\!\left[ \Xi_c(t,s) B_E B_E^\top \Xi_c^\top(t,s) \right] \mathrm{d}s\). Applying \eqref{eq:state-transition}, we simplify this integral to \(A(Y)W(t) + W(t)A(Y)^\top\). Consequently, we find that \(W\) satisfies the differential Lyapunov equation
\begin{equation}\label{eq:Wode}
 W' = A(Y)W + WA(Y)^\top + B_E B_E^\top, \qquad W(0) = \mathbf{0}_{\R^{2P\times2P}}.
\end{equation}

The enclosure procedure consists of three steps. First, we represent the numerical trajectory by an explicit interpolant. Next, we measure its residuals in \eqref{eq:energy-ode} and \eqref{eq:Wode}. Finally, we use a logarithmic-norm estimate to convert these residuals into a rigorous error bound. Accordingly, let \(\widehat Y\) and \(\widehat W\) be the exact continuous piecewise-cubic Hermite functions determined by the stored knot values and derivatives. The numerical solver supplies these data. The validation encloses the residuals of the resulting functions. We therefore include errors from coefficient construction and polynomial evaluation in the residual enclosures. For the stored piecewise-cubic control, the proof partition contains every control knot. Consequently, each proof cell lies within one control-polynomial interval, and the control is polynomial on that cell. We define the residuals \(r_Y\) and \(r_W\) by
\begin{align}
 r_Y &= \widehat{Y}' - A_0 \widehat{Y} - N\bigl(\widehat{Y}\bigr) - B_E c, \label{eq:rx}\\
 r_W &= \widehat{W}' - A\bigl(\widehat{Y}\bigr) \widehat{W} - \widehat{W} A\bigl(\widehat{Y}\bigr)^\top - B_E B_E^\top. \label{eq:rW}
\end{align}
By \eqref{eq:energy-ode} and \eqref{eq:Wode}, we observe that these residuals vanish identically for the exact pair \(\left(Y_c,W\right)\). For the interpolants, they measure the defects in the differential equations. The estimates below convert these defects, together with the initial discrepancies, into bounds for the state and Gramian errors. We combine all polynomial terms before interval evaluation. In particular, we evaluate \(\widehat{Y}'-A_0\widehat{Y}-B_Ec\) and \(\widehat{W}'-A_0\widehat{W}-\widehat{W}A_0^\top-B_EB_E^\top\) as centered polynomials on each cell, which preserves cancellations in the oscillatory linear wave dynamics.

Let \(0=t_0<t_1<\cdots<t_{N_{\rm cell}}=T\) be the proof partition, and set \(I_i=[t_i,t_{i+1}]\) for \(i=0,\ldots,N_{\rm cell}-1\). On each cell, we define the residuals almost everywhere. Let \(\rho_{Y,i}\) and \(\rho_{W,i}\) be nonnegative bounds satisfying
\[
 \esssup_{t\in I_i}\norm{r_Y(t)}_2 \le\rho_{Y,i},
 \qquad
 \esssup_{t\in I_i}\norm{r_W(t)}_F \le\rho_{W,i}.
\]
We evaluate the nonlinear spatial integrals on \((0,1)\) using the composite trapezoidal rule with \(n_x=4096\) subintervals. We set \(\Delta x=1/n_x\) and \(x_\ell=\ell\Delta x\). For \(\varphi\in C^2([0,1])\), we define
\(Q_{\Delta x}(\varphi) =\Delta x\left( \dfrac{\varphi(0)+\varphi(1)}{2} +\sum_{\ell=1}^{n_x-1}\varphi(x_\ell) \right).\)
The quadrature remainder satisfies
\begin{equation}\label{eq:spatial-quadrature-remainder}
 \left|
 \int_0^1\varphi(x)\,\mathrm{d}x
 -Q_{\Delta x}(\varphi)
 \right|
 \le
 \dfrac{(\Delta x)^2}{12}
 \norm{\varphi''}_{L^\infty(0,1)}.
\end{equation}

At each time, we set \(\widehat a=\Lambda_P^{-1/2}\widehat Y_1\) and write \(u=u_{\widehat a}=\sum_{\ell=1}^P\widehat a_\ell e_\ell\). For the integrands defining \(G_P(\widehat a)\), we use
\(\partial_x^2\left(g(u)e_j\right) = \left( g''(u)(\partial_xu)^2 +g'(u)\partial_x^2u \right)e_j +2g'(u)(\partial_xu)e_j' +g(u)e_j''.\)
For the integrands defining \(DG_P(\widehat a)\), we use
\begin{align*}
 \partial_x^2\left(g'(u)e_je_k\right)
 ={}&
 \left(
   g'''(u)(\partial_xu)^2
   +g''(u)\partial_x^2u
 \right)e_je_k \\
 &+2g''(u)(\partial_xu)
   \left(e_j'e_k+e_je_k'\right) \\
 &+g'(u)
   \left(e_j''e_k+2e_j'e_k'+e_je_k''\right).
\end{align*}
Thus, this quadrature remainder estimate for the Jacobian requires a bound on \(g'''\). For \(g(s)=\beta\tanh s\) with \(\beta\ge0\), we have
\[
 g'''(s)
 =-2\beta\left(1-\tanh^2s\right)\left(1-3\tanh^2s\right),
 \qquad
 \norm{g'''}_{L^\infty(\R)}\le2\beta.
\]

On each time cell \(I_i\), we bound these spatial second derivatives uniformly in \((t,x)\in I_i\times[0,1]\). We then enlarge each outward-rounded quadrature enclosure by the corresponding remainder from \eqref{eq:spatial-quadrature-remainder}. We propagate the resulting enclosures of \(G_P(\widehat a)\) and \(DG_P(\widehat a)\) through \eqref{eq:rx} and \eqref{eq:rW} when computing \(\rho_{Y,i}\) and \(\rho_{W,i}\). We use IEEE~754 binary64 arithmetic with rounding to nearest and unit roundoff $\mathbf u:=2^{-53}$. For an arithmetic expression $E$, we write $\operatorname{fl}(E)$ for the result of its specified floating-point evaluation, including rounding at each arithmetic operation. For each integer $n\ge1$ satisfying $n\mathbf u<1$, we define
\[
 \gamma^{\rm fl}_n:=\dfrac{n\mathbf u}{1-n\mathbf u}.
\]
Let $a,b\in\mathbb R^n$ have binary64-representable entries, interpreted as exact real numbers. If we evaluate their dot product by separately rounded multiplications and sequential summation, with no underflow or overflow, then
\[
 \left|\operatorname{fl}(a^\top b)-a^\top b\right|
 \le\gamma^{\rm fl}_n\sum_{j=1}^n|a_jb_j|.
\]
We evaluate the error bounds outward. An underflow contribution requires an additional absolute error allowance, and computations producing overflow cannot supply an accepted enclosure. We enclose quantities not exactly representable in binary floating point by outward-rounded interval arithmetic. Custom outward-rounded double-precision routines cover the remaining linear algebra operations. We do not use long-double computations as proof evidence.

Throughout the \(\tanh\) enclosure analysis, we assume that \(\beta\ge0\).
\begin{proposition}\label{prop:tubes}
We consider the one-dimensional Galerkin system on \(\Omega=(0,1)\) with \(\alpha=0\), \(e_j(x)=\sqrt{2}\sin(j\pi x)\) for \(j=1,\ldots,P\), and \(g(u)=\beta\tanh u\), where \(\beta\ge0\). We fix \(c\in L^2(0,T;\R^P)\). Let \(Y_c\) solve \eqref{eq:energy-ode} with the prescribed initial state, and let \(W\) solve \eqref{eq:Wode} along \(Y_c\), with \(W(0)=\mathbf{0}_{\R^{2P\times2P}}\).

Let \(0=t_0<t_1<\cdots<t_{N_{\rm cell}}=T\) be a time partition, and let \(I_i=[t_i,t_{i+1}]\) and \(\Delta t_i=t_{i+1}-t_i\). We suppose that \(\widehat Y:[0,T]\to\R^{2P}\) and \(\widehat W:[0,T]\to\R^{2P\times2P}\) are continuous and polynomial of degree at most three on each cell. We define their residuals by \eqref{eq:rx} and \eqref{eq:rW}, with derivatives taken almost everywhere.

We set
\[
 C_{\infty,P}
 =\left(\dfrac{2}{\pi^2}\sum_{j=1}^{P}\dfrac{1}{j^2}\right)^{1/2},
 \qquad \kappa=\dfrac{\beta}{2\pi},
 \qquad K_D=\dfrac{4\beta C_{\infty,P}}{3\sqrt{3}\,\pi},
\]
and, for \(\nu,\tau\ge0\), we define
\(\phi(\nu,\tau) = \begin{cases} \dfrac{e^{\nu\tau}-1}{\nu}, & \nu>0,\\ \tau, & \nu=0 \end{cases}\).
For \(i=0,\ldots,N_{\rm cell}\), we define the exact errors
\[
 \varepsilon_i^Y =\sup_{0\le t\le t_i}\norm{Y_c(t)-\widehat Y(t)}_2,
 \qquad
 \varepsilon_i^W =\norm{W(t_i)-\widehat W(t_i)}_F.
\]

For each \(i=0,\ldots,N_{\rm cell}-1\), we suppose that finite nonnegative constants \(\rho_{Y,i}\), \(\rho_{W,i}\), and \(\widehat W_i^{\max}\) satisfy
\[
\begin{aligned}
 \esssup_{t\in I_i}\norm{r_Y(t)}_2&\le\rho_{Y,i}, \qquad \esssup_{t\in I_i}\norm{r_W(t)}_F\le\rho_{W,i},\\
 \sup_{t\in I_i}\norm{\widehat W(t)}_F&\le\widehat W_i^{\max}.
\end{aligned}
\]
Then the exact errors satisfy
\begin{align}
 \varepsilon_{i+1}^Y
 &\le
 e^{\kappa\Delta t_i}\varepsilon_i^Y
 +\phi(\kappa,\Delta t_i)\rho_{Y,i},
 \label{eq:estate-rec}\\
 \varepsilon_{i+1}^W
 &\le
 e^{2\kappa\Delta t_i}\varepsilon_i^W
 +\phi(2\kappa,\Delta t_i)
 \left(
  \rho_{W,i}
  +2K_D\varepsilon_{i+1}^Y\widehat W_i^{\max}
 \right).\label{eq:W-rec}
\end{align}

Consequently, the terminal-state residual satisfies
\begin{equation}\label{eq:epsbar}
 \norm{\F_P(c)}_2
 \le
 \norm{\widehat Y(T)-Y_T}_2
 +\varepsilon_{N_{\rm cell}}^Y
 =:\bar\varepsilon.
\end{equation}
We suppose that a verified Cholesky computation certifies the positive definiteness of
\[
\dfrac{1}{2}\left( \widehat W(T)+\widehat W(T)^\top \right) -\underline\lambda_{\rm chol}I_{2P},
\]
for some \(\underline\lambda_{\rm chol}>0\); see \cite{Rump2010}. Then, we have
\begin{equation}\label{eq:lambdabar}
 \lambda_{\min}\left(W(T)\right)
 \ge
 \underline\lambda_{\rm chol}
 -\varepsilon_{N_{\rm cell}}^W
 =:\underline\lambda.
\end{equation}
If \(\underline\lambda>0\), the HUM right inverse satisfies
\begin{equation}\label{eq:validated-right-inverse-bound}
 \norm{R_c}_{ \mathcal L(\R^{2P},L^2(0,T;\R^P)) } \le \underline\lambda^{-1/2} =:\bar\gamma.
\end{equation}
The terminal residual bound \eqref{eq:epsbar} and the right-inverse estimate \eqref{eq:validated-right-inverse-bound}, which follows from the Gramian eigenvalue bound \eqref{eq:lambdabar}, supply the inputs \(\varepsilon\) and \(\gamma\) required by \Cref{thm:certificate}.
\end{proposition}

\begin{proof}
We set \(\mathcal E_Y=Y_c-\widehat Y\). Subtracting \eqref{eq:rx} from \eqref{eq:energy-ode} and applying the mean-value theorem to \(N\), we obtain
\begin{equation}\label{eq:err-state}
    \mathcal E_Y'
    =\left(A_0+\overline J\right)\mathcal E_Y-r_Y,
    \qquad
    \overline J
    =\int_0^1 DN\!\left(\widehat Y+s\mathcal E_Y\right)\,\mathrm ds.
\end{equation}
Since \( N(Y) =\left(\mathbf{0},-G_P\!\left(\Lambda_P^{-1/2}Y_1\right)\right)^\top,\) each derivative \(DN\), and hence \(\overline J\), has the block form
\[
    \overline J=\begin{pmatrix}0&0\\-X&0\end{pmatrix},
 \qquad
 \norm{X}_2\le\norm{DG_P}_2\,\norm{\Lambda_P^{-1/2}}_2\le\dfrac{\beta}{\pi},
\]
by the projected derivative bound proved in the Supplementary Material, \(\norm{g'}_{L^\infty(\R)}=\beta\), and \(\norm{\Lambda_P^{-1/2}}_2=\mu_1^{-1/2}=\pi^{-1}\). Since \(A_0\) is skew-symmetric, \(\mu_2\!\left(A_0+\overline J\right)=\mu_2\!\left(\overline J\right)\), and the symmetric part of \(\overline J\) is \(-\dfrac{1}{2}\begin{pmatrix}0&X^\top\\X&0\end{pmatrix}\), whose eigenvalues are \(\pm\dfrac{1}{2}\) times the singular values of \(X\), by the singular value decomposition of \(X\). The largest singular value is \(\norm{X}_2\), so \( \mu_2\!\left(A_0+\overline J\right)=\dfrac{1}{2}\norm{X}_2 \le\dfrac{\beta}{2\pi}=\kappa.\) Since the skew-symmetric part of a matrix does not contribute to its quadratic form, we have
\begin{equation}\label{eq:state-error-quadratic-form}
 \ip{\mathcal E_Y}{\left(A_0+\overline J\right)\mathcal E_Y}
 \le \mu_2\!\left(A_0+\overline J\right)\norm{\mathcal E_Y}_2^2
 \le \kappa\norm{\mathcal E_Y}_2^2.
\end{equation}
To estimate the norm without dividing by \(\norm{\mathcal E_Y}_2\), which may vanish, we fix \(\delta>0\) and define
\(q_\delta(t) =\left(\norm{\mathcal E_Y(t)}_2^2+\delta^2\right)^{1/2}.\)
The function \(q_\delta\) is absolutely continuous on each time cell. Using \eqref{eq:err-state}, the quadratic-form bound \eqref{eq:state-error-quadratic-form}, and the Cauchy--Schwarz inequality, we obtain, for almost every \(t\in I_i\),
\begin{align*}
 q_\delta'(t) &= \dfrac{ \ip{\mathcal E_Y(t)} {\left(A_0+\overline J(t)\right)\mathcal E_Y(t)} -\ip{\mathcal E_Y(t)}{r_Y(t)} }{q_\delta(t)}\\
 &\le \kappa\dfrac{\norm{\mathcal E_Y(t)}_2^2}{q_\delta(t)} +\dfrac{\norm{\mathcal E_Y(t)}_2}{q_\delta(t)} \norm{r_Y(t)}_2 \le \kappa q_\delta(t)+\rho_{Y,i}.
\end{align*}
Here, we used \(\kappa\ge0\), \(\norm{\mathcal E_Y(t)}_2\le q_\delta(t)\), and the residual bound on \(I_i\). Since
\(\phi(\nu,\tau) =\int_0^\tau e^{\nu(\tau-s)}\,\mathrm{d}s,\)
applying Gr\"onwall's inequality, we arrive at
\(q_\delta(t_i+\tau) \le e^{\kappa\tau}q_\delta(t_i) +\phi(\kappa,\tau)\rho_{Y,i}, \ 0\le\tau\le\Delta t_i.\)
Letting \(\delta\to0^+\) and using \(\norm{\mathcal E_Y(t_i)}_2\le\varepsilon_i^Y\), we conclude that
\(\norm{\mathcal E_Y(t_i+\tau)}_2 \le e^{\kappa\tau}\varepsilon_i^Y +\phi(\kappa,\tau)\rho_{Y,i}, \ 0\le\tau\le\Delta t_i.\)
Because \(\kappa\ge0\) and \(\rho_{Y,i}\ge0\), the right-hand side is nondecreasing in \(\tau\), and its value at \(\tau=\Delta t_i\) is at least \(\varepsilon_i^Y\). Therefore,
\(\varepsilon_{i+1}^Y = \max\left\{ \varepsilon_i^Y,\, \sup_{t\in I_i}\norm{\mathcal E_Y(t)}_2 \right\} \le e^{\kappa\Delta t_i}\varepsilon_i^Y +\phi(\kappa,\Delta t_i)\rho_{Y,i}.\)
We thereby establish \eqref{eq:estate-rec}.

We next set \(\mathcal E_W=W-\widehat W\). Subtracting \eqref{eq:rW} from \eqref{eq:Wode}, we obtain
\begin{align}
    \mathcal E_W'
    &=
    A(Y_c)\mathcal E_W+\mathcal E_W A(Y_c)^\top
    +\left[A(Y_c)-A(\widehat Y)\right]\widehat W
    +\widehat W\left[A(Y_c)-A(\widehat Y)\right]^\top
    -r_W. \label{eq:err-W}
\end{align}
For matrices \(Z,A\in\R^{2P\times2P}\), we set \(S=\dfrac{1}{2}\left(A+A^\top\right)\). Then
\(\ip{Z}{AZ+ZA^\top}_F = \operatorname{tr}\!\left(SZZ^\top\right) +\operatorname{tr}\!\left(SZ^\top Z\right) \le 2\lambda_{\max}(S)\norm{Z}_F^2,\)
because \(ZZ^\top\) and \(Z^\top Z\) are positive semidefinite and both have trace \(\norm{Z}_F^2\). Since \(A(Y_c)=A_0+DN(Y_c)\) has the same block structure as \(A_0+\overline J\), we again have \(\lambda_{\max}(S)\le\kappa\). Thus, the Lyapunov operator \(Z\mapsto A(Y_c)Z+ZA(Y_c)^\top\) has logarithmic norm at most \(2\kappa\) with respect to \(\norm{\cdot}_F\), which means
\begin{equation}\label{eq:gramian-error-quadratic-form}
 \ip{Z}{A(Y_c)Z+ZA(Y_c)^\top}_F \le 2\kappa\norm{Z}_F^2,
 \qquad Z\in\R^{2P\times2P}.
\end{equation}
For the inhomogeneous terms, \(A(Y_c)-A(\widehat Y)\) is the compression of multiplication by \(g'(u_{a_c})-g'(u_{\widehat a})\) composed with \(\Lambda_P^{-1/2}\). We write \(a_c=\Lambda_P^{-1/2}(Y_c)_1\) and \(\widehat a=\Lambda_P^{-1/2}\widehat Y_1\). Using the compression representation of \(DG_P\), we obtain
\[
\begin{aligned}
 \|DG_P(a_c)-DG_P(\widehat a)\|_2 &\le \|g'(u_{a_c})-g'(u_{\widehat a})\|_{L^\infty} \le \|g''\|_{L^\infty} \|u_{a_c}-u_{\widehat a}\|_{L^\infty}\\
 &\le \dfrac{4\beta}{3\sqrt3}\, C_{\infty,P}\, \|\Lambda_P^{1/2}(a_c-\widehat a)\|_2.
\end{aligned}
\]
Only the lower-left block of \(A\) depends on the state. Consequently,
\(\|A(Y_c)-A(\widehat Y)\|_2 \le \dfrac{1}{\pi} \|DG_P(a_c)-DG_P(\widehat a)\|_2 \le K_D\|\mathcal E_Y\|_2, K_D=\dfrac{4\beta C_{\infty,P}}{3\sqrt3\,\pi}\).
Using \(\norm{AB}_F\le\norm{A}_2\norm{B}_F\), we bound each of the two matrix products in \eqref{eq:err-W} on \(I_i\) by \(K_D\varepsilon_{i+1}^Y\widehat W_i^{\max}\). Here, we use \(\varepsilon_{i+1}^Y\) because it bounds the state error throughout \(I_i\). We set \(\Delta A=A(Y_c)-A(\widehat Y)\). The two perturbation terms satisfy
\(\|\Delta A\,\widehat W+\widehat W\,\Delta A^\top\|_F \le 2\|\Delta A\|_2\|\widehat W\|_F \le 2K_D\varepsilon_{i+1}^Y \widehat W_i^{\max},\)
throughout \(I_i\). Therefore, using the logarithmic-norm bound \eqref{eq:gramian-error-quadratic-form} and the regularization argument used to prove \eqref{eq:estate-rec}, we deduce that
\[
\begin{aligned}
 \|\mathcal E_W(t_{i+1})\|_F \le{}& e^{2\kappa \Delta t_i}\|\mathcal E_W(t_i)\|_F +\phi(2\kappa,\Delta t_i)
 \left(\rho_{W,i}+2K_D\varepsilon_{i+1}^Y\widehat W_i^{\max}\right),
\end{aligned}
\]
which is the claimed recurrence.

By the endpoint identity \eqref{eq:retained-endpoint}, \(\F_P(c)=Y_c(T)-Y_T\). Applying the triangle inequality, we obtain
\(\norm{\F_P(c)}_2 \le \norm{\widehat Y(T)-Y_T}_2+\norm{\mathcal E_Y(T)}_2,\)
and thus establish \eqref{eq:epsbar}. To prove \eqref{eq:lambdabar}, we recall that \(W(T)\) is symmetric. Since symmetrization is the orthogonal projection onto the space of symmetric matrices with respect to the Frobenius inner product,
\(\norm{ W(T) -\dfrac{1}{2}\left(\widehat W(T)+\widehat W(T)^\top\right) }_F \le \norm{\mathcal E_W(T)}_F \le \varepsilon_{N_{\rm cell}}^W.\)
Using Weyl's inequality and \(\norm{\cdot}_2\le\norm{\cdot}_F\), we obtain the eigenvalue bound \eqref{eq:lambdabar}. Finally, if \(\underline\lambda>0\), then \(W(T)\succ0\). Since the function \(\lambda\mapsto\lambda^{-1/2}\) is decreasing on \((0,\infty)\), we apply the right-inverse formula \eqref{eq:gamma} to obtain
\(\gamma_P(c) = \lambda_{\min}\!\left(W(T)\right)^{-1/2} \le \underline\lambda^{-1/2} = \bar\gamma.\)
\end{proof}

For the numerical validation, the exact errors \(\varepsilon_i^Y\) and \(\varepsilon_i^W\) are not available, so we replace them with computable quantities \(b_i^Y\) and \(b_i^W\) that are guaranteed to be at least as large. We start from a pair satisfying
\(b_0^Y\ge\norm{Y_c(0)-\widehat Y(0)}_2, \ b_0^W\ge\norm{\widehat W(0)}_F.\)
We propagate these majorants using interval arithmetic with outward rounding. In the formulas below, \(\operatorname{up}(x)\) and \(\operatorname{down}(x)\) denote certified upper and lower bounds for the exact value of the entire expression \(x\), respectively. Thus, these operations account for every intermediate arithmetic operation and function evaluation.

For each time cell, we first compute
\(b_{i+1}^Y = \operatorname{up}\left( e^{\kappa\Delta t_i}b_i^Y +\phi(\kappa,\Delta t_i)\rho_{Y,i} \right)\)
and then compute
\(b_{i+1}^W = \operatorname{up}\left( e^{2\kappa\Delta t_i}b_i^W +\phi(2\kappa,\Delta t_i) \left( \rho_{W,i} +2K_D b_{i+1}^Y\widehat W_i^{\max} \right) \right).\)
We use certified upper bounds for all nonnegative constants and residual bounds in these expressions. The evaluation of \(\phi\) includes its limiting value \(\phi(0,\tau)=\tau\). By induction in \eqref{eq:estate-rec} and \eqref{eq:W-rec}, these computations ensure
\(\varepsilon_i^Y\le b_i^Y\), \(\varepsilon_i^W\le b_i^W \) for \(i=0,\ldots,N_{\rm cell}.\)

For the numerical certificate, we choose the bounds in \eqref{eq:epsbar} and \eqref{eq:lambdabar} as
\[
\bar\varepsilon = \operatorname{up}\left( \norm{\widehat Y(T)-Y_T}_2+b_{N_{\rm cell}}^Y \right),
\qquad \underline\lambda = \operatorname{down}\left( \underline\lambda_{\rm chol}-b_{N_{\rm cell}}^W \right).
\]
The norm evaluation includes all rounding errors in forming \(\widehat Y(T)-Y_T\). Here, \(\underline\lambda_{\rm chol}\) is the shift certified by the verified Cholesky computation. Consequently,
\(\norm{\F_P(c)}_2\le\bar\varepsilon, \ \lambda_{\min}\left(W(T)\right)\ge\underline\lambda.\)
If \(\underline\lambda>0\), we compute
\(\bar\gamma = \operatorname{up}\left(\underline\lambda^{-1/2}\right),\)
which guarantees \(\norm{R_c}\le\bar\gamma\) in the operator norm specified in \Cref{prop:tubes}.

Let \(\bar L\) be a certified upper bound for \(L_P^{\mathrm{end}}\). We evaluate the indicator in \eqref{eq:validated-control-radius} upward:
\(\bar h = \operatorname{up}\left( 2\bar L\bar\gamma^2\bar\varepsilon \right).\)
If \(\bar h<1\), we obtain a certified lower bound for the denominator of the radius formula by computing
\(d_{\rm low} = \operatorname{down}\left( 1+\sqrt{1-\bar h} \right)>0.\)
We then report the upward-rounded radius
\(\rho_{\rm cert} = \operatorname{up}\left( \dfrac{2\bar\gamma\bar\varepsilon}{d_{\rm low}} \right).\)
For the present tanh problem, the derivative bound holds globally, so the neighborhood condition imposes no additional restriction. Therefore, if \(\underline\lambda>0\) and \(\bar h<1\), we apply \Cref{thm:certificate} to obtain a control \(c^*\) satisfying
\(\F_P(c^*)=\mathbf{0}_{\R^{2P}}, \ \norm{c^*-c}_{L^2(0,T;\R^P)} \le\rho_{\rm cert}.\)

For the three \(P=8\), \(\beta=1\) candidates in \Cref{tab:validated}, the nominal terminal errors are of order \(10^{-9}\), whereas the validated bounds satisfy \(\bar\varepsilon\le4.5\times10^{-2}\). The certificate uses these validated bounds, which include the trajectory enclosure error. For each of these three candidates, the verified Gramian lower bound satisfies \(\underline\lambda>0\) and the upward-rounded indicator satisfies \(\bar h<1\). Thus, \Cref{thm:certificate} guarantees an exact Galerkin steering control within the corresponding \(L^2(0,T;\R^P)\) distance \(\rho_{\rm cert}\) of the computed candidate. In the refinement study in \Cref{tab:indicator-refinement}, refining the validation time partition reduces the certified radii while keeping the three candidate controls fixed.

\section{Validated cancellation-preserving Galerkin response}\label{sec:response}

Whereas the endpoint certificate applies locally near a particular candidate, the following contraction result for source freezing is candidate-independent.

For \(\Omega=(0,1)\), \(\alpha=0\), and \(g(u)=\beta\tanh u\) with \(\beta\ge0\), the optimal Lipschitz constant from \(H_0^1(\Omega)\), equipped with the gradient norm, into \(L^2(\Omega)\) is \(L_g=\beta/\pi\). Since \(\|\mathcal B\|=1\) and \(\|\mathscr W_T\|\le T\), we have \(\Gamma_T\ge T^{-1/2}\) and therefore
\begin{equation}\label{eq:continuous-coarse-contraction-lower-bound}
 q_\infty\ge\dfrac{2T\beta}{\pi}.
\end{equation}
At \(T=2.5\), the lower bound \eqref{eq:continuous-coarse-contraction-lower-bound} exceeds one for every reported \(\beta\ge1\), so the coarse estimate does not establish contraction.

We instead estimate the endpoint-corrected displacement response of the finite-dimensional system directly. This response combines the source contribution and its endpoint correction before taking the operator norm. In the modal \(L^2\) norm defined below, we invoke \Cref{prop:chi} to bound
\(\ell_g\chi_{8,2.5} \le 0.160561\,\beta<1, \ 0\le\beta\le6.\)

Throughout this section, we work in \(\mathcal Z_P:=L^2\!\left(0,T;\R^P\right)\) with \(\norm{s}_{\mathcal Z_P}^2=\displaystyle\int_0^T\norm{s(t)}_{\R^P}^2\,\mathrm dt\), where \(\norm{\cdot}_{\R^P}\) is the Euclidean norm of \(\R^P\).

We fix \(s\in\mathcal Z_P\). We consider \(a''+\Lambda_Pa+s+M_Pc=\mathbf{0}_{\R^P}\) with the prescribed initial and terminal displacement and velocity. We minimize the Euclidean modal cost \(\|c\|_{\mathcal Z_P}^2=\displaystyle\int_0^T\|c(t)\|_2^2\,\mathrm dt\) over all controls satisfying these endpoint constraints.

This minimizer exists and is unique. Indeed, since \(M_P\) is invertible, we establish feasibility by choosing a cubic displacement \(a\) with the prescribed endpoint values and derivatives and setting \(c=-M_P^{-1}\left(a''+\Lambda_Pa+s\right)\). The endpoint depends continuously and affinely on \(c\), so the feasible controls form a nonempty closed affine subset of \(\mathcal Z_P\). Applying the Hilbert projection theorem, we obtain a unique minimum-norm control.

We write \(\mathcal H_P(s)\) for the displacement associated with this control. Solving the affine initial-value problem, we find \(\mathcal H_P(s)\in H^2(0,T;\R^P)\subset\mathcal Z_P\). From the Gramian formula below, we see that the selected control, and hence \(\mathcal H_P(s)\), depends continuously and affinely on \(s\). We define
\(\mathcal T_P(a):=\mathcal H_P\!\left(G_P(a)\right), \ \mathscr H_{P,T}(s):= \mathcal H_P(s)-\mathcal H_P(\mathbf{0}_{\mathcal Z_P}).\)
Thus, \(\mathscr H_{P,T}\) is a bounded linear operator on \(\mathcal Z_P\), and we set
\begin{equation}\label{eq:chi}
  \chi_{P,T}:=
  \norm{\mathscr H_{P,T}}_{\mathcal L(\mathcal Z_P)}.
\end{equation}
Whenever the contraction condition below holds, \(\mathcal T_P\) has a unique fixed point, which we denote by \(a^*\). For a matrix-valued function \(L:(0,T)^2\to\R^{P\times P}\), we use the notation
\(\|L\|_{L^2_F} = \left( \int_0^T\int_0^T \|L(t,\tau)\|_F^2\,\mathrm{d}\tau\,\mathrm{d}t \right)^{1/2}, \ \|L(t,\tau)\|_F^2 = \sum_{j,k=1}^{P}|L_{jk}(t,\tau)|^2.\)
Here, the subscript \(F\) indicates the Frobenius matrix norm.

\begin{theorem}\label{thm:response-contraction}
We assume that \(g\colon\R\to\R\) satisfies
\begin{equation}\label{eq:response-g-lipschitz}
 |g(r)-g(s)|\le\ell_g|r-s|, \qquad r,s\in\R.
\end{equation}
Then, \(\mathcal T_P\) is Lipschitz continuous on \(\mathcal Z_P\) with constant at most \(\ell_g\chi_{P,T}\). If \(\ell_g\chi_{P,T}<1\), then \(\mathcal T_P\) has a unique fixed point \(a^*\in\mathcal Z_P\), and the fixed-point iteration converges to it in \(\mathcal Z_P\) from every \(a^{(0)}\in\mathcal Z_P\). This fixed point has an \(H^2(0,T;\R^P)\) representative and, together with its selected control \(c^*\), satisfies \eqref{eq:retained} and all prescribed endpoint conditions. Moreover, for every \(a\in\mathcal Z_P\),
\begin{equation}\label{eq:aposteriori}
    \norm{a-a^*}_{\mathcal Z_P}
    \le
    \dfrac{\norm{a-\mathcal T_P(a)}_{\mathcal Z_P}}
    {1-\ell_g\chi_{P,T}}.
\end{equation}
\end{theorem}

\begin{proof}
For \(a\in\mathcal Z_P\), we interpret \(G_P(a)\) pointwise by \(G_P(a)(t)=G_P(a(t))\) for almost every \(t\in(0,T)\). Using Bessel's inequality, the scalar Lipschitz bound \eqref{eq:response-g-lipschitz}, and orthonormality, we obtain
\(\|G_P(a(t))-G_P(b(t))\|_2 \le \|g(u_{a(t)})-g(u_{b(t)})\|_{L^2(\Omega)} \le \ell_g\|u_{a(t)}-u_{b(t)}\|_{L^2(\Omega)} =\ell_g\|a(t)-b(t)\|_2\).
Squaring and integrating in time, we obtain
\begin{equation}\label{eq:GP-lip}
  \norm{G_P(a)-G_P(b)}_{\mathcal Z_P}
  \le \ell_g\norm{a-b}_{\mathcal Z_P},
  \ a,b\in\mathcal Z_P.
\end{equation}
In particular,
\(\|G_P(a)\|_{\mathcal Z_P} \le \ell_g\|a\|_{\mathcal Z_P} +\sqrt T\,\|G_P(\mathbf{0}_{\R^P})\|_2,\)
so \(G_P\) maps \(\mathcal Z_P\) into itself. Consequently, \(\mathcal T_P\) is a well-defined self-map of the complete space \(\mathcal Z_P\).

Next, \(\mathcal H_P\) depends on its argument affinely, and \(\mathscr H_{P,T}\) is its linear part, so
\(\mathcal T_P(a)-\mathcal T_P(b) =\mathcal H_P\!\left(G_P(a)\right)-\mathcal H_P\!\left(G_P(b)\right) =\mathscr H_{P,T}\!\left(G_P(a)-G_P(b)\right).\)
Taking norms and using \eqref{eq:chi} and \eqref{eq:GP-lip}, we conclude that \(\mathcal T_P\) is Lipschitz continuous on \(\mathcal Z_P\) with constant at most \(\ell_g\chi_{P,T}\). If \(\ell_g\chi_{P,T}<1\), then \(\mathcal T_P\) is a contraction, and the Banach fixed-point theorem lets us conclude that it has the unique fixed point \(a^*\) and that the iterates converge to it. Since \(a^*=\mathcal H_P(G_P(a^*))\) in \(\mathcal Z_P\), we choose the \(H^2\) representative supplied by \(\mathcal H_P\). Its associated control \(c^*\) satisfies \((a^*)''+\Lambda_Pa^*+G_P(a^*)+M_Pc^* =\mathbf{0}_{\R^P},\) almost everywhere, together with the prescribed initial and terminal displacements and velocities.

Let \(a\in\mathcal Z_P\). Inserting \(\mathcal T_P(a)\) and using \(a^*=\mathcal T_P(a^*)\) with the Lipschitz continuity of \(\mathcal{T}_P\), we obtain \(\norm{a-a^*}_{\mathcal Z_P} \le\norm{a-\mathcal T_P(a)}_{\mathcal Z_P} +\ell_g\chi_{P,T}\norm{a-a^*}_{\mathcal Z_P}\), and therefore establish \eqref{eq:aposteriori}.
\end{proof}

For the numerical certificate, we take \(\Omega=(0,1)\), \(\alpha=0\), \(T=2.5\), and the boundary collar \(\omega_\delta=(0,\delta)\cup(1-\delta,1)\) with \(\delta=0.3\).

We write \(y=(a,a')^\top\) and set
\[
 \mathsf A=\begin{pmatrix}0&I_P\\-\Lambda_P&0\end{pmatrix},\qquad
 \mathsf B=\begin{pmatrix}0\\-M_P\end{pmatrix},\qquad
 F_s=\begin{pmatrix}0\\-I_P\end{pmatrix},\qquad
 \pi_1=\begin{pmatrix}I_P&0\end{pmatrix}.
\]
We also introduce
\[
 W_T=\int_0^T e^{\mathsf A(T-r)}\mathsf B\mathsf B^\top
 e^{\mathsf A^\top(T-r)}\,\mathrm dr,
 \qquad
 C(t)=\int_0^t e^{\mathsf A(t-r)}\mathsf B\mathsf B^\top
 e^{\mathsf A^\top(T-r)}\,\mathrm dr.
\]
Since \(M_P\) is invertible,
\(\begin{pmatrix}\mathsf B&\mathsf A\mathsf B\end{pmatrix} = \begin{pmatrix}\mathbf{0}&-M_P\\-M_P&\mathbf{0}\end{pmatrix}\)
has rank \(2P\). Hence, the linear pair \((\mathsf A,\mathsf B)\) is controllable, and its Gramian \(W_T\) is positive definite for every \(T>0\). Here, \(W_T\) is expressed in the unscaled coordinates \(y=(a,a')^\top\).

Let \(y_0\) and \(y_T\) denote the prescribed endpoint states, and set
\(d_0=y_T-e^{\mathsf AT}y_0, \ d_s=\int_0^T e^{\mathsf A(T-\tau)}F_s\,s(\tau)\,\mathrm d\tau.\)
For source \(s\), the unique minimum-norm control is
\(c_s(t)= \mathsf B^\top e^{\mathsf A^\top(T-t)} W_T^{-1}(d_0-d_s).\)
Thus, its difference from the zero-source control is
\(c_s(t)-c_0(t)= -\mathsf B^\top e^{\mathsf A^\top(T-t)} W_T^{-1}d_s.\)
Subtracting the corresponding trajectories and substituting this control difference, we obtain
\((\mathscr H_{P,T}s)(t)=\int_0^T K(t,\tau)s(\tau)\,\mathrm d\tau,\)
\(K(t,\tau)=\pi_1\!\left[ \mathbf 1_{\{\tau\le t\}}e^{\mathsf A(t-\tau)}F_s -C(t)W_T^{-1}e^{\mathsf A(T-\tau)}F_s\right].\)
The correction term is smooth on \([0,T]^2\). The jump of the direct term across \(t=\tau\) is \(\pi_1F_s=\mathbf{0}_{\R^{P\times P}}\), so \(K\) is continuous across the diagonal. Its piecewise classical first derivatives are bounded and coincide with its weak derivatives. Hence,
\(K\in H^1\!\left((0,T)^2;\R^{P\times P}\right).\)

Let \(\Pi_{N_t}\) be the orthogonal projection onto the functions that are constant on each of \(N_t\) uniform time cells, and let \(\Delta=T/N_t\). We write \(\mathscr H=\mathscr H_{P,T}\). Let \(\Pi_t\) and \(\Pi_\tau\) denote cellwise averaging of the kernel in its first and second variables, respectively. The kernel of \(\Pi_{N_t}\mathscr H\Pi_{N_t}\) is \(\Pi_t\Pi_\tau K\). Using
\(K-\Pi_t\Pi_\tau K =(I-\Pi_t)K+\Pi_t(I-\Pi_\tau)K,\)
the contractivity of \(\Pi_t\) in \(L^2\), and the cellwise Poincar\'e--Wirtinger inequality, we obtain
\begin{equation}\label{eq:response-projection-error}
\begin{aligned}
\|\mathscr H-\Pi_{N_t}\mathscr H\Pi_{N_t}\|
_{\mathcal L(\mathcal Z_P)}
&\le \|K-\Pi_t\Pi_\tau K\|_{L^2_F}
\le \dfrac{\Delta}{\pi}
\left(
\|\partial_tK\|_{L^2_F}
+\|\partial_\tau K\|_{L^2_F}
\right).
\end{aligned}
\end{equation}
Here, the first inequality bounds the operator norm by the Hilbert--Schmidt norm. For \(P=8\), \(T=2.5\), and \(\delta=0.3\), the verified derivative bounds are
\begin{equation}\label{eq:response-kernel-derivative-bounds}
 \|\partial_tK\|_{L^2_F}\le2.618750484691, \qquad \|\partial_\tau K\|_{L^2_F}\le2.654926696062.
\end{equation}
Taking \(N_t=600\), we have \(\Delta=T/N_t=1/240\). Substituting \eqref{eq:response-kernel-derivative-bounds} into \eqref{eq:response-projection-error}, we obtain
\begin{align}
 \|\mathscr H-\Pi_{600}\mathscr H\Pi_{600}\|_
 {\mathcal L(\mathcal Z_8)}
 &\le
 \dfrac{2.618750484691+2.654926696062}{240\pi} <0.006994431597.\label{eq:response-projection-remainder}
\end{align}

Let \(I_1,\ldots,I_{N_t}\) be the time cells and set \(\psi_i=\Delta^{-1/2}\mathbf1_{I_i}\). In the orthonormal basis formed by these scalar functions and the standard coordinate vectors of \(\R^P\), the projected operator has a matrix \(H_{N_t}\in\R^{N_tP\times N_tP}\) with blocks
\((H_{N_t})_{ij} =\dfrac{1}{\Delta} \int_{I_i}\int_{I_j}K(t,\tau)\,\mathrm d\tau\,\mathrm dt.\)
Consequently,
\(\|\Pi_{N_t}\mathscr H\Pi_{N_t}\| _{\mathcal L(\mathcal Z_P)} =\|H_{N_t}\|_2.\)
We put \(B_{N_t}=H_{N_t}^\top H_{N_t}\). Let \(Q\) have \(r\) orthonormal columns, and let \(\theta_1\ge\cdots\ge\theta_r\) denote the eigenvalues of \(Q^\top B_{N_t}Q\). We order the eigenvalues of \(B_{N_t}\), counted with multiplicity, as \(\lambda_1(B_{N_t})\ge\cdots\ge\lambda_{N_tP}(B_{N_t})\ge0\). By eigenvalue interlacing, we obtain
\begin{equation}\label{eq:trace-ritz-upper-bound}
 \|H_{N_t}\|_2^2\le\operatorname{tr}(B_{N_t})-\sum_{j=2}^r\theta_j.
\end{equation}
Since \(\theta_j\le\lambda_j(B_{N_t})\), we subtract the captured trace from \(\operatorname{tr}(B_{N_t})\) to bound \(\lambda_1(B_{N_t})\), obtaining \eqref{eq:trace-ritz-upper-bound}. We use a numerical singular value decomposition to supply a trial matrix \(V\in\R^{N_tP\times r}\), whose stored entries we interpret as exact real numbers. We verify that its Gram matrix \(G=V^\top V\) is positive definite. The exact matrix \(Q=VG^{-1/2}\) then has orthonormal columns and spans the same trial subspace. Thus, the Ritz values are the eigenvalues of
\(G^{-1/2}(V^\top B_{N_t}V)G^{-1/2},\)
or equivalently the generalized eigenvalues of \((V^\top B_{N_t}V,G)\).

The verified computation supplies the following bounds:
\begin{equation}\label{eq:trace-ritz-certified-inputs}
 \operatorname{tr}(B_{N_t})\le\overline\tau, \qquad \underline\sigma \le\sum_{j=2}^r\theta_j.
\end{equation}
Using the input bounds \eqref{eq:trace-ritz-certified-inputs}, we apply the trace--Ritz estimate \eqref{eq:trace-ritz-upper-bound} to obtain
\begin{equation}\label{eq:trace-ritz-norm-bound}
 \|H_{N_t}\|_2 \le\left(\overline\tau-\underline\sigma\right)^{1/2}.
\end{equation}
We evaluate the right-hand side with upward rounding.

Let \(X\) be a numerical approximation to \(W_T^{-1}\). Using outward enclosures of \(W_T\) and the matrix product, we compute
\begin{equation*}
 \|I-W_TX\|_\infty\le\overline q_X<1, \qquad \|X\|_\infty\le\overline r_X.
\end{equation*}
Writing \(E=I-W_TX\), we have \(W_T^{-1}=X(I-E)^{-1}\). Expanding the inverse by the Neumann series, we obtain
\begin{equation*}
 \|W_T^{-1}-X\|_\infty \le \dfrac{\overline r_X\,\overline q_X} {1-\overline q_X}.
\end{equation*}
We evaluate this expression with a certified upper bound, using a positive lower bound for its denominator.

For \(P=8\), \(T=2.5\), and \(\delta=0.3\), the kernel of \(\mathscr H_{P,T}\) admits a closed form in terms of matrix exponentials and the controllability Gramian of the selected linear modes. In the computer-assisted proof, we discretize the time variables by projecting them onto the piecewise-constant functions associated with a grid of \(N_t=600\) cells: we control the projection remainder with cellwise Poincar\'e--Wirtinger inequalities and bound the projected operator using a trace--Ritz argument. Let \(\widehat H_{N_t}\) denote the assembled numerical matrix, with its stored entries interpreted as exact real numbers. Applying the verified trace--Ritz bound \eqref{eq:trace-ritz-norm-bound} to \(\widehat H_{N_t}\) in place of \(H_{N_t}\), that is, with \(B_{N_t}\) replaced by \(\widehat H_{N_t}^\top\widehat H_{N_t}\), and accounting for assembly error, we obtain
\begin{equation}\label{eq:response-matrix-enclosures}
 \|\widehat H_{N_t}\|_2\le\eta_{\rm Ritz}, \qquad \|H_{N_t}-\widehat H_{N_t}\|_2\le\eta_{\rm arith}.
\end{equation}
If \(\eta_{\rm proj}\) bounds the projection remainder, we combine this bound with \eqref{eq:response-matrix-enclosures} to obtain
\begin{equation*}
 \chi_{P,T} \le \eta_{\rm Ritz}+\eta_{\rm arith}+\eta_{\rm proj}.
\end{equation*}
The trace--Ritz calculation must include verified bounds for the quadratic forms of \(\widehat H_{N_t}\). The assembly-error bound must account for the approximation of the exact cell integrals, including the verified Gramian inverse and the arithmetic used to assemble the matrix.

\begin{proposition}\label{prop:chi}
For \(\Omega=(0,1)\), \(\alpha=0\), \(P=8\), \(T=2.5\), \(\omega=\omega_{0.3}\), and the Euclidean modal cost, we have
\(\chi_{8,2.5} \le0.160560155027<0.160561<\dfrac{1}{6}.\)
Consequently, for \(g(u)=\beta\tanh u\), the exact finite-mode source-frozen map is globally contractive for \(0 \le \beta\le6\).
\end{proposition}

\begin{proof}
We obtain the outward upper bounds in Table~\ref{tab:chi-components} from the outward-rounded computation described in the Supplementary Material. We justify its projection-remainder entry by \eqref{eq:response-projection-remainder}.
\begin{table}[ht]
\centering
\begin{tabular}{l c}
\toprule
validated component & outward upper bound\\
\midrule
projected trace--Ritz term & 0.152429487261\\
projected arithmetic enclosure & 0.001136236169\\
projection remainder & 0.006994431597\\
total \(\chi_{8,2.5}\) & 0.160560155027\\
\bottomrule
\end{tabular}

\caption{Outward-enclosed terms in the finite-mode response certificate.}
\label{tab:chi-components}
\end{table}
We add the three displayed component bounds to obtain
\(\chi_{8,2.5} \le0.160560155027<0.160561<\dfrac{1}{6}.\)
For \(g(u)=\beta\tanh u\) with \(\beta\ge0\), the scalar Lipschitz constant is \(\ell_g=\beta\). Hence, for \(0\le\beta\le6\),
\(\ell_g\chi_{8,2.5} \le6(0.160560155027) =0.963360930162<1.\)
The contraction conclusion follows from \Cref{thm:response-contraction}.
\end{proof}

The enclosure in \Cref{prop:chi} and the endpoint certificates in \Cref{sec:validated-endpoint} are complementary. With the first, we validate a global operator norm and establish candidate-independent convergence for the finite-mode source-frozen map. With the second, we validate a local exact-control neighborhood around controls produced by every method, including methods to which no contraction theorem applies.

\section{Neural--Spectral candidate construction}\label{sec:neural}

The Neural--Spectral formulation exploits the exact spatial eigenstructure of the wave operator and uses a neural approximation only for the temporal modal coefficients. Let \(N_{\rm w}\in\mathbb N\) be the common hidden-layer width and let \(\theta =(v^a_{jm},w^a_{jm},b^a_{jm},v^c_{jm},w^c_{jm},b^c_{jm})_{j,m}\in\R^{6PN_{\rm w}}\) collect all real network parameters. We write \(a_\theta=(a_{j,\theta})_{j=1}^P\) and \(c_\theta=(c_{j,\theta})_{j=1}^P\), and define
\(\widetilde u_\theta=\sum_{j=1}^P a_{j,\theta}e_j\), \(\widetilde\Phi_\theta=(\widetilde u_\theta,\partial_t\widetilde u_\theta)\), \(\widetilde f_\theta=\sum_{j=1}^P c_{j,\theta}e_j.\)
The applied control is \(\mathbf1_\omega\widetilde f_\theta\). For each selected mode \(j\), we represent the state coefficients \(a_{j,\theta}\) and control coefficients \(c_{j,\theta}\) by independent one-hidden-layer \(\tanh\) networks of width \(N_{\rm w}\) as follows:
\begin{equation}\label{eq:neural-ansatz}
\begin{aligned}
 a_{j,\theta}(t)&=\sum_{m=1}^{N_{\rm w}}v^a_{jm}\tanh\!\left(w^a_{jm}t+b^a_{jm}\right),\\
 c_{j,\theta}(t)&=\sum_{m=1}^{N_{\rm w}}v^c_{jm}\tanh\!\left(w^c_{jm}t+b^c_{jm}\right).
\end{aligned}
\end{equation}
The sine basis enforces the homogeneous spatial Dirichlet condition exactly. The initial and terminal conditions are imposed through penalties in the training objective. For \(N_c\) equally spaced interior nodes, we set \(t_q=qT/(N_c+1)\), \(q=1,\ldots,N_c\), and define
\begin{align*}
 e_{\rm IC}(\theta) &=\left\|\mathsf E_P\binom{a_\theta(0)-a_0}{a'_\theta(0)-a_1}\right\|_2, \qquad e_{\rm TC}(\theta) =\left\|\mathsf E_P\binom{a_\theta(T)-a_T}{a'_\theta(T)-a'_T}\right\|_2,\\
 \mathcal E_{\rm end}(\theta)&=e_{\rm IC}(\theta)^2+e_{\rm TC}(\theta)^2.
\end{align*}
We evaluate \(a'_\theta\) and \(a''_\theta\) by automatic differentiation of the network. This avoids finite-difference truncation error, although floating-point evaluation errors remain. A numerical approximation of \(G_P(a_\theta)\) introduces a separate residual-evaluation error. For collocation nodes \(t_1,\dots,t_{N_c}\in(0,T)\), we minimize
\begin{equation}\label{eq:neural-objective}
\begin{aligned}
 \mathcal J(\theta)
 &=\dfrac{T}{N_c+1}\sum_{q=1}^{N_c}
 \norm{a''_\theta(t_q)+\Lambda_Pa_\theta(t_q)
 +G_P\!\left(a_\theta(t_q)\right)+M_Pc_\theta(t_q)}_2^2 +500\,\mathcal E_{\mathrm{end}}(\theta).
\end{aligned}
\end{equation}
The endpoint weight \(500\) is fixed in advance for the one-dimensional experiments. The two-dimensional modal objective in \Cref{sec:results} has the same form, with endpoint weight \(100\). The first term in \eqref{eq:neural-objective} measures the residual at the prescribed interior nodes. A bound for the continuous residual requires control of the error between this sum and the time integral, as in \Cref{cor:discrete-neural-reliability}.

After training, let \(\widetilde c\) denote the control function used for forward validation and subsequent correction. We integrate the Galerkin equation from the prescribed initial data and evaluate its terminal error independently of the network trajectory. Because we apply the endpoint certificate to the independently validated control, its validity does not depend on convergence of the training optimizer.

Allowing the width \(N_{\rm w}\) to vary, the temporal network class is dense in \(H^2(0,T)\) for the displacement coefficients and in \(L^2(0,T)\) for the control coefficients. A continuous-time residual estimate then bounds the difference between a proposed trajectory and the forward trajectory driven by its extracted control. Finally, under the local regularity and right-invertibility assumptions of \Cref{thm:eventual-neural-certificate}, controls converging to an exact finite-dimensional control eventually satisfy the hypotheses of the endpoint certificate.
\begin{assumption}\label{ass:activation}
The scalar activation \(\sigma\in C^2(\R)\) is bounded and nonconstant. It acts componentwise on vector arguments. The implemented choice is \(\sigma=\tanh\).
\end{assumption}

\begin{lemma}\label{lem:temporal-uap}
Under \Cref{ass:activation}, the shallow-network class
\[
 \mathcal N_\sigma
 =\left\{\sum_{m=1}^{N_{\rm w}}v_m\sigma(w_mt+b_m):
 N_{\rm w}\in\mathbb N,\ v_m,w_m,b_m\in\R\right\}
\]
is dense in \(H^2(0,T)\) and, consequently, in \(L^2(0,T)\).
\end{lemma}

\begin{proof}
The interval \((0,T)\) has the segment property, so the restrictions of functions in \(C_c^\infty(\R)\) are dense in \(H^2(0,T)\) by \cite[Theorem~3.18]{Adams1975}. We fix \(a\in H^2(0,T)\), \(\varepsilon>0\) and choose \(q\in C_c^\infty(\R)\) such that
\(\norm{a-q}_{H^2(0,T)}<\varepsilon/2.\)
The hypotheses in \Cref{ass:activation} are precisely those used in the bounded, nonconstant \(C^m\) case of \cite[Theorem~3]{Hornik1991}. Taking \(m=2\), we obtain simultaneous uniform approximation of the function and its first two derivatives on the compact interval. Thus, by that theorem, we choose \(n\in\mathcal N_\sigma\) such that
\(\max_{0\le k\le2}\norm{q^{(k)}-n^{(k)}}_{L^\infty(0,T)} <\dfrac{\varepsilon}{2\sqrt{3T}}.\)
Hence, \(\norm{q-n}_{H^2(0,T)}<\varepsilon/2\), and we use the triangle inequality to prove the first density statement. We obtain the second from the continuous embedding \(H^2(0,T)\hookrightarrow L^2(0,T)\) and the density of smooth functions in \(L^2(0,T)\).
\end{proof}

We call a pair \((a^*,c^*)\) admissible if it satisfies \eqref{eq:retained} almost everywhere and has \((a^*(0),a^{*\prime}(0))=(a_0,a_1)\) and \((a^*(T),a^{*\prime}(T))=(a_T,a'_T)\).
\begin{theorem}\label{thm:neural-density}
Let \(P\) be fixed, and let \((a^*,c^*)\) be a Galerkin admissible pair with \(a^*\in H^2(0,T;\R^P)\) and \(c^*\in L^2(0,T;\R^P)\). For every \(\varepsilon>0\), there are finite-width networks of the form \eqref{eq:neural-ansatz} such that
\(\norm{a_\theta-a^*}_{H^2(0,T;\R^P)} +\norm{c_\theta-c^*}_{L^2(0,T;\R^P)}<\varepsilon.\)
We assume that \(G_P\) is Lipschitz on an open neighborhood of the compact trajectory range
\(K^*=\{a^*(t):0\le t\le T\}\subset\R^P.\)
Along every sequence for which the displayed approximation error tends to zero, the Galerkin residual tends to zero in \(L^2(0,T;\R^P)\), and the initial and terminal displacement and velocity mismatches tend to zero.
\end{theorem}

\begin{proof}
By \Cref{lem:temporal-uap}, we approximate every \(a_j^*\) in \(H^2(0,T)\) and every \(c_j^*\) in \(L^2(0,T)\). We zero-pad unused hidden units to use one common finite width. For a prescribed \(\delta>0\), we choose the coefficientwise errors so that
\[
 \sum_{j=1}^P\norm{a_{j,\theta}-a_j^*}_{H^2}^2
 +\sum_{j=1}^P\norm{c_{j,\theta}-c_j^*}_{L^2}^2<\delta^2.
\]
Subtracting the exact Galerkin equation from the network residual, we obtain
\(r_\theta=(a_\theta-a^*)''+\Lambda_P(a_\theta-a^*) +G_P(a_\theta)-G_P(a^*)+M_P(c_\theta-c^*).\)
We set \(e=a_\theta-a^*\) and \(d=c_\theta-c^*\). By the embedding \(H^2(0,T;\R^P)\hookrightarrow C^1([0,T];\R^P)\), we obtain uniform convergence of \(a_\theta\) to \(a^*\). Since \(K^*\) is compact, sufficiently accurate approximations remain in the open neighborhood from the Lipschitz hypothesis of \Cref{thm:neural-density}. If \(L_{G,P}\) is its Lipschitz constant, then
\(\|r_\theta\|_{L^2} \le \left(1+\|\Lambda_P\|_2+L_{G,P}\right)\|e\|_{H^2} +\|M_P\|_2\|d\|_{L^2}.\)
Applying the same embedding, we obtain
\(\max_{t\in\{0,T\}} \left(\|e(t)\|_2+\|e'(t)\|_2\right) \le C_T\|e\|_{H^2}.\)
Both bounds tend to zero with the approximation error. Finally,
\(\|e\|_{H^2}+\|d\|_{L^2} \le\sqrt2 \left(\|e\|_{H^2}^2+\|d\|_{L^2}^2\right)^{1/2},\)
so we choose \(\delta<\varepsilon/\sqrt2\) to obtain the asserted approximation tolerance.
\end{proof}

\begin{proposition}\label{prop:neural-residual-density}
We fix \(P\) and a Neural--Spectral control \(c_\theta\in\mathcal D_{P,T}\). Let \(a_{c_\theta}\) be its Galerkin trajectory on \([0,T]\) with the prescribed initial data. We assume that \(G_P\) is locally Lipschitz on a neighborhood of the trajectory range. Then, for every \(\varepsilon>0\), there is a finite-width state network \(a_\vartheta\) such that
\[
 \left\|a_\vartheta''+\Lambda_Pa_\vartheta
 +G_P(a_\vartheta)+M_Pc_\theta\right\|_{L^2(0,T;\R^P)}
 <\varepsilon.
\]
\end{proposition}

\begin{proof}
The trajectory \(a_{c_\theta}\) is continuous on \([0,T]\), so its range is compact. Since \(c_\theta\in L^2\) and
\(a_{c_\theta}'' =-\Lambda_Pa_{c_\theta} -G_P(a_{c_\theta})-M_Pc_\theta,\)
the right-hand side belongs to \(L^2(0,T;\R^P)\). Hence, \(a_{c_\theta}\in H^2(0,T;\R^P)\).

We apply \Cref{lem:temporal-uap} coefficientwise. By the embedding \(H^2(0,T;\R^P)\hookrightarrow C^1([0,T];\R^P)\), we obtain uniform convergence, so sufficiently accurate approximations remain in the neighborhood from the local-Lipschitz hypothesis of \Cref{prop:neural-residual-density}. For \(e=a_\vartheta-a_{c_\theta}\), we subtract the exact equation to obtain
\(r_\vartheta =e''+\Lambda_Pe +G_P(a_\vartheta)-G_P(a_{c_\theta}).\)
Consequently,
\(\|r_\vartheta\|_{L^2} \le \left(1+\|\Lambda_P\|_2+L_{G,P}\right)\|e\|_{H^2},\)
where \(L_{G,P}\) is a Lipschitz constant of \(G_P\) on that neighborhood. We can make this bound smaller than \(\varepsilon\).
\end{proof}

We next approximate a given full-PDE steering pair in the energy and control norms. For the joint approximation result below, we fix an enumeration \((e_j)_{j\in\mathbb N}\) of the complete orthonormal Dirichlet eigenbasis of \(L^2(\Omega)\), with corresponding eigenvalues \((\mu_j)_{j\in\mathbb N}\). For each \(P\in\mathbb N\), we use the first \(P\) modes in this enumeration. Let \(\pi_P\) and \(\Pi_P\) be the scalar and energy projections,
\(\pi_Pu=\sum_{j=1}^P\langle u,e_j\rangle_{L^2}e_j,\ \Pi_P(u,z)=(\pi_Pu,\pi_Pz).\)

\begin{theorem}\label{thm:joint-neural-density}
Let \(\sigma\) satisfy \Cref{ass:activation} and \((\Phi^*,f^*)\) be a mild steering pair for \eqref{eq:pde}, where
\(\Phi^*=(u^*,\partial_tu^*) \in C([0,T];\E), \ f^*\in L^2((0,T)\times\omega).\)
In the control estimate below, we identify \(f^*\) with its zero extension to \((0,T)\times\Omega\). We suppose that, for every \(j\in\mathbb N\), the modal displacement coefficient
\(a_j^*(t) = \langle u^*(t),e_j\rangle_{L^2(\Omega)}, \ t\in[0,T],\)
belongs to \(H^2(0,T)\).

For every \(\varepsilon>0\), there exist \(P,N_{\rm w}\in\mathbb N\) and temporal networks \(a_{j,\theta},c_{j,\theta}\) of common width \(N_{\rm w}\), for \(j\in\{1,\ldots,P\}\), such that
\[
 \widetilde u_\theta(t,x)=\sum_{j=1}^{P}a_{j,\theta}(t)e_j(x),
 \qquad
 \widetilde f_\theta(t,x)=\sum_{j=1}^{P}c_{j,\theta}(t)e_j(x),
\]
and, with \(\widetilde\Phi_\theta =(\widetilde u_\theta,\partial_t\widetilde u_\theta)\),
\begin{align}
 \sup_{t\in[0,T]} \norm{\Phi^*(t)-\widetilde\Phi_\theta(t)}_{\E}
 &<\varepsilon, \label{eq:joint-state-density}\\
 \norm{f^*-\mathbf1_\omega\widetilde f_\theta}_{L^2((0,T)\times\Omega)}
 &<\varepsilon.
\end{align}
\end{theorem}

\begin{proof}
We fix an arbitrary \(\varepsilon>0\). For each \(P\in\mathbb N\), let \(\pi_P\) and \(\Pi_P\) be the scalar and energy projections defined above. By spectral orthogonality, we have
\[\norm{(u,z)}_{\E}^2 = \sum_{j=1}^{\infty} \left( (\mu_j+\alpha) |\langle u,e_j\rangle_{L^2(\Omega)}|^2 + |\langle z,e_j\rangle_{L^2(\Omega)}|^2 \right),\]
for every \((u,z)\in\E\). Consequently, \(\Pi_P\) and \(I-\Pi_P\) are contractions on \(\E\), and \(\Pi_P\) converges strongly to the identity as \(P\to\infty\).

Because \(\Phi^*\) is continuous, its range \(\mathcal K = \{\Phi^*(t):t\in[0,T]\}\) is compact in \(\E\). The uniformly bounded projections converge strongly, and we use their uniform convergence on this compact set. Indeed, for each \(\eta>0\), we choose a finite \(\eta\)-net in \(\mathcal K\) and then take \(P\) large enough that the projection error at every point of the net is less than \(\eta\). The contraction property of \(I-\Pi_P\) then bounds the error at every point of \(\mathcal K\) by \(2\eta\). Hence,
\(\sup_{t\in[0,T]} \norm{(I-\Pi_P)\Phi^*(t)}_{\E} \longrightarrow0.\)

For the zero-extended control, Parseval's identity and integration in time give
\[
 \norm{f^*-\pi_Pf^*}_{L^2((0,T)\times\Omega)}^2
 =\sum_{j>P}\norm{\langle f^*(\cdot),e_j\rangle_{L^2(\Omega)}}_{L^2(0,T)}^2
 \longrightarrow0.
\]
We therefore choose one \(P\in\mathbb N\) such that
\[
 \sup_{t\in[0,T]}\norm{(I-\Pi_P)\Phi^*(t)}_{\E}<\dfrac{\varepsilon}{2},
 \qquad
 \norm{f^*-\pi_Pf^*}_{L^2((0,T)\times\Omega)}<\dfrac{\varepsilon}{2}.
\]

For each \(j\in\{1,\ldots,P\}\), we set
\[
 c_j^*(t) = \langle f^*(t),e_j\rangle_{L^2(\Omega)}
 \quad\text{for almost every }t\in(0,T).
\]
By \Cref{lem:temporal-uap}, we can choose control networks such that
\[
 \sum_{j=1}^{P}\norm{c_j^*-c_{j,\theta}}_{L^2(0,T)}^2
 <\dfrac{\varepsilon^2}{4}.
\]
Since \(\mathbf1_\omega f^*=f^*\) and multiplication by \(\mathbf1_\omega\) is an \(L^2\) contraction, we obtain
\begin{align*}
 &\norm{f^*-\mathbf1_\omega\widetilde f_\theta}_{L^2((0,T)\times\Omega)} \quad\le\norm{f^*-\widetilde f_\theta}_{L^2((0,T)\times\Omega)}\\
 &\le \norm{f^*-\pi_Pf^*}_{L^2((0,T)\times\Omega)} + \left( \sum_{j=1}^{P} \norm{c_j^*-c_{j,\theta}}_{L^2(0,T)}^2 \right)^{1/2} <\varepsilon.
\end{align*}

For the state coefficients, we use the mild wave equation to obtain
\[
 (a_j^*)'(t)=\langle\partial_tu^*(t),e_j\rangle_{L^2(\Omega)},
 \qquad t\in[0,T].
\]
Using \Cref{lem:temporal-uap} and the continuous embedding \(H^2(0,T)\hookrightarrow C^1([0,T])\), we choose state networks such that
\[
 \sum_{j=1}^{P}\left(
 (\mu_j+\alpha)\norm{a_j^*-a_{j,\theta}}_{L^\infty(0,T)}^2
 +\norm{(a_j^*)'-a_{j,\theta}'}_{L^\infty(0,T)}^2
 \right)<\dfrac{\varepsilon^2}{4}.
\]
Only finitely many coefficients and fixed spectral weights occur in this inequality. Using spectral orthogonality, we obtain
\[
 \sup_{t\in[0,T]}\norm{\Pi_P\Phi^*(t)-\widetilde\Phi_\theta(t)}_{\E}
 <\dfrac{\varepsilon}{2}.
\]
Combining this estimate with the state projection error, we prove \eqref{eq:joint-state-density}. Finally, we add hidden units with zero output weights to give all state and control networks one common finite width.
\end{proof}

\begin{remark}
The joint density result concerns representational capacity; see \Cref{thm:joint-neural-density}. We use \Cref{thm:neural-reliability} and its discrete version, \Cref{cor:discrete-neural-reliability}, to convert a computed residual into a forward statement. Finally, we apply \Cref{thm:eventual-neural-certificate} to upgrade sufficiently accurate controls to exact steering. Keeping these three steps distinct identifies precisely what approximation theory supplies, what we must check after training, and what the Newton--HUM certificate adds.
\end{remark}

Let \(\widetilde a\in H^2(0,T;\R^P)\) and \(\widetilde c\in\mathcal D_{P,T}\) be an extracted pair. Let \(a_{\widetilde c}\) denote the exact finite-dimensional solution driven by \(\widetilde c\) from the prescribed initial data. We define its continuous residual by
\(r_\theta =\widetilde a''+\Lambda_P\widetilde a +G_P(\widetilde a)+M_P\widetilde c.\)

\begin{theorem}\label{thm:neural-reliability}
We assume that \(G_P\) is continuously differentiable on an open neighborhood of
\[
 \left\{a_{\widetilde c}(t)+s\left(\widetilde a(t)-a_{\widetilde c}(t)\right):
 0\le t\le T,\ 0\le s\le1\right\},
\]
and that
\begin{equation}\label{eq:neural-jacobian-bound}
 \sup_{t\in[0,T],\,s\in[0,1]}
 \left\|DG_P\!\left(a_{\widetilde c}(t)
 +s\left(\widetilde a(t)-a_{\widetilde c}(t)\right)\right)\right\|_2
 \le\ell_1.
\end{equation}
We set \(\kappa=\ell_1/(2\pi)\). Then
\begin{equation}\label{eq:neural-reliability-estimate}\begin{aligned}
 &\sup_{0\le t\le T}
 \left\|\mathsf E_P\binom{\widetilde a(t)-a_{\widetilde c}(t)}
 {\widetilde a'(t)-a'_{\widetilde c}(t)}\right\|_2\le e^{\kappa T}\left(
 \left\|\mathsf E_P\binom{\widetilde a(0)-a_0}{\widetilde a'(0)-a_1}\right\|_2
 +\|r_\theta\|_{L^1(0,T;\R^P)}\right).
\end{aligned}\end{equation}
The bound depends on \(T\) and the Jacobian bound \eqref{eq:neural-jacobian-bound}, and requires no optimizer-convergence assumption.
\end{theorem}

\begin{proof}
Let \(e=\widetilde a-a_{\widetilde c}\) and use the Jacobian bound \eqref{eq:neural-jacobian-bound} in the mean-value representation
\begin{align*}
 G_P(\widetilde a)-G_P(a_{\widetilde c})=J_e(t)e, \qquad J_e(t)=\int_0^1DG_P\bigl(a_{\widetilde c}(t)+s e(t)\bigr)\,\mathrm{d}s, \qquad
 \norm{J_e(t)}_2\le\ell_1.
\end{align*}
Subtracting the two equations, we obtain \(e''+\Lambda_Pe+J_e(t)e=r_\theta\). We set \(y=\mathsf E_P(e,e')^\top\). In these coordinates
\[
 y'=(A_0+K_e(t))y+\begin{pmatrix}0\\r_\theta\end{pmatrix},
 \qquad
 K_e(t)=\begin{pmatrix}0&0\\-J_e(t)\Lambda_P^{-1/2}&0\end{pmatrix}.
\]
The free block \(A_0\) is skew-symmetric. The symmetric part
\(S_e(t)=\frac{1}{2}\bigl(K_e(t)+K_e(t)^\top\bigr)\) satisfies
\[
 \norm{S_e(t)}_2
 =\dfrac{1}{2}\norm{J_e(t)\Lambda_P^{-1/2}}_2
 \le\dfrac{\ell_1}{2\pi}=\kappa.
\]
Since \(y\) is absolutely continuous, we differentiate its squared norm for almost every \(t\in(0,T)\) and obtain
\(\frac{1}{2}\frac{\mathrm d}{\mathrm dt}\norm{y(t)}_2^2 =\ip{y(t)}{y'(t)} =\ip{y(t)}{S_e(t)y(t)} +\ip{y(t)}{\begin{pmatrix}0\\r_\theta(t)\end{pmatrix}}\).
The skew-symmetric parts contribute zero to this inner product. Applying the bound on \(S_e(t)\) and the Cauchy--Schwarz inequality, we obtain
\[
 \ip{y(t)}{y'(t)}
 \le \kappa\norm{y(t)}_2^2+\norm{y(t)}_2\norm{r_\theta(t)}_2.
\]
For \(\delta>0\), we set \(q_\delta(t)=\left(\|y(t)\|_2^2+\delta^2\right)^{1/2}\). Then, \(q_\delta'(t) \le\kappa q_\delta(t)+\|r_\theta(t)\|_2,\) for almost every \(t\in(0,T)\). Integrating this inequality and letting \(\delta\to0^+\), we obtain \(\|y(t)\|_2 \le e^{\kappa t}\|y(0)\|_2 +\int_0^t e^{\kappa(t-s)} \|r_\theta(s)\|_2\,\mathrm ds.\) Since \(\kappa\ge0\), we take the supremum over \([0,T]\) to obtain the reliability estimate in the theorem, namely \eqref{eq:neural-reliability-estimate}.
\end{proof}

For each \(N_{\rm q} \in \mathbb{N}\), we choose quadrature nodes \(t_q\in[0,T]\) and positive weights \(w_q\), specified independently of the network parameters. The nodes and weights may depend on \(N_{\rm q}\). We fix positive endpoint weights \(\lambda_{\mathrm{IC}},\lambda_{\mathrm{TC}}\) independently of \(P,N_{\rm w},N_{\rm q}\). For the objective \eqref{eq:neural-objective}, \(N_{\rm q}=N_c\), \(w_q=T/(N_{\rm q}+1)\), and \(\lambda_{\mathrm{IC}}=\lambda_{\mathrm{TC}}=500\). For \(B>0\), we set \(\Theta_{P,N_{\rm w}}^B =\left\{\vartheta\in\R^{6PN_{\rm w}}:\|\vartheta\|_2\le B\right\}\), and we assume that the computed parameter \(\theta\) lies in this set. For each \(\vartheta\in\Theta_{P,N_{\rm w}}^B \), we define \(r_\vartheta:=a_\vartheta''+\Lambda_Pa_\vartheta+G_P(a_\vartheta)+M_Pc_\vartheta\). The loss function is
\(\mathcal L_{P,N_{\rm q}}(\vartheta)=\sum_{q=1}^{N_{\rm q}}w_q\|r_\vartheta(t_q)\|_2^2 +\lambda_{\mathrm{IC}}\,e_{\mathrm{IC}}(\vartheta)^2 +\lambda_{\mathrm{TC}}\,e_{\mathrm{TC}}(\vartheta)^2\). We further assume that we have certified the quadrature error uniformly on \(\Theta_{P,N_{\rm w}}^B \), namely
\begin{equation}\label{eq:neural-uniform-quadrature}
 \left|\,\|r_\vartheta\|_{\mathcal Z_P}^2
 -\sum_{q=1}^{N_{\rm q}}w_q\|r_\vartheta(t_q)\|_2^2\right|
 \le\eta_{N_{\rm q}}(P,N_{\rm w};B)^2.
\end{equation}
Finally, we introduce the three error quantities
\begin{align*}
 E_{\mathrm{net}}&=\inf_{\vartheta\in\Theta_{P,N_{\rm w}}^B } \left(\|r_\vartheta\|_{\mathcal Z_P}^2+e_{\mathrm{IC}}(\vartheta)^2 +e_{\mathrm{TC}}(\vartheta)^2\right)^{1/2},\\
 E_{\mathrm{opt}}&=\left(\mathcal L_{P,N_{\rm q}}(\theta) -\inf_{\vartheta\in\Theta_{P,N_{\rm w}}^B }\mathcal L_{P,N_{\rm q}}(\vartheta)\right)^{1/2}, \ E_{\mathrm{tar}}(P)=\|(I-\Pi_P)\Phi_T\|_\E.
\end{align*}

\begin{corollary}\label{cor:discrete-neural-reliability}
We assume the hypotheses of \Cref{thm:neural-reliability} for the pair \((\widetilde a,\widetilde c)=(a_\theta,c_\theta)\), with \(\theta\in\Theta_{P,N_{\rm w}}^B \), and the uniform quadrature bound \eqref{eq:neural-uniform-quadrature}. With \(T\), \(\ell_1\), and the endpoint weights fixed, there exists a constant \(C=C(T,\ell_1,\lambda_{\mathrm{IC}},\lambda_{\mathrm{TC}})>0\), independent of \(P,N_{\rm w},N_{\rm q}\), such that
\[
 \sup_{0\le t\le T}
 \norm{\mathsf E_P\!\left(\widetilde a(t)-a_{\widetilde c}(t),
 \widetilde a'(t)-a_{\widetilde c}'(t)\right)^\top}_2
 \le C\!\left(E_{\mathrm{net}}+\eta_{N_{\rm q}}+E_{\mathrm{opt}}\right).
\]
Moreover, if \(\Phi^P_{\widetilde c}(T)\) denotes the spatial reconstruction of the exact finite-dimensional solution, then
\[
 \|\Phi^P_{\widetilde c}(T)-\Phi_T\|_\E
 \le E_{\mathrm{tar}}(P)
 +C\!\left(E_{\mathrm{net}}+\eta_{N_{\rm q}}+E_{\mathrm{opt}}\right).
\]
In particular, if \(\mathcal L_{P,N_{\rm q}}(\theta)\le\delta^2\), we obtain
\(\|\Phi^P_{\widetilde c}(T)-\Phi_T\|_\E \le E_{\mathrm{tar}}(P)+C\!\left(\delta+\eta_{N_{\rm q}}\right).\)
\end{corollary}

\begin{proof}
Let
\(\lambda_-:=\min\{1,\lambda_{\rm IC},\lambda_{\rm TC}\}\),  \(\lambda_+:=\max\{1,\lambda_{\rm IC},\lambda_{\rm TC}\}.\)
Using the certified quadrature inequality, we obtain, for every \(\vartheta\in\Theta_{P,N_{\rm w}}^B \),
\begin{equation}\label{eq:discrete-neural-loss-comparison}\begin{aligned}
 &\norm{r_\vartheta}_{L^2\!\left(0,T;\R^P\right)}^2+e_{\rm IC}(\vartheta)^2+e_{\rm TC}(\vartheta)^2
 \le\lambda_-^{-1}\mathcal L_{P,N_{\rm q}}(\vartheta)+\eta_{N_{\rm q}}^2,\\
 &\mathcal L_{P,N_{\rm q}}(\vartheta)
 \le\lambda_+\left(
 \norm{r_\vartheta}_{L^2\!\left(0,T;\R^P\right)}^2+e_{\rm IC}(\vartheta)^2+e_{\rm TC}(\vartheta)^2+\eta_{N_{\rm q}}^2
 \right).
\end{aligned}\end{equation}
Taking the infimum in the second inequality of \eqref{eq:discrete-neural-loss-comparison} and using the definition of the optimization gap, we obtain
\begin{equation}\label{eq:neural-objective-bound}
 \sqrt{\mathcal L_{P,N_{\rm q}}(\theta)}
 \le\sqrt{\lambda_+}\,(E_{\rm net}+\eta_{N_{\rm q}})+E_{\rm opt}.
\end{equation}
We set
\(
 S_\theta
 =\left(\|r_\theta\|_{\mathcal Z_P}^2
 +e_{\rm IC}(\theta)^2+e_{\rm TC}(\theta)^2\right)^{1/2}.
\)
Combining \eqref{eq:discrete-neural-loss-comparison} and \eqref{eq:neural-objective-bound}, we obtain
\begin{equation*}
 S_\theta \le C_0(E_{\rm net}+\eta_{N_{\rm q}}+E_{\rm opt}),
 \qquad C_0=1+\sqrt{\lambda_+/\lambda_-}.
\end{equation*}
By the Cauchy--Schwarz inequality and \Cref{thm:neural-reliability}, we bound the forward trajectory discrepancy by
\(
 e^{\kappa T}\left(e_{\rm IC}(\theta)+\sqrt T\,\|r_\theta\|_{\mathcal Z_P}\right)
 \le e^{\kappa T}\sqrt{1+T}\,S_\theta.
\)

At the terminal time, we insert the network state and the projected target to obtain
\(\|\Phi^P_{c_\theta}(T)-\Phi_T\|_\E \le E_{\rm tar}(P)+e_{\rm TC}(\theta) +e^{\kappa T}\sqrt{1+T}\,S_\theta\).
Thus, both conclusions hold, for example, with
\[
 C=\left(1+e^{\kappa T}\sqrt{1+T}\right)
 \left(1+\sqrt{\lambda_+/\lambda_-}\right).
\]
Finally, if \(\mathcal L_{P,N_{\rm q}}(\theta)\le\delta^2\), we apply the first inequality in \eqref{eq:discrete-neural-loss-comparison} to obtain
\(S_\theta\le\lambda_-^{-1/2}\delta+\eta_{N_{\rm q}},\)
and thus we establish the attained-loss estimate.
\end{proof}

The next result supplies the theoretical bridge from neural approximation to exact control certification.

\begin{theorem}\label{thm:eventual-neural-certificate}
We fix \(P\), write \(\mathcal Z_P=L^2(0,T;\R^P)\) with its usual norm, and equip \(\R^{2P}\) with the Euclidean norm \(\|\cdot\|_2\). Let \(c^*\in\mathcal D_{P,T}\) satisfy \(\F_P(c^*)=\mathbf{0}_{\R^{2P}}\). We assume that \(\F_P\) is continuously differentiable on an open neighborhood \(U\subset\mathcal D_{P,T}\) of \(c^*\) and that its derivative is \(L\)-Lipschitz in the operator norm from \(\mathcal Z_P\) to \(\R^{2P}\), meaning
\[
 \|D\F_P(c)-D\F_P(\widetilde c)\|_{\mathcal L(\mathcal Z_P,\R^{2P})}
 \le L\|c-\widetilde c\|_{\mathcal Z_P},
 \qquad c,\widetilde c\in U.
\]
We also assume that \(D\F_P(c^*)\) has a right inverse \(R^*\in\mathcal L(\R^{2P},\mathcal Z_P)\) with \(\|R^*\|_{\mathcal L(\R^{2P},\mathcal Z_P)}\le\gamma^*\).

Let \(c_n\) be Neural--Spectral controls converging to \(c^*\) in \(\mathcal Z_P\), and set \(\gamma=2\gamma^*\). For all sufficiently large \(n\), the minimum-norm HUM right inverse \(R_{c_n}\in\mathcal L(\R^{2P},\mathcal Z_P)\) exists and satisfies \(\|R_{c_n}\|_{\mathcal L(\R^{2P},\mathcal Z_P)}\le\gamma\). Moreover, with
\(\varepsilon_n=\|\F_P(c_n)\|_2\),  \(h_n=2L\gamma^2\varepsilon_n,\)
we have \(h_n\to0\), and the hypotheses of \Cref{thm:certificate,thm:newton} hold on a neighborhood of \(c_n\) contained in \(U\). In particular, the exact full-step Newton--HUM iteration starting from \(c_n\) converges to an exact finite-dimensional control.
\end{theorem}

\begin{proof}
We choose \(\rho>0\) such that \(B_{2\rho}(c^*) =\left\{c \in \mathcal Z_P: \left\|c-c^*\right\|<2\rho\right\}\subset U\) and
\[
 \left\| \left(D\F_P(c)-D\F_P(c^*)\right)R^* \right\|_{\mathcal L(\R^{2P},\R^{2P})}<\dfrac{1}{2},
\]
for every \(c\in B_{2\rho}(c^*)\). Since \(D\F_P(c^*)R^*=I\), we apply the Neumann series to show that \(D\F_P(c)R^*\) is invertible. Hence, \(\widetilde R_c=R^*\left(D\F_P(c)R^*\right)^{-1}\) is a right inverse of \(D\F_P(c)\) with \(\|\widetilde R_c\|_{\mathcal L(\R^{2P},\mathcal Z_P)}\le2\gamma^*=\gamma\).

The minimum-norm HUM right inverse \(R_c\) therefore exists throughout this ball. For every \(\xi\in\R^{2P}\), we use its minimizing property to obtain \(\|R_c\xi\|_{\mathcal Z_P} \le\|\widetilde R_c\xi\|_{\mathcal Z_P}\), and consequently
\begin{equation}\label{eq:eventual-neural-uniform-right-inverse}
 \|R_c\|_{\mathcal L(\R^{2P},\mathcal Z_P)}\le\gamma, \qquad c\in B_{2\rho}(c^*).
\end{equation}

For all sufficiently large \(n\), it holds that
\begin{equation}\label{eq:eventual-neural-neighborhood-inclusion}
 c_n\in B_\rho(c^*), \qquad B_\rho(c_n)\subset B_{2\rho}(c^*).
\end{equation}
Since \(\F_P\) is continuous and \(c_n\to c^*\), we have \(\varepsilon_n\to0\). Thus, \(h_n<1\) eventually, and the certificate radius \(r_n= \dfrac{2\varepsilon_n}{1+\sqrt{1-h_n}}\) satisfies \(\gamma r_n\to0\). In particular, \(\gamma r_n<\rho\) for all sufficiently large \(n\), so \Cref{thm:certificate} applies at \(c_n\).

For the Newton--HUM method, we set \(q_{{\rm N},n}=\dfrac{L\gamma^2}{2}\varepsilon_n =\dfrac{h_n}{4}\). Then, \(q_{{\rm N},n}\to0\) and \(\dfrac{\gamma\varepsilon_n}{1-q_{{\rm N},n}}\to0\). Hence, eventually, \(q_{{\rm N},n}<1\) and \(\gamma\varepsilon_n/(1-q_{{\rm N},n})<\rho\). The bound \eqref{eq:eventual-neural-uniform-right-inverse} applies on \(B_\rho(c_n)\) by the inclusion \eqref{eq:eventual-neural-neighborhood-inclusion} and verifies the hypotheses of \Cref{thm:newton}. Applying \Cref{lem:temporal-uap} componentwise to \(c^*\) produces a sequence of Neural--Spectral controls converging to \(c^*\) in \(\mathcal Z_P\).
\end{proof}

The experiments preserve this analytical separation. We score full-budget Neural--Spectral controls by independent forward integration before correction and then record the Newton--HUM history and cost separately. This paired reporting quantifies the quality of each learned candidate and the numerical improvement and computational cost of its Newton--HUM correction. We report a rigorous endpoint certificate for the corrected seed-\(0\), \(\beta=1\) candidate in \Cref{tab:validated}.

\section{Candidate generators and matched protocol}\label{sec:algorithms}

We compare two deterministic candidate generators and the Neural--Spectral candidate of \Cref{sec:neural}. All use the same Galerkin equation \eqref{eq:retained}, actuator matrix, and \(501\)-node correction representation, and we assess them through independent endpoint validation.

\subsection{Endpoint least-squares/Newton--HUM method}

Our first method applies Newton's method to the endpoint equation \(\F_P(c)=\mathbf0_{\R^{2P}}\), using the minimum-norm HUM right inverse from \Cref{prop:linearized-right-inverse} to compute each correction. At iteration \(n\), we integrate the nonlinear trajectory, the variational fundamental matrix, and the Gramian, then form \(\delta c_n=-R_{c_n}\F_P(c_n)\). In the exact continuous-time formulation, \(\delta c_n\) is the unique minimum-\(\mathcal Z_P\)-norm solution of \(D\F_P(c_n)\delta c_n=-\F_P(c_n)\). Equivalently, it is the minimum-norm minimizer of the linearized endpoint least-squares problem \(\min_d\dfrac{1}{2} \|\F_P(c_n)+D\F_P(c_n)d\|_2^2\). We test \(\varsigma_n=1,1/2,1/4,\ldots\) and reject each trial whose forward integration fails to reach \(T\) with a finite state. Among the remaining trials, we accept the first whose independently evaluated terminal energy error is at most \(1-10^{-4}\varsigma_n\) times its current value.

\subsection{Stabilized spectral source-freezing method}

Starting from an endpoint-compatible trajectory, we freeze \(s^{(n)}=G_P(a^{(n)})\).
Let \((\widehat a^{(n+1)},\widehat c^{(n+1)})\) be the minimum-modal-cost steering pair for this frozen source. We choose a relaxation parameter \(\eta\in(0,1]\). The trajectory update is \(a^{(n+1)}=(1-\eta)a^{(n)}+\eta\widehat a^{(n+1)}\). We monitor the unrelaxed proposal residual \(r_{\rm fp}^{(n)}=\|\widehat a^{(n+1)}-a^{(n)}\|_{\mathcal Z_P}\).
Upon internal convergence, we return the affine-proposal control \(\widehat c^{(n+1)}\) and validate it by integrating the nonlinear Galerkin equation from the prescribed initial data. We then check every returned control against the external endpoint gate \eqref{eq:common-gate}. In the matched tanh runs of \Cref{tab:tanh}, every returned source-freezing control has a nonlinear Galerkin endpoint error between \(2.0\times10^{-8}\) and \(3.9\times10^{-8}\), slightly above the common tolerance \(10^{-8}\). We therefore apply the damped Newton--HUM iteration described above as a common final step. In every run, it accepts one full step and reduces \(E_P\) below \(10^{-8}\). \Cref{tab:tanh} reports \(E_P\) before and after this correction, and the reported solve times include it.

We work with the unscaled matrices \(\mathsf A,\mathsf B\) of \Cref{sec:response}, and we write \(y_0=(a_0,a_1)^\top\) and \(y_T=(a_T,a'_T)^\top\). We fix a frozen source \(s\). We set \(b_s(t)=(\mathbf{0},-s(t))^\top\) and
\(d_s^{\mathrm{aff}}=y_T-e^{\mathsf A T}y_0 -\int_0^T e^{\mathsf A(T-t)}b_s(t)\,\mathrm dt,\)
so that \(d_s^{\mathrm{aff}}\) measures, at time \(T\), the gap between the target and the uncontrolled affine evolution. For the dynamics \(y'=\mathsf Ay+\mathsf Bc+b_s\), the control of minimum modal cost is
\(c_s(t)=\mathsf B^\top e^{\mathsf A^\top(T-t)}W_T^{-1}d_s^{\mathrm{aff}}.\)
Since \(M_P\succ0\), the matrix
\(\left(\mathsf B\;\;\mathsf A\mathsf B\right) =\begin{pmatrix}0&-M_P\\-M_P&0\end{pmatrix}\)
has rank \(2P\). The affine Galerkin pair is therefore controllable, and \(W_T\succ0\) for every \(T>0\). In particular, a single factorization of \(W_T\) serves all frozen sources and all compatible pairs of initial and target data.

The same construction distinguishes the modal coefficient cost from the physical applied-control cost. Let \(\Sigma\in\R^{P\times P}\) be symmetric positive definite. We minimize \(J_{\Sigma}(c)=\displaystyle\int_0^T c(t)^\top \Sigma c(t)\,\mathrm{d}t\) over \(c\in L^2(0,T;\R^P)\) subject to the affine endpoint constraint
\(\int_0^T e^{\mathsf A(T-t)}\mathsf Bc(t)\,\mathrm{d}t = d_s^{\mathrm{aff}}.\)

Let \(\Sigma^{1/2}\) denote the symmetric positive definite square root of \(\Sigma\) and set \(v=\Sigma^{1/2}c\). Then,
\[
\begin{aligned}
 J_{\Sigma}(c)&=\norm{v}_{L^2(0,T;\R^P)}^2, \qquad \int_0^T e^{\mathsf A(T-t)}\mathsf B\Sigma^{-1/2}v(t)\,\mathrm dt=d_s^{\mathrm{aff}}.
\end{aligned}
\]
Thus, the change of variables reduces the weighted problem to the Euclidean minimum-norm problem. Its control matrix is \(\mathsf B\Sigma^{-1/2}\), and its controllability Gramian is
\(W_{T,\Sigma} = \int_0^T e^{\mathsf A(T-t)} \mathsf B \Sigma^{-1}\mathsf B^\top e^{\mathsf A^\top(T-t)} \,\mathrm{d}t.\)
Since \(\Sigma^{-1/2}\) is invertible, the transformed pair is controllable whenever \((\mathsf A,\mathsf B)\) is controllable. Hence, we conclude that \(W_{T,\Sigma}\) is positive definite. Applying the minimum-norm formula, we obtain
\(v_{\Sigma}(t) = \Sigma^{-1/2}\mathsf B^\top e^{\mathsf A^\top(T-t)} W_{T,\Sigma}^{-1}d_s^{\mathrm{aff}}, \ t\in[0,T].\)
Returning to \(c_{\Sigma}=\Sigma^{-1/2}v_{\Sigma}\), we obtain the unique minimizer
\begin{equation}\label{eq:weighted-control-formula}
 c_{\Sigma}(t)
 =
 \Sigma^{-1}\mathsf B^\top
 e^{\mathsf A^\top(T-t)}
 W_{T,\Sigma}^{-1}d_s^{\mathrm{aff}},
 \qquad t\in[0,T].
\end{equation}
Using the definition of \(W_{T,\Sigma}\), we compute its cost:
\begin{align*}
 J_{\Sigma}(c_{\Sigma}) &= (d_s^{\mathrm{aff}})^\top W_{T,\Sigma}^{-1}W_{T,\Sigma}W_{T,\Sigma}^{-1} d_s^{\mathrm{aff}} = (d_s^{\mathrm{aff}})^\top W_{T,\Sigma}^{-1}d_s^{\mathrm{aff}}.
\end{align*}

For the localized modal control, we use the definition of \(M_P\) to obtain
\begin{align*}
 \int_0^T\int_\omega
 \left|\sum_{j=1}^{P}c_j(t)e_j(x)\right|^2
 \,\mathrm{d}x\,\mathrm{d}t
 &=
 \int_0^T
 \sum_{j,k=1}^{P}
 c_j(t)c_k(t)
 \left(\int_\omega e_j(x)e_k(x)\,\mathrm{d}x\right)
 \,\mathrm{d}t\\
 &=
 \int_0^T c(t)^\top M_Pc(t)\,\mathrm{d}t.
\end{align*}
Consequently, we identify \(\Sigma=I_P\) with the squared modal coefficient norm and \(\Sigma=M_P\) with the squared physical applied-control norm.

The one-dimensional affine source-frozen solves and minimum-norm Newton--HUM corrections use the Euclidean modal metric, corresponding to \(\Sigma=I_P\). We nevertheless report
\(\left( \int_0^T c(t)^\top M_Pc(t)\,\mathrm{d}t \right)^{1/2}\)
as the principal physical control-amplitude metric. The two-dimensional linear comparison in \Cref{tab:2d-gramian} explicitly considers both \(\Sigma=I_P\) and \(\Sigma=M_P\) and reports both norms. Reporting the physical norm does not imply that every candidate minimizes the physical control cost.

\paragraph{\bf Convergence and relaxation.}
Let \(\mathcal X_P=C\!\left([0,T];\R^P\right)\), equipped with the supremum norm, and let \(\mathcal T_P\) denote the exact source-frozen map. We fix a closed invariant ball, and we let \(C_{P,T}^{\mathrm{aff}}\) bound the ratio \(\|\mathcal H_P(s_1)-\mathcal H_P(s_2)\|_{\mathcal X_P}/\|s_1-s_2\|_{\mathcal Z_P}\), while \(L_{P,\mathrm{loc}}\) bounds \(\|G_P(a)-G_P(b)\|_{\mathcal Z_P}/\|a-b\|_{\mathcal X_P}\) on that ball. Combining the two, we obtain
\(\|\mathcal T_P(a)-\mathcal T_P(b)\|_{\mathcal X_P} \le C_{P,T}^{\mathrm{aff}}L_{P,\mathrm{loc}}\|a-b\|_{\mathcal X_P}.\)
Let \(q=C_{P,T}^{\mathrm{aff}}L_{P,\mathrm{loc}}<1\). On the stated closed invariant ball in \(\mathcal X_P\), the map \(\mathcal T_P\) is a contraction. For \(0<\eta\le1\), we define
\(\mathcal T_{P,\eta}=(1-\eta)I+\eta\mathcal T_P.\)
Convexity of the ball preserves invariance, and
\(\|\mathcal T_{P,\eta}(a)-\mathcal T_{P,\eta}(b)\|_{\mathcal X_P} \le \left(1-\eta+\eta q\right)\|a-b\|_{\mathcal X_P}.\)
Thus, the relaxed iteration converges geometrically to the same fixed point from every initial trajectory in that ball. In \(\mathcal Z_P\), we instead apply \Cref{thm:response-contraction} to obtain the global factor \(q=\ell_g\chi_{P,T}\). When this factor is below one, the same relaxation argument applies in the \(\mathcal Z_P\) norm.

For any $0<\eta\le1$, we find that the derivative of the relaxed map evaluates to $(1-\eta)I+\eta D\mathcal T_P$. If its norm has a strict bound below one near the fixed point, we obtain local convergence. For the local spectral test, let \(J\) be the Jacobian of a fixed finite-dimensional time discretization of the source-frozen map at one of its fixed points. The relaxed Jacobian is
\(J_\eta=(1-\eta)I+\eta J.\)
For each \(\lambda\in\sigma(J)\) and \(\eta>0\), the inequality \(|1+\eta(\lambda-1)|^2<1\) holds if and only if \(2(\operatorname{Re}\lambda-1) +\eta|\lambda-1|^2<0\).
Consequently, a positive \(\eta\) satisfying \(r(J_\eta)<1\) exists if and only if \(\operatorname{Re}\lambda<1\) for every \(\lambda\in\sigma(J)\). Under this condition, the stabilizing under-relaxation parameters satisfy
\(0<\eta\le1, \ \eta< \min_{\lambda\in\sigma(J)} \dfrac{2(1-\operatorname{Re}\lambda)} {|\lambda-1|^2}.\)
If the discretized map is continuously differentiable near the fixed point, we use this strict spectral condition to conclude that the relaxed iteration converges locally. We record the detailed affine estimate and the local convergence statements in the Supplementary Material.

\subsection{Neural--Spectral candidate and endpoint correction}

For each learned candidate, the Newton--HUM method acts on the extracted control rather than on network parameters.

We check the corrected controls against the numerical endpoint gate \eqref{eq:common-gate} by independent forward integration. Passing this numerical test alone does not constitute a rigorous endpoint certificate. For the seed-\(0\), \(P=8\), \(\beta=1\) learned candidate in \Cref{tab:validated}, we also apply the validated endpoint pipeline to prove the existence of a nearby exact Galerkin steering control.

The main matched timing table compares the two deterministic solvers. The source-freezing controls attain the common gate only after the final Newton--HUM step described above. We report the full-budget neural candidate study separately so that its approximation and correction roles can be read without conflating them.

\subsection{Common stopping, refinement, timing, and metrics}

The matched one-dimensional single-query solves use the exact endpoint metric
\(E_P(c)=\|\F_P(c)\|_2.\)
In the numerical stopping test, we evaluate
\(E_P^{\rm num}(c) =\|Y_c^{\rm num}(T)-Y_T\|_2,\)
where \(Y_c^{\rm num}(T)\) is the independently computed Galerkin endpoint. A run passes the numerical gate only if the forward integration reaches \(T\) with a finite state and
\begin{equation}\label{eq:common-gate}
E_P^{\rm num}(c)\le10^{-8}.
\end{equation}
The tables report numerical approximations to the endpoint metrics unless we explicitly identify a rigorous enclosure. The many-query and two-dimensional studies use the criteria declared in their respective descriptions. The primary matched levels are \(P=8,12,16\). For a validator with \(Q\ge P\) modes, we define \((M_{Q,P})_{jk}=\displaystyle\int_\omega e_je_k\,\mathrm dx\), \(j\le Q\), \(k\le P\). The same localized control is validated through
\(a_Q''+\Lambda_Qa_Q+G_Q(a_Q)+M_{Q,P}c=\mathbf{0}_{\R^Q},\)
from the projections of the same physical initial data. We set
\(E_Q(c)=\left\|\mathsf E_Q \binom{a_Q(T)-a_{T,Q}}{a'_Q(T)-a'_{T,Q}}\right\|_2\), \(E_{\rm free}=E_Q(\mathbf{0}_{L^2(0,T;\R^P)}),\)
for the explicitly declared validator. Here,
\(a_{T,Q} =\left(\langle u_T^0,e_j\rangle_{L^2}\right)_{j=1}^Q\), \(a'_{T,Q} =\left(\langle u_T^1,e_j\rangle_{L^2}\right)_{j=1}^Q.\)
We define the endpoint metric whenever the corresponding forward solution exists through \(T\). Thus, \(E_{2P}\), \(E_{4P}\) and \(E_{128}\) use the same applied field, including modes excited by localization. The matched one-dimensional deterministic protocol evaluates the same control in the \(2P\)- and \(4P\)-mode systems, without retuning. For the Neural--Spectral study at \(P=8\), the reported higher-mode validators use \(Q=32\) and \(Q=64\).

The dedicated truncation study uses the \(\beta=1\) least-squares/Newton--HUM controls at \(P\in\{8,12,16,24,32\}\) and evaluates each control in the same \(Q=128\) nonlinear system. Within each comparison, the validator receives the same continuous-time applied control. The reported control metrics are
\(\norm{c}_{\rm modal}^2=\int_0^T\norm{c(t)}_2^2\,\mathrm dt\), \(\|\mathbf1_\omega f_P\|_{L^2((0,T)\times\Omega)}^2 =\int_0^T c(t)^\top M_Pc(t)\,\mathrm dt.\)
The second is the squared physical applied-control norm and is the principal amplitude metric.

For the focusing problem, we apply coefficient continuation followed by the Newton--HUM method.

\section{Validated indicators and matched numerical results}\label{sec:results}

Unless stated otherwise, we use \(T=2.5\), \(\alpha=0\), the collar \(\omega_{0.3}\) in the one-dimensional single-query experiments, \(a(0)=2^{-1/2}(1,0.5,-0.3,0,\ldots,0)\), \(a'(0)=\mathbf{0}_{\R^P}\), and the rest state as the target. The many-query campaign varies the initial state, and the two-dimensional benchmark has the separate data specified below.

\subsection{Validated endpoint indicators for independently computed controls}

We apply the endpoint pipeline to three stored piecewise-cubic controls at \(P=8\) and \(\beta=1\): one from the endpoint least-squares/Newton--HUM method, one from the stabilized source-freezing method followed by one Newton--HUM step, and one from Neural--Spectral training followed by Newton--HUM completion. Within each row, the nominal endpoint, residual tubes, Gramian enclosure, and control norm refer to the same stored control function. Each is re-integrated with \(500\) validation time cells and \(4096\) spatial subintervals. The proof quantity \(\bar\varepsilon\) includes the state residual tube, while \(\underline\lambda\) combines the Gramian tube with an interval shifted-Cholesky verification of the nominal symmetric Gramian. We compute a certified upper bound
\(\bar\gamma =\operatorname{up}\!\left(\underline\lambda^{-1/2}\right).\)
The certificate uses the full-precision outward enclosures of \(\bar\varepsilon\), \(\underline\lambda\), \(\bar\gamma\), and the derivative Lipschitz constant. The table displays rounded summaries of these quantities.

\begin{table}[ht]
\centering
\resizebox{\textwidth}{!}{\begin{tabular}{l c c c c c c c}
\toprule
candidate & \(\varepsilon_{\rm nom}\) & \(\bar\varepsilon\) & \(\underline\lambda\) & \(\bar\gamma\) & \(\bar h\) & \(\rho_{\mathrm{cert}}\) & \(\|c\|_{L^2\!\left(0,T;\R^P\right)}\)\\
\midrule
least-squares/Newton--HUM method & \(3.77\times10^{-9}\) & \(4.48\times10^{-2}\) & 0.190145 & 2.2933 & 0.3027 & 0.111927 & 5.228\\
stabilized spectral source-freezing method + Newton--HUM method & \(2.20\times10^{-9}\) & \(4.49\times10^{-2}\) & 0.189922 & 2.2946 & 0.3041 & 0.112439 & 5.226\\
Neural--Spectral method + Newton--HUM method & \(1.96\times10^{-9}\) & \(2.59\times10^{-2}\) & 0.224576 & 2.1102 & 0.1481 & 0.056815 & 23.749\\
\bottomrule
\end{tabular}}

\caption{Validated \(P=8\), \(\beta=1\) endpoint indicators. Entries are rounded summaries of the full-precision enclosures used in the certificate. We show the nominal defect \(\varepsilon_{\rm nom}\) for comparison. It is computed along with the validation run, with \(4096\) spatial quadrature subintervals, and therefore differs slightly from the value \(E_P\) of the same control in \Cref{tab:tanh}, which uses the nominal validator. We compute the indicator and control-distance bound from the full-precision data, using the outward upper endpoint of \eqref{eq:Lendpoint}.}
\label{tab:validated}
\end{table}

For each candidate, the verified Gramian lower bound is positive, and the validated indicator is below one. Hence, \Cref{thm:certificate} guarantees an exact finite-dimensional steering control within the corresponding reported \(L^2\)-control neighborhood. Applying the same proof pipeline to deterministic and learned candidates shows that certification does not depend on how a candidate is generated. Measured in the modal norm, the certified radii are approximately \(2.1\%\), \(2.2\%\), and \(0.24\%\) of the corresponding control norms. We do not use indicator differences to rank the solvers, since the 500-cell Hermite tube contributes most of \(\bar\varepsilon\). These endpoint certificates cover \(P=8\), \(\beta=1\), whereas the separate candidate-independent response certificate establishes contraction through \(\beta=6\). For \(\beta=6\), we evaluate \eqref{eq:Lendpoint} to obtain \(L_8^{\rm end}\approx43.8173\). For scale, if a verified right-inverse bound of \(2.3\) were available, the smallness condition would require \(\bar\varepsilon < \dfrac{1}{2L_8^{\rm end}(2.3)^2} \approx2.1571\times10^{-3}\). This illustrative threshold is about twenty times smaller than the displayed \(500\)-cell residual bounds for the two deterministic \(\beta=1\) candidates. To certify a \(\beta=6\) candidate, we would need to compute its own residual and variational-Gramian enclosures. Separately, we report the focusing computation below as a successful Galerkin solve rather than as an endpoint-enclosed case.

\paragraph{\bf Demonstrated proof-cell refinement.}
The stored controls use \(500\) cubic-spline intervals. Keeping each control fixed, we refine only the proof partition to \(1000\), \(2000\), and \(4000\) cells, always preserving the original knots. The endpoint and variational ODEs are re-integrated at every level. Thus, the experiment measures the enclosure discretization rather than changing the candidate.

\begin{table}[ht]
\centering
\resizebox{\textwidth}{!}{\begin{tabular}{l c c c c c c}
\toprule
candidate & \(\bar\varepsilon_{500}\) & \(\bar\varepsilon_{4000}\) & \(\underline\lambda_{500}\) & \(\underline\lambda_{4000}\) & \(\bar h_{500}\) & \(\bar h_{4000}\)\\
\midrule
least-squares/Newton--HUM method & \(4.478\times10^{-2}\) & \(5.571\times10^{-3}\) & 0.190145 & 0.266340 & 0.30268 & 0.02688\\
stabilized spectral source-freezing method + Newton--HUM method & \(4.494\times10^{-2}\) & \(5.590\times10^{-3}\) & 0.189922 & 0.266358 & 0.30410 & 0.02697\\
Neural--Spectral method + Newton--HUM method & \(2.589\times10^{-2}\) & \(3.244\times10^{-3}\) & 0.224576 & 0.271362 & 0.14815 & 0.01537\\
\bottomrule
\end{tabular}}
\caption{Endpoints of the nested proof-cell refinement study. In the Supplementary Material, we provide the complete 500--4000-cell records, including \(\bar\gamma\) and the control-distance radius.}
\label{tab:indicator-refinement}
\end{table}

From \(500\) to \(4000\) proof cells, the endpoint-defect upper bounds decrease by factors of approximately \(8.04\), \(8.04\), and \(7.98\), respectively. The verified Gramian lower bounds also increase, and the corresponding indicators decrease by factors of approximately \(11.26\), \(11.27\), and \(9.64\). The reported control-distance bounds decrease to \(0.010868\), \(0.010906\), and \(0.006253\).

Because each stored control remains fixed, these changes measure improvements in the validated enclosures. The decrease in \(\bar h\) reflects both the smaller endpoint-defect bound and the improved right-inverse bound. Across different candidates, \(\bar h\) depends on the candidate and on the sharpness of its enclosures, so we should not use it alone to rank solver performance.

Figure~\ref{fig:indicator-refinement} shows how the enclosed endpoint defect and certificate indicator change under proof-cell refinement.
\begin{figure}[ht]
\centering
\includegraphics[width=.88\textwidth]{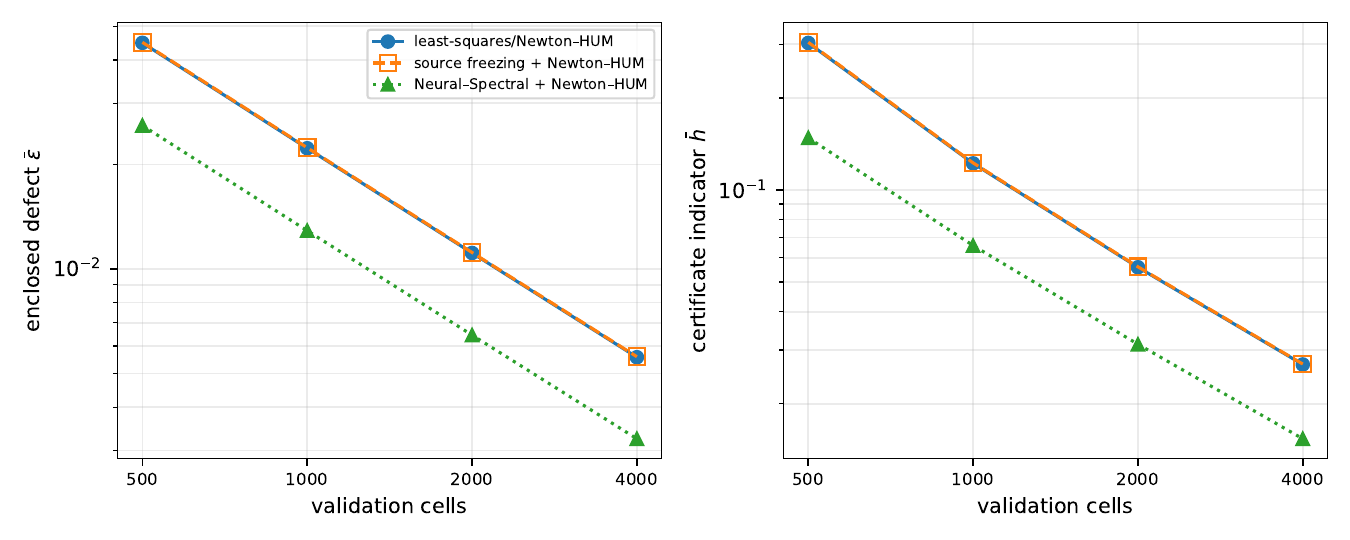}
\caption{Nested proof-cell convergence of the enclosed endpoint defect and certificate indicator for the same stored controls.}
\label{fig:indicator-refinement}
\end{figure}

\subsection{Matched tanh comparison}

\Cref{tab:tanh} reports all \(\beta\in\{1,6\}\) rows. The least-squares/Newton--HUM method passes the common Galerkin gate at every \(P\). The source-freezing controls reach endpoint errors between \(2.02\times10^{-8}\) and \(3.93\times10^{-8}\) and pass the gate after one Newton--HUM step. For each fixed \((P,\beta)\), we evaluate both methods in the same \(Q=4P\) nonlinear system. The quantity \(E_{4P}\), which we call the richer-validator error, measures its total terminal energy error, including the additional modes and changes in the first \(P\) modes caused by nonlinear coupling. These numerical evaluations provide a cross-resolution diagnostic without supplying a certified full-PDE error bound.

\begin{table}[htbp]
\centering
\tiny
\resizebox{\textwidth}{!}{\begin{tabular}{c l c c c c c c}
\toprule
\(P\) & method & \(\beta\) & \(E_P\) before correction & \(E_P\) & \(E_{4P}\) & \(\|\mathbf 1_\omega f_P\|_{L^2((0,T)\times\Omega)}\) & solve (s)\\
\midrule
8 & least-squares/Newton--HUM method & 1 & --- & \(3.97\times10^{-9}\) & \(1.74\times10^{-2}\) & 5.139 & 1.63\\
8 & stabilized spectral source-freezing method & 1 & \(2.02\times10^{-8}\) & \(2.41\times10^{-9}\) & \(1.77\times10^{-2}\) & 5.137 & 3.66\\
8 & least-squares/Newton--HUM method & 6 & --- & \(5.10\times10^{-9}\) & \(1.06\times10^{-2}\) & 6.297 & 2.97\\
8 & stabilized spectral source-freezing method & 6 & \(2.70\times10^{-8}\) & \(3.22\times10^{-9}\) & \(2.07\times10^{-2}\) & 6.550 & 12.98\\
12 & least-squares/Newton--HUM method & 1 & --- & \(4.41\times10^{-9}\) & \(7.75\times10^{-3}\) & 5.091 & 3.27\\
12 & stabilized spectral source-freezing method & 1 & \(2.40\times10^{-8}\) & \(1.87\times10^{-9}\) & \(7.96\times10^{-3}\) & 5.089 & 6.32\\
12 & least-squares/Newton--HUM method & 6 & --- & \(1.27\times10^{-9}\) & \(5.23\times10^{-3}\) & 6.233 & 4.81\\
12 & stabilized spectral source-freezing method & 6 & \(3.13\times10^{-8}\) & \(9.88\times10^{-10}\) & \(9.77\times10^{-3}\) & 6.483 & 19.50\\
16 & least-squares/Newton--HUM method & 1 & --- & \(3.20\times10^{-9}\) & \(6.89\times10^{-3}\) & 5.060 & 4.65\\
16 & stabilized spectral source-freezing method & 1 & \(2.91\times10^{-8}\) & \(1.11\times10^{-9}\) & \(7.09\times10^{-3}\) & 5.058 & 9.56\\
16 & least-squares/Newton--HUM method & 6 & --- & \(1.41\times10^{-9}\) & \(4.53\times10^{-3}\) & 6.190 & 7.69\\
16 & stabilized spectral source-freezing method & 6 & \(3.93\times10^{-8}\) & \(9.35\times10^{-10}\) & \(8.92\times10^{-3}\) & 6.438 & 27.66\\
\bottomrule
\end{tabular}
}
\caption{Matched comparison for \(g(u)=\beta\tanh u\). For each \((P,\beta)\), the two methods use the same numerical stopping rule, \(Q=4P\) validator, and physical-control metric. The column ``\(E_P\) before correction'' reports the Galerkin endpoint error of the control returned by source freezing. Each such control exceeds the common tolerance \(10^{-8}\) and receives one accepted full Newton--HUM step. The remaining columns refer to the corrected control, and the reported solve times include this correction. The validator dimension increases with \(P\).}
\label{tab:tanh}
\end{table}

At \(\beta=1\), we observe nearly equal physical norms and comparable errors in the richer validator. The least-squares/Newton--HUM method is faster at all three refinements. At \(\beta=6\), it also attains a smaller physical norm and smaller error in the richer validator, while the stabilized source-freezing method, followed by one Newton--HUM step, supplies a second high-accuracy control at every refinement. Together, these rows establish a deterministic reference against which Neural--Spectral candidates and future learned control operators can be assessed.

Figure~\ref{fig:tanh} compares the richer endpoint errors and physical applied-control norms under the matched protocol.
\begin{figure}[ht]
\centering
\includegraphics[width=.94\textwidth]{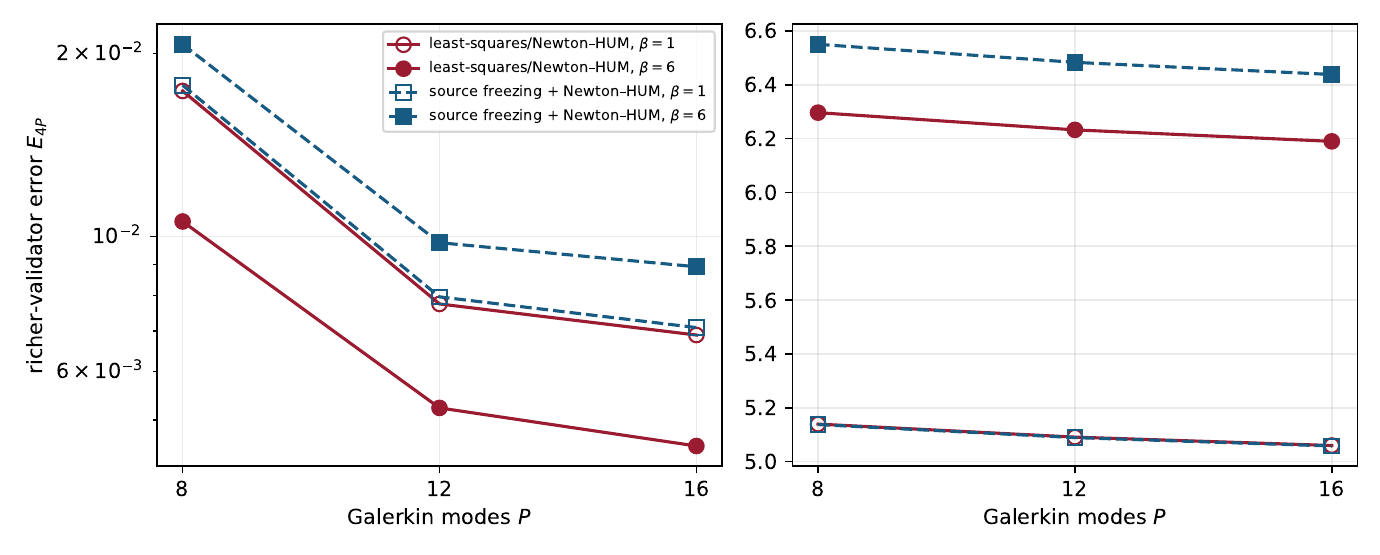}
\caption{Richer terminal errors and physical applied-control norms under the matched protocol.}
\label{fig:tanh}
\end{figure}

\subsection{Fixed-validator truncation study}

To separate validator resolution from control resolution, we extend the \(\beta=1\) least-squares/Newton--HUM method sequence to \(P=24,32\) and evaluate every control, without retuning, in the same nonlinear \(Q=128\) system. The uncontrolled endpoint error in that validator is \(E_{\rm free}=3.7764\).

Table~\ref{tab:fixed-validator} reports the endpoint errors and physical control norms in the fixed-validator study.
\begin{table}[ht]
\centering
\begin{tabular}{c c c c c}
\toprule
\(P\) & \(E_P\) & \(E_{128}\) & \(E_{128}/E_{\rm free}\) & physical norm\\
\midrule
8  & \(3.97\times10^{-9}\) & \(1.747\times10^{-2}\) & \(4.626\times10^{-3}\) & 5.139\\
12 & \(4.41\times10^{-9}\) & \(7.818\times10^{-3}\) & \(2.070\times10^{-3}\) & 5.091\\
16 & \(3.20\times10^{-9}\) & \(6.923\times10^{-3}\) & \(1.833\times10^{-3}\) & 5.060\\
24 & \(3.61\times10^{-9}\) & \(4.342\times10^{-3}\) & \(1.150\times10^{-3}\) & 5.043\\
32 & \(3.16\times10^{-9}\) & \(2.053\times10^{-3}\) & \(5.436\times10^{-4}\) & 5.038\\
\bottomrule
\end{tabular}
\caption{One-dimensional \(\beta=1\) truncation study against one fixed \(Q=128\) nonlinear validator. Every control satisfies the same Galerkin gate and is transferred without retuning.}
\label{tab:fixed-validator}
\end{table}

The fixed-validator error decreases at every tested refinement, by a factor of \(8.51\) from \(P=8\) to \(P=32\). The physical norm changes from \(5.043\) to \(5.038\) between the last two levels, a relative change of approximately \(0.1\%\).

Writing \(c^{(P)}\) for the control computed at level \(P\), we define the adjacent-level refinement exponents by
\(p_i= \dfrac{ \log\!\left( E_{128}(c^{(P_i)})/ E_{128}(c^{(P_{i+1})}) \right)} {\log(P_{i+1}/P_i)}.\)
The resulting values are \(1.98\), \(0.42\), \(1.15\), and \(2.60\). These exponents describe the observed decay along the constructed control sequence. They do not establish an asymptotic convergence rate. We use the common validator to test how the computed controls transfer to a larger modal space. A full-PDE error bound additionally requires validated forward integration and a rigorous estimate of the difference between the \(128\)-mode trajectory and the full PDE solution. An exact-control certificate at the PDE level must also verify the corresponding continuous-space hypotheses of \Cref{thm:certificate}.

Figure~\ref{fig:fixed-validator} shows the terminal errors
in the common \(Q=128\) validator.
\begin{figure}[ht]
\centering
\includegraphics[width=.6\textwidth]{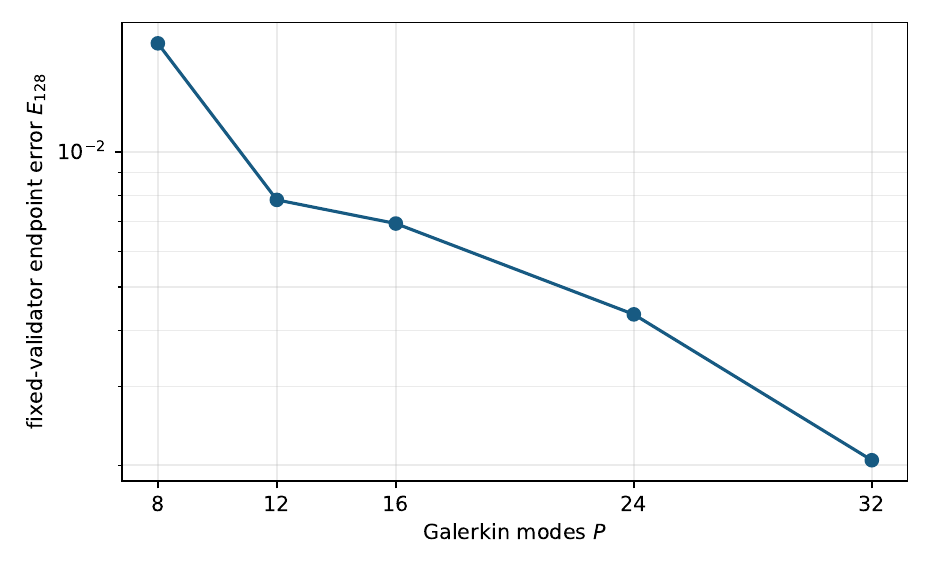}
\caption{Terminal error of the \(P\)-mode controls in the common \(Q=128\) nonlinear validator.}
\label{fig:fixed-validator}
\end{figure}

\subsection{Full-budget Neural--Spectral candidates and Newton--HUM completion}

The Neural--Spectral method reports three independent seeds for each tanh coefficient. We deliberately separate candidate construction from endpoint completion: the independent nonlinear forward integrator first scores every raw control, and then we apply the same endpoint Newton--HUM correction used before directly to the extracted control. This organization measures both the approximation delivered by the neural ansatz and the effectiveness of the subsequent Newton--HUM correction.

Table~\ref{tab:neural-1d} reports the raw neural candidates and their Newton--HUM completions.
\begin{table}[htbp]
\centering
\tiny
\resizebox{\textwidth}{!}{\begin{tabular}{c r c c c c c c c}
\toprule
\(\beta\) & seed & raw \(E_8\) & raw \(E_{32}\)
& corrected \(E_{32}\) & corrected physical norm
& train (s) & corrected \(E_8\) & correction (s)\\
\midrule
1 & 0  & \(2.97\times10^{-1}\) & \(3.29\times10^{-1}\)
& \(1.42\times10^{-1}\) & 16.024 & 1830.8
& \(2.04\times10^{-9}\) & 1.3\\
1 & 10 & \(4.01\times10^{-1}\) & \(4.25\times10^{-1}\)
& \(1.42\times10^{-1}\) & 10.427 & 1788.6
& \(6.87\times10^{-9}\) & 1.3\\
1 & 20 & \(2.31\times10^{-1}\) & \(2.74\times10^{-1}\)
& \(1.48\times10^{-1}\) & 10.874 & 1767.1
& \(3.02\times10^{-9}\) & 1.2\\
6 & 0  & \(4.38\times10^{-1}\) & \(4.49\times10^{-1}\)
& \(9.66\times10^{-2}\) & 14.727 & 1758.7
& \(9.53\times10^{-9}\) & 1.2\\
6 & 10 & \(4.93\times10^{-1}\) & \(5.16\times10^{-1}\)
& \(1.52\times10^{-1}\) & 13.274 & 1732.0
& \(9.94\times10^{-9}\) & 1.3\\
6 & 20 & \(2.37\times10^{-1}\) & \(2.68\times10^{-1}\)
& \(1.26\times10^{-1}\) & 11.527 & 1721.1
& \(4.72\times10^{-9}\) & 1.2\\
\bottomrule
\end{tabular}
}
\caption{Full-budget Neural--Spectral tanh candidates at \(P=8\) and their Newton--HUM completions. We compute endpoint errors by independent nonlinear forward integration and evaluate the same raw or corrected control at both reported modal levels without retuning. We report training and correction times separately.}
\label{tab:neural-1d}
\end{table}

The raw endpoint columns are independent forward errors rather than training losses and therefore provide a measure of candidate quality. Across the six runs, the independently computed raw Galerkin defects range from \(0.231\) to \(0.493\). We report no endpoint certificate for these raw candidates. In all six runs, two accepted Newton--HUM steps reduce the computed \(E_8\) error below \(10^{-8}\). This number is determined by the recorded initial candidates and should not be read as a general convergence property of the method. Correction also lowers the computed \(E_{32}\) error in every run, although the corrected values remain between \(0.0966\) and \(0.152\), as reported in the Supplementary Material. At the common reconstruction level \(P=8\), neural training takes \(1721\)--\(1831\) seconds, followed by \(1.2\)--\(1.3\) seconds for endpoint correction. The corrected physical control norms range from \(10.427\) to \(16.024\). The corresponding deterministic runs of the least-squares/Newton--HUM method take \(1.63\) seconds at \(\beta=1\) and \(2.97\) seconds at \(\beta=6\), with physical norms \(5.139\) and \(6.297\), respectively. Thus, the deterministic method requires less reported computation time and lower physical control effort in these single-query tests. In these tests, the neural approach therefore contributes structured candidates that the Newton--HUM correction can complete, while the deterministic references set the efficiency and effort benchmark. The validated row in \Cref{tab:validated} converts one completed learned control into a proof-level exact-control neighborhood. The Supplementary Material summarizes seed-by-seed richer errors and corrected norms.

\subsection{Focusing cubic: continued Newton--HUM globalization}

For the focusing benchmark, we consider \(g_\beta(u)=-\beta u^3\), with \(\beta\ge0\), and use coefficient continuation to reach the target value \(\beta=11\).
For the focusing dynamics, each accepted continuation stage must complete the numerical forward integration through \(T\) with a finite terminal state and satisfy the Galerkin endpoint tolerance. We apply the Newton--HUM method to the endpoint map at each coefficient value along the tested continuation sequence. We report the smallest finite final value of \(r_{\rm fp}^{(n)}\) among the tested relaxation runs.

Table~\ref{tab:focusing} reports the results at the final continuation coefficient \(\beta=11\).
\begin{table}[htbp]
\centering
\resizebox{\textwidth}{!}{\begin{tabular}{c l c c c c c}
\toprule
\(P\) & method & status & best final proposal residual & \(E_P\) & \(E_{2P}\) & solve (s)\\
\midrule
8 & least-squares/Newton--HUM method & passed & not applicable & \(5.32\times10^{-9}\) & \(4.09\times10^{-2}\) & 62.40\\
8 & stabilized spectral source-freezing method & budget completed & \(5.87\times10^{-2}\) & not finite & not finite & 63.53\\
12 & least-squares/Newton--HUM method & passed & not applicable & \(3.07\times10^{-9}\) & \(2.06\times10^{-2}\) & 95.41\\
12 & stabilized spectral source-freezing method & budget completed & \(5.61\times10^{-2}\) & not finite & not finite & 113.87\\
16 & least-squares/Newton--HUM method & passed & not applicable & \(2.87\times10^{-9}\) & \(1.93\times10^{-2}\) & 137.57\\
16 & stabilized spectral source-freezing method & budget completed & \(5.44\times10^{-2}\) & not finite & not finite & 164.53\\
\bottomrule
\end{tabular}}
\caption{Focusing cubic at \(\beta=11\). Entries marked ``not finite'' indicate that independent forward integration of the selected candidate did not produce a finite terminal state at \(T\), so we report no endpoint error. Solve times of the source-freezing rows also include the three forward validations at \(P\), \(2P\), and \(4P\).}
\label{tab:focusing}
\end{table}

We find that the continued Newton--HUM method passes \eqref{eq:common-gate} at \(P=8,12,16\), with Galerkin errors \(5.32\times10^{-9}\), \(3.07\times10^{-9}\), and \(2.87\times10^{-9}\), respectively. The \(2P\) errors decrease from \(4.09\times10^{-2}\) to \(1.93\times10^{-2}\), and the physical control norm decreases from \(5.13\) to \(5.03\). The source-freezing runs attempt each prescribed relaxation value and perform 230, 231, and 231 iterations in total at \(P=8,12,16\), respectively. We record their best source-update residuals as \(0.0587\), \(0.0561\), and \(0.0544\). Independent nonlinear forward integration of the selected candidates does not produce a finite terminal state at \(T\), and the protocol returns no accepted control. At each of the three tested reconstruction levels, we use coefficient continuation followed by the Newton--HUM method to reach \(\beta=11\) and satisfy the numerical Galerkin endpoint tolerance. These results provide evidence for the effectiveness of the tested continuation sequence on this benchmark.

Figure~\ref{fig:focusing} shows the entry and corrected
terminal defects at each continuation stage.
\begin{figure}[ht]
\centering
\includegraphics[width=.78\textwidth]{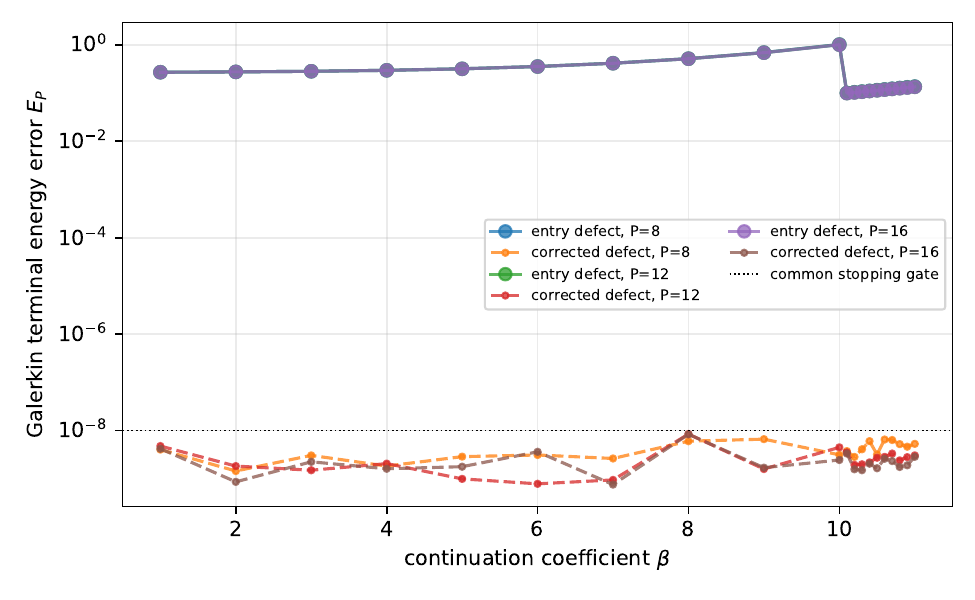}
\caption{Entry and corrected terminal defects at every focusing continuation stage. All three refinements pass the common \(10^{-8}\) gate through \(\beta=11\).}
\label{fig:focusing}
\end{figure}

The results illustrate complementary roles for the methods. The source-freezing method uses reusable affine steering solves, with convergence guaranteed under the respective contraction conditions of \Cref{thm:full-energy,thm:full-exact-response,thm:response-contraction}. Endpoint continuation succeeds in the reported focusing tests, where the prescribed source-freezing protocol returns no accepted control. Under the hypotheses of \Cref{thm:eventual-neural-certificate}, sufficiently accurate neural approximations lie in the region where the Newton--HUM correction converges to an exact finite-dimensional steering control.

\subsection{Many-query reuse and two-dimensional physical metrics}

We consider \(64\) steering problems with \(P=16\), \(T=2.5\), the collar \(\omega_{0.3}\), nonlinearity \(g(u)=\beta u^3\) with \(\beta=8\), and Euclidean modal control cost. For each query \(q\in\{1,\ldots,64\}\), the prescribed initial coefficients \(a_0^{(q)},a_1^{(q)}\in\R^{16}\) satisfy
\((a_0^{(q)})^\top\Lambda_{16}a_0^{(q)} +\norm{a_1^{(q)}}_2^2=1\), and the target is \(a^{(q)}(T)=\mathbf{0}_{\R^{16}}\), \((a^{(q)})'(T)=\mathbf{0}_{\R^{16}}\). Thus, every query has an initial energy norm of one.

At each source-freezing step, the affine dynamics have the same matrices
\[
 \mathsf A=
 \begin{pmatrix}
  0&I_{16}\\
  -\Lambda_{16}&0
 \end{pmatrix},
 \qquad
 \mathsf B=
 \begin{pmatrix}
  0\\-M_{16}
 \end{pmatrix}.
\]
Consequently, the Gramian \(W_T= \int_0^T e^{\mathsf A(T-t)} \mathsf B\mathsf B^\top e^{\mathsf A^\top(T-t)} \,\mathrm{d}t\) is independent of the query and the frozen source. We therefore reuse one factorization of \(W_T\) for all affine subproblems. The initial state and frozen source enter through \(d_s^{\mathrm{aff}}\), which changes between solves. The campaign uses 64 low-mode directions, normalizes each to energy norm one, and steers them to rest with one reused factorization. Because this experiment probes reuse rather than the matched single-query accuracy gate, it uses the relative gate, fixed before the runs, \(E_{16}^{\rm num}<10^{-2}E_{\rm free}\), where \(E_{\rm free}\) is the numerically evaluated uncontrolled terminal error in the same 16-mode system for the corresponding query. The response certificate in \Cref{prop:chi} is specific to the \(P=8\), \(T=2.5\) tanh problem and therefore does not certify this \(P=16\) campaign. Convergence and iteration counts here remain empirical results for the sampled initial states.

Table~\ref{tab:manyquery} summarizes the many-query campaign.
\begin{table}[ht]
\centering
\begin{tabular}{l c}
\toprule
quantity & value\\
\midrule
fixed-point convergence & 64/64\\
relative steering criterion & 64/64\\
median \(E_{16}^{\rm num}\) among passing cases & \(3.1456\times10^{-7}\)\\
median \(E_{16}^{\rm num}/E_{\rm free}\) among converged cases & \(3.1374\times10^{-7}\)\\
\bottomrule
\end{tabular}
\caption{Many-query campaign at \(P=16\) and \(\beta=8\). Initial states share one prescribed energy, and all affine subproblems reuse the same Gramian factorization.}
\label{tab:manyquery}
\end{table}

All 64 sampled instances reach the fixed-point tolerance and satisfy the relative steering criterion. The median terminal error is \(3.1456\times10^{-7}\), the median querywise relative terminal error is \(3.1374\times10^{-7}\), and the median iteration count is four.

For the two-dimensional linear benchmark, we take \(\Omega=(0,1)^2\), \(g=0\), \(\alpha=0\), \(T=7\), and \(\omega_{0.3}=\left\{(x,y)\in(0,1)^2:\min(x,1-x,y,1-y)<0.3\right\}\). We use the eigenfunctions \(e_{jk}=2\sin(j\pi x)\sin(k\pi y)\) with \(\mu_{jk}=\pi^2\!\left(j^2+k^2\right)\), and we order the pairs \((j,k)\), \(1\le j,k\le K\) lexicographically. The reconstruction set corresponds to \(K=4\), that is, \(P=16\) modes, whereas the finer validation set corresponds to \(K=8\), that is, \(64\) modes. We prescribe the initial state and the target as
\(u^0=\sin(\pi x)\sin(\pi y) +\dfrac{1}{2}\sin(\pi x)\sin(2\pi y)\), \(u^1=\mathbf{0}_{L^2(\Omega)}\), \(\Phi_T=\mathbf{0}_{\E},\)
so that \(a_{11}(0)=1/2\) and \(a_{12}(0)=1/4\), while every other coefficient vanishes. In the formulas for the physical cost, we write \(M=M_{16}\), and we read \(E_{K\times K}\) as the energy error at time \(T\) over this set of tensor modes.

For a physical-space neural control \(f_\theta\), the forcing coefficient of a validation mode is
\(d_{jk}^{\theta}(t) =\int_{\omega_{0.3}} f_\theta(t,x,y)e_{jk}(x,y)\,\mathrm dx\,\mathrm dy.\)
For a modal control, the vector of these coefficients is \(M_{Q,P}c(t)\), where \(P=16\) and \(Q=16\) or \(64\). We evaluate every control through the same terminal-state energy norm and the same physical applied-control norm,
\(\left\|\mathbf1_{\omega_{0.3}}f\right\|_{L^2((0,T)\times\Omega)} = \left( \int_0^T\int_{\omega_{0.3}} |f(t,x,y)|^2\,\mathrm dx\,\mathrm dy\,\mathrm dt \right)^{1/2}.\)
For modal controls, the square of this norm equals \(\int_0^T c(t)^\top M_{16}c(t)\,\mathrm dt\).

For the projected \(4\times4\) linear wave system, we use the weighted Gramian formula \eqref{eq:weighted-control-formula} to obtain the exact minimum-cost Galerkin steering control. We compute its numerical realizations for \(\Sigma=I_{16}\) and \(\Sigma=M_{16}\), which minimize the modal coefficient cost and the physical applied-control cost, respectively. We insert both computed controls without retuning into the \(8\times8\) validator.

\begin{table}[ht]
\centering
\begin{tabular}{l c c c c}
\toprule
cost & \(E_{4\times4}\) & \(E_{8\times8}\) & modal norm & physical norm\\
\midrule
free evolution & \(2.8318\) & \(2.8318\) & 0 & 0\\
modal coefficient & \(1.53\times10^{-10}\) & \(4.00\times10^{-3}\) & 2.103 & 2.035\\
physical energy & \(1.37\times10^{-10}\) & \(7.12\times10^{-3}\) & 2.622 & 1.968\\
\bottomrule
\end{tabular}

\caption{Weighted HUM/Gramian references for the two-dimensional linear benchmark. The physical-energy weighting attains the smaller applied-field norm, while the coefficient weighting attains the smaller richer endpoint error.}
\label{tab:2d-gramian}
\end{table}

Using the conservation of the uncontrolled linear-wave energy, we compute
\(E_{\rm free}=\dfrac{\pi\sqrt{13}}{4}\approx2.8318.\)
We obtain \(4\times4\) controlled endpoint errors of \(1.53\times10^{-10}\) and \(1.37\times10^{-10}\). At the \(8\times8\) validation level, we find that the coefficient-cost control has an error \(3.996\times10^{-3}\) with the physical norm \(2.035\), whereas the physical-energy control has an error \(7.118\times10^{-3}\) with the smaller norm \(1.968\). We can reuse this metric-aware comparison for neural and deterministic controls alike to make the effort–accuracy tradeoff explicit.

We next compare the two weighted-Gramian references with three neural approaches: the Neural--Spectral method, a hard-boundary PINN~\cite{LagarisLikasFotiadis1998}, and a vanilla PINN~\cite{Raissi2019}. The hard-boundary PINN enforces the homogeneous Dirichlet condition through its state parametrization, whereas the vanilla PINN imposes this condition through a penalty in the training loss. We apply the same validator and physical norm to every full-budget two-dimensional neural run. In \Cref{tab:neural-2d}, we report the median and full seed ranges alongside the two numerical weighted-Gramian references. We use the common validators to compare controls from the reported method-specific training protocols. We find that the Neural--Spectral method attains a median richer endpoint error equal to \(31.6\%\) of \(E_{\rm free}\). Its median wall-clock time for training and validation is \(2039.8\) seconds, compared with \(176.9\) seconds for the hard-boundary PINN, while the single vanilla-PINN run takes \(192.4\) seconds. These results show a lower endpoint error for the tested Neural--Spectral implementation, but do not isolate the effect of architecture from the optimization protocol and computational budget.

\begin{table}[htbp]
\centering
\tiny
\resizebox{\textwidth}{!}{\begin{tabular}{l c c c c c}
\toprule
method & seeds & \(E_{4\times4}\) & \(E_{8\times8}\) & \(E_{8\times8}/E_{\rm free}\) & physical norm\\
\midrule
free evolution & deterministic & \(2.8318\) & \(2.8318\) & 1.000 & 0\\
Neural--Spectral method & 0,1,2 & \(8.73\times10^{-1}\,[8.64\times10^{-1},1.01\times10^{0}]\) & \(8.95\times10^{-1}\,[8.94\times10^{-1},1.03\times10^{0}]\) & 0.316 [0.316,0.364] & 17.52 [16.52,18.25]\\
hard-boundary PINN method & 0,1,2 & \(2.83\times10^{0}\,[2.83\times10^{0},2.83\times10^{0}]\) & \(2.83\times10^{0}\,[2.83\times10^{0},2.83\times10^{0}]\) & 1.000 [1.000,1.001] & \(4.8\times10^{-3}\,[3.6\times10^{-3},1.8\times10^{-2}]\)\\
vanilla PINN method & 0 & \(2.80\times10^{0}\) & \(2.80\times10^{0}\) & 0.987 & 0.22\\
Gramian method, modal cost & deterministic & \(1.53\times10^{-10}\) & \(4.00\times10^{-3}\) & 0.00141 & 2.04\\
Gramian method, physical cost & deterministic & \(1.37\times10^{-10}\) & \(7.12\times10^{-3}\) & 0.00251 & 1.97\\
\bottomrule
\end{tabular}
}
\caption{Linear two-dimensional learned candidates and weighted Gramian references under the same reconstruction/finer endpoint validators and physical applied-control norm. For methods with three seeds, we summarize each metric separately as its median \([\min,\max]\). The vanilla-PINN row reports one seed.}
\label{tab:neural-2d}
\end{table}

Under the tested formulations and training budgets, the Neural--Spectral method attains a median richer endpoint error of \(31.6\%\) of \(E_{\rm free}\). The hard-boundary PINN median and the single vanilla PINN run remain at approximately \(100.0\%\) and \(98.7\%\) of \(E_{\rm free}\), respectively, and apply much smaller controls. The final stored training logs report mean squared PDE residuals at the \(9\times9\times9\) interior collocation points ranging approximately from \(3.96\times10^{-8}\) to \(9.88\times10^{-7}\). These collocation residuals alone do not establish convergence of the PINN states to a solution of the controlled wave equation. These results show that the tested Neural--Spectral configuration produces a larger reduction in endpoint error, but they do not establish that modal structure alone causes the improvement. The formulations differ in parameterization, loss construction, and optimization budget. The recorded Neural--Spectral wall-clock times for training and validation range from \(2030.3\) to \(2056.4\) seconds, compared with \(137.7\) to \(190.4\) seconds for the hard-boundary PINN and \(192.4\) seconds for the vanilla PINN.

The deterministic Gramian references attain richer endpoint errors of approximately \(0.14\%\)--\(0.25\%\) of \(E_{\rm free}\) and physical control norms \(1.97\)--\(2.04\), compared with the Neural--Spectral median physical norm \(17.52\). Thus, these references attain both smaller endpoint errors and smaller applied-control norms than the reported neural candidates.

\section{Discussion and conclusion}\label{sec:conclusion}

Our central result is a verifiable bridge from computation to exact control. We use the a posteriori theorem to convert quantitative endpoint information into an exact-control existence result. To apply it, we require a continuously differentiable endpoint map, a bounded right inverse at the candidate, and a Lipschitz derivative bound on a neighborhood containing the certified correction. For the reported \(P=8\), \(\beta=1\) tanh problem, we use residual tubes and interval Gramian bounds to verify these conditions for one control from each of three computational paths. The resulting neighborhoods contain exact finite-dimensional steering controls. Refining the proof partition tightens the reported bounds while leaving each stored candidate unchanged.

The response enclosure addresses a different question. For the geometry and Euclidean modal cost specified in \Cref{prop:chi}, with \(T=2.5\), we use the bound \(\chi_{8,2.5}<0.160561\) to prove the global contraction of the exact source-frozen displacement map for \(g(u)=\beta\tanh u\) with \(0\le\beta\le6\). With this bound, we establish candidate-independent convergence for the finite-mode map and use the endpoint certificate to establish local existence near a candidate satisfying the hypotheses of \Cref{thm:certificate}.

In the matched one-dimensional tanh tests, we find that the least-squares/Newton--HUM method reaches the numerical endpoint tolerance in less reported solve time than the stabilized source-freezing method, whose controls require one final Newton--HUM step to reach that tolerance. At \(\beta=6\), we also obtain lower physical control norms and lower richer-validator endpoint errors with this method. We complete coefficient continuation in the three reported focusing tests. All six one-dimensional Neural--Spectral candidates reach the Galerkin numerical tolerance after two accepted Newton--HUM steps, and we certify one completed learned candidate at \(P=8\), \(\beta=1\). Eventual neural certifiability remains conditional on the hypotheses of \Cref{thm:eventual-neural-certificate}. In two dimensions, we find that the tested Neural--Spectral controls improve on free evolution, while the numerical weighted-Gramian references achieve smaller endpoint errors and physical control norms. Because the certificate indicator depends on both the candidate and the enclosure accuracy, it should not be interpreted as a measure of computational cost or control effort.

In the fixed \(Q=128\) study, we find numerical evidence for improved transfer to a larger Galerkin model: the computed endpoint error decreases from \(1.747\times10^{-2}\) at \(P=8\) to \(2.053\times10^{-3}\) at \(P=32\). At the same time, the Galerkin defects remain near \(10^{-9}\). A full-PDE endpoint enclosure would additionally require validated numerical integration and a rigorous comparison between the Galerkin and PDE trajectories, including the effect of unresolved modes on the resolved-mode dynamics.

\section*{Acknowledgements}

\paragraph{\bf AI Use Declaration} AI tools (ChatGPT and Claude) were used for polishing the presentation, language editing, code review, and manuscript checking. Especially, the numerical results were obtained by executing code developed with the assistance of Claude Opus 4.8. The authors reviewed and validated the results. The authors assume responsibility for all content.

\paragraph{\bf Competing Interests.} The authors declare that they have no competing interests.

\bibliographystyle{m2an-local}
\bibliography{references}
\end{document}


\maketitle

We provide auxiliary proofs and computational details to support the main article. We include the source-frozen HUM and Gramian constructions, additional convergence analysis, extended numerical results, and implementation and verification details. In the main article, we prove the energy-space results, the Neural--Spectral method results, the Newton--HUM method results, and the validated-enclosure results.

\section{Common notation and analytical hypotheses}
We use the setting and notation of the main article. Let \(\Omega=(0,1)^N\), where \(N\in\{1,2,\ldots\}\), and let \(\omega\subset\Omega\) be a nonempty open control region. We fix a final time \(T>0\) and a parameter \(\alpha\ge 0\). We equip \(H_0^1(\Omega)\) with the gradient norm \(\|v\|_{H_0^1(\Omega)}=\|\nabla v\|_{L^2(\Omega)}\) and define the energy space \(\mathcal E=H_0^1(\Omega)\times L^2(\Omega)\) with norm
\[
 \|(v,z)\|_{\mathcal E}^2
 =\|\nabla v\|_{L^2(\Omega)}^2
  +\alpha\|v\|_{L^2(\Omega)}^2
  +\|z\|_{L^2(\Omega)}^2.
\]
The control space is \(\mathcal V=L^2(0,T;L^2(\omega))\).

For vectors, \(|\cdot|\) and \(\|\cdot\|_2\) denote the Euclidean norm, with the dimension specified when needed. For matrices, \(\|\cdot\|_2\) denotes the induced Euclidean operator norm and \(\|\cdot\|_F\) the Frobenius norm. We write \(\mathcal L(X,Y)\) for the space of bounded linear maps between normed spaces \(X\) and \(Y\), and \(\mathcal L(X)=\mathcal L(X,X)\).

For prescribed endpoint states \(\Phi_0=(u^0,u^1)\in\mathcal E\) and \(\Phi_T=(u_T^0,u_T^1)\in\mathcal E\), we seek a control \(f\in\mathcal V\) and a state \(u\) satisfying
\begin{equation}\label{Wave1}
\begin{cases}
 u_{tt}-\Delta u+\alpha u+g(u)+ f\,\mathbf{1}_\omega=0
 & \text{in }(0,T)\times\Omega,\\
 u=0
 & \text{on }(0,T)\times\partial\Omega,\\
 (u,u_t)(0)=\Phi_0,\qquad (u,u_t)(T)=\Phi_T.
\end{cases}
\end{equation}
Here, \(g:\mathbb R\to\mathbb R\) is the scalar nonlinearity. For \(f\in L^2(0,T;L^2(\omega))\), we define its zero extension to \(\Omega\) by
\[
 \widetilde f(t,x)=
 \begin{cases}
  f(t,x), & x\in\omega,\\
  0, & x\in\Omega\setminus\omega,
 \end{cases}
 \qquad\text{for almost every }t\in(0,T).
\]

We use the same symbol \(g\) for the scalar function and its Nemytskii operator, defined by \(g(v)(x)=g(v(x))\). We use the following assumption only in results that directly refer to it.

\begin{assumption}
\label{ass:global_lipschitz}
The scalar function \(g:\mathbb R\to\mathbb R\) induces a globally Lipschitz Nemytskii operator from \(H_0^1(\Omega)\) into \(L^2(\Omega)\), which means there exists \(L_g\ge 0\) such that
\[
 \|g(v)-g(w)\|_{L^2(\Omega)}
 \le L_g\|v-w\|_{H_0^1(\Omega)}
 \qquad\text{for all }v,w\in H_0^1(\Omega).
\]
\end{assumption}
Assumption~\ref{ass:global_lipschitz} does not include differentiability of \(g\). When results require derivative bounds or a scalar Lipschitz bound, we state those hypotheses separately.

We formulate the steering problem in the energy space by specifying the regularity and endpoint conditions of a control--state pair.
\begin{definition}
An admissible control--state pair consists of a trajectory \(\Phi=(u,\partial_tu)\in C([0,T];H_0^1(\Omega)\times L^2(\Omega))\) and a control forcing \((0,f\mathbf 1_\omega)\) with \(f\in L^2((0,T)\times\omega)\). They satisfy the controlled equation~\eqref{Wave1} in the mild energy sense and the endpoint conditions \(\Phi(0)=\Phi_0\) and \(\Phi(T)=\Phi_T\). Whenever \(f\) is paired with functions on \(\Omega\), we identify it with its zero extension from \(\omega\) to \(\Omega\).
\end{definition}

\section{Source-frozen Picard control: HUM realization}
\label{sec:sf_picard_hum}

We formulate the source-frozen Picard iteration for the resulting Galerkin system and realize each affine steering step through the Hilbert Uniqueness Method (HUM). We then derive global and local convergence criteria for this iteration. In the main article, we derive the corresponding energy-space analysis.

\subsection{Source-frozen affine steering problem}
\label{sub:source_frozen_framework}

We fix \(P\in\{1,2,\ldots\}\) and choose orthonormal Dirichlet modes \(e_1,\ldots,e_P\) with corresponding eigenvalues \(\mu_1,\ldots,\mu_P\). We write
\[
 u_P(t,x)=\sum_{j=1}^{P}a_j(t)e_j(x),
 \qquad
 f_P(t,x)=\sum_{j=1}^{P}c_j(t)e_j(x),
\]
where \(a(t)\in\mathbb R^P\) and \(c\in L^2(0,T;\mathbb R^P)\). Here, we define \(f_P\) over the entire domain \(\Omega\). Its restriction to \(\omega\) is the control, and \(\mathbf 1_\omega f_P\) is the applied forcing on \(\Omega\), as in the main article.

We define
\begin{align}
 \Lambda_P&=\operatorname{diag}
 (\mu_1+\alpha,\ldots,\mu_P+\alpha), \notag\\
 M&=(M_{j\ell})_{j,\ell=1}^{P},
 \qquad
 M_{j\ell}=\int_\omega e_j(x)e_\ell(x)\,\mathrm{d}x. \label{eq:overlap-matrix}
\end{align}
We also write \(M_P=M\) when emphasizing the number of modes. For \(a\in\mathbb R^P\), we set \(u_a=\sum_{j=1}^Pa_je_j\) and define the projected nonlinearity by
\[
 (G_P(a))_j
 =\int_\Omega g(u_a(x))e_j(x)\,\mathrm dx,
 \qquad j\in\{1,\ldots,P\},
\]
whenever \(g(u_a)\in L^2(\Omega)\). For a measurable trajectory $a:(0,T)\to\mathbb R^P$,
we use the same notation for the pointwise extension
\[
 [G_P(a)](t):=G_P(a(t))
 \quad\text{for almost every }t\in(0,T).
\]
Whenever we use $G_P(a)$ as a frozen source, we require $G_P(a)\in L^2(0,T;\mathbb R^P)$. The hypotheses of each result below ensure this integrability on the trajectories under consideration.

We define the projected endpoint vectors by
\[
 \begin{aligned}
 a_0&=\left(\langle u^0,e_j\rangle_{L^2(\Omega)}\right)_{j=1}^{P},
 &
 a_1&=\left(\langle u^1,e_j\rangle_{L^2(\Omega)}\right)_{j=1}^{P},
 \\
 a_T&=\left(\langle u_T^0,e_j\rangle_{L^2(\Omega)}\right)_{j=1}^{P},
 &
 a'_T&=\left(\langle u_T^1,e_j\rangle_{L^2(\Omega)}\right)_{j=1}^{P},
 \\
 Y_0&=\begin{pmatrix}a_0\\a_1\end{pmatrix},
 &
 Y_T&=\begin{pmatrix}a_T\\a'_T\end{pmatrix}.
 \end{aligned}
\]
Testing the controlled equation against the selected modes, we obtain
\[
\begin{aligned}
 a''(t)+\Lambda_Pa(t)+G_P(a(t))+Mc(t)
 =\mathbf 0,\qquad \begin{pmatrix}a(0)\\a'(0)\end{pmatrix}=Y_0,
 \qquad
 \begin{pmatrix}a(T)\\a'(T)\end{pmatrix}=Y_T.
\end{aligned}
\]

For a trial trajectory \(a^{(n)}\), we define \(s^{(n)}(t)=G_P(a^{(n)}(t))\) and suppose that \(s^{(n)}\in L^2(0,T;\mathbb R^P)\). This holds, in particular, for \(a^{(n)}\in C([0,T];\mathbb R^P)\) under Assumption~\ref{ass:global_lipschitz}. We freeze this source to obtain the affine equation \(a''+\Lambda_Pa+Mc+s^{(n)}=\mathbf 0\). With \(Y=(a,a')^\top\), its first-order formulation is
\begin{equation}\label{eq:source_frozen_first_order_system}
 Y'(t)=\mathsf A_PY(t)+\mathsf B_Pc(t)+b^{(n)}(t),
 \qquad Y(0)=Y_0,\quad Y(T)=Y_T,
\end{equation}
where
\[
 \mathsf A_P=
 \begin{pmatrix}\mathbf 0&I_P\\-\Lambda_P&\mathbf 0\end{pmatrix},
 \qquad
 \mathsf B_P=\begin{pmatrix}\mathbf 0\\-M\end{pmatrix},
 \qquad
 b^{(n)}(t)=\begin{pmatrix}\mathbf 0\\-s^{(n)}(t)\end{pmatrix}.
\]
The variable \(Y\) here uses unscaled displacement and velocity coordinates. The energy coordinates of the main article are \(\mathsf E_PY\), where \(\mathsf E_P=\operatorname{diag}(\Lambda_P^{1/2},I_P)\). Accordingly, the linear and control matrices in those coordinates satisfy \(A_0=\mathsf E_P\mathsf A_P\mathsf E_P^{-1}\) and \(B_E=\mathsf E_P\mathsf B_P\). The corresponding endpoint vectors are \(\mathsf E_PY_0\) and \(\mathsf E_PY_T\).

\begin{lemma}
\label{lem:retained_pair_controllable}
For every finite set of orthonormal Dirichlet modes on \(\Omega=(0,1)^N\) and every nonempty open subset \(\omega\subset\Omega\), the overlap matrix \(M\) is symmetric positive definite. Consequently, the pair \((\mathsf A_P,\mathsf B_P)\) is controllable.
\end{lemma}

\begin{proof}
We use analyticity to prove that \(M\) is positive definite and then verify the controllability rank condition. By the definition of \(M\) in~\eqref{eq:overlap-matrix}, we have \(M^\top=M\). For \(v=(v_1,v_2,\dots,v_P)\in\mathbb R^P\), we set \(\mathbf{v}_P(x)=\displaystyle\sum_{j=1}^{P}v_je_j(x)\). Then,
\[
 v^\top Mv=\int_\omega |\mathbf{v}_P(x)|^2\,\mathrm{d}x\ge 0.
\]
We suppose that this integral vanishes. We first obtain \(\mathbf{v}_P=0\) almost everywhere on \(\omega\), and continuity then lets us conclude that \(\mathbf{v}_P=0\) at every point of \(\omega\). Each sine mode is real analytic, so \(\mathbf{v}_P\) is real analytic on the connected domain \(\Omega\). By the identity theorem for analytic functions, we conclude that \(\mathbf{v}_P=0\) on \(\Omega\). Moreover, by orthonormality of \(\{e_j\}_{j=1}^P\), we have
\[
 |v|_{\mathbb R^P}^2=\|\mathbf{v}_P\|_{L^2(\Omega)}^2=0.
\]
Thus, \(v=\mathbf{0}_{\mathbb R^P}\), and we conclude that \(M\) is positive definite.

To prove the controllability of \((\mathsf A_P,\mathsf B_P)\), we perform direct multiplication to obtain
\[
 \mathsf A_P\mathsf B_P
 =\begin{pmatrix}-M\\\mathbf 0\end{pmatrix},
 \qquad
 [\,\mathsf B_P\ \ \mathsf A_P\mathsf B_P\,]
 =\begin{pmatrix}\mathbf 0&-M\\-M&\mathbf 0\end{pmatrix}.
\]
Since \(M\) is invertible, this block matrix has rank \(2P\). Hence, the full controllability matrix has rank \(2P\), and we apply the Kalman rank criterion to establish controllability of \((\mathsf A_P,\mathsf B_P)\).
\end{proof}

This conclusion concerns each fixed finite-dimensional system. It does not assert a controllability bound uniform in \(P\) or controllability of the full wave equation.

For each frozen problem, we use the modal-coefficient cost
\begin{equation}\label{eq:modal_control_cost}
\mathcal J(c)=\int_0^T|c(t)|_{\mathbb R^P}^2\,\mathrm{d}t.
\end{equation}
The corresponding physical \(L^2\)-energy of the localized applied control is
\begin{equation}\label{eq:physical_control_cost}
\|\mathbf1_\omega f_P\|_{L^2((0,T)\times\Omega)}^2
=\int_0^Tc(t)^\top M c(t)\,\mathrm{d}t.
\end{equation}
More generally, for a symmetric positive-definite matrix \(\Sigma_P\), we define \(\mathcal J_{\Sigma_P}(c)=\int_0^T c(t)^\top \Sigma_Pc(t)\,\mathrm{d}t\). With \(\Sigma_P=I_P\), we recover the modal-coefficient cost \eqref{eq:modal_control_cost}, and with \(\Sigma_P=M\), we recover the physical control cost \eqref{eq:physical_control_cost}. The latter minimizes the physical cost within the finite-dimensional control space, subject to the projected affine dynamics and endpoint conditions. We use \(\Sigma_P=I_P\) in the source-frozen HUM and Gramian algorithms compared below and evaluate their physical control norms separately.

For each frozen source in \(L^2(0,T;\mathbb R^P)\), controllability lets us satisfy the affine endpoint constraints. The endpoint depends continuously and affinely on \(c\in L^2(0,T;\mathbb R^P)\), so the feasible controls form a nonempty closed affine subset of this Hilbert space. Applying the Hilbert projection theorem, we obtain a unique minimizer of \eqref{eq:modal_control_cost}. The corresponding affine initial-value problem has a unique trajectory, whose displacement belongs to \(H^2(0,T;\mathbb R^P)\). Thus, we can define the exact update below.

For each $n \in \{0,1,2,\ldots\}$, we define
\[
\begin{aligned}
 W_{P,T}
 &=
 \int_0^T
 e^{\mathsf A_P(T-s)}
 \mathsf B_P\mathsf B_P^\top
 e^{\mathsf A_P^\top(T-s)}
 \,\mathrm{d}s,
 \\
 d^{(n)}
 &=
 Y_T-e^{\mathsf A_PT}Y_0
 -
 \int_0^T
 e^{\mathsf A_P(T-s)}b^{(n)}(s)
 \,\mathrm{d}s.
\end{aligned}
\]
By Lemma~\ref{lem:retained_pair_controllable},
$W_{P,T}$ is positive definite. The unique control
minimizing \eqref{eq:modal_control_cost} subject to
\eqref{eq:source_frozen_first_order_system} is
\[
 c^{(n+1)}(t)
 =
 \mathsf B_P^\top
 e^{\mathsf A_P^\top(T-t)}
 W_{P,T}^{-1}d^{(n)},
 \qquad t\in[0,T].
\]
We define the corresponding displacement by
\[
\begin{aligned}
 \widehat a^{(n+1)}(t)
 &:=
 \begin{pmatrix}I_P&\mathbf 0\end{pmatrix}
 \left(
 e^{\mathsf A_Pt}Y_0
 +
 \int_0^t
 e^{\mathsf A_P(t-s)}
 \left[
 b^{(n)}(s)
 +
 \mathsf B_P\mathsf B_P^\top
 e^{\mathsf A_P^\top(T-s)}
 W_{P,T}^{-1}d^{(n)}
 \right]
 \,\mathrm{d}s
 \right),
 \qquad t\in[0,T],
 \\
 \widehat a^{(n+1)}&:=\mathcal T_P(a^{(n)}).
\end{aligned}
\]
Here, $\widehat a^{(n+1)}$ is the exact affine proposal. For a fixed relaxation parameter $\eta\in(0,1]$, we set
\[
 a^{(n+1)}
 :=(1-\eta)a^{(n)}+\eta\widehat a^{(n+1)}.
\]
The choice $\eta=1$ gives the undamped update
$a^{(n+1)}=\widehat a^{(n+1)}$. The numerical stopping quantity is the unrelaxed proposal residual
\begin{equation*}
r_{\mathrm{fp}}^{(n)}
=\|\widehat a^{(n+1)}-a^{(n)}\|_{L^2(0,T;\mathbb R^P)}.
\end{equation*}
The relaxed increment satisfies \(\|a^{(n+1)}-a^{(n)}\|_{L^2(0,T;\mathbb R^P)} =\eta r_{\mathrm{fp}}^{(n)}\). Using the unrelaxed residual therefore avoids accepting a small increment caused solely by a small relaxation parameter. This residual measures the fixed-point defect. Without an additional estimate, it does not certify the nonlinear terminal error.
\subsection{HUM realization and the source-frozen Picard--HUM algorithm}
At iteration \(n\), we fix the source \(b^{(n)}=(\mathbf 0,-G_P(a^{(n)}))^\top\). The endpoint correction required from the control is
\(
 d^{(n)}
 =Y_T-e^{\mathsf A_PT}Y_0
 -\int_0^T e^{\mathsf A_P(T-t)}b^{(n)}(t)\,\mathrm{d}t.
\)
For \(q\in\mathbb R^{2P}\), let \(p_q\) solve the adjoint problem \(-p_q'(t)=\mathsf A_P^\top p_q(t)\) with \(p_q(T)=q\). Thus, \(p_q(t)=e^{\mathsf A_P^\top(T-t)}q\). The finite-dimensional HUM construction minimizes the functional
\[
 \mathscr J_n(q)
 =\dfrac{1}{2}\int_0^T
   \left|\mathsf B_P^\top p_q(t)\right|_{\mathbb R^P}^2
   \,\mathrm{d}t
 -q^\top d^{(n)}.
\]
By Lemma~\ref{lem:retained_pair_controllable}, the quadratic part is positive definite. Indeed, if \(\mathsf B_P^\top p_q=\mathbf 0\) almost everywhere, then \(q\) is orthogonal to every endpoint reachable from the zero initial state, and controllability forces \(q=\mathbf 0\). Since the terminal adjoint space is finite dimensional, \(\mathscr J_n\) is coercive and strictly convex. Therefore, we obtain a unique minimizer \(q^{(n)}\).

Applying the first-order optimality condition, we obtain \(\int_0^T e^{\mathsf A_P(T-t)} \mathsf B_P\mathsf B_P^\top p_{q^{(n)}}(t)\,\mathrm{d}t =d^{(n)}\). Consequently, the control \(c^{(n+1)}(t)=\mathsf B_P^\top p_{q^{(n)}}(t)\) steers the frozen affine system from \(Y_0\) to \(Y_T\).

To verify its minimum-cost property, we let \(\widetilde c\in L^2(0,T;\mathbb R^P)\) be another feasible control and set \(h=\widetilde c-c^{(n+1)}\). The common endpoint conditions imply
\(
 \int_0^T e^{\mathsf A_P(T-t)}\mathsf B_Ph(t)\,\mathrm{d}t
 =\mathbf 0.
\)
Hence,
\(
 \int_0^T c^{(n+1)}(t)^\top h(t)\,\mathrm{d}t
 =(q^{(n)})^\top
 \int_0^T e^{\mathsf A_P(T-t)}
 \mathsf B_Ph(t)\,\mathrm{d}t=0.
\)
We therefore obtain
\begin{equation}\label{eq:hum-minimum-cost-pythagorean}
 \|\widetilde c\|_{L^2(0,T;\mathbb R^P)}^2
 =
 \|c^{(n+1)}\|_{L^2(0,T;\mathbb R^P)}^2
 +\|h\|_{L^2(0,T;\mathbb R^P)}^2.
\end{equation}
Thus, \(c^{(n+1)}\) is the unique minimizer of \eqref{eq:modal_control_cost}.

For every $t\in[0,T]$, the HUM displacement is
\[
\begin{aligned}
 \widehat a^{(n+1)}(t)
 &=
 \begin{pmatrix}I_P&\mathbf 0\end{pmatrix}
 \left(
 e^{\mathsf A_Pt}Y_0
 +
 \int_0^t e^{\mathsf A_P(t-s)}
 \left[
 b^{(n)}(s)
 +\mathsf B_P\mathsf B_P^\top
 e^{\mathsf A_P^\top(T-s)}q^{(n)}
 \right]\,\mathrm{d}s
 \right).
\end{aligned}
\]
We therefore define
\[
 \mathcal T_P^{\mathrm{HUM}}(a^{(n)})
 :=\widehat a^{(n+1)}.
\]
The HUM optimality condition gives $W_{P,T}q^{(n)}=d^{(n)}$, and hence $q^{(n)}=W_{P,T}^{-1}d^{(n)}$. Substituting this identity into the HUM control formula recovers the minimum-cost control used to define $\mathcal T_P$. Both constructions therefore use the same control, frozen source, and initial state. By uniqueness of the affine initial-value problem, their state trajectories coincide, so
\[
 \mathcal T_P^{\mathrm{HUM}}(a^{(n)})
 =\mathcal T_P(a^{(n)}).
\]

Algorithm~\ref{alg:sf_picard_hum} describes the iteration with exact affine solves. We assume that every frozen source generated by the iteration belongs to \(L^2(0,T;\mathbb R^P)\).

\begin{algorithm}[H]
\caption{Source-frozen Galerkin Picard--HUM algorithm}
\label{alg:sf_picard_hum}
\begin{algorithmic}[1]
\State We assemble \(\mathsf A_P\) and \(\mathsf B_P\).
\State We choose \(a^{(0)}\in H^2(0,T;\mathbb R^P)\) satisfying
\[
 \begin{pmatrix}a^{(0)}(0)\\(a^{(0)})'(0)\end{pmatrix}=Y_0,
 \qquad
 \begin{pmatrix}a^{(0)}(T)\\(a^{(0)})'(T)\end{pmatrix}=Y_T.
\]
\State We choose \(\eta\in(0,1]\), a tolerance
\(\varepsilon_{\mathrm{fp}}>0\), and an integer iteration cap
\(N_{\max}\ge 1\).
\For{\(n=0,\ldots,N_{\max}-1\)}
  \State We freeze
  \(s^{(n)}(t)=G_P(a^{(n)}(t))\).
  \State We solve the affine steering problem
  \eqref{eq:source_frozen_first_order_system} by the Hilbert Uniqueness Method,
  obtaining \(c^{(n+1)}\) and \(\widehat a^{(n+1)}\).
  \State We compute
  \[
   r_{\mathrm{fp}}^{(n)}
   =\|\widehat a^{(n+1)}-a^{(n)}\|_{L^2(0,T;\mathbb R^P)}.
  \]
  \State If \(r_{\mathrm{fp}}^{(n)}\le\varepsilon_{\mathrm{fp}}\),
  then the algorithm returns \((c^{(n+1)},\widehat a^{(n+1)})\) with status
  ``proposal tolerance met''.
  \State We set
  \[
   a^{(n+1)}
   =(1-\eta)a^{(n)}+\eta\widehat a^{(n+1)}.
  \]
\EndFor
\State The algorithm returns \((c^{(N_{\max})},\widehat a^{(N_{\max})})\)
with status ``iteration cap reached'', together with
\(r_{\mathrm{fp}}^{(N_{\max}-1)}\).
\end{algorithmic}
\end{algorithm}

Each exact affine proposal satisfies the prescribed projected endpoint conditions. Since the initial trajectory is endpoint compatible, the relaxed iterates also preserve these conditions. However, the proposal solves the equation with the source frozen at \(a^{(n)}\). Its residual in the nonlinear Galerkin equation is
\[
 \begin{aligned}
 (\widehat a^{(n+1)})''
 +\Lambda_P\widehat a^{(n+1)}
 +G_P(\widehat a^{(n+1)})
 +Mc^{(n+1)}=
 G_P(\widehat a^{(n+1)})
 -G_P(a^{(n)}).
 \end{aligned}
\]
Thus, at an exact fixed point, we obtain a solution of the nonlinear Galerkin steering problem. Meeting a positive proposal tolerance alone does not establish the nonlinear terminal accuracy. In numerical computations, we assess the extracted control by a separate forward solve of the nonlinear \(P\)-mode system, and check errors in the numerical affine solve separately.

\subsection{Galerkin convergence and under-relaxation}

We first estimate how a change in the frozen source affects the minimum-cost steering trajectory. The source-response estimate \eqref{eq:controlled_source_lipschitz} accounts for both the direct response to the source and the change in the control needed to preserve the prescribed endpoints.

\begin{lemma}
\label{lem:affine_source_trajectory_lipschitz}
We fix \(P\in\{1,2,\ldots\}\), \(T>0\), the matrices \(\mathsf A_P\) and \(\mathsf B_P\), the endpoint data \(Y_0,Y_T\), and the modal-coefficient cost \eqref{eq:modal_control_cost}. We set
\[
 \mathcal X_P=C([0,T];\mathbb R^P),
 \qquad
 \|a\|_{\mathcal X_P}
 =\max_{t\in[0,T]}|a(t)|_{\mathbb R^P},
\]
and define
\[
 K_{P,T}
 =\max_{0\le t\le T}\|e^{\mathsf A_Pt}\|_2,
 \qquad
 W_{P,T}
 =\int_0^T e^{\mathsf A_P(T-r)}
 \mathsf B_P\mathsf B_P^\top
 e^{\mathsf A_P^\top(T-r)}\,\mathrm{d}r.
\]
Here, \(\|\cdot\|_2\) denotes the matrix norm induced by the Euclidean norm.

For each $s\in L^2(0,T;\mathbb R^P)$, set
\[
\begin{aligned}
 b_s(t)&:=\begin{pmatrix}\mathbf 0\\-s(t)\end{pmatrix},
 \\
 d[s]&:=
 Y_T-e^{\mathsf A_PT}Y_0
 -
 \int_0^T e^{\mathsf A_P(T-r)}b_s(r)\,\mathrm{d}r,
 \\
 c[s](t)&:=
 \mathsf B_P^\top e^{\mathsf A_P^\top(T-t)}
 W_{P,T}^{-1}d[s],
 \qquad t\in[0,T].
\end{aligned}
\]
We define the affine source-to-displacement map by
\[
\begin{aligned}
 \mathcal H_P:
 L^2(0,T;\mathbb R^P)&\longrightarrow\mathcal X_P,
 \\
 [\mathcal H_P(s)](t)
 &:=
 \begin{pmatrix}I_P&\mathbf 0\end{pmatrix}
 \left(
 e^{\mathsf A_Pt}Y_0
 +
 \int_0^t e^{\mathsf A_P(t-r)}
 \left[b_s(r)+\mathsf B_Pc[s](r)\right]
 \,\mathrm{d}r
 \right),
 \qquad t\in[0,T].
\end{aligned}
\]
Here, $c[s]$ is the unique minimizer of \eqref{eq:modal_control_cost} for the frozen source $s$ and the prescribed endpoints. Moreover, $\mathcal H_P(s)\in H^2(0,T;\mathbb R^P)$ because the affine state belongs to $H^1(0,T;\mathbb R^{2P})$, and its velocity component is the weak derivative of its displacement component. The map $\mathcal H_P$ satisfies
\begin{equation}\label{eq:controlled_source_lipschitz}
 \|\mathcal H_P(s_1)-\mathcal H_P(s_2)\|_{\mathcal X_P}
 \le C_{P,T}^{\mathrm{aff}}
 \|s_1-s_2\|_{L^2(0,T;\mathbb R^P)},
\end{equation}
for all \(s_1,s_2\in L^2(0,T;\mathbb R^P)\), where
\[
 C_{P,T}^{\mathrm{aff}}
 =K_{P,T}\sqrt T
 \left(
 1+K_{P,T}\sqrt T\,\|\mathsf B_P\|_2
 \|W_{P,T}^{-1}\|_2^{1/2}
 \right).
\]

Under Assumption~\ref{ass:global_lipschitz}, for every \(a,b\in\mathcal X_P\),
\begin{equation}\label{eq:global_projected_nonlinearity_lipschitz}
 \|G_P(a)-G_P(b)\|_{L^2(0,T;\mathbb R^P)}
 \le L_g\sqrt{T\mu_{\max}(P)}\,\|a-b\|_{\mathcal X_P},
\end{equation}
where \(\mu_{\max}(P)=\max_{1\le j\le P}\mu_j\). Here, \(G_P\) acts on trajectories pointwise in time, which means \(G_P(a)(t)=G_P(a(t))\).
\end{lemma}
\begin{proof}
We estimate the source-induced endpoint correction, the corresponding control change, and finally the trajectory change. Lemma~\ref{lem:retained_pair_controllable} makes \(W_{P,T}\) positive definite. We fix \(s_1,s_2\in L^2(0,T;\mathbb R^P)\). For \(i\in\{1,2\}\), let
\[
 b_i(t)=\begin{pmatrix}\mathbf 0\\-s_i(t)\end{pmatrix},
 \qquad
 a_i=\mathcal H_P(s_i),
 \qquad
 Y_i=\begin{pmatrix}a_i\\a_i'\end{pmatrix},
\]
and let \(c_i\) be the corresponding minimum-cost control. We write \(\delta s=s_1-s_2\), and use the same convention for the differences of the other quantities.

The endpoint correction associated with \(s_i\) is
\[
 d_i
 =Y_T-e^{\mathsf A_PT}Y_0
 -\int_0^T e^{\mathsf A_P(T-r)}b_i(r)\,\mathrm{d}r.
\]
Since the endpoint data are the same for both sources,
\[
 \delta d
 =-\int_0^T e^{\mathsf A_P(T-r)}
 \delta b(r)\,\mathrm{d}r.
\]
Using \(|\delta b(r)|=|\delta s(r)|\) and the Cauchy--Schwarz inequality, we obtain
\[
 |\delta d|
 \le K_{P,T}\int_0^T|\delta s(r)|\,\mathrm{d}r
 \le K_{P,T}\sqrt T\,\|\delta s\|_{L^2(0,T;\mathbb R^P)}.
\]

Applying the minimum-cost control formula \eqref{eq:picard_gramian_control} to the difference in endpoint corrections, we obtain
\[
 \delta c(t)
 =\mathsf B_P^\top e^{\mathsf A_P^\top(T-t)}
 W_{P,T}^{-1}\delta d.
\]
Since \(W_{P,T}\) is symmetric positive definite, we use its defining integral to obtain
\[
 \begin{aligned}
 \|\delta c\|_{L^2(0,T;\mathbb R^P)}^2
 &=
 \delta d^\top W_{P,T}^{-1}
 \left(
 \int_0^T e^{\mathsf A_P(T-t)}
 \mathsf B_P\mathsf B_P^\top
 e^{\mathsf A_P^\top(T-t)}\,\mathrm{d}t
 \right)
 W_{P,T}^{-1}\delta d=\delta d^\top W_{P,T}^{-1}\delta d\le \|W_{P,T}^{-1}\|_2|\delta d|^2.
 \end{aligned}
\]
Therefore,
\[
 \|\delta c\|_{L^2(0,T;\mathbb R^P)}
 \le
 \|W_{P,T}^{-1}\|_2^{1/2}K_{P,T}\sqrt T\,
 \|\delta s\|_{L^2(0,T;\mathbb R^P)}.
\]

The initial states also agree, so we apply variation of constants to obtain
\[
 \delta Y(t)
 =\int_0^t e^{\mathsf A_P(t-r)}
 \left(\mathsf B_P\delta c(r)+\delta b(r)\right)
 \,\mathrm{d}r.
\]
The Euclidean projection onto the displacement block has norm one. Hence, for every \(t\in[0,T]\),
\[
 \begin{aligned}
 |\delta a(t)|
 \le |\delta Y(t)|
 &\le K_{P,T}\int_0^t
 \left(\|\mathsf B_P\|_2|\delta c(r)|+|\delta s(r)|\right)
 \,\mathrm{d}r\\
 &\le K_{P,T}\sqrt T
 \left(
 \|\mathsf B_P\|_2\|\delta c\|_{L^2(0,T;\mathbb R^P)}
 +\|\delta s\|_{L^2(0,T;\mathbb R^P)}
 \right)\\
 &\le K_{P,T}\sqrt T
 \left(
 1+K_{P,T}\sqrt T\,\|\mathsf B_P\|_2
 \|W_{P,T}^{-1}\|_2^{1/2}
 \right)
 \|\delta s\|_{L^2(0,T;\mathbb R^P)}.
 \end{aligned}
\]
Taking the supremum over \(t\in[0,T]\), we obtain the affine trajectory estimate \eqref{eq:controlled_source_lipschitz}.

For the nonlinear estimate, we fix \(a,b\in\mathcal X_P\) and set
\[
 u_a(t,x)=\sum_{j=1}^{P}a_j(t)e_j(x),
 \qquad
 u_b(t,x)=\sum_{j=1}^{P}b_j(t)e_j(x).
\]
The weak Dirichlet eigenfunction identity and orthonormality give
\[
 \int_\Omega\nabla e_j\cdot\nabla e_k\,\mathrm{d}x
 =\mu_j\int_\Omega e_je_k\,\mathrm{d}x
 =\mu_j\delta_{jk},
\]
where \(\delta_{jk}\) is the Kronecker delta. Applying Bessel's inequality and Assumption~\ref{ass:global_lipschitz}, we obtain, for each \(t\in[0,T]\),
\begin{equation}\label{eq:global_projected_pointwise_bound}
 \begin{aligned}
 |G_P(a(t))-G_P(b(t))|^2
 \le
 \|g(u_a(t))-g(u_b(t))\|_{L^2(\Omega)}^2
 &\le
 L_g^2\|\nabla(u_a(t)-u_b(t))\|_{L^2(\Omega)}^2\\
 &=L_g^2\sum_{j=1}^{P}\mu_j|a_j(t)-b_j(t)|^2\\
 &\le L_g^2\mu_{\max}(P)|a(t)-b(t)|^2.
 \end{aligned}
\end{equation}
Integrating the pointwise estimate \eqref{eq:global_projected_pointwise_bound} over \((0,T)\), we arrive at
\[
 \begin{aligned}
 \|G_P(a)-G_P(b)\|_{L^2(0,T;\mathbb R^P)}^2\le
 L_g^2\mu_{\max}(P)\int_0^T|a(t)-b(t)|^2\,\mathrm{d}t\le
 L_g^2T\mu_{\max}(P)\|a-b\|_{\mathcal X_P}^2.
 \end{aligned}
\]
Taking square roots, we obtain the projected nonlinearity estimate \eqref{eq:global_projected_nonlinearity_lipschitz}.
\end{proof}

\begin{proposition}
\label{prop:sf_picard_conditional_convergence}
We fix \(P\in\{1,2,\ldots\}\), \(T>0\), the endpoint data, and the modal-coefficient cost \eqref{eq:modal_control_cost}. Let \(\mathbb B_R=\{a\in\mathcal X_P:\|a\|_{\mathcal X_P}\le R\}\), where \(R>0\). We suppose that \(G_P:\mathbb B_R\to L^2(0,T;\mathbb R^P)\) satisfies
\begin{equation}\label{eq:projected_nonlinearity_lipschitz}
 \|G_P(a)-G_P(b)\|_{L^2(0,T;\mathbb R^P)}
 \le L_{P,R}\|a-b\|_{\mathcal X_P},
 \qquad a,b\in\mathbb B_R,
\end{equation}
for some \(L_{P,R}\ge0\). If \(\mathcal T_P^{\mathrm{HUM}}(\mathbb B_R)\subset\mathbb B_R\) and
\[
 q_{P,R}:=C_{P,T}^{\mathrm{aff}}L_{P,R}<1,
\]
then \(\mathcal T_P^{\mathrm{HUM}}\) has a unique fixed point \(a^*\) in \(\mathbb B_R\). For every \(a^{(0)}\in\mathbb B_R\), the undamped iterates remain in \(\mathbb B_R\) and satisfy
\[
 \|a^{(n)}-a^*\|_{\mathcal X_P}
 \le q_{P,R}^{\,n}\|a^{(0)}-a^*\|_{\mathcal X_P},
 \qquad n\in\{1,2,\ldots\}.
\]

Under Assumption~\ref{ass:global_lipschitz}, the projected nonlinearity is globally Lipschitz from \(\mathcal X_P\) into \(L^2(0,T;\mathbb R^P)\) with bound \(L_P^{\mathrm{glob}}=L_g\sqrt{T\mu_{\max}(P)}\). If
\begin{equation}\label{eq:picard_global_galerkin_contraction}
 q_P^{\mathrm{glob}}
 :=C_{P,T}^{\mathrm{aff}}L_g\sqrt{T\mu_{\max}(P)}<1,
\end{equation}
then \(\mathcal T_P^{\mathrm{HUM}}\) has a unique fixed point \(a^*\) in \(\mathcal X_P\), and every initial trajectory \(a^{(0)}\in\mathcal X_P\) generates undamped iterates satisfying
\[
 \|a^{(n)}-a^*\|_{\mathcal X_P}
 \le (q_P^{\mathrm{glob}})^n
 \|a^{(0)}-a^*\|_{\mathcal X_P},
 \qquad n\in\{1,2,\ldots\}.
\]

In either case, we fix \(0<\eta\le1\) and define \(\mathcal T_{P,\eta}=(1-\eta)I+\eta\mathcal T_P^{\mathrm{HUM}}\) and \(q_\eta=1-\eta+\eta q\), where \(I\) is the identity on \(\mathcal X_P\) and \(q=q_{P,R}\) or \(q=q_P^{\mathrm{glob}}\), respectively. The relaxed map has the same fixed point and is a contraction on the corresponding set, with Lipschitz constant at most \(q_\eta<1\). Its iterates satisfy
\[
 \|a^{(n)}-a^*\|_{\mathcal X_P}
 \le q_\eta^{\,n}\|a^{(0)}-a^*\|_{\mathcal X_P},
 \qquad n\in\{1,2,\ldots\},
\]
for every initial trajectory in that set.

Independently of these two contraction conditions, we suppose that \(\mathcal T_P^{\mathrm{HUM}}:U\to\mathcal X_P\) is continuously Fr\'echet differentiable on an open neighborhood \(U\) of a fixed point \(a^*\). For a fixed \(0<\eta\le1\), we assume that
\begin{equation}\label{eq:picard_local_relaxed_condition}
 \|(1-\eta)I+\eta D\mathcal T_P^{\mathrm{HUM}}(a^*)\|
 _{\mathcal L(\mathcal X_P)}<1.
\end{equation}
Then, there exist \(r>0\) and \(\theta\in(0,1)\) such that, for every initial trajectory \(a^{(0)}\in\mathcal X_P\) satisfying \(\|a^{(0)}-a^*\|_{\mathcal X_P}\le r\), the relaxed iterates remain in the closed ball of radius \(r\) centered at \(a^*\) and satisfy
\[
 \|a^{(n)}-a^*\|_{\mathcal X_P}
 \le\theta^n\|a^{(0)}-a^*\|_{\mathcal X_P},
 \qquad n\in\{1,2,\ldots\}.
\]
\end{proposition}

\begin{proof}
We first combine the affine response estimate with the nonlinear Lipschitz estimate. We then check invariance and completeness for the global assertions and use derivative continuity for the local assertion. By construction, \(\mathcal T_P^{\mathrm{HUM}}(a)=\mathcal H_P(G_P(a))\). Hence, Lemma~\ref{lem:affine_source_trajectory_lipschitz} and the projected nonlinearity estimate \eqref{eq:projected_nonlinearity_lipschitz} give
\begin{equation}\label{eq:picard-hum-ball-contraction}
 \begin{aligned}
 \|\mathcal T_P^{\mathrm{HUM}}(a)
       -\mathcal T_P^{\mathrm{HUM}}(b)\|_{\mathcal X_P}\le
 C_{P,T}^{\mathrm{aff}}
 \|G_P(a)-G_P(b)\|_{L^2(0,T;\mathbb R^P)}\le q_{P,R}\|a-b\|_{\mathcal X_P},
 \qquad a,b\in\mathbb B_R.
 \end{aligned}
\end{equation}
The space \(\mathcal X_P\) is Banach, and \(\mathbb B_R\) is nonempty and closed, so \(\mathbb B_R\) is complete. Together with the assumed invariance of \(\mathbb B_R\), the contraction estimate \eqref{eq:picard-hum-ball-contraction} lets us apply the Banach fixed-point theorem \cite[Theorem~5.7]{brezis2011functional}. We thereby obtain a unique fixed point \(a^*\in\mathbb B_R\). To bound the iterates, we use
\[
 \|a^{(n+1)}-a^*\|_{\mathcal X_P}
 \le q_{P,R}\|a^{(n)}-a^*\|_{\mathcal X_P},
\]
and we obtain the stated geometric estimate by induction.

Under Assumption~\ref{ass:global_lipschitz}, \eqref{eq:global_projected_nonlinearity_lipschitz} holds for every \(a,b\in\mathcal X_P\). Consequently,
\begin{equation}\label{eq:picard-hum-global-contraction}
 \|\mathcal T_P^{\mathrm{HUM}}(a)
       -\mathcal T_P^{\mathrm{HUM}}(b)\|_{\mathcal X_P}
 \le q_P^{\mathrm{glob}}\|a-b\|_{\mathcal X_P}.
\end{equation}
The affine solve maps every square-integrable source to a continuous displacement trajectory, so \(\mathcal T_P^{\mathrm{HUM}}\) maps \(\mathcal X_P\) into itself. When \(q_P^{\mathrm{glob}}<1\), we apply the Banach fixed-point theorem on the whole space \(\mathcal X_P\) using the contraction estimate \eqref{eq:picard-hum-global-contraction} and obtain the global conclusion.

For either contractive case, let \(a,b\) belong to the corresponding set. Since \(0<\eta\le1\),
\begin{equation}\label{eq:picard-relaxed-contraction}
 \begin{aligned}
 \|\mathcal T_{P,\eta}(a)-\mathcal T_{P,\eta}(b)\|_{\mathcal X_P}
 &\le
 (1-\eta)\|a-b\|_{\mathcal X_P}\\
 &\quad+
 \eta\|\mathcal T_P^{\mathrm{HUM}}(a)
          -\mathcal T_P^{\mathrm{HUM}}(b)\|_{\mathcal X_P}\\
 &\le(1-\eta+\eta q)\|a-b\|_{\mathcal X_P}.
 \end{aligned}
\end{equation}
Moreover, by convexity of \(\mathbb B_R\), we have \(\mathcal T_{P,\eta}(\mathbb B_R)\subset\mathbb B_R\) in the invariant-ball case. Because \(\eta>0\), \(\mathcal T_{P,\eta}(a)=a\) is equivalent to \(\mathcal T_P^{\mathrm{HUM}}(a)=a\). Thus, the fixed point is unchanged, and we obtain the relaxed geometric estimate by induction from \eqref{eq:picard-relaxed-contraction}.

For the local assertion, we set \(\rho_*=\|D\mathcal T_{P,\eta}(a^*)\|_{\mathcal L(\mathcal X_P)}<1\) and choose \(\theta\) with \(\rho_*<\theta<1\). Continuity of the derivative and openness of \(U\) allow us to choose \(r>0\) such that
\[
 \mathbb B_r(a^*)
 :=\{a\in\mathcal X_P:\|a-a^*\|_{\mathcal X_P}\le r\}
 \subset U
\]
and
\[
 \sup_{a\in\mathbb B_r(a^*)}
 \|D\mathcal T_{P,\eta}(a)\|_{\mathcal L(\mathcal X_P)}
 \le\theta.
\]
For \(a,b\in\mathbb B_r(a^*)\), the line segment joining them lies in the same ball. Applying the integral mean-value formula, we obtain
\[
 \mathcal T_{P,\eta}(a)-\mathcal T_{P,\eta}(b)
 =
 \int_0^1
 D\mathcal T_{P,\eta}(b+\tau(a-b))(a-b)\,\mathrm{d}\tau.
\]
Taking norms, we obtain
\[
 \|\mathcal T_{P,\eta}(a)-\mathcal T_{P,\eta}(b)\|_{\mathcal X_P}
 \le\theta\|a-b\|_{\mathcal X_P}.
\]
Since \(\mathcal T_{P,\eta}(a^*)=a^*\), we also obtain
\[
 \|\mathcal T_{P,\eta}(a)-a^*\|_{\mathcal X_P}
 \le\theta\|a-a^*\|_{\mathcal X_P}
 \le\theta r<r.
\]
Hence, \(\mathcal T_{P,\eta}\) maps the complete ball \(\mathbb B_r(a^*)\) into itself and is a contraction there. Its iterates converge to \(a^*\) with the stated rate.
\end{proof}

A sufficient condition for the invariant-ball hypothesis is
\[
 \|\mathcal T_P^{\mathrm{HUM}}(\mathbf 0)\|_{\mathcal X_P}
 +q_{P,R}R\le R.
\]
Indeed, for every \(a\in\mathbb B_R\),
\[
 \begin{aligned}
 \|\mathcal T_P^{\mathrm{HUM}}(a)\|_{\mathcal X_P}
 &\le
 \|\mathcal T_P^{\mathrm{HUM}}(\mathbf 0)\|_{\mathcal X_P}
 +q_{P,R}\|a\|_{\mathcal X_P}\le R.
 \end{aligned}
\]

\paragraph{\bf Exact response in the \(L^2\) trajectory norm.}
We set \(\mathcal Z_P=L^2(0,T;\mathbb R^P)\), with \(\|a\|_{\mathcal Z_P}^2=\int_0^T|a(t)|_{\mathbb R^P}^2\,\mathrm{d}t\). Using the affine map $\mathcal H_P$ defined in Lemma~\ref{lem:affine_source_trajectory_lipschitz}, we define
\[
 \mathscr H_{P,T}:\mathcal Z_P\to\mathcal Z_P,
 \qquad
 \mathscr H_{P,T}s
 :=\mathcal H_P(s)-\mathcal H_P(\mathbf 0).
\] 
Using the control formula \eqref{eq:picard_gramian_control} and source-dependent endpoint correction \eqref{eq:picard_gramian_defect}, we see that the minimum-cost control depends affinely on the source. The affine state equation then lets us conclude that \(\mathcal H_P\) is affine in \(s\). Therefore, \(\mathscr H_{P,T}\) is linear and
\[
 \mathcal H_P(s_1)-\mathcal H_P(s_2)
 =\mathscr H_{P,T}(s_1-s_2).
\]

By Lemma~\ref{lem:affine_source_trajectory_lipschitz},
\[
 \begin{aligned}
 \|\mathscr H_{P,T}s\|_{\mathcal Z_P}
 \le\sqrt T\,\|\mathscr H_{P,T}s\|_{\mathcal X_P}\le\sqrt T\,C_{P,T}^{\mathrm{aff}}\|s\|_{\mathcal Z_P}.
 \end{aligned}
\]
Thus, \(\mathscr H_{P,T}\) is a bounded linear operator on \(\mathcal Z_P\), and we define
\[
 \chi_{P,T}
 :=\|\mathscr H_{P,T}\|_{\mathcal L(\mathcal Z_P)}.
\]
This definition includes the combined effect of the source and the control correction needed to preserve the endpoints.

\begin{theorem}
\label{thm:exact_response_galerkin}
We suppose that \(g:\mathbb R\to\mathbb R\) is globally Lipschitz with constant \(\ell_g\ge0\), so that \(|g(r)-g(s)|\le\ell_g|r-s|\) for \(r,s\in\mathbb R\). Then, the exact source-frozen map \(\mathcal T_P(a)=\mathcal H_P(G_P(a))\) is well defined on \(\mathcal Z_P\) and has global Lipschitz constant at most
\[
 \widehat q_P^{\mathrm{glob}}:=\ell_g\chi_{P,T}.
\]

If \(\widehat q_P^{\mathrm{glob}}<1\), then \(\mathcal T_P\) has a unique fixed point \(a^*\in\mathcal Z_P\). For every \(a^{(0)}\in\mathcal Z_P\), the undamped iterates satisfy
\[
 \|a^{(n)}-a^*\|_{\mathcal Z_P}
 \le(\widehat q_P^{\mathrm{glob}})^n
 \|a^{(0)}-a^*\|_{\mathcal Z_P},
 \qquad n\in\{1,2,\ldots\}.
\]
Moreover, for every \(a\in\mathcal Z_P\), the exact proposal residual satisfies
\begin{equation}\label{eq:sharp_proposal_certificate}
 \|a-a^*\|_{\mathcal Z_P}
 \le
 \dfrac{\|a-\mathcal T_P(a)\|_{\mathcal Z_P}}
 {1-\widehat q_P^{\mathrm{glob}}}.
\end{equation}
For every \(0<\eta\le1\), the relaxed map \((1-\eta)I+\eta\mathcal T_P\) has the same fixed point and is a contraction with Lipschitz constant at most \(1-\eta+\eta\widehat q_P^{\mathrm{glob}}<1\).
\end{theorem}

\begin{proof}
We combine the scalar Lipschitz bound with the exact affine response norm, then apply the contraction principle on \(\mathcal Z_P\).
For \(a,b\in\mathcal Z_P\), we define
\[
 u_a(t,x)=\sum_{j=1}^{P}a_j(t)e_j(x),
 \qquad
 u_b(t,x)=\sum_{j=1}^{P}b_j(t)e_j(x).
\]
By orthonormality, \(\|u_a(t)-u_b(t)\|_{L^2(\Omega)}=|a(t)-b(t)|_{\mathbb R^P}\) for almost every \(t\in(0,T)\). Applying Bessel's inequality and the scalar Lipschitz bound, we obtain
\[
 \begin{aligned}
 |G_P(a(t))-G_P(b(t))|_{\mathbb R^P}
 &\le
 \|g(u_a(t))-g(u_b(t))\|_{L^2(\Omega)}\\
 &\le
 \ell_g\|u_a(t)-u_b(t)\|_{L^2(\Omega)}\\
 &=\ell_g|a(t)-b(t)|_{\mathbb R^P}.
 \end{aligned}
\]
Squaring and integrating in time, we obtain
\[
 \|G_P(a)-G_P(b)\|_{\mathcal Z_P}
 \le\ell_g\|a-b\|_{\mathcal Z_P}.
\]
Furthermore,
\[
 \|G_P(a)\|_{\mathcal Z_P}
 \le
 \ell_g\|a\|_{\mathcal Z_P}
 +\sqrt{T|\Omega|}\,|g(0)|.
\]
Hence, \(G_P\) maps \(\mathcal Z_P\) into itself. Since \(\mathcal H_P\) maps square-integrable sources to displacement trajectories in \(H^2(0,T;\mathbb R^P)\), the map \(\mathcal T_P:\mathcal Z_P\to\mathcal Z_P\) is well defined.

Using the affine dependence of \(\mathcal H_P\) on its source, we obtain the exact identity
\[
 \begin{aligned}
 \mathcal T_P(a)-\mathcal T_P(b)
 &=\mathcal H_P(G_P(a))
   -\mathcal H_P(G_P(b))\\
 &=\mathscr H_{P,T}
   \left(G_P(a)-G_P(b)\right).
 \end{aligned}
\]
Therefore,
\[
 \begin{aligned}
 \|\mathcal T_P(a)-\mathcal T_P(b)\|_{\mathcal Z_P}
 &\le
 \chi_{P,T}
 \|G_P(a)-G_P(b)\|_{\mathcal Z_P}\\
 &\le
 \ell_g\chi_{P,T}\|a-b\|_{\mathcal Z_P}.
 \end{aligned}
\]
If \(\widehat q_P^{\mathrm{glob}}<1\), we apply the Banach fixed-point theorem on the complete space \(\mathcal Z_P\) to obtain a unique fixed point \(a^*\) and convergence from every initial trajectory in \(\mathcal Z_P\). We obtain the geometric estimate by iterating the contraction inequality.

For every \(a\in\mathcal Z_P\), we use the fixed-point identity \(a^*=\mathcal T_P(a^*)\) to obtain
\[
 \begin{aligned}
 \|a-a^*\|_{\mathcal Z_P}
 &\le
 \|a-\mathcal T_P(a)\|_{\mathcal Z_P}
 +\|\mathcal T_P(a)-\mathcal T_P(a^*)\|_{\mathcal Z_P}\\
 &\le
 \|a-\mathcal T_P(a)\|_{\mathcal Z_P}
 +\widehat q_P^{\mathrm{glob}}\|a-a^*\|_{\mathcal Z_P}.
 \end{aligned}
\]
Rearranging, we obtain the proposal-residual estimate \eqref{eq:sharp_proposal_certificate}.

Finally, for \(0<\eta\le1\),
\[
 \begin{aligned}
 \|(1-\eta)(a-b)
       +\eta(\mathcal T_P(a)-\mathcal T_P(b))\|_{\mathcal Z_P}\le
 \left(1-\eta+\eta\widehat q_P^{\mathrm{glob}}\right)
 \|a-b\|_{\mathcal Z_P}.
 \end{aligned}
\]
Since \(\eta>0\), the relaxed map and \(\mathcal T_P\) have the same fixed points. We have therefore proved the relaxed contraction statement.
\end{proof}

\paragraph{\bf A computable endpoint-map certificate.}
For a control \(c\in L^2(0,T;\mathbb R^P)\) whose Galerkin trajectory exists on \([0,T]\), let \(a_c\) solve
\[
 a_c''+\Lambda_Pa_c+G_P(a_c)+Mc=\mathbf 0,
 \qquad
 \begin{pmatrix}a_c(0)\\a_c'(0)\end{pmatrix}=Y_0.
\]
Using the unscaled state \(Y_c=(a_c,a_c')^\top\) and the energy-scaling matrix \(\mathsf E_P\) defined above, we set \(\mathfrak F_P(c)=\mathsf E_P(Y_c(T)-Y_T)\). Here, \(\mathfrak F_P(c)\in\mathbb R^{2P}\) is the energy-scaled endpoint defect, and \(|\mathfrak F_P(c)|_{\mathbb R^{2P}}\) is the \(P\)-mode terminal energy error. Under the global derivative bounds in the following proposition, we define this map for every \(c\in L^2(0,T;\mathbb R^P)\).

\begin{proposition}
\label{prop:galerkin_endpoint_lipschitz}
We assume that \(\Omega=(0,1)\), \(\alpha=0\), and \(e_j(x)=\sqrt2\sin(j\pi x)\) for \(j\in\{1,\ldots,P\}\). Let \(g\in C^2(\mathbb R)\) satisfy \(\|g'\|_{L^\infty(\mathbb R)}\le\ell_1\) and \(\|g''\|_{L^\infty(\mathbb R)}\le m_2\), where \(\ell_1,m_2\ge0\). We set \(C_{\infty,P}=\left(\dfrac{2}{\pi^2}\displaystyle\sum_{j=1}^P\dfrac{1}{j^2}\right)^{1/2}\) and \(\kappa=\ell_1/(2\pi)\). We define
\[
 H_\kappa(t)=\int_0^t e^{2\kappa(t-s)}\,\mathrm{d}s
 =
 \begin{cases}
 t,&\kappa=0,\\[1mm]
 \dfrac{e^{2\kappa t}-1}{2\kappa},&\kappa>0,
 \end{cases}
\]
and
\begin{equation}\label{eq:explicit_endpoint_L}
 \mathcal L_P
 :=\dfrac{m_2C_{\infty,P}\|M_P\|_2^2}{\pi}
 \int_0^T e^{\kappa(T-s)}H_\kappa(s)\,\mathrm{d}s.
\end{equation}
Then, \(\mathfrak F_P\) is twice continuously Fr\'echet differentiable on \(L^2(0,T;\mathbb R^P)\), and
\begin{equation}\label{eq:galerkin_endpoint_DF_lipschitz}
 \|D\mathfrak F_P(c)-D\mathfrak F_P(d)\|
 _{\mathcal L(L^2(0,T;\mathbb R^P),\mathbb R^{2P})}
 \le
 \mathcal L_P\|c-d\|_{L^2(0,T;\mathbb R^P)}
\end{equation}
for all \(c,d\in L^2(0,T;\mathbb R^P)\). The norm on \(\mathbb R^{2P}\) is Euclidean.

For \(g(u)=\beta\tanh u\) with \(\beta\ge0\), one may take
\[
 \ell_1=\beta,
 \qquad
 m_2=\dfrac{4\beta}{3\sqrt3},
 \qquad
 \kappa=\dfrac{\beta}{2\pi}.
\]
For every \(\kappa\ge0\),
\[
 \int_0^T e^{\kappa(T-s)}H_\kappa(s)\,\mathrm{d}s
 =
 \begin{cases}
 \dfrac{T^2}{2},&\kappa=0,\\[2mm]
 \dfrac{e^{\kappa T}(\cosh(\kappa T)-1)}{\kappa^2},
 &\kappa>0.
 \end{cases}
\]
\end{proposition}

\begin{proof}
We bound the first and second variations in energy coordinates. We then integrate the second-derivative bound along a control segment to obtain the Lipschitz constant for the endpoint derivative.
For \(v\in\mathbb R^P\), we write
\[
 u_v(x)=\sum_{j=1}^P v_je_j(x).
\]
Differentiating the projected nonlinearity, we obtain
\[
 \begin{aligned}
 (DG_P(a)v)_j
 &=\int_0^1 g'(u_a)u_ve_j\,\mathrm{d}x,\\
 (D^2G_P(a)[v,w])_j
 &=\int_0^1 g''(u_a)u_vu_we_j\,\mathrm{d}x.
 \end{aligned}
\]
Using Bessel's inequality and orthonormality, we have
\[
 |DG_P(a)v|
 \le\ell_1\|u_v\|_{L^2(0,1)}
 =\ell_1|v|.
\]
Thus, the projected first-order vector field is globally Lipschitz in the state. Since an \(L^2\) control is also integrable on \([0,T]\), the finite-dimensional initial-value problem has a unique solution throughout this interval.

The projected vector field is \(C^2\). On bounded sets of control trajectories, continuous dependence bounds the corresponding state trajectories on \([0,T]\). Applying Taylor's formula for this finite-dimensional vector field to its integral equation, we obtain continuous Fr\'echet derivatives with respect to the \(L^2\) control. These derivatives satisfy the variational equations below.

We fix a control \(c\in L^2(0,T;\mathbb R^P)\) and directions \(h,k\in L^2(0,T;\mathbb R^P)\). The first displacement variation \(z_h\) satisfies
\[
 z_h''+\Lambda_Pz_h+DG_P(a_c)z_h+M_Ph
 =\mathbf 0,
 \qquad z_h(0)=z_h'(0)=\mathbf 0.
\]
We define
\[
 r_h(t)^2
 =|\Lambda_P^{1/2}z_h(t)|^2+|z_h'(t)|^2.
\]
Since \(\Lambda_P=\operatorname{diag}((j\pi)^2)_{j=1}^P\), we have
\[
 |z_h|\le\pi^{-1}|\Lambda_P^{1/2}z_h|.
\]
Taking the scalar product of the variational equation with \(z_h'\), we obtain, almost everywhere,
\begin{equation}\label{eq:prop-first-var-energy-differential}
 \begin{aligned}
 \dfrac{1}{2}\dfrac{\mathrm d}{\mathrm dt}r_h^2
 &=-\langle DG_P(a_c)z_h,z_h'\rangle
   -\langle M_Ph,z_h'\rangle\\
 &\le
 \dfrac{\ell_1}{\pi}
 |\Lambda_P^{1/2}z_h|\,|z_h'|
 +\|M_P\|_2|h|r_h\\
 &\le
 \dfrac{\ell_1}{2\pi}r_h^2+\|M_P\|_2|h|r_h.
 \end{aligned}
\end{equation}
In the last step, we used \(2ab\le a^2+b^2\). To integrate without dividing by \(r_h\), we fix \(\varepsilon>0\) and set \(r_{h,\varepsilon}=(r_h^2+\varepsilon^2)^{1/2}\). Using the differential energy inequality \eqref{eq:prop-first-var-energy-differential}, we obtain
\begin{equation}\label{eq:first-var-regularized-gronwall}
 r_{h,\varepsilon}'\le\kappa r_{h,\varepsilon}+\|M_P\|_2|h|,
 \qquad r_{h,\varepsilon}(0)=\varepsilon.
\end{equation}
Multiplying the regularized estimate \eqref{eq:first-var-regularized-gronwall} by \(e^{-\kappa t}\), integrating on \([0,t]\), and letting \(\varepsilon\to0^+\), we obtain
\begin{equation}\label{eq:first-var-gronwall-bound}
 r_h(t)
 \le\|M_P\|_2\int_0^t e^{\kappa(t-s)}|h(s)|\,\mathrm{d}s.
\end{equation}
This bound holds for every \(t\in[0,T]\). Hence, applying the Cauchy--Schwarz inequality to \eqref{eq:first-var-gronwall-bound}, we obtain
\begin{equation}\label{eq:first-var-energy-bound}
 r_h(t)
 \le\|M_P\|_2H_\kappa(t)^{1/2}
 \|h\|_{L^2(0,t;\mathbb R^P)}.
\end{equation}

The mixed second displacement variation \(w_{h,k}\) satisfies
\[
 \begin{aligned}
 w_{h,k}''+\Lambda_Pw_{h,k}
 +DG_P(a_c)w_{h,k}
 &=-D^2G_P(a_c)[z_h,z_k],\\
 w_{h,k}(0)=w_{h,k}'(0)&=\mathbf 0.
 \end{aligned}
\]
For every \(v\in\mathbb R^P\) and \(x\in[0,1]\), the Cauchy--Schwarz inequality and \(|e_j(x)|\le\sqrt2\) imply
\begin{equation}\label{eq:modal-infty-bound}
 \begin{aligned}
 |u_v(x)|
 &\le
 \left(\sum_{j=1}^P(j\pi)^2|v_j|^2\right)^{1/2}
 \left(\sum_{j=1}^P\dfrac{|e_j(x)|^2}{(j\pi)^2}\right)^{1/2}\\
 &\le C_{\infty,P}|\Lambda_P^{1/2}v|.
 \end{aligned}
\end{equation}
Using the modal inverse estimate \eqref{eq:modal-infty-bound}, Bessel's inequality, and \(\|u_{z_k}\|_{L^2}=|z_k|\le\pi^{-1}r_k\), we obtain
\begin{equation}\label{eq:second-variation-source-bound}
 \begin{aligned}
 |D^2G_P(a_c)[z_h,z_k]|
 &\le m_2\|u_{z_h}u_{z_k}\|_{L^2(0,1)}\\
 &\le m_2\|u_{z_h}\|_{L^\infty(0,1)}
           \|u_{z_k}\|_{L^2(0,1)}\\
 &\le\dfrac{m_2C_{\infty,P}}{\pi}r_h r_k.
 \end{aligned}
\end{equation}
Using the variation-of-constants formula and Gr\"onwall's inequality from \eqref{eq:first-var-gronwall-bound}, together with the first-variation energy bound \eqref{eq:first-var-energy-bound} and the source estimate \eqref{eq:second-variation-source-bound}, we obtain
\begin{equation}\label{eq:second-variation-bound}
 \begin{aligned}
 |D^2\mathfrak F_P(c)[h,k]|
 &=\left|
 \mathsf E_P\begin{pmatrix}w_{h,k}(T)\\w_{h,k}'(T)\end{pmatrix}
 \right|\\
 &\le
 \dfrac{m_2C_{\infty,P}}{\pi}
 \int_0^T e^{\kappa(T-s)}r_h(s)r_k(s)\,\mathrm{d}s\\
 &\le
 \mathcal L_P
 \|h\|_{L^2(0,T;\mathbb R^P)}
 \|k\|_{L^2(0,T;\mathbb R^P)}.
 \end{aligned}
\end{equation}
Here, we bounded both partial-time control norms by their norms on \((0,T)\).

For \(c,d,h\in L^2(0,T;\mathbb R^P)\), we apply the integral mean-value formula to obtain
\[
 \bigl(D\mathfrak F_P(c)-D\mathfrak F_P(d)\bigr)h
 =
 \int_0^1
 D^2\mathfrak F_P(d+\tau(c-d))[c-d,h]\,\mathrm{d}\tau.
\]
Applying the second-variation bound \eqref{eq:second-variation-bound} and taking the supremum over directions \(h\in L^2(0,T;\mathbb R^P)\) with \(\|h\|_{L^2}=1\), we obtain the endpoint-derivative estimate \eqref{eq:galerkin_endpoint_DF_lipschitz}.

For \(g(u)=\beta\tanh u\) with \(\beta\ge0\),
\[
 |g'(u)|\le\beta,
 \qquad
 |g''(u)|=2\beta|\tanh u|(1-\tanh^2u).
\]
The function \(x\mapsto x(1-x^2)\) has derivative \(1-3x^2\), so it attains its maximum \(2/(3\sqrt3)\) on \([0,1]\) at \(x=1/\sqrt3\). We therefore take the stated values of \(\ell_1\) and \(m_2\), including the case \(\beta=0\).

For \(\kappa>0\), we integrate directly to obtain
\[
 \begin{aligned}
 \int_0^T e^{\kappa(T-s)}H_\kappa(s)\,\mathrm{d}s
 &=
 \dfrac{e^{\kappa T}}{2\kappa}
 \int_0^T(e^{\kappa s}-e^{-\kappa s})\,\mathrm{d}s\\
 &=\dfrac{e^{\kappa T}(\cosh(\kappa T)-1)}{\kappa^2}.
 \end{aligned}
\]
For \(\kappa=0\), the integral is \(\displaystyle\int_0^T s\,\mathrm{d}s=T^2/2\).
\end{proof}

We define
\begin{equation}\label{supp:eq:L-end-definition}
 L_P^{\mathrm{end}}
 :=
 \dfrac{m_2C_{\infty,P}}{\pi}
 \int_0^T e^{\kappa(T-s)}H_\kappa(s)\,\mathrm{d}s.
\end{equation}
Since $\|M_P\|_2\le1$, the constant in \eqref{eq:explicit_endpoint_L} satisfies
\[
 \mathcal L_P
 =\|M_P\|_2^2L_P^{\mathrm{end}}
 \le L_P^{\mathrm{end}}.
\]
Thus, $L_P^{\mathrm{end}}$ is also a global Lipschitz bound for the endpoint derivative in \eqref{eq:galerkin_endpoint_DF_lipschitz}. Retaining the factor $\|M_P\|_2^2$ gives a potentially smaller bound for the specified actuator matrix.

Let \(\Phi_c(t,s)\) be the transition matrix of the unscaled linearized system
\begin{equation}\label{eq:unscaled-variational-transition}
 \partial_t\Phi_c(t,s)
 =
 \begin{pmatrix}
 \mathbf 0&I_P\\
 -\Lambda_P-DG_P(a_c(t))&\mathbf 0
 \end{pmatrix}
 \Phi_c(t,s),
 \qquad
 \Phi_c(s,s)=I_{2P}.
\end{equation}
Then,
\[
 D\mathfrak F_P(c)h
 =\mathsf E_P\int_0^T\Phi_c(T,s)\mathsf B_Ph(s)\,\mathrm{d}s.
\]
Consequently, the energy-scaled linearized Gramian is
\[
 W^E_{P,c}
 =\mathsf E_P\left(
 \int_0^T\Phi_c(T,s)\mathsf B_P\mathsf B_P^\top
 \Phi_c(T,s)^\top\,\mathrm{d}s
 \right)\mathsf E_P^\top
 =D\mathfrak F_P(c)D\mathfrak F_P(c)^*.
\]
If \(W^E_{P,c}\) is positive definite, then the operator
\[
 R_c=D\mathfrak F_P(c)^*(W^E_{P,c})^{-1}
\]
is a right inverse of \(D\mathfrak F_P(c)\) and returns the minimum-\(L^2\) control direction for each endpoint direction. Moreover,
\[
 \|R_c\xi\|_{L^2(0,T;\mathbb R^P)}^2
 =\xi^\top(W^E_{P,c})^{-1}\xi,
\]
so
\[
 \gamma_P(c):=\|R_c\|
 =\lambda_{\min}(W^E_{P,c})^{-1/2}.
\]

Under the hypotheses of Proposition~\ref{prop:galerkin_endpoint_lipschitz}, we fix \(c\in L^2(0,T;\mathbb R^P)\) and assume that \(W^E_{P,c}\) is positive definite. We set
\[
 L=\mathcal L_P,
 \qquad
 \gamma=\gamma_P(c),
 \qquad
 \varepsilon=|\mathfrak F_P(c)|.
\]
For \(L>0\), the a posteriori endpoint certificate in the main article applies when
\[
 2L\gamma^2\varepsilon<1.
\]
The derivative bound is global, so no finite neighborhood-radius restriction is needed. If \(L=0\), then the endpoint map is affine, and \(c-R_c\mathfrak F_P(c)\) reaches the projected target exactly.

For a computer-assisted certificate, we use verified upper bounds in place of the endpoint defect and right-inverse norm, and we check the smallness condition using a verified upper bound for \(L\).

We now examine the coarse contraction bounds for the tanh nonlinearity. For the full-space bound, we recall the notation of the main article. Let \(S(t)\) be the unitary group for the homogeneous linear wave equation on \(\mathcal E\), and define
\[
 \mathcal B:L^2(\omega)\to\mathcal E,
 \qquad
 \mathcal Bv=(0,\widetilde v),
\]
where \(\widetilde v\) is the zero extension of \(v\) to \(\Omega\). The controllability operator and Gramian are
\[
 \mathcal C_Tf
 =-\int_0^T S(T-s)\mathcal Bf(s)\,\mathrm{d}s,
 \qquad
 \mathscr W_T=\mathcal C_T\mathcal C_T^*.
\]
When \(\mathscr W_T\) is boundedly invertible on \(\mathcal E\), we set
\[
 \Gamma_T=\|\mathscr W_T^{-1}\|_{\mathcal L(\mathcal E)}^{1/2},
 \qquad
 C_T^{\mathrm{aff}}
 =\sqrt T\left(1+\sqrt T\,\|\mathcal B\|\Gamma_T\right).
\]
For a valid Nemytskii Lipschitz constant \(L_g\), the coarse full-space contraction bound is
\[
 q_\infty=C_T^{\mathrm{aff}}\sqrt T\,L_g.
\]
We define the Galerkin bound \(q_P^{\mathrm{glob}}\) in \eqref{eq:picard_global_galerkin_contraction}.

\begin{proposition}
\label{prop:coarse_criteria_fail}
Let \(\Omega=(0,1)\), \(T=2.5\), and \(g(u)=\beta\tanh u\) with \(\beta\ge0\). Equip \(H_0^1(0,1)\) with the norm \(\|u\|_{H_0^1(0,1)}=\|u_x\|_{L^2(0,1)}\). The optimal Lipschitz constant of the Nemytskii map \(g:H_0^1(0,1)\to L^2(0,1)\) is
\[
 L_g^*=\dfrac{\beta}{\pi}.
\]

If the full-space Gramian \(\mathscr W_T\) is boundedly invertible, then the coarse full-space bound satisfies
\[
 q_\infty\ge\dfrac{2T\beta}{\pi}
 =\dfrac{5\beta}{\pi}.
\]
For the space spanned by the first eight Dirichlet sine modes,
\[
 q_8^{\mathrm{glob}}
 \ge T L_g^*\sqrt{\mu_{\max}(8)}
 =20\beta.
\]
Both inequalities hold for every choice of a valid constant \(L_g\ge L_g^*\) in the corresponding coarse bound. In particular, for every \(\beta\ge1\), both lower bounds exceed one, so neither coarse criterion certifies contraction.
\end{proposition}
\begin{proof}
Since \(|\tanh r-\tanh s|\le|r-s|\), we apply the sharp Poincar\'e inequality on \((0,1)\) to obtain
\[
 \begin{aligned}
 \|g(u)-g(v)\|_{L^2(0,1)}
 \le\beta\|u-v\|_{L^2(0,1)}\le\dfrac{\beta}{\pi}\|u_x-v_x\|_{L^2(0,1)}
 \end{aligned}
\]
for all \(u,v\in H_0^1(0,1)\). Hence, \(L_g^*\le\beta/\pi\).

To prove optimality, take \(u=\varepsilon e_1\) and \(v=0\), where \(e_1(x)=\sqrt2\sin(\pi x)\) and \(\varepsilon>0\). We have \(\|e_1\|_{L^2}=1\) and \(\|e_1'\|_{L^2}=\pi\). Moreover,
\[
 \dfrac{\tanh(\varepsilon e_1(x))}{\varepsilon}
 \longrightarrow e_1(x),
 \qquad
 \left|
 \dfrac{\tanh(\varepsilon e_1(x))}{\varepsilon}
 \right|\le|e_1(x)|.
\]
For every \(x\in(0,1)\), the quotient \(\tanh(\varepsilon e_1(x))/\varepsilon\) converges to \(e_1(x)\) as \(\varepsilon\to0^+\). Moreover,
\[
 \left|
 \dfrac{\tanh(\varepsilon e_1(x))}{\varepsilon}-e_1(x)
 \right|^2
 \le 4|e_1(x)|^2
 \quad\text{for every }\varepsilon>0.
\]
Since \(4|e_1|^2\in L^1(0,1)\), the dominated convergence theorem gives us convergence in \(L^2(0,1)\). Consequently,
\[
 \dfrac{\|g(\varepsilon e_1)-g(0)\|_{L^2(0,1)}}
      {\|(\varepsilon e_1)'\|_{L^2(0,1)}}
 =
 \dfrac{\beta}{\pi}
 \left\|\dfrac{\tanh(\varepsilon e_1)}{\varepsilon}\right\|_{L^2(0,1)}
 \longrightarrow\dfrac{\beta}{\pi}.
\]
Thus, we deduce that \(L_g^*=\beta/\pi\).

For the full-space estimate, \(\mathcal B\) is an isometry from \(L^2(\omega)\) into \(\mathcal E\), so \(\|\mathcal B\|=1\). Using the unitarity of \(S(t)\) and the Cauchy--Schwarz inequality, we obtain
\[
 \begin{aligned}
 \|\mathcal C_Tf\|_{\mathcal E}
 &\le\int_0^T\|\mathcal Bf(s)\|_{\mathcal E}\,\mathrm{d}s\le\sqrt T\,\|f\|_{L^2(0,T;L^2(\omega))}.
 \end{aligned}
\]
Consequently,
\[
 \|\mathscr W_T\|
 =\|\mathcal C_T\mathcal C_T^*\|
 =\|\mathcal C_T\|^2
 \le T.
\]
Since \(\mathscr W_T\) is boundedly invertible,
\[
 1=\|I_{\mathcal E}\|
 \le\|\mathscr W_T\|\,\|\mathscr W_T^{-1}\|,
\]
and hence
\[
 \Gamma_T
 =\|\mathscr W_T^{-1}\|^{1/2}
 \ge T^{-1/2}.
\]
We therefore obtain
\[
 C_T^{\mathrm{aff}}
 =\sqrt T\left(1+\sqrt T\,\Gamma_T\right)
 \ge2\sqrt T.
\]
Using \(L_g\ge L_g^*=\beta/\pi\), we conclude that
\[
 q_\infty
 =C_T^{\mathrm{aff}}\sqrt T\,L_g
 \ge\dfrac{2T\beta}{\pi}.
\]

For the Galerkin estimate, the interval defining \(K_{P,T}\) contains \(t=0\), so
\[
 K_{P,T}\ge\|e^{\mathsf A_P0}\|_2=1.
\]
Using the expression for \(C_{P,T}^{\mathrm{aff}}\), we obtain \(C_{P,T}^{\mathrm{aff}}\ge\sqrt T\), and
\[
 q_P^{\mathrm{glob}}
 =C_{P,T}^{\mathrm{aff}}L_g\sqrt{T\mu_{\max}(P)}
 \ge T L_g^*\sqrt{\mu_{\max}(P)}.
\]
For the first eight sine modes, \(\mu_{\max}(8)=(8\pi)^2\). Substituting \(T=2.5\), we obtain
\[
 q_8^{\mathrm{glob}}
 \ge2.5\,\dfrac{\beta}{\pi}\,8\pi
 =20\beta.
\]
Finally, since \(5/\pi>1\) and \(20>1\), we obtain the conclusion for \(\beta\ge1\).
\end{proof}
These conclusions concern the coarse upper bounds for the Lipschitz constants of the source-frozen maps. A coarse bound greater than one does not imply that the map itself is noncontractive or that its iteration diverges. A sharper estimate may still establish contraction, as in the finite-mode response criterion based on \(\ell_g\chi_{P,T}\).

\begin{proposition}
\label{prop:validated_chi_certificate}
We consider \(\Omega=(0,1)\), \(\alpha=0\), the first \(8\) Dirichlet sine modes, \(T=5/2\), \(\omega=(0,3/10)\cup(7/10,1)\), and the Euclidean modal-coefficient cost \(\displaystyle\int_0^T|c(t)|_{\mathbb R^8}^2\,\mathrm{d}t\). Then, using the computer-assisted enclosure, we obtain
\begin{equation}\label{eq:validated_chi_bound}
 \chi_{8,2.5}
 \le0.16056015502558602
 <0.160561<\dfrac{1}{6}.
\end{equation}
Consequently, for \(g(u)=\beta\tanh u\) with \(0\le\beta\le6\), the exact eight-mode source-frozen map is globally contractive on \(\mathcal Z_8\).
\end{proposition}

\begin{proof}
We enclose a projected response operator and its projection remainder separately. The trace--Ritz estimate controls the projected norm, while cellwise derivative bounds control the remainder. We write \(\mathscr H=\mathscr H_{8,2.5}\), and let \(\Pi_N\) be the orthogonal \(L^2\) projection onto functions that are constant on each of \(N=600\) uniform time intervals. The cell width is \(h=T/N=1/240\). Let \(K\) be the displacement kernel defined in \eqref{supp:eq:K-explicit}. We write
\[
 \|K\|_{L^2_F}^2
 =
 \int_0^T\int_0^T
 \|K(t,s)\|_F^2\,\mathrm{d}s\,\mathrm{d}t,
\]
where \(\|\cdot\|_F\) denotes the Frobenius matrix norm. Thus, \(\|\cdot\|_{L^2_F}\) denotes the \(L^2\) norm over \((0,T)^2\) with the Frobenius norm on the matrix values. We use the same convention for the weak derivatives of \(K\). With \(\pi_1=\begin{pmatrix}I_8&0\end{pmatrix}\) and \(F=\begin{pmatrix}0\\-I_8\end{pmatrix}\), the kernel is continuous across the diagonal because \(\pi_1F=\mathbf 0\). Its weak first derivatives are square integrable, as established in Section~\ref{supp:sec:response-kernel}.

Applying the projection-remainder estimate \eqref{supp:eq:projection-remainder}, which we proved in Section~\ref{supp:sec:response-kernel} using cellwise Poincar\'e--Wirtinger inequalities in the two kernel variables, we obtain
\[
 \|\mathscr H-\Pi_N\mathscr H\Pi_N\|
 \le\dfrac{h}{\pi}
 \left(\|\partial_tK\|_{L^2_F}
       +\|\partial_sK\|_{L^2_F}\right).
\]
The required derivative enclosures are
\[
 \|\partial_tK\|_{L^2_F}\le2.618750484691,
 \qquad
 \|\partial_sK\|_{L^2_F}\le2.654926696062.
\]
From these bounds, we obtain
\[
 \|\mathscr H-\Pi_N\mathscr H\Pi_N\|
 \le0.006994431597.
\]

The range of \(\Pi_N\) is
\[
 \operatorname{range}\Pi_N
 :=\{\Pi_Nv:v\in L^2(0,T;\mathbb R^8)\},
\]
which consists of the \(\mathbb R^8\)-valued functions that are constant almost everywhere on each time cell. Since \(N=600\), this space has dimension \(8N=4800\). Let \(H_N\in\mathbb R^{4800\times4800}\) represent the restriction of \(\Pi_N\mathscr H\Pi_N\) to this space in an orthonormal cellwise-constant basis. Let \(\widehat H_N\) denote the assembled binary64 matrix, with its stored entries interpreted as exact real numbers. Since \(\Pi_N\) is an orthogonal projection,
\[
 \|\Pi_N\mathscr H\Pi_N\|=\|H_N\|_2.
\]
Using the verified trace--Ritz calculation based on Lemma~\ref{supp:lem:trace-ritz}, we obtain
\[
 \|\widehat H_N\|_2\le0.152429487261,
\]
while we use the assembly and arithmetic enclosure to obtain
\[
 \|H_N-\widehat H_N\|_2\le0.001136236169.
\]
Thus, using the triangle inequality, we obtain
\[
 \|\Pi_N\mathscr H\Pi_N\|
 \le\|\widehat H_N\|_2+\|H_N-\widehat H_N\|_2
 \le0.152429487261+0.001136236169.
\]
Combining the projected-operator and projection bounds, we arrive at
\[
 \begin{aligned}
 \chi_{8,2.5}
 &\le0.152429487261+0.001136236169+0.006994431597\\
 &=0.160560155027\\
 &<0.160561<\dfrac{1}{6}.
 \end{aligned}
\]

We round the displayed contributions upward for readability. From the full-precision outward calculation described in Section~\ref{supp:sec:response-kernel}, we obtain the tighter value in \eqref{eq:validated_chi_bound}.

For \(\beta\ge0\), the function \(g(u)=\beta\tanh u\) satisfies
\[
 |g'(u)|=\dfrac{\beta}{\cosh^2u}\le\beta,
 \qquad u\in\mathbb R.
\]
Thus, we may take \(\ell_g=\beta\). For every \(a,b\in\mathcal Z_8\), the source-response estimate gives
\[
 \|\mathcal T_8(a)-\mathcal T_8(b)\|_{\mathcal Z_8}
 \le\beta\chi_{8,2.5}\|a-b\|_{\mathcal Z_8}.
\]
For \(0\le\beta\le6\), we have
\[
 \beta\chi_{8,2.5}
 \le6(0.160560155027)
 =0.963360930162<1.
\]
We obtain the conclusion by applying Theorem~\ref{thm:exact_response_galerkin}.
\end{proof}

\begin{remark}
We use the finite-mode response enclosure to bound the linear operator \(\mathscr H_{P,T}\) and establish global contraction of the exact source-frozen map under the stated smallness condition. This result is independent of the endpoint data and the initial trial trajectory.

The endpoint certificates in the main article instead enclose the nonlinear endpoint defect and a right-inverse bound at a specified candidate control. Combined with the derivative Lipschitz bound and the certificate smallness condition, they establish the existence of a nearby control reaching the projected target exactly. Both conclusions concern the specified Galerkin system. Neither alone certifies exact control of the full PDE.
\end{remark}

\begin{proposition}
\label{prop:local_spectral_radius}
We consider \(\mathcal X_P=C([0,T];\mathbb R^P)\) with the supremum norm. We assume that \(g\in C^1(\mathbb R)\), and let \(a^*\in\mathcal X_P\) be a fixed point of the exact source-frozen map \(\mathcal T_P\).

Then, \(\mathcal T_P\) is continuously Fr\'echet differentiable near \(a^*\), and
\[
 D\mathcal T_P(a^*)
 =\mathscr H_{P,T}DG_P(a^*),
\]
where \(DG_P(a^*)\) acts pointwise in time:
\[
 \bigl(DG_P(a^*)v\bigr)(t)
 =DG_P(a^*(t))v(t).
\]

For a fixed \(0<\eta\le1\), we define
\[
 \mathcal T_{P,\eta}=(1-\eta)I+\eta\mathcal T_P,
 \qquad
 J_\eta=(1-\eta)I+\eta D\mathcal T_P(a^*).
\]
Let \(r(J_\eta)\) denote the spectral radius of \(J_\eta\), with the spectrum taken in the complexification of \(\mathcal X_P\). If \(r(J_\eta)<1\), then the relaxed iteration converges locally and geometrically to \(a^*\). More precisely, for every \(q\in(r(J_\eta),1)\), there exist \(\delta>0\) and \(C\ge1\) such that
\[
 \|a^{(n)}-a^*\|_{\mathcal X_P}
 \le Cq^n\|a^{(0)}-a^*\|_{\mathcal X_P},
 \qquad n\in\{0,1,2,\ldots\},
\]
whenever \(\|a^{(0)}-a^*\|_{\mathcal X_P}\le\delta\).

In particular, taking \(\eta=1\), we obtain local convergence of the undamped iteration whenever \(r(D\mathcal T_P(a^*))<1\).
\end{proposition}
\begin{proof}
We first differentiate the projected Nemytskii map on the trajectory space. We then use an equivalent norm in which the relaxed derivative is a contraction. Derivative continuity transfers that estimate to a neighborhood of the fixed point.
For \(a,v\in\mathbb R^P\), we write
\[
 u_a(x)=\sum_{j=1}^Pa_je_j(x),
 \qquad
 u_v(x)=\sum_{j=1}^Pv_je_j(x).
\]
Since the Galerkin space is finite dimensional and its basis functions are bounded, we can differentiate \(G_P:\mathbb R^P\to\mathbb R^P\) continuously when \(g\in C^1(\mathbb R)\), obtaining
\[
 (DG_P(a)v)_j
 =\int_\Omega g'(u_a(x))u_v(x)e_j(x)\,\mathrm{d}x.
\]

To verify differentiability on the trajectory space, we consider \(a\) in a bounded neighborhood of \(a^*\). The values of \(a(t)\) and of sufficiently small perturbations \(a(t)+v(t)\) lie in a fixed compact subset of \(\mathbb R^P\). On this set, \(DG_P\) is uniformly continuous. Thus, there is a modulus \(\omega_D(\varepsilon)\to0\) as \(\varepsilon\to0^+\) such that
\[
 \begin{aligned}
 \|G_P(a+v)-G_P(a)
       -DG_P(a)v\|_{\mathcal X_P}\le
 \omega_D(\|v\|_{\mathcal X_P})\|v\|_{\mathcal X_P}.
 \end{aligned}
\]
Indeed, we obtain this bound by applying the integral mean-value formula at each time and taking the supremum. The same uniform continuity lets us conclude that \(a\mapsto DG_P(a)\) is continuous in the operator norm. Hence, the projected Nemytskii map is continuously Fr\'echet differentiable on \(\mathcal X_P\).

Using the continuous embedding \(\mathcal X_P\hookrightarrow L^2(0,T;\mathbb R^P)\) and Lemma~\ref{lem:affine_source_trajectory_lipschitz}, we obtain
\[
 \|\mathscr H_{P,T}s\|_{\mathcal X_P}
 \le C_{P,T}^{\mathrm{aff}}\sqrt T\,\|s\|_{\mathcal X_P}.
\]
Using
\[
 \mathcal T_P(a)
 =\mathcal H_P(\mathbf 0)+\mathscr H_{P,T}G_P(a),
\]
we therefore apply the chain rule to obtain
\[
 D\mathcal T_P(a^*)
 =\mathscr H_{P,T}DG_P(a^*).
\]

We fix \(\eta\in(0,1]\) and suppose that \(r(J_\eta)<1\). For a prescribed \(q\in(r(J_\eta),1)\), we choose \(\vartheta\in\mathbb R\) such that
\[
 r(J_\eta)<\vartheta<q.
\]
By the spectral-radius formula, we choose \(C\ge1\) satisfying
\[
 \|J_\eta^n\|_{\mathcal L(\mathcal X_P)}
 \le C\vartheta^n,
 \qquad n\in\{0,1,2,\ldots\}.
\]
For \(v\in\mathcal X_P\), we define
\[
 \|v\|_*=\sup_{n\in\{0,1,2,\ldots\}}
 \vartheta^{-n}\|J_\eta^nv\|_{\mathcal X_P}.
\]
This is an equivalent norm, since
\[
 \|v\|_{\mathcal X_P}\le\|v\|_*
 \le C\|v\|_{\mathcal X_P}.
\]
Moreover,
\[
 \begin{aligned}
 \|J_\eta v\|_*
 &=\vartheta
 \sup_{m\in\{1,2,\ldots\}}\vartheta^{-m}
 \|J_\eta^mv\|_{\mathcal X_P}\\
 &\le\vartheta\|v\|_*.
 \end{aligned}
\]

Continuity of \(D\mathcal T_{P,\eta}\) also holds in the equivalent operator norm. Since \(D\mathcal T_{P,\eta}(a^*)=J_\eta\), we can choose \(r_*>0\) such that
\[
 \|D\mathcal T_{P,\eta}(a)\|_*\le q
 \qquad\text{whenever }\|a-a^*\|_*\le r_*.
\]
Applying the integral mean-value formula, we conclude that \(\mathcal T_{P,\eta}\) is \(q\)-Lipschitz on this ball. Because \(\mathcal T_{P,\eta}(a^*)=a^*\),
\[
 \|\mathcal T_{P,\eta}(a)-a^*\|_*
 \le q\|a-a^*\|_*
 \le qr_*<r_*.
\]
Thus, the closed ball is invariant, and the iterates satisfy
\[
 \|a^{(n)}-a^*\|_*
 \le q^n\|a^{(0)}-a^*\|_*.
\]
Taking \(\delta=r_*/C\) and using norm equivalence, we obtain
\[
 \|a^{(n)}-a^*\|_{\mathcal X_P}
 \le Cq^n\|a^{(0)}-a^*\|_{\mathcal X_P}.
\]
We have proved the relaxed statement and, with \(\eta=1\), the undamped statement.
\end{proof}

\begin{remark}
Using Proposition~\ref{prop:coarse_criteria_fail}, we see that the coarse bounds do not certify contraction for the reported parameters with \(\beta\ge1\). In contrast, we use \eqref{eq:validated_chi_bound} to establish global contraction of the exact eight-mode map for \(0\le\beta\le6\).

We give a separate local criterion at a fixed point in Proposition~\ref{prop:local_spectral_radius}. Eigenvalues of a time-discretized Jacobian concern the discretized map. Applying them to the exact trajectory map requires an additional argument controlling the discretization and spectral errors. The finite-mode response enclosure does not by itself establish a corresponding full-PDE contraction result.
\end{remark}

\begin{remark}
The coarse contraction theorem uses the supremum norm on \(\mathcal X_P\). We cannot apply its error estimates directly to a residual measured only in \(\mathcal Z_P\).

Under the exact-response condition of Theorem~\ref{thm:exact_response_galerkin}, we use the exact residual \(\|a-\mathcal T_P(a)\|_{\mathcal Z_P}\) to bound the trajectory error by \eqref{eq:sharp_proposal_certificate}.
\end{remark}

\begin{proposition}
Let \(m\in\{1,2,\ldots\}\), and let \(\mathcal S:U\to\mathbb R^m\) represent a finite-dimensional time discretization of the source-frozen trajectory map, where \(U\subset\mathbb R^m\) is open. We suppose that \(\mathcal S\) is continuously differentiable near a fixed point \(x^*\in U\), and we set \(J=D\mathcal S(x^*)\). For \(\eta>0\), we define the relaxed map by \(\mathcal S_\eta(x)=(1-\eta)x+\eta\mathcal S(x)\) for \(x\in U\). There exists \(\eta>0\) such that
\[
 r(D\mathcal S_\eta(x^*))
 =r((1-\eta)I_m+\eta J)<1
\]
if and only if
\[
 \operatorname{Re}\lambda<1
 \qquad\text{for every }\lambda\in\sigma(J).
\]

Under this spectral condition, we define
\[
 \eta_{\max}
 =\min_{\lambda\in\sigma(J)}
 \dfrac{2(1-\operatorname{Re}\lambda)}
      {|\lambda-1|^2}.
\]
Then, \(\eta_{\max}>0\), and the exact range giving strict linearized stability is
\[
 0<\eta<\eta_{\max}.
\]
For every such \(\eta\), the nonlinear relaxed iteration converges locally and geometrically to \(x^*\). If relaxation is restricted to \(0<\eta\le1\), then the admissible range is
\[
 \{\eta\in(0,1]:\eta<\eta_{\max}\}.
\]

If \(J\) has an eigenvalue with \(\operatorname{Re}\lambda\ge1\), then no positive \(\eta\) makes the relaxed derivative have spectral radius strictly less than one.
\end{proposition}
\begin{proof}
Since \(\mathcal S(x^*)=x^*\), we also have \(\mathcal S_\eta(x^*)=x^*\), and
\[
 D\mathcal S_\eta(x^*)=(1-\eta)I_m+\eta J.
\]
The eigenvalues of this matrix are \(1+\eta(\lambda-1)\), with \(\lambda\in\sigma(J)\). This spectral identity does not require \(J\) to be diagonalizable.

For \(\eta>0\), we expand directly to obtain
\[
 \begin{aligned}
 |1+\eta(\lambda-1)|^2
 &=1+2\eta(\operatorname{Re}\lambda-1)
   +\eta^2|\lambda-1|^2.
 \end{aligned}
\]
Therefore, \(|1+\eta(\lambda-1)|<1\) is equivalent to \(2(\operatorname{Re}\lambda-1)+\eta|\lambda-1|^2<0\). If this inequality holds, then necessarily \(\operatorname{Re}\lambda<1\). Conversely, if \(\operatorname{Re}\lambda<1\), then \(\lambda\ne1\), and the inequality is equivalent to
\[
 0<\eta<
 \dfrac{2(1-\operatorname{Re}\lambda)}
      {|\lambda-1|^2}.
\]
Because the spectrum of \(J\) is finite, we can choose a common positive parameter exactly when every eigenvalue satisfies \(\operatorname{Re}\lambda<1\). Under this condition, we intersect the admissible intervals to obtain \(0<\eta<\eta_{\max}\).

The minimum defining \(\eta_{\max}\) is attained. At \(\eta=\eta_{\max}\), at least one relaxed eigenvalue has modulus one, so we exclude the upper endpoint. For \(0<\eta<\eta_{\max}\), the derivative of \(\mathcal S_\eta\) has spectral radius less than one. Applying the equivalent-norm argument in Proposition~\ref{prop:local_spectral_radius} in \(\mathbb R^m\), we obtain local geometric convergence.

Finally, we restrict \(\eta\) to \((0,1]\) to obtain the stated under-relaxation range.
\end{proof}

\begin{remark}
The proposition characterizes strict stability of the linearized discrete iteration. If \(\lambda=1\) is an eigenvalue of \(J\), then it remains an eigenvalue of \((1-\eta)I_m+\eta J\) for every \(\eta>0\). In this case, the linearized criterion is inconclusive about nonlinear convergence.

Moreover, failure of a finite set of tested relaxation parameters does not establish that no stabilizing parameter exists. The admissible interval may lie below the smallest tested positive value.
\end{remark}

\section{Source-frozen Picard control: reusable Gramian realization}
\label{sec:picard_gramian}
We express the affine steering step through the controllability Gramian of the finite-dimensional linear system. For fixed selected modes, actuator geometry, linear operator, time horizon, and control metric, this Gramian is independent of the frozen source and endpoint data, so its factorization can be reused. We then prove that the exact Gramian and HUM realizations produce the same proposed trajectory and inherit the same convergence criteria.

\subsection{First-order source-frozen system}
We use the modal basis, matrices \(\Lambda_P\) and \(M\), projected nonlinearity \(G_P\), and endpoint vectors \(Y_0,Y_T\) from Section~\ref{sub:source_frozen_framework}. In particular, we use the unscaled coordinates
\[
 Y=\begin{pmatrix}a\\a'\end{pmatrix}\in\mathbb R^{2P},
\]
with
\[
 \mathsf A_P=
 \begin{pmatrix}
 \mathbf 0&I_P\\
 -\Lambda_P&\mathbf 0
 \end{pmatrix},
 \qquad
 \mathsf B_P=
 \begin{pmatrix}
 \mathbf 0\\-M
 \end{pmatrix}.
\]
As in the HUM formulation, we define \(f_P=\sum_{j=1}^Pc_je_j\) on \(\Omega\), and \(\mathbf 1_\omega f_P\) is the applied forcing.

For a trial trajectory \(a^{(n)}\), we define
\[
 s^{(n)}(t)=G_P(a^{(n)}(t)),
 \qquad
 b^{(n)}(t)=
 \begin{pmatrix}
 \mathbf 0\\-s^{(n)}(t)
 \end{pmatrix},
\]
and assume that \(s^{(n)}\in L^2(0,T;\mathbb R^P)\). The frozen affine steering problem is
\begin{equation}\label{eq:pg_source_frozen_first_order_system}
 Y'(t)=\mathsf A_PY(t)+\mathsf B_Pc(t)+b^{(n)}(t),
 \qquad
 Y(0)=Y_0,\quad Y(T)=Y_T.
\end{equation}
Among the controls \(c\in L^2(0,T;\mathbb R^P)\) satisfying these endpoint constraints, the Gramian realization minimizes
\[
 \mathcal J_{\mathrm{PG}}(c)
 =\int_0^T|c(t)|_{\mathbb R^P}^2\,\mathrm{d}t
 =\mathcal J(c).
\]
Thus, the affine dynamics, endpoint constraints, and objective coincide with those of the HUM formulation. In the following subsection, we derive the Gramian representation of its unique minimum-cost control.
\subsection{Reusable controllability Gramian and frozen control formula}

By Lemma~\ref{lem:retained_pair_controllable}, the pair \((\mathsf A_P,\mathsf B_P)\) is controllable. Hence, for every \(T>0\), its controllability Gramian
\[
 W_T
 =\int_0^T e^{\mathsf A_P(T-s)}
 \mathsf B_P\mathsf B_P^\top
 e^{\mathsf A_P^\top(T-s)}\,\mathrm{d}s
\]
is symmetric positive definite. Here, \(W_T=W_{P,T}\) in the notation of Lemma~\ref{lem:affine_source_trajectory_lipschitz}. We suppress the fixed index \(P\).

Applying the variation of constants to \eqref{eq:pg_source_frozen_first_order_system}, we obtain
\[
 Y(T)
 =e^{\mathsf A_PT}Y_0
 +\int_0^T e^{\mathsf A_P(T-s)}b^{(n)}(s)\,\mathrm{d}s
 +\int_0^T e^{\mathsf A_P(T-s)}\mathsf B_Pc(s)\,\mathrm{d}s.
\]
Therefore, the terminal condition is equivalent to
\[
 \int_0^T e^{\mathsf A_P(T-s)}\mathsf B_Pc(s)\,\mathrm{d}s
 =d^{(n)},
\]
where
\begin{equation}\label{eq:picard_gramian_defect}
 d^{(n)}
 =Y_T-e^{\mathsf A_PT}Y_0
 -\int_0^T e^{\mathsf A_P(T-s)}b^{(n)}(s)\,\mathrm{d}s.
\end{equation}
Thus, \(d^{(n)}\) is the target minus the uncontrolled frozen-source endpoint, expressed in unscaled coordinates. The unique minimum-modal-cost control is
\begin{equation}\label{eq:picard_gramian_control}
 c^{(n+1)}(t)
 =\mathsf B_P^\top e^{\mathsf A_P^\top(T-t)}
 W_T^{-1}d^{(n)}.
\end{equation}
Indeed, substituting into the endpoint constraint, we verify
\[
 \int_0^T e^{\mathsf A_P(T-s)}
 \mathsf B_Pc^{(n+1)}(s)\,\mathrm{d}s
 =W_TW_T^{-1}d^{(n)}
 =d^{(n)}.
\]
The control formula \eqref{eq:picard_gramian_control} is the HUM formula from Section~\ref{sec:sf_picard_hum}, with the adjoint terminal value \(q^{(n)}=W_T^{-1}d^{(n)}\). By the orthogonal norm-splitting identity \eqref{eq:hum-minimum-cost-pythagorean}, we obtain minimum cost and uniqueness without a second proof. Substituting the control formula into the cost integral, we compute the optimal value
\[
 \mathcal J_{\mathrm{PG}}(c^{(n+1)})
 =(d^{(n)})^\top W_T^{-1}d^{(n)}.
\]
More generally, let \(\Sigma_P\in\mathbb R^{P\times P}\) be a fixed symmetric positive-definite matrix and consider
\[
 \mathcal J_{\Sigma_P}(c)
 =\int_0^T c(t)^\top \Sigma_Pc(t)\,\mathrm{d}t.
\]
Introduce the transformed control
\[
 v(t)=\Sigma_P^{1/2}c(t),
 \qquad
 c(t)=\Sigma_P^{-1/2}v(t).
\]
Then,
\[
 \mathcal J_{\Sigma_P}(c)
 =\|v\|_{L^2(0,T;\mathbb R^P)}^2,
\]
and the affine system becomes
\[
 Y'
 =\mathsf A_PY+\mathsf B_P\Sigma_P^{-1/2}v+b^{(n)}.
\]
The transformation between \(c\) and \(v\) is invertible, so the transformed system has the same reachable endpoints. Its controllability Gramian is therefore positive definite and equals
\[
 \begin{aligned}
 W_{T,\Sigma}
 &=\int_0^T e^{\mathsf A_P(T-s)}
 (\mathsf B_P\Sigma_P^{-1/2})
 (\mathsf B_P\Sigma_P^{-1/2})^\top
 e^{\mathsf A_P^\top(T-s)}\,\mathrm{d}s\\
 &=\int_0^T e^{\mathsf A_P(T-s)}
 \mathsf B_P\Sigma_P^{-1}\mathsf B_P^\top
 e^{\mathsf A_P^\top(T-s)}\,\mathrm{d}s.
 \end{aligned}
\]
Applying the minimum-norm formula to \(v\) and transforming back, we obtain
\[
 c_\Sigma^{(n+1)}(t)
 =\Sigma_P^{-1}\mathsf B_P^\top
 e^{\mathsf A_P^\top(T-t)}
 W_{T,\Sigma}^{-1}d^{(n)}.
\]
The optimal weighted cost is
\[
 \mathcal J_{\Sigma_P}(c_\Sigma^{(n+1)})
 =(d^{(n)})^\top W_{T,\Sigma}^{-1}d^{(n)}.
\]

Taking \(\Sigma_P=I_P\) recovers the modal-cost control formula \eqref{eq:picard_gramian_control}. Taking \(\Sigma_P=M\) minimizes the physical applied-control cost \eqref{eq:physical_control_cost} within the modal control space, subject to the same affine dynamics and endpoints. The nonlinear experiments use \(\Sigma_P=I_P\) for each frozen affine solve and report the physical control norm separately. The minimum-cost property applies to each frozen subproblem. It does not establish optimality for the nonlinear steering problem.

For fixed \(\mathsf A_P\), \(\mathsf B_P\), \(T\), and control metric \(\Sigma_P\), the Gramian \(W_{T,\Sigma}\) is independent of the frozen source and endpoint data. Its factorization can therefore be reused throughout the source-frozen iteration and across different projected endpoint pairs.

For the modal cost, one factors \(W_T\) once, solves \(W_Tq^{(n)}=d^{(n)}\) at each iteration, and then evaluates \(c^{(n+1)}(t)=\mathsf B_P^\top e^{\mathsf A_P^\top(T-t)}q^{(n)}\). Thus, the implementation can apply the formula without forming \(W_T^{-1}\) explicitly.

This reuse concerns the fixed linear system obtained by freezing the nonlinear source. The linearized Gramian used by the Newton--HUM method depends on the current nonlinear trajectory and generally changes between iterations.

\subsection{The Picard--Gramian algorithm}

For each $a\in L^2(0,T;\mathbb R^P)$ such that $G_P(a)\in L^2(0,T;\mathbb R^P)$, define
\[
 \mathcal T_P^{\mathrm{PG}}(a)
 :=\mathcal H_P(G_P(a)).
\]
By the uniqueness of the minimum-modal-cost affine control,
\[
 \mathcal T_P^{\mathrm{PG}}(a)
 =\mathcal T_P^{\mathrm{HUM}}(a)
 =\mathcal T_P(a)
\]
whenever these maps are defined. At iteration \(n\), we write \(\widehat a^{(n+1)}=\mathcal T_P^{\mathrm{PG}}(a^{(n)})\). The unrelaxed proposal residual is
\begin{equation}\label{eq:pg_picard_residual}
r_{\mathrm{fp,PG}}^{(n)}
=
\|\widehat a^{(n+1)}-a^{(n)}\|_
{L^2(0,T;\mathbb R^P)}.
\end{equation}

The following algorithm describes one run with a fixed relaxation parameter. We assume that every frozen source encountered belongs to \(L^2(0,T;\mathbb R^P)\). The affine solves and residual evaluations are understood to be exact.

\begin{algorithm}[H]
\caption{Reusable Picard--Gramian control algorithm}
\begin{algorithmic}[1]
\State We assemble \(\mathsf A_P\) and \(\mathsf B_P\),
compute the positive-definite Gramian \(W_T\), and
factorize it once.

\State We choose
\(a^{(0)}\in H^2(0,T;\mathbb R^P)\)
satisfying the prescribed displacement and velocity
conditions at \(t=0\) and \(t=T\).
We choose \(\eta\in(0,1]\),
\(\varepsilon_{\mathrm{fp}}>0\), and an integer
\(N_{\max}\ge 1\).

\For{\(n=0,\ldots,N_{\max}-1\)}
  \State We freeze
  \(s^{(n)}(t)=G_P(a^{(n)}(t))\)
  and set
  \(b^{(n)}(t)=
  \binom{\mathbf 0}{-s^{(n)}(t)}\).

  \State We form \(d^{(n)}\) from
  \eqref{eq:picard_gramian_defect}.

  \State We use the stored factorization to solve
  \(W_Tq^{(n)}=d^{(n)}\), and set
  \[
  c^{(n+1)}(t)
  =
  \mathsf B_P^\top
  e^{\mathsf A_P^\top(T-t)}q^{(n)}.
  \]

  \State We solve the frozen affine initial-value problem
  \eqref{eq:pg_source_frozen_first_order_system}
  with initial state \(Y_0\), source \(b^{(n)}\),
  and control \(c^{(n+1)}\), obtaining the displacement
  proposal \(\widehat a^{(n+1)}\).

  \State We compute \(r_{\mathrm{fp,PG}}^{(n)}\)
  using the proposal-residual definition \eqref{eq:pg_picard_residual}.

  \State If
  \(r_{\mathrm{fp,PG}}^{(n)}
  \le\varepsilon_{\mathrm{fp}}\), then
  the algorithm returns
  \((c^{(n+1)},\widehat a^{(n+1)})\),
  the residual \(r_{\mathrm{fp,PG}}^{(n)}\), and
  the status ``proposal tolerance met.''

  \State If \(n<N_{\max}-1\), then we set
  \[
  a^{(n+1)}
  =(1-\eta)a^{(n)}
  +\eta\widehat a^{(n+1)}.
  \]
\EndFor

\State The algorithm returns
\((c^{(N_{\max})},\widehat a^{(N_{\max})})\),
the residual \(r_{\mathrm{fp,PG}}^{(N_{\max}-1)}\),
and the status ``iteration cap reached.''
\end{algorithmic}
\end{algorithm}

Under exact affine solves, each proposal satisfies the prescribed displacement and velocity endpoint conditions. The relaxed iterates preserve these conditions because the initialization is endpoint-compatible. The control returned by the algorithm is paired with its affine proposal, rather than with the relaxed iterate.

Meeting the proposal tolerance does not by itself certify the nonlinear endpoint error. In a numerical implementation, affine-solve errors and residual-evaluation errors must also be distinguished from the exact affine endpoint conditions in \eqref{eq:pg_source_frozen_first_order_system} and the proposal residual in \eqref{eq:pg_picard_residual}. The returned control is therefore evaluated by the separate nonlinear endpoint-validation procedure used in the computational protocol.

\subsection{Equivalence with the source-frozen Picard--HUM algorithm and convergence}
\begin{proposition}
\label{prop:hum_gramian_equivalence}
We fix \(T>0\), the Galerkin space, the matrices \(\mathsf A_P\) and \(\mathsf B_P\), endpoint states \(Y_0,Y_T\in\mathbb R^{2P}\), and a frozen source \(s\in L^2(0,T;\mathbb R^P)\). Both constructions use the control space \(L^2(0,T;\mathbb R^P)\) and the modal-coefficient cost \eqref{eq:modal_control_cost}. Under the controllability established in Lemma~\ref{lem:retained_pair_controllable}, the exact HUM construction and the Gramian formula \eqref{eq:picard_gramian_control} produce the same unique minimum-cost control and the same affine state trajectory.

Consequently, \(\mathcal T_P^{\mathrm{PG}}(a)=\mathcal T_P^{\mathrm{HUM}}(a)\) for every trial trajectory \(a\) such that \(G_P(a)\in L^2(0,T;\mathbb R^P)\).
\end{proposition}

\begin{proof}
We identify the HUM optimality equation with the Gramian equation. We set
\[
b(t)=\binom{\mathbf 0}{-s(t)},
\qquad
d=Y_T-e^{\mathsf A_PT}Y_0
-\int_0^T e^{\mathsf A_P(T-t)}b(t)\,\mathrm{d}t.
\]
For terminal adjoint data \(q\in\mathbb R^{2P}\), we define
\[
p_q(t)=e^{\mathsf A_P^\top(T-t)}q.
\]
Then
\[
-p_q'(t)=\mathsf A_P^\top p_q(t),
\qquad p_q(T)=q.
\]
With the control convention \(c(t)=\mathsf B_P^\top p_q(t)\), the HUM functional is
\[
\begin{aligned}
\mathcal F(q)
=\dfrac{1}{2}\int_0^T
  |\mathsf B_P^\top p_q(t)|^2\,\mathrm{d}t-d^\top q=\dfrac{1}{2}q^\top W_Tq-d^\top q.
\end{aligned}
\]
Since \(W_T\) is positive definite, this functional has the unique minimizer \(q_*=W_T^{-1}d\). Thus, the HUM control is
\[
c_{\mathrm{HUM}}(t)
=
\mathsf B_P^\top
e^{\mathsf A_P^\top(T-t)}W_T^{-1}d,
\]
which is exactly the Gramian formula \eqref{eq:picard_gramian_control}. Using the minimum-cost identity \eqref{eq:hum-minimum-cost-pythagorean}, we conclude that this control uniquely minimizes \eqref{eq:modal_control_cost} subject to the affine endpoint constraint.

Both constructions now solve the same affine initial-value problem with the same control and source. By uniqueness of that initial-value problem, we conclude that their state trajectories agree. Taking \(s=G_P(a)\), we establish equality of the proposal maps.
\end{proof}

\begin{remark}
Proposition~\ref{prop:hum_gramian_equivalence} concerns the continuous-time finite-dimensional affine problem. It does not, by itself, imply the equality of numerical implementations that use different control representations, quadrature rules, or time integrators. An analogous discrete equivalence requires the HUM and Gramian constructions to represent the same discrete endpoint operator and control inner product.

In computation, errors in the Gramian, source integral, linear solve, control representation, and forward integration can produce a nonzero affine endpoint defect. Moreover, the affine proposal solves the equation with the frozen source, so its endpoint accuracy alone does not establish nonlinear endpoint accuracy. Therefore, the returned control must undergo the separate nonlinear forward validation specified in the computational protocol. Refinement studies assess numerical sensitivity. A rigorous endpoint certificate additionally requires validated error bounds.
\end{remark}

\begin{corollary}
Under the global or invariant-ball contraction hypotheses of Proposition~\ref{prop:sf_picard_conditional_convergence}, the exact Picard--Gramian iteration has the same contraction and geometric convergence estimates as the source-frozen Picard--HUM algorithm, in the same trajectory norm and on the same domain. For \(0<\eta\le1\), the relaxed map has Lipschitz constant at most \(q_\eta=1-\eta+\eta q<1\), where \(q<1\) is the corresponding contraction bound.

Separately, if the common proposal map is continuously Fr\'echet differentiable near a fixed point and satisfies the local derivative condition \eqref{eq:picard_local_relaxed_condition}, then the exact relaxed Picard--Gramian iteration converges locally to that fixed point.

With the same initial trajectory and relaxation parameter, the two exact algorithms produce identical proposal controls, proposal trajectories, and relaxed iterates for as long as both are continued.
\end{corollary}

\begin{proof}
By Proposition~\ref{prop:hum_gramian_equivalence}, we obtain equality of the proposal maps and hence of their relaxed maps. We then apply Proposition~\ref{prop:sf_picard_conditional_convergence} under the respective hypotheses to obtain the contraction and local convergence conclusions. Starting from the same initial trajectory, we use the equality of the proposal controls and trajectories at each iteration, followed by the same relaxation update, to prove equality of the iterates by induction.
\end{proof}

\begin{remark}
We consider \(g(u)=\pm\beta u^3\), with \(\beta\ge0\). For \(|u|,|v|\le R_u\), we use factorization to obtain
\[
|g(u)-g(v)|
=\beta|u-v|\,|u^2+uv+v^2|
\le3\beta R_u^2|u-v|.
\]

To apply this estimate on the ball \(\mathbb B_R\subset\mathcal X_P\), we suppose that
\[
\left\|\sum_{j=1}^P a_j(t)e_j(x)\right\|_
{L^\infty((0,T)\times\Omega)}
\le R_u
\qquad\text{for every }a\in\mathbb B_R.
\]
For \(a,b\in\mathbb B_R\), we write \(u_a(t,x)=\sum_{j=1}^P a_j(t)e_j(x)\) and define \(u_b\) similarly. By Bessel's inequality and orthonormality of the selected eigenfunctions in \(L^2(\Omega)\), we obtain
\[
\begin{aligned}
|G_P(a(t))-G_P(b(t))|
&\le
\|g(u_a(t,\cdot))-g(u_b(t,\cdot))\|_{L^2(\Omega)}\\
&\le3\beta R_u^2
\|u_a(t,\cdot)-u_b(t,\cdot)\|_{L^2(\Omega)}\\
&=3\beta R_u^2|a(t)-b(t)|.
\end{aligned}
\]
Integrating in time, we obtain
\[
\|G_P(a)-G_P(b)\|_
{L^2(0,T;\mathbb R^P)}
\le
3\beta R_u^2\sqrt T\,
\|a-b\|_{\mathcal X_P}.
\]
Thus, one may take \(L_{P,R}=3\beta R_u^2\sqrt T\) in \eqref{eq:projected_nonlinearity_lipschitz}. The contraction theorem additionally requires the invariance of \(\mathbb B_R\) and
\[
C_{P,T}^{\mathrm{aff}}\,
3\beta R_u^2\sqrt T<1.
\]

For an individual bounded trajectory \(u_P\), the quantity
\[
L_{\mathrm{eff}}
=
3\beta\|u_P\|_{L^\infty((0,T)\times\Omega)}^2
\]
measures the largest magnitude of the scalar derivative \(g'(u_P)\) along that trajectory. A maximum computed from finitely many samples approximates this quantity but does not, without an additional enclosure, bound the supremum over the whole space-time domain. Even an exact value for one trajectory does not supply a uniform amplitude bound for all trajectories in \(\mathbb B_R\).

Local convergence of the relaxed iteration follows under the differentiability and derivative-bound hypotheses associated with \eqref{eq:picard_local_relaxed_condition}. The scalar diagnostic \(L_{\mathrm{eff}}\) alone does not verify those hypotheses.
\end{remark}

\begin{remark}
Both realizations use the continuous-time Galerkin system obtained from the prescribed spatial eigenfunctions. This preserves their spatial eigenstructure but does not, by itself, remove errors introduced by time discretization.

For example, we consider an oscillatory mode satisfying \(z'=-i\omega z\), with frequency \(\omega>0\) and time step \(\Delta t>0\), and we set \(\theta=\omega\Delta t\). Here, \(i\) is the imaginary unit and \(\omega\) denotes the frequency only in this example, not the control region. Over one time step, the exact multiplier is \(e^{-i\theta}\), whereas explicit Euler, implicit Euler, and Crank--Nicolson give, respectively,
\[
1-i\theta,
\qquad
\dfrac{1}{1+i\theta},
\qquad
\dfrac{1-i\theta/2}{1+i\theta/2}
=
e^{-2i\arctan(\theta/2)}.
\]
The first two multipliers have moduli \(\sqrt{1+\theta^2}>1\) and \((1+\theta^2)^{-1/2}<1\). Crank--Nicolson preserves the modulus, but its numerical frequency is
\[
\widetilde\omega
=
\dfrac{2}{\Delta t}
\arctan\left(\dfrac{\omega\Delta t}{2}\right),
\]
which differs from \(\omega\). Thus, preserving modal amplitudes does not ensure accurate modal phases.

Using the exact linear propagator \(e^{\mathsf A_Pt}\) avoids these particular time-stepping multiplier errors in the homogeneous Galerkin dynamics. Numerical evaluation of the propagator, source and control integrals, and nonlinear trajectories still requires error assessment.

For finite-difference wave approximations, high-frequency spurious solutions can prevent observability estimates from remaining uniform under mesh refinement; see \cite[abstract]{Zuazua2005}. Controllability of each fixed finite-mode system does not, by itself, establish uniform control bounds or convergence of the resulting controls to a PDE control. Here, truncation studies and independent forward integration assess numerical accuracy, while the validated enclosures establish the stated rigorous claims for the Galerkin models.
\end{remark}



\section{Neural--Spectral candidate record}

\subsection{One-dimensional tanh candidates}

For the one-dimensional benchmark with \(g(u)=\beta\tanh u\), we report the trained candidates at \(P=8\) for \(\beta\in\{1,6\}\). The endpoint errors reported below come from independent forward integrations of the extracted controls, rather than from the training loss. We then apply the endpoint Newton--HUM corrector to the extracted control coefficients. Each correction uses the time-dependent linearized Gramian. We report candidate construction and endpoint correction as separate computational stages.

The control has eight modal coefficients throughout the validation study. For \(Q\in\{8,32,64\}\), let \(a_Q\) denote the trajectory obtained by applying this same localized control to the \(Q\)-mode nonlinear system:
\[
a_Q''+\Lambda_Qa_Q+G_Q(a_Q)+M_{Q,8}c=0,
\qquad
(M_{Q,8})_{jk}
=\int_\omega e_j(x)e_k(x)\,\mathrm{d}x,
\]
where \(1\le j\le Q\) and \(1\le k\le8\). We use the projections of the same physical initial and terminal data at every validation level. The endpoint metric is
\[
E_Q(c)
=
\left(
\sum_{j=1}^Q
\left[
(\mu_j+\alpha)
|a_{Q,j}(T)-a_{T,Q,j}|^2
+
|a'_{Q,j}(T)-a'_{T,Q,j}|^2
\right]
\right)^{1/2}.
\]
Thus, changing \(Q\) changes the forward validator, while the applied control remains fixed.

The physical norm column refers to the corrected control and reports
\[
\|\mathbf 1_\omega f_8^{\mathrm{corr}}\|_
{L^2((0,T)\times\Omega)}
=
\left(
\int_0^T
(c^{\mathrm{corr}}(t))^\top
M c^{\mathrm{corr}}(t)\,\mathrm{d}t
\right)^{1/2}.
\]
Here, \(f_8^{\mathrm{corr}} =\sum_{j=1}^8c_j^{\mathrm{corr}}e_j\) is the globally defined modal field.

Table~\ref{supp:tab-neural-1d} reports the neural candidates, their corrected endpoint errors, and the corrected physical control norms.
\begin{table}[htbp]
\centering
\caption{One-dimensional Neural--Spectral method candidates and Newton--HUM method corrections. We evaluate endpoint errors by separate nonlinear forward integration. The final column records only the numerical eight-mode endpoint criterion. The richer errors and physical norms refer to the corrected controls.}
\label{supp:tab-neural-1d}
\resizebox{\textwidth}{!}{%
\begin{tabular}{c r c c c c c c c}
\toprule
\(\beta\) & seed & raw \(E_8\) & corrected \(E_8\)
& corrected \(E_{32}\) & corrected \(E_{64}\)
& corrected physical norm & accepted Newton--HUM steps
& \(E_8\le10^{-8}\)\\
\midrule
1 & 0 & \(2.97\times10^{-1}\) & \(2.04\times10^{-9}\) & \(1.42\times10^{-1}\) & \(1.43\times10^{-1}\) & 16.024 & 2 & yes\\
1 & 10 & \(4.01\times10^{-1}\) & \(6.87\times10^{-9}\) & \(1.42\times10^{-1}\) & \(1.42\times10^{-1}\) & 10.427 & 2 & yes\\
1 & 20 & \(2.31\times10^{-1}\) & \(3.02\times10^{-9}\) & \(1.48\times10^{-1}\) & \(1.48\times10^{-1}\) & 10.874 & 2 & yes\\
6 & 0 & \(4.38\times10^{-1}\) & \(9.53\times10^{-9}\) & \(9.66\times10^{-2}\) & \(9.71\times10^{-2}\) & 14.727 & 2 & yes\\
6 & 10 & \(4.93\times10^{-1}\) & \(9.94\times10^{-9}\) & \(1.52\times10^{-1}\) & \(1.53\times10^{-1}\) & 13.274 & 2 & yes\\
6 & 20 & \(2.37\times10^{-1}\) & \(4.72\times10^{-9}\) & \(1.26\times10^{-1}\) & \(1.27\times10^{-1}\) & 11.527 & 2 & yes\\
\bottomrule
\end{tabular}%
}
\end{table}

Figure~\ref{supp:fig-neural-1d} compares the independently integrated endpoint errors before and after Newton--HUM correction.
\begin{figure}[htbp]
\centering
\includegraphics[width=.78\textwidth]{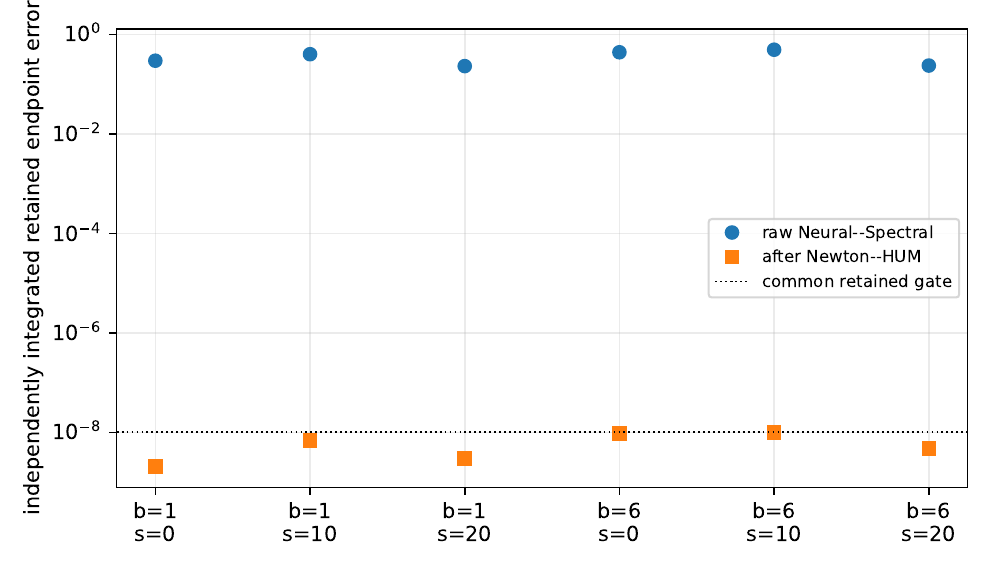}
\caption{Independently integrated raw Neural--Spectral endpoint errors and the corresponding Newton--HUM completions. The plot keeps candidate construction and deterministic completion visible as distinct computational stages.}
\label{supp:fig-neural-1d}
\end{figure}

Across the six reported runs, the raw eight-mode endpoint errors range from \(0.231\) to \(0.493\). Two accepted Newton--HUM steps reduce these errors to values between \(2.04\times10^{-9}\) and \(9.94\times10^{-9}\). Thus, every corrected candidate satisfies the numerical criterion \(E_8\le10^{-8}\).

The corrected richer errors range from \(9.66\times10^{-2}\) to \(1.52\times10^{-1}\) at \(Q=32\), and from \(9.71\times10^{-2}\) to \(1.53\times10^{-1}\) at \(Q=64\). These values quantify the remaining endpoint error when the same controls are transferred to larger modal systems. In particular, satisfying the eight-mode criterion does not imply comparable accuracy in those systems.

The six runs demonstrate successful endpoint correction of the reported neural candidates. They do not establish a two-step convergence guarantee for other seeds or parameter values. The main article reports training and correction times separately and compares the corrected physical norms with those of the deterministic reference controls.

For the corrected \(P=8\), \(\beta=1\), seed-\(0\) control \(\bar c\), we use the validated endpoint calculation reported in Section~\ref{supp:indicator-record} to obtain \(\bar h\le0.148148<1\) on the \(500\)-cell proof partition. We apply the corresponding certificate to establish the existence of an exact steering control \(c^*\) for the eight-mode nonlinear system satisfying \(\|c^*-\bar c\|_{L^2(0,T;\mathbb R^8)}\le0.056815\). Refining the proof partition to \(4000\) cells while keeping \(\bar c\) fixed, we obtain the reported bounds \(\bar h\le0.015367\) and \(\|c^*-\bar c\|_{L^2(0,T;\mathbb R^8)}\le0.006253\). These radii use the modal coefficient norm. The certificate concerns an exact control near the stored candidate for the finite-dimensional system. It does not assert that the stored candidate itself steers exactly or certify a full-PDE control. Among the six neural runs reported here, this is the candidate for which we report an endpoint certificate.

\subsection{Two-dimensional neural formulations under common physical metrics}

The two-dimensional benchmark uses the linear wave equation with \(\alpha=0\) and \(g=0\) on \(\Omega=(0,1)^2\), with homogeneous Dirichlet boundary conditions, final time \(T=7\), and a boundary collar of width \(0.3\). The initial data are
\[
u(0,x,y)
=
\sin(\pi x)\sin(\pi y)
+\dfrac{1}{2}\sin(\pi x)\sin(2\pi y),
\qquad
u_t(0,x,y)=0,
\]
and the target is the rest state. We construct modal controls in the \(4\times4\) Galerkin space and evaluate them in both the \(4\times4\) system and an \(8\times8\) forward validator. 

\subsubsection{Exact collar geometry and weighted Gramian references}

Let
\[
e_{jk}(x,y)=2\sin(j\pi x)\sin(k\pi y),
\qquad 1\le j,k\le4.
\]
These functions are orthonormal in \(L^2(\Omega)\). We order the pairs lexicographically, with \(j\) as the outer index and \(k\) as the inner index. For \(a(t),c(t)\in\mathbb R^{16}\), we write
\[
u_{16}(t,x,y)
=\sum_{j,k=1}^4 a_{jk}(t)e_{jk}(x,y),
\qquad
f_c(t,x,y)
=\sum_{j,k=1}^4 c_{jk}(t)e_{jk}(x,y).
\]
Here, we define \(f_c\) on all of \(\Omega\), and the applied control is \(\mathbf 1_{\omega_\delta}f_c\).

For \(0<\delta<1/2\), we define
\[
\omega_\delta
=
\left\{(x,y)\in(0,1)^2:
\min(x,1-x,y,1-y)<\delta\right\}.
\]
Its complement is the central square \([\delta,1-\delta]^2\), up to sets of measure zero. Thus, we apply separation of variables to obtain
\[
\begin{aligned}
M_{(j,k),(\ell,m)}
=\int_{\omega_\delta}
  e_{jk}(x,y)e_{\ell m}(x,y)\,\mathrm{d}x\,\mathrm{d}y=\delta_{j\ell}\delta_{km}-O_{j\ell}O_{km},
\end{aligned}
\]
where \(\delta_{j\ell}\) denotes the Kronecker delta. We write \(\otimes\) for the Kronecker product. With the
specified lexicographic ordering, its entries satisfy
\[
 (O\otimes O)_{(j,k),(\ell,m)}
 =O_{j\ell}O_{km},
 \qquad j,k,\ell,m\in\{1,2,3,4\}.
\]
Consequently,
\[
M=I_{16}-O\otimes O,
\qquad
O_{j\ell}
=
\int_\delta^{1-\delta}
2\sin(j\pi x)\sin(\ell\pi x)\,\mathrm{d}x.
\]
The benchmark uses \(\delta=0.3\). We can evaluate the entries explicitly:
\[
O_{j\ell}
=
\begin{cases}
\displaystyle
\left[
\dfrac{\sin((j-\ell)\pi x)}{(j-\ell)\pi}
-
\dfrac{\sin((j+\ell)\pi x)}{(j+\ell)\pi}
\right]_\delta^{1-\delta},
& j\ne\ell,\\[3ex]
\displaystyle
1-2\delta+\dfrac{\sin(2j\pi\delta)}{j\pi},
& j=\ell.
\end{cases}
\]
These expressions avoid spatial quadrature in the assembly of \(M\). Their numerical evaluation still incurs floating-point error. If \(a\) is the Galerkin displacement vector and \(c\) is the coefficient vector of \(f_c\), then
\[
 a''+\Lambda_{16}a+Mc=0,
 \qquad
 \Lambda_{16}=\operatorname{diag}\{\pi^2(j^2+k^2)\}_{j,k=1}^4.
\]
For \(c\in L^2(0,T;\mathbb R^{16})\), orthonormality on \(\Omega\) and the definition of \(M\) give
\[
\begin{aligned}
\|f_c\|_{L^2((0,T)\times\Omega)}^2
&=\int_0^T c(t)^\top c(t)\,\mathrm{d}t,\\
\|\mathbf 1_{\omega_\delta}f_c\|_
{L^2((0,T)\times\Omega)}^2
&=\int_0^T c(t)^\top M c(t)\,\mathrm{d}t.
\end{aligned}
\]
Thus, the modal-coefficient cost measures the norm of the global modal field, whereas the physical cost measures the norm of the localized applied field.

We use the unscaled state \(Y=(a^\top,(a')^\top)^\top\in\mathbb R^{32}\), and set
\[
\mathsf A
=
\begin{pmatrix}
0&I_{16}\\
-\Lambda_{16}&0
\end{pmatrix},
\qquad
\mathsf B
=
\begin{pmatrix}
0\\-M
\end{pmatrix}.
\]
For the benchmark data, the only nonzero initial displacement coefficients are \(a_{11}(0)=1/2\) and \(a_{12}(0)=1/4\). The initial velocity vector and terminal state are \(\mathbf 0\). For a symmetric positive-definite weight \(\Sigma\), we define
\begin{equation*}
 W_{T,\Sigma}=\int_0^T e^{\mathsf A(T-s)}\mathsf B \Sigma^{-1}\mathsf B^\top
 e^{\mathsf A^\top(T-s)}\,\mathrm{d}s.
\end{equation*}
We set \(d=Y_T-e^{\mathsf AT}Y_0\). The minimum-\(\Sigma\)-cost control is
\begin{equation}\label{supp:eq:weighted-control}
 c_\Sigma^*(t)=\Sigma^{-1}\mathsf B^\top e^{\mathsf A^\top(T-t)}W_{T,\Sigma}^{-1}d.
\end{equation}

\begin{proposition}
Let \(0<\delta<1/2\), \(T>0\), and \(Y_0,Y_T\in\mathbb R^{32}\). For either \(\Sigma=I_{16}\) or \(\Sigma=M\), the control \eqref{supp:eq:weighted-control} is the unique minimizer of
\[
J_\Sigma(c)=\int_0^T c(t)^\top \Sigma c(t)\,\mathrm{d}t
\]
among \(c\in L^2(0,T;\mathbb R^{16})\) satisfying the projected endpoint constraint. The choices \(\Sigma=I_{16}\) and \(\Sigma=M\) minimize, respectively, the modal coefficient norm and the physical applied-control norm in this control class.
\end{proposition}

\begin{proof}
We apply the weighted-control construction of Section~\ref{sec:picard_gramian}. The selected sine modes are orthonormal Dirichlet modes on \((0,1)^2\), and \(0<\delta<1/2\) makes \(\omega_\delta\) a nonempty open subset. Lemma~\ref{lem:retained_pair_controllable} therefore makes \(M\) positive definite and \((\mathsf A,\mathsf B)\) controllable. For either weight \(\Sigma=I_{16}\) or \(\Sigma=M\), the invertible change of control \(v=\Sigma^{1/2}c\) preserves the reachable endpoints, so \(W_{T,\Sigma}\) is positive definite.

The weighted-control formula \eqref{supp:eq:weighted-control} is the formula we proved in Section~\ref{sec:picard_gramian}, with \(P=16\), control matrix \(\mathsf B\), and zero frozen source. We have verified its hypotheses, so we apply it to obtain the unique minimizer of \(J_\Sigma\). For \(\Sigma=M\), the physical norm identity \eqref{eq:physical_control_cost} lets us identify \(J_\Sigma\) with the applied-control energy. For \(\Sigma=I_{16}\), we identify it with the modal coefficient energy.
\end{proof}

\subsubsection{Richer validation and numerical evaluation}
For a modal control with coefficient vector \(c(t)\in\mathbb R^{16}\), the richer validator applies the same localized field \(\mathbf 1_{\omega_\delta}f_c\) without retuning. Extend the definition of \(e_{jk}\) to \(1\le j,k\le8\), and define the rectangular overlap matrix by
\[
(M_{64,16})_{(j,k),(\ell,m)}
=
\int_{\omega_\delta}
e_{jk}(x,y)e_{\ell m}(x,y)\,
\mathrm{d}x\,\mathrm{d}y,
\]
where \(1\le j,k\le8\) and \(1\le\ell,m\le4\). Thus, \(M_{64,16}\in\mathbb R^{64\times16}\). The richer trajectory satisfies
\[
a_{64}''+\Lambda_{64}a_{64}+M_{64,16}c=0,
\qquad
\Lambda_{64}
=
\operatorname{diag}
\left\{\pi^2(j^2+k^2)\right\}_{j,k=1}^8,
\]
with initial data obtained by projecting the prescribed physical initial state onto these modes. Localization can produce nonzero forcing coefficients outside the selected \(4\times4\) set.

For comparison with the modal controls, we also consider the vanilla and hard-boundary physics-informed neural networks (PINNs), which represent the state and control directly as functions of \((t,x,y)\). We refer to these two formulations collectively as field PINNs. The hard-boundary formulation enforces the homogeneous Dirichlet condition through the state parametrization, whereas the vanilla formulation uses a boundary penalty in the training loss. For a control \(f_\theta(t,x,y)\) produced by either formulation, the forcing coefficient of each validation mode is evaluated directly as
\[
d^\theta_{jk}(t)
=
\int_{\omega_\delta}
f_\theta(t,x,y)e_{jk}(x,y)\,
\mathrm{d}x\,\mathrm{d}y.
\]
Both constructions therefore project the applied field onto every validation mode. For the field PINN methods, we do not first replace \(f_\theta\) by its projection onto the \(4\times4\) modal space. We compute their physical control norm from
\[
\|\mathbf 1_{\omega_\delta}f_\theta\|_
{L^2((0,T)\times\Omega)}^2
=
\int_0^T\int_{\omega_\delta}
|f_\theta(t,x,y)|^2\,\mathrm{d}x\,\mathrm{d}y\,\mathrm{d}t.
\]

For each neural run, we evaluate the extracted control after optimization using separate \(4\times4\) and richer forward integrations. The reported endpoint errors are numerical evaluations, not rigorous endpoint-error enclosures.

For the zero target, we define the endpoint energy norm at validation level \(K\times K\) by
\[
E_{K\times K}
=
\left(
\sum_{j,k=1}^K
\left[
\pi^2(j^2+k^2)|a_{jk}(T)|^2
+
|a'_{jk}(T)|^2
\right]
\right)^{1/2}.
\]
The uncontrolled linear dynamics conserve this norm. Since the only nonzero initial displacement coefficients are \(a_{11}(0)=1/2\) and \(a_{12}(0)=1/4\), and the initial velocity vanishes, it holds that
\[
E_{\mathrm{free}}^2
=
2\pi^2\left(\dfrac{1}{2}\right)^2
+
5\pi^2\left(\dfrac{1}{4}\right)^2
=
\dfrac{13\pi^2}{16}.
\]
Hence,
\[
E_{\mathrm{free}}
=
\dfrac{\pi\sqrt{13}}4
\approx2.8318.
\]
This baseline is the same for the \(4\times4\) and \(8\times8\) validators because both contain every initially excited mode.

Table~\ref{supp:tab-neural-2d-seeds} reports the individual runs in the two-dimensional neural benchmark.
\begin{table}[htbp]
\centering
\caption{Seed-by-seed results for the two-dimensional neural benchmark. We compute endpoint errors by separate forward integration of the extracted controls. The physical norm measures the localized applied field. The time column reports wall-clock time for training and validation, and the parameter column reports the number of trainable parameters in each formulation.}
\label{supp:tab-neural-2d-seeds}
\resizebox{\textwidth}{!}{\begin{tabular}{l r c c c c c}
\toprule
method & seed & parameters & \(E_{4\times4}\) & \(E_{8\times8}\) & physical norm & wall (s)\\
\midrule
free evolution & -- & -- & \(2.8318\) & \(2.8318\) & 0 & --\\
hard-boundary PINN method & 0 & 10242 & \(2.83\times10^{0}\) & \(2.83\times10^{0}\) & 0.005 & 190.4\\
hard-boundary PINN method & 1 & 10242 & \(2.83\times10^{0}\) & \(2.83\times10^{0}\) & 0.004 & 176.9\\
hard-boundary PINN method & 2 & 10242 & \(2.83\times10^{0}\) & \(2.83\times10^{0}\) & 0.018 & 137.7\\
Neural--Spectral method & 0 & 6144 & \(8.73\times10^{-1}\) & \(8.95\times10^{-1}\) & 17.516 & 2030.3\\
Neural--Spectral method & 1 & 6144 & \(1.01\times10^{0}\) & \(1.03\times10^{0}\) & 18.250 & 2056.4\\
Neural--Spectral method & 2 & 6144 & \(8.64\times10^{-1}\) & \(8.94\times10^{-1}\) & 16.520 & 2039.8\\
vanilla PINN method & 0 & 6962 & \(2.80\times10^{0}\) & \(2.80\times10^{0}\) & 0.221 & 192.4\\
\bottomrule
\end{tabular}
}
\end{table}

Figure~\ref{supp:fig-neural-2d} compares the richer endpoint errors and physical applied-control norms of the neural candidates and weighted Gramian references.
\begin{figure}[htbp]
\centering
\includegraphics[width=.76\textwidth]{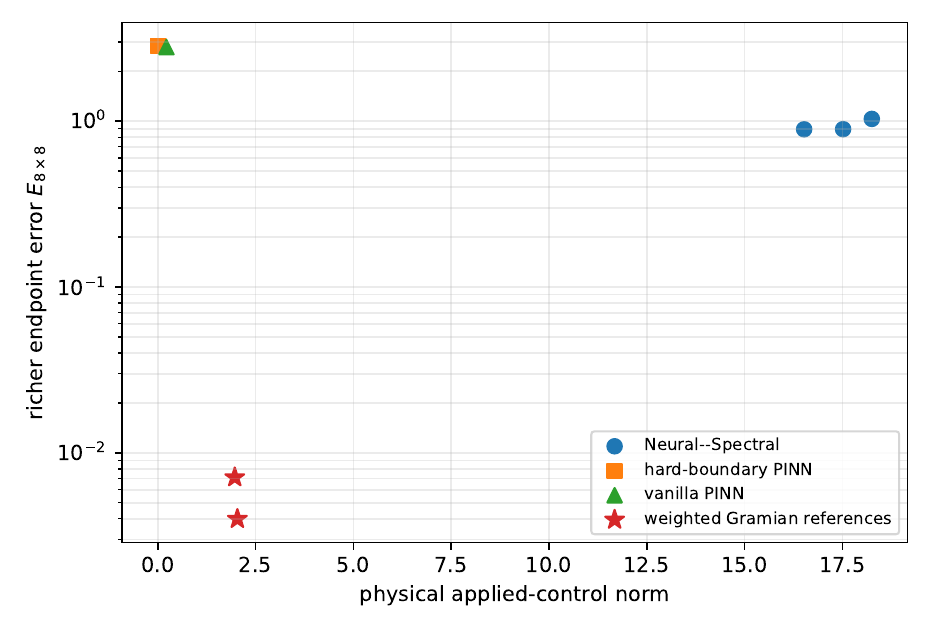}
\caption{Richer endpoint error versus physical applied-control norm for all frozen neural runs and the two weighted Gramian references.}
\label{supp:fig-neural-2d}
\end{figure}

Under our architectures and training budgets, the Neural--Spectral method runs attain smaller endpoint errors than the reported field-PINN method runs. Their median richer error is \(0.895\), approximately \(31.6\%\) of \(E_{\mathrm{free}}\), and their median physical control norm is \(17.516\). The field-PINN method endpoint errors remain close to the free-evolution baseline, with substantially smaller reported control norms. The final stored training logs report mean squared PDE residuals at the \(9\times9\times9\) interior collocation points ranging approximately from \(3.96\times10^{-8}\) to \(9.88\times10^{-7}\). These collocation residuals alone do not establish convergence of the PINN states to a solution of the controlled wave equation.

The numerical weighted-Gramian references reported in the main article attain richer errors of approximately \(4.00\times10^{-3}\) and \(7.12\times10^{-3}\), with physical norms \(2.035\) and \(1.968\), respectively. Thus, these deterministic references attain both smaller endpoint errors and smaller physical norms than the reported Neural--Spectral method controls.

These comparisons concern the tested formulations, seeds, and optimization budgets. The formulations also differ in parameterization, loss construction, and training cost, so the experiment does not isolate the effect of modal structure alone. We use the richer validator as a finite-dimensional refinement diagnostic, but we do not obtain a rigorous bound on the full-PDE endpoint error.

\section{Focusing benchmark and branch record}
We consider the focusing nonlinearity \[g_\beta(u)=-\beta u^3\] on \(\Omega=(0,1)\), with \(\alpha=0\), final time \(T=2.5\), and actuator region
\[
\omega=(0,0.3)\cup(0.7,1).
\]
For each \(P\in\{8,12,16\}\), the initial modal data are
\[
a(0)=2^{-1/2}(1,0.5,-0.3,0,\ldots,0)^\top,
\qquad
a'(0)=\mathbf 0,
\]
and the target is rest.

Starting from the linear problem at \(\beta=0\), the endpoint Newton--HUM procedure uses the coefficient schedule
\[
1,2,\ldots,10,\;10.1,10.2,\ldots,11.0.
\]
This schedule contains \(20\) positive-coefficient continuation stages. The linear initialization is not one of those stages. Before advancing, each stage must produce a successful forward integration through \(T\) and satisfy the numerical \(P\)-mode endpoint criterion \(E_P\le10^{-8}\).

At the final coefficient \(\beta=11\), stabilized source freezing tests the ordered relaxation ladder \(\eta\in\{1,0.75,0.5,0.25,0.125,0.0625,0.03125\}\), with at most \(40\) Picard iterations per relaxation value and earlier termination under the prescribed stopping rules.

The best proposal residual is the smallest finite numerical value recorded at termination of these attempts for the unrelaxed quantity
\[
r_{\mathrm{fp}}^{(n)}
=
\|\widehat a^{(n+1)}-a^{(n)}\|_
{L^2(0,T;\mathbb R^P)}.
\]
We use this statistic to measure the consistency of the source-frozen update. It is distinct from the endpoint error we obtain by applying an extracted control to the nonlinear system.

At each reported refinement, the source-freezing candidate submitted to nonlinear forward validation fails to produce an accepted finite endpoint value. Accordingly, we report no numerical value of \(E_P\) for that candidate. This computational outcome does not establish finite-time blow-up of the exact controlled solution.

Continuation followed by the endpoint Newton--HUM method reaches the numerical endpoint criterion at all three tested truncation levels. The main article additionally reports richer errors \(E_{2P}\) of \(4.09\times10^{-2}\), \(2.06\times10^{-2}\), and \(1.93\times10^{-2}\) for \(P=8,12,16\), respectively. These values assess the transfer of each computed control to a larger modal system.

The results support the effectiveness of this continuation protocol for the reported benchmark. They do not establish convergence uniformly in \(P\), a rigorous continuation branch, or a full-PDE steering certificate. We report the focusing computations as numerical results, separately from the validated endpoint certificates for the tanh examples.

In Table~\ref{supp:tab-focusing-record}, we report the focusing benchmark results at \(\beta=11\).
\begin{table}
\centering
\caption{Focusing benchmark at the final coefficient \(\beta=11\). The rows for the Newton--HUM method use \(20\) continuation stages, while the rows for the source-freezing method test seven relaxation values. We take the best proposal residual over the recorded source-freezing attempts. ``Not finite'' means that nonlinear forward validation did not produce an accepted finite endpoint value. It is not a numerical estimate of \(E_P\).}
\label{supp:tab-focusing-record}
\resizebox{\textwidth}{!}{\begin{tabular}{c l c c c c c}
\toprule
\(P\) & method & stages/attempts & reported iterations & best proposal residual
& final \(E_P\) & time (s)\\
\midrule
8 & least-squares/Newton--HUM method & 20 continuation stages & 62 & not applicable & \(5.321\times10^{-9}\) & 62.40\\
8 & stabilized spectral source-freezing method & 7 relaxation attempts & 230 & \(5.874\times10^{-2}\) & not finite & 63.53\\
12 & least-squares/Newton--HUM method & 20 continuation stages & 61 & not applicable & \(3.071\times10^{-9}\) & 95.41\\
12 & stabilized spectral source-freezing method & 7 relaxation attempts & 231 & \(5.608\times10^{-2}\) & not finite & 113.87\\
16 & least-squares/Newton--HUM method & 20 continuation stages & 61 & not applicable & \(2.868\times10^{-9}\) & 137.57\\
16 & stabilized spectral source-freezing method & 7 relaxation attempts & 231 & \(5.441\times10^{-2}\) & not finite & 164.53\\
\bottomrule
\end{tabular}
}
\end{table}

\section{Validated endpoint evidence record}
\label{supp:indicator-record}
This section records the endpoint-validation quantities for three controls with \(P=8\), \(T=2.5\), and \(g(u)=\tanh u\): a control from the least-squares/Newton--HUM method, a control from the stabilized source-freezing method followed by one Newton--HUM step, and a Neural--Spectral method control corrected by the Newton--HUM method. These endpoint certificates concern the eight-mode nonlinear system.

For each fixed stored control \(c\), let
\[
Y_c(t)
=
\begin{pmatrix}
\Lambda_8^{1/2}a_c(t)\\
a_c'(t)
\end{pmatrix}
\]
denote its exact state in energy coordinates. In this section and the enclosure calculations below, \(Y_c\) denotes this scaled state, rather than the unscaled state used in the affine steering formulas. We set \(Y_T^E=\mathsf E_8Y_T\) for the corresponding target. For $0\le s\le t\le T$, set
\[
 \Xi_c(t,s):=\mathsf E_8\Phi_c(t,s)\mathsf E_8^{-1},
 \qquad
 B_E:=\mathsf E_8\mathsf B_8
 =\begin{pmatrix}\mathbf 0\\-M_8\end{pmatrix},
\]
where \(\Phi_c\) is the unscaled transition matrix defined by \eqref{eq:unscaled-variational-transition}. We define
\[
 W_c(t)
 :=
 \int_0^t
 \Xi_c(t,s)B_EB_E^\top\Xi_c(t,s)^\top
 \,\mathrm{d}s,
 \qquad t\in[0,T].
\]
We write \(\widehat Y\) and \(\widehat W\) for their numerical Hermite representations.

For a proof partition with \(N\) cells, the accumulated bounds satisfy
\[
\sup_{0\le t\le T}
\|Y_c(t)-\widehat Y(t)\|_2
\le b_N^Y,
\qquad
\|W_c(T)-\widehat W(T)\|_F
\le b_N^W.
\]
The state and Gramian residual columns report \(\max_i\rho_{Y,i}\) and \(\max_i\rho_{W,i}\), where the cellwise bounds use the Euclidean and Frobenius norms, respectively. These residual bounds are distinct from the accumulated solution-error bounds \(b_N^Y\) and \(b_N^W\).

We set \(\widehat W_{\mathrm{sym}}(T)= \dfrac{1}{2}\left(\widehat W(T)+\widehat W(T)^\top\right)\). Let $\underline\lambda_{\mathrm{chol}}>0$ be a shift
for which an interval Cholesky calculation certifies
\[
 \widehat W_{\mathrm{sym}}(T)
 -\underline\lambda_{\mathrm{chol}}I_{16}
 \succ0.
\]
Here, $\succ0$ denotes positive definiteness, and the
verification includes rounding errors in forming the
shifted symmetric matrix.
Consequently,
\[
 \lambda_{\min}\left(\widehat W_{\mathrm{sym}}(T)\right)
 \ge\underline\lambda_{\mathrm{chol}}.
\]
Since symmetrization does not increase the Frobenius norm, we apply Weyl's inequality to obtain
\[
\lambda_{\min}\bigl(W_c(T)\bigr)
\ge
\underline\lambda_{\mathrm{chol}}-b_N^W.
\]
For an exact real expression $x$, let
$\operatorname{up}(x)$ and $\operatorname{down}(x)$
denote certified upper and lower evaluations,
respectively, including all intermediate rounding
and function-evaluation errors.
We define
\[
\begin{aligned}
 \underline\lambda
 &:=
 \operatorname{down}\left(
 \underline\lambda_{\mathrm{chol}}-b_N^W
 \right),
 \\
 \bar\varepsilon
 &:=
 \operatorname{up}\left(
 \|\widehat Y(T)-Y_T^E\|_2+b_N^Y
 \right).
\end{aligned}
\]
These definitions give
\[
 \lambda_{\min}\left(W_c(T)\right)
 \ge\underline\lambda,
 \qquad
 \|\mathfrak F_8(c)\|_2\le\bar\varepsilon.
\]

Table~\ref{supp:tab-endpoint-evidence} reports the endpoint validation bounds for the three stored controls.
\begin{table}[htbp]
\centering
\caption{Reported endpoint-validation bounds for \(P=8\), \(\beta=1\), \(T=2.5\), using \(N=500\) time cells and \(4096\) spatial subintervals. The control-distance bounds use the modal \(L^2(0,T;\mathbb R^8)\) norm. Displayed decimals summarize the validation output. The certificate uses the full-precision outward bounds.}
\label{supp:tab-endpoint-evidence}
\resizebox{\textwidth}{!}{\begin{tabular}{l c c c c c}
\toprule
candidate & \(b_N^Y\) & \(b_N^W\)
& \(\max_i\rho_{Y,i}\) & \(\max_i\rho_{W,i}\)
& control-distance bound\\
\midrule
least-squares/Newton--HUM method & \(4.478\times10^{-2}\) & \(8.716\times10^{-2}\) & \(2.461\times10^{-2}\) & \(1.679\times10^{-2}\) & 0.111927\\
stabilized spectral source-freezing method + Newton--HUM method & \(4.494\times10^{-2}\) & \(8.743\times10^{-2}\) & \(2.461\times10^{-2}\) & \(1.683\times10^{-2}\) & 0.112439\\
Neural--Spectral method + Newton--HUM method & \(2.589\times10^{-2}\) & \(5.364\times10^{-2}\) & \(2.124\times10^{-2}\) & \(9.815\times10^{-3}\) & 0.056815\\
\bottomrule
\end{tabular}
}
\end{table}

If $\underline\lambda>0$, we define
\[
 \bar\gamma
 :=
 \operatorname{up}\left(\underline\lambda^{-1/2}\right).
\]
This quantity bounds the operator norm of the minimum-modal-norm right inverse of $D\mathfrak F_8(c)$.

For the present parameters $P=8$, $T=2.5$, and $g(u)=\tanh u$, we set
\[
 \bar L
 :=
 \operatorname{up}\left(L_8^{\mathrm{end}}\right),
 \qquad
 \bar h
 :=
 \operatorname{up}\left(
 2\bar L\,\bar\gamma^2\bar\varepsilon
 \right),
\]
where $L_8^{\mathrm{end}}$ is defined in \eqref{supp:eq:L-end-definition}. Its evaluation includes certified enclosures of all constants and elementary functions.

If $\bar h<1$, we compute
\[
\begin{aligned}
 d_{\mathrm{low}}
 &:=
 \operatorname{down}\left(
 1+\sqrt{1-\bar h}
 \right)>0,
 \\
 \rho_{\mathrm{cert}}
 &:=
 \operatorname{up}\left(
 \dfrac{2\bar\gamma\bar\varepsilon}{d_{\mathrm{low}}}
 \right).
\end{aligned}
\]
The derivative bound holds globally, so the a posteriori theorem in the main article yields a control $c^*\in L^2(0,T;\mathbb R^8)$ satisfying
\[
 \mathfrak F_8(c^*)=\mathbf 0,
 \qquad
 \|c^*-c\|_{L^2(0,T;\mathbb R^8)}
 \le\rho_{\mathrm{cert}}.
\]
The table reports this certified control-distance bound.

\subsection{Nested proof-cell refinement}

Each stored candidate is a fixed piecewise-cubic control with \(500\) polynomial intervals. We refine the proof partition to \(1000\), \(2000\), and \(4000\) cells while preserving every original control knot. Each proof cell therefore lies within a single polynomial interval of the stored control.

At every refinement level, we evaluate the same stored control function and recompute the nominal state and Gramian representations, as well as their validated residual bounds. Subdivision restricts the existing control polynomial to smaller intervals. It does not refit the control. Thus, the exact endpoint defect and the exact linearized Gramian associated with that candidate remain fixed, while their computed enclosures may change.

The table below records all four proof partitions, with \(500\), \(1000\), \(2000\), and \(4000\) cells.
\begin{center}
\resizebox{\textwidth}{!}{\begin{tabular}{l r c c c c c}
\toprule
candidate & cells & $\bar\varepsilon$ & $\underline\lambda_{\min}$ & $\bar\gamma$ & $\bar h$ & radius\\
\midrule
least-squares/Newton--HUM & 500 & $4.48\times10^{-2}$ & 0.190145 & 2.2933 & 0.3027 & 0.111927\\
least-squares/Newton--HUM & 1000 & $2.23\times10^{-2}$ & 0.234156 & 2.0666 & 0.1222 & 0.047508\\
least-squares/Newton--HUM & 2000 & $1.11\times10^{-2}$ & 0.255670 & 1.9777 & 0.0559 & 0.022305\\
least-squares/Newton--HUM & 4000 & $5.57\times10^{-3}$ & 0.266340 & 1.9377 & 0.0269 & 0.010868\\
source freezing + Newton--HUM & 500 & $4.49\times10^{-2}$ & 0.189922 & 2.2946 & 0.3041 & 0.112439\\
source freezing + Newton--HUM & 1000 & $2.23\times10^{-2}$ & 0.234072 & 2.0669 & 0.1227 & 0.047689\\
source freezing + Newton--HUM & 2000 & $1.12\times10^{-2}$ & 0.255654 & 1.9778 & 0.0561 & 0.022385\\
source freezing + Newton--HUM & 4000 & $5.59\times10^{-3}$ & 0.266358 & 1.9376 & 0.0270 & 0.010906\\
Neural--Spectral + Newton--HUM & 500 & $2.59\times10^{-2}$ & 0.224576 & 2.1102 & 0.1481 & 0.056815\\
Neural--Spectral + Newton--HUM & 1000 & $1.29\times10^{-2}$ & 0.251687 & 1.9933 & 0.0658 & 0.026136\\
Neural--Spectral + Newton--HUM & 2000 & $6.45\times10^{-3}$ & 0.264844 & 1.9431 & 0.0313 & 0.012638\\
Neural--Spectral + Newton--HUM & 4000 & $3.24\times10^{-3}$ & 0.271362 & 1.9197 & 0.0154 & 0.006253\\
\bottomrule
\end{tabular}
}
\end{center}

Comparing the reported \(500\)- and \(4000\)-cell results, the endpoint-defect upper bounds decrease by factors of approximately eight, while the certificate indicators decrease by factors between \(9.6\) and \(11.3\). The verified Gramian lower bounds also increase.

Because the stored controls remain fixed, these changes measure improvements in the validation bounds rather than changes in the exact controlled trajectories. They demonstrate substantial enclosure tightening over the tested resolutions. Nested refinement alone does not guarantee monotonic improvement or convergence of the bounds to the corresponding exact quantities.

Figure~\ref{supp:fig-indicator-refinement} shows the enclosed
endpoint defects and certificate indicators for the four
proof partitions.
\begin{figure}[htbp]
\centering
\includegraphics[width=.88\textwidth]{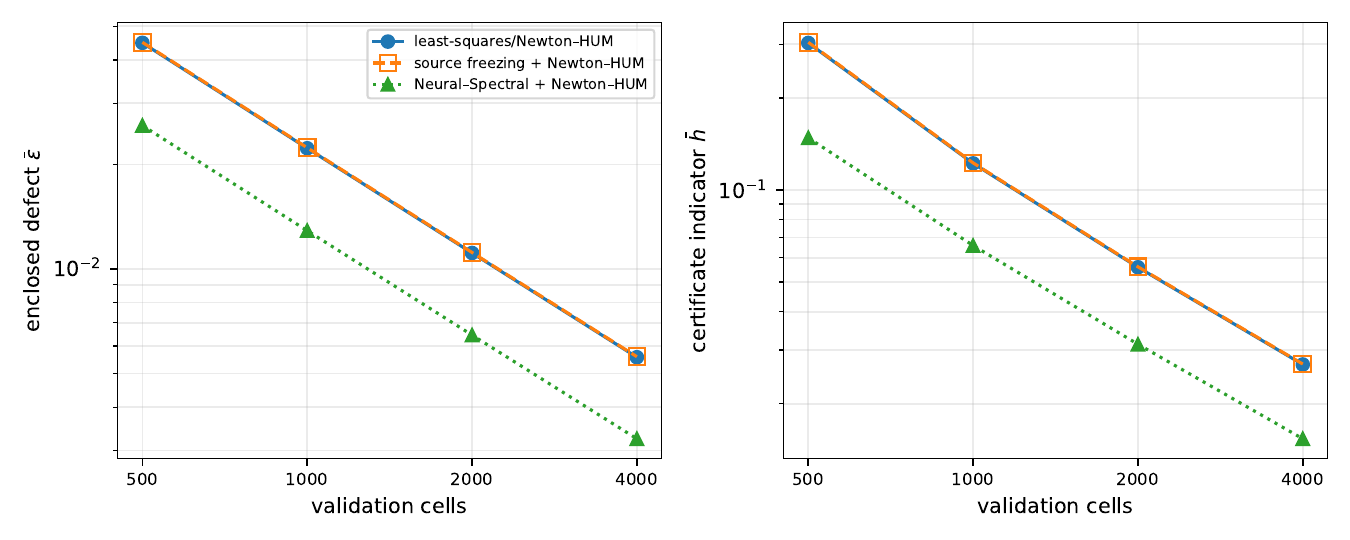}
\caption{Nested proof-cell refinement for the same three stored controls. We display the enclosed endpoint defect in the left panel and the a posteriori indicator in the right panel.}
\label{supp:fig-indicator-refinement}
\end{figure}

\section{Detailed derivation of the endpoint-derivative constant}

This section derives the closed-form bound \(L_P^{\mathrm{end}}\) reported in the main article from the first- and second-variation estimates. We use the energy-scaled state variable \(Y\) for the endpoint map.

Throughout this section, we work on \(\Omega=(0,1)\) with \(\alpha=0\) and the first \(P\) normalized sine modes. We assume that \(g\in C^2(\mathbb{R})\), \(\ell_1:=\norm{g'}_{L^\infty(\mathbb{R})}<\infty\), and \(m_2:=\norm{g''}_{L^\infty(\mathbb{R})}<\infty\).
The projected vector field is globally Lipschitz, so the Galerkin initial-value problem has a solution on \([0,T]\) for every \(c\in L^2(0,T;\mathbb{R}^P)\). Its endpoint map is twice continuously Fr\'echet differentiable. Differentiating the integral form of the Galerkin equation, we obtain the variational equations below.

For \(h\in L^2(0,T;\mathbb{R}^P)\), let \(z_h\) denote the displacement variation at \(c\) in direction \(h\). Its energy-coordinate vector is \(\mathsf E_P(z_h,z_h')^\top\). For each control direction \(h\), we have
\begin{equation}\label{eq:first-var}
 z_h''+\Lambda_Pz_h+DG_P(a_c)z_h+M_Ph=\mathbf{0}_{\mathbb{R}^P},
 \qquad z_h(0)=z_h'(0)=\mathbf{0}_{\mathbb{R}^P}.
\end{equation}
We define its energy by \(r_h(t)^2=\norm{\Lambda_P^{1/2}z_h(t)}_2^2+\norm{z_h'(t)}_2^2\).
Since the smallest selected frequency is \(\pi\),
\begin{equation}\label{eq:poincare-modal}
 \norm{z_h}_2\le\pi^{-1}\norm{\Lambda_P^{1/2}z_h}_2\le\pi^{-1}r_h.
\end{equation}
Writing \(J_Pa=\sum_{j=1}^Pa_je_j\) for the isometry of \(\mathbb{R}^P\) onto the selected modes, we have \(G_P=J_P^{*}\circ g\circ J_P\) and hence \(DG_P(a)=J_P^{*}\,g'(u_a)\,J_P\): the projected Nemytskii derivative is the compression to their span of multiplication by \(g'(u_a)\). Since \(J_P\) is an isometry and multiplication by \(\varphi\) has \(L^2\)-operator norm \(\norm{\varphi}_{L^\infty}\),
\begin{equation*}
 \norm{DG_P(a)}_2\le\norm{g'}_{L^\infty(\mathbb{R})}=\ell_1, \qquad a\in\mathbb{R}^P,\ P\in\mathbb N.
\end{equation*}
Differentiating the squared energy and substituting the first-variation equation \eqref{eq:first-var}, we obtain, for almost every \(t\in(0,T)\),
\begin{align*}
 \dfrac{1}{2}\dfrac{\mathrm d}{\mathrm dt}r_h^2
 &=-\ip{DG_P(a_c)z_h}{z_h'}
   -\ip{M_Ph}{z_h'}\notag\\
 &\le
 \dfrac{\ell_1}{\pi}
 \norm{\Lambda_P^{1/2}z_h}_2\norm{z_h'}_2
 +\norm{M_P}_2\norm{h(t)}_2r_h\notag\\
 &\le
 \dfrac{\ell_1}{2\pi}r_h^2
 +\norm{M_P}_2\norm{h(t)}_2r_h.
\end{align*}
We use \(2ab\le a^2+b^2\) in the last inequality. Consequently, for almost every \(t\in(0,T)\), we have
\begin{equation}\label{eq:first-var-energy-differential}
 r_h'(t)\le\kappa r_h(t)+\norm{M_P}_2\norm{h(t)}_2,
 \qquad \kappa=\dfrac{\ell_1}{2\pi}.
\end{equation}

To justify the integrated estimate without dividing by \(r_h\), we fix \(\varepsilon>0\) and set \(r_{h,\varepsilon}=(r_h^2+\varepsilon^2)^{1/2}\). Using the differential bound \eqref{eq:first-var-energy-differential} and \(r_h\le r_{h,\varepsilon}\), we obtain
\(r_{h,\varepsilon}' \le\kappa r_{h,\varepsilon} +\norm{M_P}_2\norm{h(t)}_2.\)
Since \(r_{h,\varepsilon}(0)=\varepsilon\), we integrate and let \(\varepsilon\to0^+\) to obtain
\(r_h(t)\le \norm{M_P}_2\int_0^t e^{\kappa(t-s)}\norm{h(s)}_2\,\mathrm ds.\)
Using the variation of constants and the Cauchy--Schwarz inequality, we obtain
\begin{equation}\label{eq:first-var-bound}
 r_h(t)
 \le\norm{M_P}_2
 \left(\int_0^t e^{2\kappa(t-s)}\,\mathrm{d}s\right)^{1/2}
 \norm{h}_{L^2(0,t)}
 =\norm{M_P}_2H_\kappa(t)^{1/2}\norm{h}_{L^2(0,t)},
\end{equation}
where
\[
 H_\kappa(t)=
 \begin{cases}
 t,&\kappa=0,\\
 (e^{2\kappa t}-1)/(2\kappa),&\kappa>0
 \end{cases}.
\]

For any zero-initial-data solution \(\zeta\) of \(\zeta''+\Lambda_P\zeta+DG_P(a_c)\zeta=q\), we apply the regularized-energy argument using the differential bound \eqref{eq:first-var-energy-differential} to obtain
\begin{equation}\label{eq:inhomogeneous-energy-gronwall}
 \left\|\mathsf E_P\binom{\zeta(T)}{\zeta'(T)}\right\|_2
 \le\int_0^T e^{\kappa(T-s)}\norm{q(s)}_2\,\mathrm ds.
\end{equation}
For two control directions \(h,k\in L^2(0,T;\mathbb R^P)\), the mixed second variation \(w_{h,k}\) satisfies
\begin{equation}\label{eq:second-var}
 w_{h,k}''+\Lambda_Pw_{h,k}+DG_P(a_c)w_{h,k}
 =-D^2G_P(a_c)[z_h,z_k],
 \qquad w_{h,k}(0)=w_{h,k}'(0)=\mathbf{0}_{\mathbb{R}^P}.
\end{equation}
The modal inverse estimate
\begin{equation}\label{eq:modal-infty}
 \norm{\sum_{j=1}^Pz_je_j}_{L^\infty(0,1)}
 \le C_{\infty,P}\norm{\Lambda_P^{1/2}z}_2,
 \qquad
 C_{\infty,P}=\left(\dfrac{2}{\pi^2}\sum_{j=1}^P\dfrac{1}{j^2}\right)^{1/2},
\end{equation}
follows by the Cauchy--Schwarz inequality after writing \(z_je_j=(j\pi z_j)e_j/(j\pi)\). With \(m_2=\norm{g''}_{L^\infty(\mathbb{R})}\), Bessel's inequality and \eqref{eq:poincare-modal}--\eqref{eq:modal-infty} give
\begin{equation}\label{eq:second-var-source-bound}
\begin{aligned}
 \norm{D^2G_P(a_c)[z_h,z_k]}_2
 &\le m_2\norm{u_{z_h}u_{z_k}}_{L^2}
 \le m_2\norm{u_{z_h}}_{L^\infty}\norm{u_{z_k}}_{L^2}\\
 &\le\dfrac{m_2C_{\infty,P}}\pi r_h r_k.
\end{aligned}
\end{equation}
Applying \eqref{eq:inhomogeneous-energy-gronwall} to \eqref{eq:second-var} with the source bound \eqref{eq:second-var-source-bound}, and substituting \eqref{eq:first-var-bound}, we obtain
\begin{align}
 \left\|\mathsf E_P\binom{w_{h,k}(T)}{w'_{h,k}(T)}\right\|_2
 &\le\int_0^T e^{\kappa(T-s)}
 \dfrac{m_2C_{\infty,P}}\pi r_h(s)r_k(s)\,\mathrm{d}s\notag\\
 &\le \dfrac{m_2C_{\infty,P}\norm{M_P}_2^2}{\pi}
 \left(\int_0^T e^{\kappa(T-s)}H_\kappa(s)\,\mathrm{d}s\right)
 \norm{h}_{L^2}\norm{k}_{L^2}.\label{eq:second-var-bound}
\end{align}
Since the target is independent of the control,
\[
 D^2\mathfrak{F}_P(c)[h,k]
 =\mathsf E_P\binom{w_{h,k}(T)}{w_{h,k}'(T)}.
\]
We set
\[
 \mathcal L_P
 :=\dfrac{m_2C_{\infty,P}\norm{M_P}_2^2}{\pi}
 \int_0^T e^{\kappa(T-s)}H_\kappa(s)\,\mathrm ds.
\]
Thus, using \eqref{eq:second-var-bound}, we obtain the following estimate, uniformly in \(c\):
\begin{equation}\label{eq:second-endpoint-variation-bound}
 \norm{D^2\mathfrak{F}_P(c)[h,k]}_2 \le\mathcal L_P\norm{h}_{L^2(0,T;\mathbb{R}^P)}\norm{k}_{L^2(0,T;\mathbb{R}^P)}.
\end{equation}

For \(c,d\in L^2(0,T;\mathbb{R}^P)\) and a direction \(h\in L^2(0,T;\mathbb{R}^P)\), we apply the fundamental theorem of calculus along the segment \(d+s(c-d)\) to obtain
\begin{equation}\label{eq:endpoint-derivative-ftc}
 \bigl(D\mathfrak{F}_P(c)-D\mathfrak{F}_P(d)\bigr)h
 =\int_0^1D^2\mathfrak{F}_P\bigl(d+s(c-d)\bigr)[c-d,h]\,\mathrm ds.
\end{equation}
Applying \eqref{eq:second-endpoint-variation-bound} to \eqref{eq:endpoint-derivative-ftc} and taking the supremum over \(\norm{h}_{L^2}=1\), we obtain
\begin{equation*}
 \norm{D\mathfrak{F}_P(c)-D\mathfrak{F}_P(d)}_{\mathcal L(L^2(0,T;\mathbb{R}^P),\mathbb{R}^{2P})} \le\mathcal L_P\norm{c-d}_{L^2(0,T;\mathbb{R}^P)}.
\end{equation*}

For \(g(u)=\beta\tanh u\) with \(\beta\ge0\), we have
\(\ell_1=\beta, \ m_2=\dfrac{4\beta}{3\sqrt3}.\)
Indeed, writing \(q=|\tanh u|\in[0,1)\), we obtain \(|g''(u)|=2\beta q(1-q^2)\), whose maximum occurs at \(q=1/\sqrt3\).

For \(\kappa>0\), we integrate directly to obtain
\begin{equation}\label{eq:integral-H-kappa}\begin{aligned}
 \int_0^T e^{\kappa(T-s)}H_\kappa(s)\,\mathrm ds &= \dfrac{1}{2\kappa} \int_0^T \left(e^{\kappa(T+s)}-e^{\kappa(T-s)}\right) \,\mathrm ds = \dfrac{\left(e^{\kappa T}-1\right)^2}{2\kappa^2}\\
 &= \dfrac{e^{\kappa T} \left(\cosh(\kappa T)-1\right)}{\kappa^2}.
\end{aligned}\end{equation}
For \(\kappa=0\), the integral equals \(T^2/2\). Substituting the \(\kappa>0\) identity \eqref{eq:integral-H-kappa} and this limiting value, and using \(\norm{M_P}_2\le1\), we obtain the closed-form bound stated in the main article, with its continuous interpretation at \(\beta=0\). In that case, the endpoint map is affine, and its derivative has a Lipschitz constant of zero. In the validated evaluation, we use \(\norm{M_P}_2\le1\) directly instead of a floating-point spectral norm. This bound holds because multiplication by \(\mathbf1_\omega\) is an orthogonal contraction before projection.

\paragraph{\bf Detailed enclosure arithmetic.}
In Section~\ref{app:residual-details}, we give the complete cellwise Hermite residual bounds, directed spatial-quadrature remainders, matrix-product variation estimates, and separate state and Gramian floating-point guards.

\paragraph{\bf Detailed finite-mode response decomposition.}
In Section~\ref{supp:sec:response-kernel}, we derive the exact kernel and the projection, trace--Ritz, and inverse-residual bounds.

\section{Cellwise residual enclosures and spatial quadrature}\label{app:residual-details}

We give the formulas for the endpoint and variational residual bounds, including the errors that their numerical evaluation must enclose. We fix a proof cell \(I_i=[t_i,t_i+h]\), with \(h>0\), and write \(\tau=t-t_i\). For stored endpoint values \(y_i,y_{i+1}\) and stored derivative values \(d_i,d_{i+1}\), we define the cubic Hermite interpolant by
\[
\widehat y(t_i+\tau)
=b_0+b_1\tau+b_2\tau^2+b_3\tau^3,
\qquad 0\le\tau\le h,
\]
where
\begin{equation}\label{supp:eq:hermite-coefficients}
\begin{aligned}
b_0&=y_i,
&
b_1&=d_i,\\
b_2&=
\dfrac{3(y_{i+1}-y_i)/h-2d_i-d_{i+1}}{h},
&
b_3&=
\dfrac{2(y_i-y_{i+1})/h+d_i+d_{i+1}}{h^2}.
\end{aligned}
\end{equation}
These coefficients satisfy
\[
\widehat y(t_i)=y_i,\qquad
\widehat y(t_i+h)=y_{i+1},\qquad
\widehat y'(t_i)=d_i,\qquad
\widehat y'(t_i+h)=d_{i+1}.
\]
The stored data need not be exact ODE solution values: the residual enclosure measures the failure of the resulting interpolant to satisfy the differential equation. The coefficient formulas \eqref{supp:eq:hermite-coefficients} define exact real-arithmetic quantities, and we enclose their numerical evaluation. We translate these coefficients to the midpoint variable \(\sigma=t-(t_i+h/2)\). We set \(t_m=t_i+h/2\) and \(\sigma=t-t_m\). For a normed space $E$ and a continuous function $v:I_i\to E$, define its maximum deviation from its midpoint value by
\[
 \operatorname{var}_{I_i}(v)
 :=
 \sup_{t\in I_i}\|v(t)-v(t_m)\|_E.
\]
This notation denotes midpoint deviation, not total variation. We use the Euclidean norm for state vectors, the Frobenius norm for Gramian matrices, and the induced Euclidean operator norm for Jacobian matrices.

For a polynomial
\[
 p(t_m+\sigma)=\sum_{k=0}^3p_k\sigma^k,
 \qquad -h/2\le\sigma\le h/2,
\]
we define
\[
 V_i(p):=
 \sum_{k=1}^3\|p_k\|_E
 \left(\dfrac{h}{2}\right)^k.
\]
Since $p(t_m)=p_0$, the triangle inequality gives
\begin{equation}\label{eq:poly-variation}
 \operatorname{var}_{I_i}(p)\le V_i(p).
\end{equation}
A numerical upper bound for \(V_i(p)\) must account for coefficient translation, coefficient uncertainty, powers, products, norms, and summation. All norms and the final positive sum in \eqref{eq:poly-variation} are rounded upward.

For the state enclosure, we use the one-dimensional sine basis with \(\alpha=0\) and \(g(u)=\beta\tanh u\), where \(\beta\ge0\). We write the energy-coordinate state as \(Y=(Y_1,Y_2)^\top\), with \(Y_1=\Lambda_P^{1/2}a\), and define
\[
N(Y)
=
\binom{\mathbf 0}
{-G_P(\Lambda_P^{-1/2}Y_1)}.
\]
Orthonormal projection and the scalar bound \(|g'(u)|\le\beta\) imply
\[
\begin{aligned}
\|N(Y)-N(Z)\|_2
\le
\beta\|\Lambda_P^{-1/2}(Y_1-Z_1)\|_2\le\dfrac{\beta}{\pi}\|Y-Z\|_2.
\end{aligned}
\]

On a proof cell contained in one polynomial interval of the stored control, we define
\[
p_Y(t)
=
\widehat Y'(t)-A_0\widehat Y(t)-B_Ec(t),
\qquad
r_Y(t)=p_Y(t)-N(\widehat Y(t)).
\]
The function \(p_Y\) is a polynomial of degree at most three. We set
\[
Y_m=\widehat Y(t_m),
\qquad
a_m=\Lambda_P^{-1/2}(Y_m)_1.
\]
Let \(\widetilde G_i\in\mathbb R^P\) be the computed approximation to \(G_P(a_m)\), with a validated bound
\[
\|G_P(a_m)-\widetilde G_i\|_2
\le\eta_{G,i},
\qquad
\widetilde N_i=\binom{\mathbf 0}{-\widetilde G_i}.
\]
Then
\[
\begin{aligned}
\sup_{t\in I_i}\|r_Y(t)\|_2
\le
\|p_Y(t_m)-\widetilde N_i\|_2
+V_i(p_Y)+\dfrac{\beta}{\pi}V_i(\widehat Y)
+\eta_{G,i}
=:S_i.
\end{aligned}
\]
Indeed, we add and subtract \(p_Y(t_m)\), \(\widetilde N_i\), and \(N(Y_m)\), and then apply the triangle inequality, the midpoint-deviation bounds, and the Lipschitz estimate for \(N\).

Let $\widehat S_i$ be the computed approximation to $S_i$. Let $\eta_{{\rm fl},Y,i}\ge0$ be a certified upper bound for the positive evaluation error $S_i-\widehat S_i$. This allowance accounts for rounding errors in constructing and evaluating $p_Y$, computing the midpoint-deviation bounds and the Euclidean norm, and summing the terms defining $S_i$. It also includes uncertainty in the represented coefficients and matrices, except for contributions already enclosed by $\eta_{G,i}$. Thus,
\[
 S_i-\widehat S_i\le\eta_{{\rm fl},Y,i},
 \qquad
 S_i\le\widehat S_i+\eta_{{\rm fl},Y,i}.
\]
An upward-rounded evaluation of \(\widehat S_i+\eta_{{\rm fl},Y,i}\) therefore supplies the required bound \(\rho_{Y,i}\), satisfying
\[
\sup_{t\in I_i}\|r_Y(t)\|_2\le\rho_{Y,i}.
\]

The stored control is a fixed piecewise-cubic function. Each proof cell is contained in one of its polynomial intervals, so the control restriction on that cell remains cubic. Refinement preserves the stored control function.

The bound \(\eta_{G,i}\) must include the spatial quadrature remainder and the numerical errors needed to enclose \(G_P(a_m)\), including errors in the reconstructed midpoint displacement and nonlinear function evaluations. The guard \(\eta_{{\rm fl},Y,i}\) covers the remaining polynomial and linear-algebra evaluations in \(S_i\), including the representations of \(A_0\) and \(B_E\).

If independently rounded cell polynomials have nonzero value mismatches at their common endpoints, then the propagation of the error bounds must include validated bounds for these jumps. Cell-interior residual bounds alone do not account for such mismatches.

Let \(n_x\ge1\) be an integer. For \(F\in C^2([0,1])\), we define the composite trapezoidal rule in exact arithmetic by
\[
Q_{n_x}(F)
=
\dfrac{1}{n_x}
\left[
\dfrac{F(0)+F(1)}2
+\sum_{\ell=1}^{n_x-1}
F\left(\dfrac{\ell}{n_x}\right)
\right].
\]
Integrating the linear interpolation remainder on each subinterval and summing, we obtain
\begin{equation}\label{eq:trap-error}
\left|
\int_0^1F(x)\,\mathrm{d}x-Q_{n_x}(F)
\right|
\le
\dfrac{\|F''\|_{L^\infty(0,1)}}{12n_x^2}.
\end{equation}
This estimate bounds the analytical quadrature remainder. Numerical errors in the nodes, function values, weights, and summation require additional enclosures.

For these one-dimensional estimates, take
\[
e_j(x)=\sqrt2\sin(j\pi x),
\qquad
\Lambda_P=\operatorname{diag}\{(j\pi)^2\}_{j=1}^P,
\]
as appropriate for \(\alpha=0\). For \(u=\sum_{j=1}^Pa_je_j\), we set \(q=\Lambda_P^{1/2}a\), so \(q_j=j\pi a_j\). Then,
\[
u_x(x)=\sqrt2\sum_{j=1}^Pq_j\cos(j\pi x),
\qquad
u_{xx}(x)
=-\sqrt2\sum_{j=1}^P(j\pi)q_j\sin(j\pi x).
\]
Applying the Cauchy--Schwarz inequality, we have
\begin{align}
\|u_x\|_{L^\infty(0,1)}
&\le\sqrt{2P}\,\|q\|_2
=:U_1(q),\\
\|u_{xx}\|_{L^\infty(0,1)}
&\le
\sqrt2
\left(\sum_{j=1}^P(j\pi)^2\right)^{1/2}
\|q\|_2
=:U_2(q).
\end{align}
At the midpoint of a proof cell, these estimates use \(q=(Y_m)_1\). When we enclose that vector by intervals, we use an upper bound for its Euclidean norm.

For \(g(u)=\beta\tanh u\), with \(\beta\ge0\), we set
\[
m_0=\|g\|_\infty=\beta,\qquad
m_1=\|g'\|_\infty=\beta,\qquad
m_2=\|g''\|_\infty=\dfrac{4\beta}{3\sqrt3}.
\]
For the source integrand \(F_j(x)=g(u(x))e_j(x)\), the chain and product rules give
\[
F_j''
=
g''(u)u_x^2e_j
+g'(u)u_{xx}e_j
+2g'(u)u_xe_j'
+g(u)e_j''.
\]
Consequently,
\[
\begin{aligned}
\|F_j''\|_\infty
\le
\sqrt2\,m_2U_1(q)^2
+\sqrt2\,m_1U_2(q)+2\sqrt2\,j\pi m_1U_1(q)
+\sqrt2\,(j\pi)^2m_0
=:C_j(q).
\end{aligned}
\]
Let \(G^Q(a)\in\mathbb R^P\) have entries \(G^Q_j(a)=Q_{n_x}(F_j)\). Then
\[
\|G_P(a)-G^Q(a)\|_2
\le
\dfrac{1}{12n_x^2}
\left(\sum_{j=1}^PC_j(q)^2\right)^{1/2}
=:\eta_G^{\mathrm{quad}}(q).
\]

At a proof-cell midpoint, the full bound \(\eta_{G,i}\) must also enclose the difference between \(G^Q(a_m)\) and the computed vector \(\widetilde G_i\). Thus, if
\[
\|G^Q(a_m)-\widetilde G_i\|_2
\le\eta_{G,i}^{\mathrm{eval}},
\]
then we may take an outward upper bound satisfying
\[
\eta_{G,i}
\ge
\eta_G^{\mathrm{quad}}((Y_m)_1)
+\eta_{G,i}^{\mathrm{eval}}.
\]
The evaluation term includes the effects of uncertainty in the midpoint data and numerical evaluation of the quadrature expression.

The Jacobian of the projected nonlinearity has entries
\[
J(a)_{jk}
=
\int_0^1g'(u(x))e_j(x)e_k(x)\,\mathrm{d}x.
\]
For \(H_{jk}(x)=g'(u(x))e_j(x)e_k(x)\), differentiating twice, we obtain
\[
\begin{aligned}
H_{jk}''
=g'''(u)u_x^2e_je_k
+g''(u)u_{xx}e_je_k+2g''(u)u_x(e_j'e_k+e_je_k')+g'(u)(e_j''e_k+2e_j'e_k'+e_je_k'').
\end{aligned}
\]
We set \(m_3=\|g'''\|_\infty=2\beta\). To verify this value, we write \(z=\tanh s\in(-1,1)\) and differentiate to obtain \(g'''(s)=-2\beta(1-z^2)(1-3z^2)\). The polynomial \(p(t)=1-4t+3t^2\) decreases from \(1\) to \(-1/3\) on \([0,2/3]\) and increases to \(0\) on \([2/3,1]\). Hence, \(|p(t)|\le1\) on \([0,1]\), with equality at \(t=0\), and we obtain the stated norm, including \(\beta=0\). Using the sine-basis bounds, we obtain
\[
\begin{aligned}
\|H_{jk}''\|_\infty
\le{}&
2m_3U_1(q)^2+2m_2U_2(q)+4\pi(j+k)m_2U_1(q)
+2\pi^2(j+k)^2m_1
=:C_{jk}^J(q).
\end{aligned}
\]
Let \(J^Q(a)_{jk}=Q_{n_x}(H_{jk})\). Applying the trapezoidal error estimate \eqref{eq:trap-error} entrywise, we obtain
\[
\|J(a)-J^Q(a)\|_F
\le
\dfrac{1}{12n_x^2}
\left(
\sum_{j,k=1}^P
\bigl(C_{jk}^J(q)\bigr)^2
\right)^{1/2}
=:\eta_J^{\mathrm{quad}}(q).
\]
The full midpoint bound \(\eta_{J,i}\) additionally includes numerical evaluation errors, so that the computed matrix \(\widetilde J_i\) satisfies
\begin{equation}\label{eq:jacobian-midpoint-error}
\|J(a_m)-\widetilde J_i\|_2
\le
\|J(a_m)-\widetilde J_i\|_F
\le\eta_{J,i}.
\end{equation}

In energy coordinates,
\[
D(Y):=DN(Y)
=
\begin{pmatrix}
0&0\\
-J(a)\Lambda_P^{-1/2}&0
\end{pmatrix},
\qquad
a=\Lambda_P^{-1/2}Y_1.
\]
Define the exact matrix
\[
D_i^\sharp:=
\begin{pmatrix}
0&0\\
-\widetilde J_i\Lambda_P^{-1/2}&0
\end{pmatrix}.
\]
Then,
\begin{equation}\label{eq:energy-jacobian-midpoint-error}
\left\|D(Y_m)-D_i^\sharp\right\|_2
\le \dfrac{\eta_{J,i}}{\pi}.
\end{equation}
We must also account for numerical errors in constructing the scaled matrix when evaluating the Gramian residual.

We define the polynomial part of the Gramian residual by
\[
p_W
=
\widehat W'-A_0\widehat W-\widehat WA_0^\top-B_EB_E^\top.
\]
We set
\[
D_t=D(\widehat Y(t)),
\qquad
D_m=D(Y_m),
\qquad
W_m=\widehat W(t_m),
\]
where \(Y_m=\widehat Y(t_m)\). For the tanh nonlinearity in the stated energy coordinates,
\[
\|D_t\|_2\le\dfrac{\beta}{\pi}.
\]
The identity
\[
D_t\widehat W(t)-D_mW_m
=
(D_t-D_m)W_m
+D_t\bigl(\widehat W(t)-W_m\bigr)
\]
and its right-product counterpart imply
\begin{equation}\label{eq:product-variation}
\begin{aligned}
\sup_{t\in I_i}
\left\|
D_t\widehat W(t)+\widehat W(t)D_t^\top
-D_mW_m-W_mD_m^\top
\right\|_F\le
2\left[
\operatorname{var}_{I_i}(D\circ\widehat Y)\|W_m\|_F
+\dfrac{\beta}{\pi}
\operatorname{var}_{I_i}(\widehat W)
\right].
\end{aligned}
\end{equation}
Here, the variation of \(D\circ\widehat Y\) uses the operator norm. For energy states \(Y=(q,v)\) and \(Z=(\widetilde q,\widetilde v)\) in \(\mathbb R^{2P}\), we set \(u_Y=\displaystyle\sum_{j=1}^P q_je_j/(j\pi)\) and \(u_Z=\displaystyle\sum_{j=1}^P\widetilde q_je_j/(j\pi)\). Applying the modal inverse estimate \eqref{eq:modal-infty-bound} to the coefficient vector \(\Lambda_P^{-1/2}(q-\widetilde q)\), we obtain \(\|u_Y-u_Z\|_\infty\le C_{\infty,P}\|Y-Z\|_2\). Compressing the multiplication operator \(g'(u_Y)-g'(u_Z)\) to the orthonormal modes, we bound its norm by \(m_2\|u_Y-u_Z\|_\infty\). The energy scaling contributes \(\|\Lambda_P^{-1/2}\|_2=\pi^{-1}\). Combining these bounds, we prove
\begin{equation}\label{eq:energy-jacobian-lipschitz}
\|D(Y)-D(Z)\|_2\le K_D\|Y-Z\|_2,
\qquad
K_D=\dfrac{4\beta C_{\infty,P}}{3\sqrt3\,\pi},
\end{equation}
with
\[
C_{\infty,P}
=
\left(\dfrac{2}{\pi^2}\sum_{j=1}^P\dfrac{1}{j^2}\right)^{1/2},
\]
and therefore, by \eqref{eq:energy-jacobian-lipschitz},
\begin{equation}\label{eq:energy-jacobian-path-variation}
\operatorname{var}_{I_i}(D\circ\widehat Y)
\le K_D V_i(\widehat Y).
\end{equation}

Let \(\widetilde D_i\) be the numerical approximation
to \(D_i^\sharp\). Choose \(\delta_{D,i}\) to include
both the Jacobian approximation error and the numerical
matrix-construction error, so that
\begin{equation}\label{eq:energy-jacobian-computation-error}
\left\|D(Y_m)-\widetilde D_i\right\|_2\le\delta_{D,i}.
\end{equation}
Using the Jacobian error bound \eqref{eq:jacobian-midpoint-error}, we obtain the contribution \(\eta_{J,i}/\pi\) in \eqref{eq:energy-jacobian-midpoint-error}. We must also include the additional numerical errors from constructing the energy-scaled matrix in \(\delta_{D,i}\).

The exact Gramian residual is
\[
r_W(t)
=
p_W(t)-D_t\widehat W(t)-\widehat W(t)D_t^\top.
\]
Using the product-variation estimate \eqref{eq:product-variation}, the Jacobian path-variation estimate \eqref{eq:energy-jacobian-path-variation}, and the midpoint matrix-error bound \eqref{eq:energy-jacobian-computation-error}, we obtain
\begin{equation}\label{eq:gramian-cell-residual-bound}
\begin{aligned}
\sup_{t\in I_i}\|r_W(t)\|_F
\le{}&
\|p_W(t_m)-\widetilde D_iW_m
-W_m\widetilde D_i^\top\|_F
+V_i(p_W)\\
&+2\left[
\left(K_DV_i(\widehat Y)+\delta_{D,i}\right)
\|W_m\|_F
+\dfrac{\beta}{\pi}V_i(\widehat W)
\right]
=:S_{W,i}.
\end{aligned}
\end{equation}
Let $\widehat S_{W,i}$ be the computed approximation to the scalar bound $S_{W,i}$ defined in \eqref{eq:gramian-cell-residual-bound}. Let $\eta_{{\rm fl},W,i}\ge0$ be a certified upper bound for the positive evaluation error $S_{W,i}-\widehat S_{W,i}$, including the remaining arithmetic errors not already enclosed by $\delta_{D,i}$. Thus,
\[
 S_{W,i}-\widehat S_{W,i}
 \le\eta_{{\rm fl},W,i},
 \qquad
 S_{W,i}
 \le\widehat S_{W,i}+\eta_{{\rm fl},W,i}.
\]
We choose $\rho_{W,i}$ as a certified upper evaluation of $\widehat S_{W,i}+\eta_{{\rm fl},W,i}$. Consequently,
\begin{equation*}
\sup_{t\in I_i}\|r_W(t)\|_F\le\rho_{W,i}.
\end{equation*}

We use IEEE~754 binary64 arithmetic with rounding to nearest and unit roundoff $\mathbf u:=2^{-53}$. For an arithmetic expression $E$, the notation $\operatorname{fl}(E)$ denotes the result of its specified floating-point evaluation, including rounding at each arithmetic operation. The stored floating-point inputs are interpreted as exact real numbers; any uncertainty in those inputs is enclosed separately.

For each integer $n\ge1$ satisfying $n\mathbf u<1$, we define
\[
 \gamma^{\rm fl}_n:=\dfrac{n\mathbf u}{1-n\mathbf u}.
\]
For $a,b\in\mathbb R^n$ with binary64-representable entries, suppose that the dot product is evaluated by separately rounded multiplications and sequential summation, with no underflow or overflow. Then,
\[
\left|
\operatorname{fl}(a^\top b)-a^\top b
\right|
\le\gamma^{\rm fl}_n\sum_{j=1}^n|a_jb_j|.
\]
For a real matrix $C$, write $|C|$ for its entrywise absolute value, so that $(|C|)_{jk}:=|C_{jk}|$. Inequalities between matrices in the following estimate are entrywise. Let $p,q\ge1$ be integers, and let $A\in\mathbb R^{p\times n}$ and $B\in\mathbb R^{n\times q}$ have binary64-representable entries. If each entry of the computed product uses the dot-product evaluation specified above, with no underflow or overflow, then
\[
 |\operatorname{fl}(AB)-AB|
 \le\gamma^{\rm fl}_n|A||B|.
\]

Let $S\ge0$ be an exact scalar quantity and $\widehat S$ its computed approximation. Let $\delta_{\rm uf}\ge0$ be a certified absolute error allowance for underflow in this computation. Choose an integer $m\ge1$ such that $m\mathbf u<1$ and $\gamma_m<1$, with $\gamma_m$ defined above. Suppose that an error analysis of the actual evaluation procedure establishes
\[
 |\widehat S-S|
 \le\gamma_m S+\delta_{\rm uf}.
\]
Then,
\[
S\le\dfrac{\widehat S+\delta_{\rm uf}}{1-\gamma_m}.
\]
For a computable upper enclosure, we use a verified upper bound for \(\gamma_m\), round the numerator upward, round the positive denominator downward, and round the final quotient upward.

The validation argument does not rely on the nominal integrator's internal error estimates. We establish its validity using verified residual bounds, initial discrepancies, cell-interface value jumps when present, and their rigorous propagation. A less accurate nominal representation may produce bounds too large to close the certificate.

\section{Exact finite-mode response kernel and certificate decomposition}
\label{supp:sec:response-kernel}

This section supplies the explicit operator and arithmetic decomposition behind Proposition~\ref{prop:validated_chi_certificate}. We fix \(P\), \(T>0\), and the prescribed endpoint states \(y_0,y_T\in\mathbb R^{2P}\). In this section, we use the unscaled state \(y=(a^\top,(a')^\top)^\top\) and write
\[
A=\mathsf A_P
=\begin{pmatrix}0&I_P\\-\Lambda_P&0\end{pmatrix},
\qquad
B=\mathsf B_P
=\begin{pmatrix}0\\-M\end{pmatrix},
\qquad
F=\begin{pmatrix}0\\-I_P\end{pmatrix}.
\]
For \(s\in\mathcal Z_P=L^2(0,T;\mathbb R^P)\), the affine steering problem is
\[
y'=Ay+Bc+Fs,
\qquad
y(0)=y_0,\qquad y(T)=y_T.
\]
The displacement projection is
\[
\pi_1=\begin{pmatrix}I_P&0\end{pmatrix}.
\]
By Lemma~\ref{lem:retained_pair_controllable}, the Gramian
\[
W_T
=
\int_0^T e^{A(T-r)}BB^\top e^{A^\top(T-r)}
\,\mathrm{d}r
\]
is positive definite.

Let \(c^{[s]}\) and \(c^{[0]}\) denote the exact minimum-modal-cost controls for sources \(s\) and zero, respectively, with the same prescribed endpoints. We set
\[
d_0=y_T-e^{AT}y_0,
\qquad
d_s=
\int_0^T e^{A(T-\sigma)}Fs(\sigma)\,\mathrm{d}\sigma.
\]
The endpoint correction required in the source-\(s\) problem is \(d_0-d_s\). Thus,
\[
c^{[s]}(r)
=
B^\top e^{A^\top(T-r)}W_T^{-1}(d_0-d_s).
\]
For the zero source, we use the same formula with \(d_s=\mathbf 0\) to obtain \(c^{[0]}\). Their difference is
\[
c_s(r):=c^{[s]}(r)-c^{[0]}(r)
=
-B^\top e^{A^\top(T-r)}W_T^{-1}d_s.
\]

For each $s\in\mathcal Z_P$, define
\[
 y^{[s]}(t)
 :=
 e^{At}y_0
 +
 \int_0^t e^{A(t-r)}
 \left[Bc^{[s]}(r)+Fs(r)\right]\,\mathrm{d}r,
 \qquad t\in[0,T],
\]
and let $y^{[0]}$ be this expression with the zero source and its corresponding control $c^{[0]}$. Set
\[
 \delta y_s(t):=y^{[s]}(t)-y^{[0]}(t),
 \qquad t\in[0,T].
\]
Since both trajectories satisfy the same prescribed initial and terminal conditions, we obtain
\[
\delta y_s'=A\delta y_s+Bc_s+Fs,
\qquad
\delta y_s(0)=\delta y_s(T)=\mathbf 0.
\]
Applying variation of constants, we obtain
\[
\delta y_s(t)
=
\int_0^t e^{A(t-\sigma)}Fs(\sigma)\,\mathrm{d}\sigma
-C(t)W_T^{-1}d_s,
\]
where
\[
C(t)
=
\int_0^t
e^{A(t-r)}BB^\top e^{A^\top(T-r)}
\,\mathrm{d}r.
\]
By the definition of the source-response operator, \(\mathscr H_{P,T}s=\pi_1\delta y_s\). Substituting the expression for \(d_s\), we obtain
\[
(\mathscr H_{P,T}s)(t)
=
\int_0^T K(t,\sigma)s(\sigma)\,\mathrm{d}\sigma,
\]
where
\begin{equation}\label{supp:eq:K-explicit}
K(t,\sigma)
=
\pi_1\left[
\mathbf 1_{\{\sigma\le t\}}e^{A(t-\sigma)}F
-
C(t)W_T^{-1}e^{A(T-\sigma)}F
\right].
\end{equation}
In particular, \(C(0)=0\) and \(C(T)=W_T\), so the full-state response difference has zero initial and terminal states. The second term represents the control correction that cancels the source-induced endpoint contribution. Estimating the two terms separately by the triangle inequality can lose cancellation and produce a coarser operator-norm bound. The validated computation therefore encloses the combined kernel in \eqref{supp:eq:K-explicit}.

The displacement kernel \(K\) is continuous across the diagonal \(t=\sigma\). Indeed, the possible jump of the direct term is
\[
\pi_1F=\mathbf{0}_{\mathbb R^{P\times P}},
\]
and the endpoint-correction term is smooth. The kernel is smooth on each side of the diagonal, although its first derivatives can have different one-sided limits there.

Consequently, its classical derivatives off the diagonal represent the first weak derivatives of \(K\), with no diagonal Dirac contribution. These derivatives remain bounded on the two closed triangles, so
\[
K\in H^1\bigl((0,T)^2;\mathbb R^{P\times P}\bigr).
\]

Let \(I_i=((i-1)h,ih)\), \(i=1,\ldots,N\), with \(h=T/N\). The orthogonal projection \(\Pi_N:\mathcal Z_P\to\mathcal Z_P\) onto cellwise constant functions is
\[
(\Pi_Nv)(t)
=
\dfrac{1}{h}\int_{I_i}v(r)\,\mathrm{d}r,
\qquad t\in I_i.
\]
For a matrix-valued kernel, let \(\mathcal P_t\) and \(\mathcal P_\sigma\) denote the corresponding cell-averaging projections in the first and second variables. The operator \(\Pi_N\mathscr H_{P,T}\Pi_N\) has kernel \(\mathcal P_t\mathcal P_\sigma K\).

We recall the \(L^2_F\) norm introduced in the proof of Proposition~\ref{prop:validated_chi_certificate}: for a square-integrable matrix-valued function \(L:(0,T)^2\to\mathbb R^{P\times P}\), we write
\[
\|L\|_{L^2_F}^2
=
\int_0^T\int_0^T
\|L(t,\sigma)\|_F^2\,\mathrm{d}\sigma\,\mathrm{d}t.
\]
Since \(K\in H^1((0,T)^2;\mathbb R^{P\times P})\), we apply the cellwise Poincar\'e--Wirtinger inequality to obtain
\[
\|K-\mathcal P_tK\|_{L^2_F}
\le\dfrac{h}{\pi}\|\partial_tK\|_{L^2_F},
\qquad
\|K-\mathcal P_\sigma K\|_{L^2_F}
\le\dfrac{h}{\pi}\|\partial_\sigma K\|_{L^2_F}.
\]
Using
\[
K-\mathcal P_t\mathcal P_\sigma K
=
(K-\mathcal P_tK)
+\mathcal P_t(K-\mathcal P_\sigma K)
\]
and the contractivity of \(\mathcal P_t\), we obtain
\[
\|K-\mathcal P_t\mathcal P_\sigma K\|_{L^2_F}
\le
\dfrac{h}{\pi}
\left(
\|\partial_tK\|_{L^2_F}
+\|\partial_\sigma K\|_{L^2_F}
\right).
\]
The \(L^2_F\)-norm of the kernel bounds the operator norm of an integral operator. Therefore,
\begin{equation}\label{supp:eq:projection-remainder}
\begin{aligned}
\|\mathscr H_{P,T}
-\Pi_N\mathscr H_{P,T}\Pi_N\|_
{\mathcal L(\mathcal Z_P)}\le
\dfrac{h}{\pi}
\left(
\|\partial_tK\|_{L^2_F}
+\|\partial_\sigma K\|_{L^2_F}
\right),
\qquad h=T/N.
\end{aligned}
\end{equation}

For \(P=8\), \(T=2.5\), and the stated actuator geometry, the certificate encloses the derivative norms using a \(300\times300\) box partition of \((0,T)^2\). Boxes intersecting the diagonal include derivative values from both sides. The resulting bounds, rounded upward for display, are
\[
\|\partial_tK\|_{L^2_F}\le2.618750484691,
\qquad
\|\partial_\sigma K\|_{L^2_F}\le2.654926696062.
\]

The operator projection uses a separate partition with \(N=600\) cells, so \(h=2.5/600=1/240\). Using full-precision derivative bounds and outward arithmetic, we obtain
\[
\|\mathscr H_{8,2.5}
-\Pi_{600}\mathscr H_{8,2.5}\Pi_{600}\|_
{\mathcal L(\mathcal Z_8)}
\le0.006994431596564443.
\]
The derivative-box partition certifies the derivative norms, and the projection partition determines the factor \(h/\pi\).

Let \(\boldsymbol e_k\) be the \(k\)-th standard basis vector of \(\mathbb R^P\). An orthonormal basis for the range of \(\Pi_N\) is
\[
\psi_{i,k}(t)
=
h^{-1/2}\mathbf 1_{I_i}(t)\boldsymbol e_k,
\qquad
1\le i\le N,\quad 1\le k\le P.
\]
In this basis, the exact projected matrix \(H_N\in\mathbb R^{NP\times NP}\) has \(P\times P\) blocks
\[
(H_N)_{ij}
=
\dfrac{1}{h}
\int_{I_i}\int_{I_j}
K(t,\sigma)\,\mathrm{d}\sigma\,\mathrm{d}t.
\]
Consequently,
\[
\|H_N\|_2
=
\|\Pi_N\mathscr H_{P,T}\Pi_N\|_
{\mathcal L(\mathcal Z_P)}.
\]
For \(P=8\) and \(N=600\), the matrix dimension is \(NP=4800\).

Let $\widehat H_N$ denote the floating-point matrix
assembled by the implementation, with its stored
entries interpreted as exact real numbers.
Let $\delta_N\ge0$ be the certified upper bound
produced by the assembly verification for the
spectral-norm error $\|H_N-\widehat H_N\|_2$.
For $P=8$ and $N=600$, this verification gives
\begin{equation}\label{supp:eq:projected-matrix-assembly-error}
 \|H_N-\widehat H_N\|_2
 \le\delta_N\le0.001136236168948356.
\end{equation}
\begin{equation}\label{supp:eq:chi-matrix-decomposition}
\chi_{P,T}\le \|\widehat H_N\|_2+\delta_N+ \frac{h}{\pi} \left( \|\partial_tK\|_{L^2_F} +\|\partial_\sigma K\|_{L^2_F} \right).
\end{equation}
The trace--Ritz calculation below bounds \(\|\widehat H_N\|_2\).

Let \(\widehat H_N\in\mathbb R^{n\times n}\), with \(n=NP=4800\), be the represented floating-point matrix introduced above. We define \(\mathcal B=\widehat H_N^\top\widehat H_N\), where this product denotes the exact mathematical product of the represented entries.

We use a randomized range finder with fixed seed \(947231\) to produce a trial matrix \(Q\in\mathbb R^{n\times r}\), with \(r=500\). We use this random construction to select a candidate subspace, then verify the bounds used in the proof. In particular, we do not assume that the computed columns of \(Q\) are exactly orthonormal. We account separately for the error between \(\widehat H_N\) and the exact projected matrix \(H_N\).

\begin{lemma}\label{supp:lem:trace-ritz}
Let \(n\in\{1,2,\ldots\}\), and let \(B\in\mathbb R^{n\times n}\) be symmetric positive semidefinite, with eigenvalues \(\lambda_1\ge\cdots\ge\lambda_n\ge0\). Let \(r\in\{1,\ldots,n\}\), and let \(U\in\mathbb R^{n\times r}\) satisfy \(U^\top U=I_r\). If \(\theta_1\ge\cdots\ge\theta_r\) are the eigenvalues of \(U^\top BU\), then
\[
 \|B\|_2=\lambda_1
 \le \operatorname{tr}(B)-\sum_{j=2}^{r}\theta_j.
\]
Consequently, for a real matrix \(H\in\mathbb R^{m\times n}\) with \(m\in\{1,2,\ldots\}\) and \(B=H^*H\) (so \(H^*=H^\top\)),
\[
 \|H\|_2\le
 \left(\operatorname{tr}(H^*H)-\sum_{j=2}^{r}\theta_j\right)^{1/2}.
\]
\end{lemma}
\begin{proof}
The trace of \(B\) is the sum of its eigenvalues. Thus, we can bound the largest eigenvalue from above by subtracting a lower bound for the sum of the remaining eigenvalues from \(\operatorname{tr}(B)\).

The range of \(U\), defined by
\[
 \operatorname{range}U
 :=\{Uz:z\in\mathbb R^r\}\subseteq\mathbb R^n,
\]
is the subspace spanned by its columns. Since \(U^\top U=I_r\), these columns form an orthonormal basis of this \(r\)-dimensional subspace. Thus, \(U^\top BU\) represents the operator obtained by applying \(B\) to a vector in \(\operatorname{range}U\) and then projecting the result orthogonally onto \(\operatorname{range}U\). For each \(j\in\{1,\ldots,r\}\), we apply the max--min characterization of the eigenvalues of a real symmetric matrix to obtain
\[
 \theta_j
 =\max_{\substack{V\subseteq\operatorname{range}U\\\dim V=j}}
   \min_{\substack{x\in V\\\|x\|_2=1}}x^\top Bx
 \le\max_{\substack{V\subseteq\mathbb R^n\\\dim V=j}}
   \min_{\substack{x\in V\\\|x\|_2=1}}x^\top Bx
 =\lambda_j.
\]
Here, \(V\) ranges over linear subspaces. The inequality holds because the first maximization uses a subset of those in the second. Moreover, since \(B\) is positive semidefinite, we have \(\lambda_j\ge0\) for every \(j\in\{1,\ldots,n\}\). Therefore,
\[
 \sum_{j=2}^{r}\theta_j
 \le\sum_{j=2}^{r}\lambda_j
 \le\sum_{j=2}^{n}\lambda_j.
\]
Here, a sum over an empty index set is understood to be zero. Consequently,
\[
 \begin{aligned}
 \lambda_1
 =\operatorname{tr}(B)-\sum_{j=2}^{n}\lambda_j\le\operatorname{tr}(B)-\sum_{j=2}^{r}\theta_j.
 \end{aligned}
\]
Since \(B\) is symmetric and positive semidefinite, its spectral norm equals its largest eigenvalue. We therefore obtain
\begin{equation}\label{eq:trace-ritz-matrix-bound}
 \|B\|_2
 \le\operatorname{tr}(B)-\sum_{j=2}^{r}\theta_j.
\end{equation}

To obtain the assertion for \(H\), we apply the matrix bound \eqref{eq:trace-ritz-matrix-bound} to \(B=H^*H\). This matrix is positive semidefinite because
\[
 x^*H^*Hx=\|Hx\|_2^2\ge0
\]
for every \(x\in\mathbb R^n\). Its largest eigenvalue is \(\|H\|_2^2\), so
\begin{equation}\label{eq:trace-ritz-squared-bound}
 \|H\|_2^2
 \le\operatorname{tr}(H^*H)-\sum_{j=2}^{r}\theta_j.
\end{equation}
The right-hand side of \eqref{eq:trace-ritz-squared-bound} is nonnegative because it bounds \(\|H\|_2^2\) from above. Taking square roots, we obtain
\[
 \|H\|_2
 \le
 \left(
 \operatorname{tr}(H^*H)-\sum_{j=2}^{r}\theta_j
 \right)^{1/2},
\]
as required.
\end{proof}

For the computed trial matrix \(Q\), we define
\[
G=Q^\top Q,
\qquad
B_Q=Q^\top\mathcal BQ.
\]
For real symmetric matrices \(A\) and \(B\) of the same size, we write \(A\preceq B\), equivalently \(B\succeq A\), when \(B-A\) is positive semidefinite. The verification establishes
\[
0<\underline g I_r\preceq G\preceq\overline g I_r.
\]
Thus, \(Q\) has full column rank. The matrix
\[
U=QG^{-1/2}
\]
has orthonormal columns and spans the same subspace. We use \(U\) only in the analysis. The implementation does not need to construct it.

The exact compression associated with this orthonormal basis is
\[
S=U^\top\mathcal BU
=G^{-1/2}B_QG^{-1/2}.
\]
Let its eigenvalues be \(\theta_1\ge\cdots\ge\theta_r\ge0\). We suppose the computation supplies
\[
\operatorname{tr}(B_Q)\ge\underline t,
\qquad
\lambda_{\max}(B_Q)\le\overline\lambda_Q,
\qquad
\operatorname{tr}(\mathcal B)\le\overline\tau.
\]
Since \(G^{-1}\succeq\overline g^{-1}I_r\) and \(B_Q\succeq0\), the trace of \((G^{-1}-\overline g^{-1}I_r)B_Q\) is nonnegative: it equals the trace of the positive semidefinite matrix \(B_Q^{1/2}(G^{-1}-\overline g^{-1}I_r)B_Q^{1/2}\). Also, \(\|G^{-1/2}\|_2^2\le\underline g^{-1}\). Thus,
\[
\operatorname{tr}(S)
=\operatorname{tr}(G^{-1}B_Q)
\ge\dfrac{\underline t}{\overline g},
\qquad
\theta_1
\le\dfrac{\overline\lambda_Q}{\underline g}.
\]
Consequently,
\[
\sum_{j=2}^r\theta_j
=
\operatorname{tr}(S)-\theta_1
\ge
\max\left\{
0,\,
\dfrac{\underline t}{\overline g}
-\dfrac{\overline\lambda_Q}{\underline g}
\right\}.
\]
The implementation evaluates a lower bound \(\underline s_{\mathrm{tail}}\) for this last expression using directed rounding. We then apply Lemma~\ref{supp:lem:trace-ritz} to obtain
\[
\|\widehat H_N\|_2
\le
\left(
\overline\tau-\underline s_{\mathrm{tail}}
\right)^{1/2},
\]
with the subtraction and square root enclosed upward.

To bound \(\lambda_{\max}(B_Q)\), the implementation uses a numerical eigenvector matrix \(V\) as a change of coordinates. It verifies
\begin{equation}\label{eq:verified-V-gram-lower-bound}
V^\top V\succeq\underline v I_r,
\qquad
\underline v>0,
\end{equation}
and encloses
\begin{equation}\label{eq:verified-transformed-matrix-bound}
\|V^\top B_QV\|_\infty\le m_V.
\end{equation}
Since \(V\) is square and invertible,
\begin{equation}\label{eq:trace-ritz-Q-factorization}
B_Q
=
V^{-\top}(V^\top B_QV)V^{-1}.
\end{equation}
From the lower bound \eqref{eq:verified-V-gram-lower-bound}, we obtain \(\|V^{-1}\|_2^2\le\underline v^{-1}\). The middle matrix in the factorization \eqref{eq:trace-ritz-Q-factorization} is symmetric, so its spectral norm is its spectral radius and is at most its induced maximum-row-sum norm. Combining submultiplicativity with \eqref{eq:verified-transformed-matrix-bound}, we obtain
\begin{equation}\label{eq:certified-lambda-max-bound}
\lambda_{\max}(B_Q)
\le
\dfrac{\|V^\top B_QV\|_2}{\underline v}
\le\dfrac{m_V}{\underline v}.
\end{equation}
The bound \eqref{eq:certified-lambda-max-bound} supplies \(\overline\lambda_Q\) without treating the eigenvalues returned by the numerical eigensolver as certified bounds.

Using the verified trace--Ritz calculation, we obtain
\begin{equation}\label{supp:eq:trace-ritz-computed-bound}
\|\widehat H_N\|_2\le0.152429487261.
\end{equation}
Combining the trace--Ritz bound \eqref{supp:eq:trace-ritz-computed-bound} with the assembly-error bound \eqref{supp:eq:projected-matrix-assembly-error}, we obtain
\[
\begin{aligned}
\|\Pi_N\mathscr H_{P,T}\Pi_N\|_{\mathcal L(\mathcal Z_P)}=\|H_N\|_2&\le0.152429487261+0.001136236169\\
&=0.153565723430.
\end{aligned}
\]
These displayed bounds are rounded upward, and the final enclosure uses the unrounded values.

Let \(X\) be the binary64 approximation to \(W_T^{-1}\), interpreted as a matrix with exactly represented entries. Here, \(\|\cdot\|_\infty\) denotes the induced maximum-row-sum matrix norm. Using an interval enclosure of the exact Gramian \(W_T\), the computation verifies
\[
\|E\|_\infty\le\overline\rho<1,
\qquad
E=I-W_TX.
\]
Since \(\|E^k\|_\infty\le\overline\rho^{\,k}\) for every \(k\in\{0,1,2,\ldots\}\) and \(\overline\rho<1\), we obtain norm convergence of the series \(\sum_{k=0}^{\infty}E^k\). Multiplying its partial sums by \(I-E\), we obtain \(I-E^{m+1}\), which tends to \(I\) as \(m\to\infty\) through the nonnegative integers. Hence, we identify the sum as \((I-E)^{-1}\). Since \(W_TX=I-E\) and the matrices are square, we conclude that \(W_T\) is invertible and
\[
W_T^{-1}=X(I-E)^{-1}.
\]
Therefore,
\[
W_T^{-1}-X
=
X\sum_{k=1}^{\infty}E^k,
\]
and, by submultiplicativity, we obtain
\[
\|W_T^{-1}-X\|_\infty
\le
\dfrac{\|X\|_\infty\,\overline\rho}
{1-\overline\rho}.
\]

Conservative decimal displays of the verified input bounds are
\[
\overline\rho\le6.792207230727\times10^{-14},
\qquad
\|X\|_\infty\le1240.082316165.
\]
These imply
\[
\|W_T^{-1}-X\|_\infty
\le8.422896075\times10^{-11}.
\]
The final enclosure uses the tighter bound computed from the full-precision inputs.

Substituting the trace--Ritz bound \eqref{supp:eq:trace-ritz-computed-bound}, the projected-matrix assembly error \eqref{supp:eq:projected-matrix-assembly-error}, and the projection remainder \eqref{supp:eq:projection-remainder} into \eqref{supp:eq:chi-matrix-decomposition}, we obtain
\begin{equation*}
\begin{aligned}
\chi_{8,2.5}
&\le
0.152429487261
+0.001136236169
+0.006994431597=
0.160560155027
<0.160561.
\end{aligned}
\end{equation*}
We round each displayed contribution upward. We use the full-precision outward computation to obtain the tighter enclosure
\[
\chi_{8,2.5}\le0.16056015502558602.
\]

\bibliographystyle{m2an-local}
\bibliography{references}